%% file: main.tex
\documentclass{article}
\usepackage[inline]{asymptote}

\usepackage[table,dvipsnames]{xcolor}

\input{header}
\usepackage{singer-macros}
\usepackage{enumitem}
\usepackage{extarrows}
\usepackage{tikz}
\usepackage{rotating}
\usepackage{makecell}
\usepackage{subcaption}

\title{Coboundary expansion inside Chevalley coset complex HDXs}
\author{
Ryan O'Donnell\thanks{\texttt{odonnell@cs.cmu.edu}. Carnegie Mellon University.  Supported in part by a grant from Google Quantum AI.}${~}^{\tiny{\textcircled{r}}}$
 \and
Noah G. Singer\thanks{\texttt{ngsinger@cs.cmu.edu}. Carnegie Mellon University.  Supported in part by an NSF Graduate Research Fellowship (Award DGE 2140739). \\
{\tiny \textcircled{r}} The author ordering was randomized.}${~}^{\tiny{\textcircled{r}}}$}

\date{}

\usepackage{tikz-3dplot}

\begin{document}

\maketitle

\begin{abstract}
    Recent major results in property testing~\cite{BLM24,DDL24} and PCPs~\cite{BMV24} were unlocked by moving to high-dimensional expanders (HDXs) constructed from $\widetilde{C}_d$-type buildings, rather than the long-known $\widetilde{A}_d$-type ones.
    At the same time, these building quotient HDXs are not as easy to understand as the more elementary (and more symmetric/explicit) \emph{coset complex} HDXs constructed by Kaufman--Oppenheim~\cite{KO18} (of $A_d$-type) and O'Donnell--Pratt~\cite{OP22} (of $B_d$-, $C_d$-, $D_d$-type).
    Motivated by these considerations, we study the $B_3$-type generalization of a recent work of Kaufman--Oppenheim~\cite{KO21}, which showed that the $A_3$-type coset complex HDXs have good $1$-coboundary expansion in their links, and thus yield $2$-dimensional topological expanders.

    The crux of Kaufman--Oppenheim's proof of $1$-coboundary expansion was: (1)~identifying a group-theoretic result by Biss and Dasgupta~\cite{BD01} on small presentations for the $A_3$-unipotent group over~$\BF_q$; (2)~``lifting'' it to an analogous result for an $A_3$-unipotent group over polynomial extensions~$\BF_q[x]$.

    For our $B_3$-type generalization, the analogue of~(1) appears to not hold.  We manage to circumvent this with a significantly more involved strategy: (1)~getting a computer-assisted proof of vanishing $1$-cohomology of $B_3$-type unipotent groups over~$\BF_5$; (2)~developing significant new ``lifting'' technology to deduce the required quantitative $1$-cohomology results in $B_3$-type unipotent groups over $\BF_{5^k}[x]$.
\end{abstract}

\newpage

\tableofcontents

\newpage

\input{macros}

\usetikzlibrary{positioning,calc,shapes,arrows,backgrounds,fit}
\usetikzlibrary{shapes, positioning}

\tikzset{
node distance=0.5cm and 2cm,
algnode0/.style = {draw, fill=magenta!40, anchor=center, minimum width=1cm, minimum height=.75cm},
algnode1/.style = {draw, fill=yellow!40, anchor=center, minimum width=1cm, minimum height=.75cm},
algnode2/.style = {draw, fill=cyan!40, anchor=center, minimum width=1cm, minimum height=.75cm},
algnode0M/.style = {draw, fill=magenta!40!gray, anchor=center, minimum width=1cm, minimum height=.75cm},
algnode1M/.style = {draw, fill=yellow!40!gray, anchor=center, minimum width=1cm, minimum height=.75cm},
algnode2M/.style = {draw, fill=cyan!40!gray, anchor=center, minimum width=1cm, minimum height=.75cm},
algdummy/.style = {anchor=center, minimum width=1cm, minimum height=.5cm},
algedge/.style = {very thick},
algedge1/.style = {very thick, dotted},
algback/.style = {draw, dashed, rounded corners, inner sep=0.3cm},
algbackh/.style = {draw=none, fill=gray!40, rounded corners, inner sep=0.15cm},
algedgelabel/.style = {midway, draw, solid, thin, circle, fill=white, inner sep=1pt}
}
\allowdisplaybreaks

\input{sections/01-intro}
\input{sections/02-prelim}

\input{sections/03-chevalley}
\input{sections/04-triangulations}
\input{sections/05-lifting}
\input{sections/06-story}

\input{sections/07-A3}

\input{sections/08-B3}

\bibliographystyle{alpha}
\bibliography{HDX}

\appendix
\input{sections/AA-explicit-realizations}

\input{sections/BB-algebra}
\input{sections/CC-biss-dasgupta}

\end{document}

%% file: header.tex
\usepackage[margin=1in]{geometry}
\usepackage[utf8]{inputenc}
\usepackage[T1]{fontenc}
\usepackage{bm}
\usepackage{graphicx}
\usepackage{color,xcolor}
\usepackage{amsmath,amsfonts, amssymb, amsthm, thmtools,mathtools}
\usepackage{algorithm, algpseudocode}
\usepackage[many]{tcolorbox}

\PassOptionsToPackage{hyphens}{url}
\usepackage[colorlinks=true, allcolors=blue]{hyperref}
\usepackage[nameinlink]{cleveref}
\crefname{ineq}{inequality}{inequalities}
\creflabelformat{ineq}{#2{\upshape(#1)}#3}
\crefname{fact}{fact}{facts}
\creflabelformat{fact}{#2{\upshape(#1)}#3}
\crefname{obs}{observation}{observations}
\creflabelformat{obs}{#2{\upshape(#1)}#3}

\crefname{figure}{figure}{figures}
\crefname{relation}{relation}{relations}

\newcommand{\excl}{{(!!)}}
\declaretheorem[numberwithin=section]{theorem}
\declaretheorem[sibling=theorem]{lemma}

\declaretheorem[sibling=theorem]{proposition}
\declaretheorem[sibling=theorem,name={\excl Proposition}]{bigproposition}
\declaretheorem[sibling=theorem]{corollary}

\declaretheorem[sibling=theorem]{relation}
\declaretheorem[sibling=theorem,name={Relation\excl}]{bigrelation}
\theoremstyle{definition}
\declaretheorem[sibling=theorem]{definition}
\declaretheorem[sibling=theorem]{fact}
\declaretheorem[sibling=theorem]{remark}
\declaretheorem[sibling=theorem]{notation}

\declaretheorem[sibling=theorem]{example}

%% file: macros.tex
\newcommand{\CoCo}{\mathcal{C}\mathcal{C}}
\newcommand{\CoCoCo}{\mathcal{C}\mathcal{C}^\circ}
\newcommand{\col}[2]{\binom{#1}{#2}}

\newcommand{\dist}{\mathrm{dist}}

\newcommand{\ACplx}[3]{\mathfrak{A}^{#1}_{#2}(#3)}
\newcommand{\BCplx}[3]{\mathfrak{B}^{#1}_{#2}(#3)} 
\newcommand{\BCplxSkel}[3]{\widehat{\mathfrak{B}}^{#1}_{#2}(#3)} 

\newcommand{\UnipA}[1]{U_{A_3}(\BF_{#1})}
\newcommand{\GrUnipA}[1]{U_{A_3}(\BF_{#1}[x]_{\leq 1})}
\newcommand{\UnipBSm}[1]{U_{B_3}^{\mathrm{sm}}(\BF_{#1})}
\newcommand{\GrUnipBSm}[1]{U_{B_3}^{\mathrm{sm}}(\BF_{#1}[x]_{\leq 1})}
\newcommand{\UnipBLg}[1]{U_{B_3}^{\mathrm{lg}}(\BF_{#1})}
\newcommand{\GrUnipBLg}[1]{U_{B_3}^{\mathrm{lg}}(\BF_{#1}[x]_{\leq 1})}

\newcommand{\LinkCplxA}[1]{\mathfrak{LA}_3(\BF_{#1})}
\newcommand{\GrLinkCplxA}[1]{\mathfrak{LA}_3(\BF_{#1}[x]_{\leq 1})}
\newcommand{\LinkCplxBSm}[1]{\mathfrak{LB}_3^{\mathrm{sm}}(\BF_{#1})}
\newcommand{\GrLinkCplxBSm}[1]{\mathfrak{LB}_3^{\mathrm{sm}}(\BF_{#1}[x]_{\leq 1})}
\newcommand{\LinkCplxBLg}[1]{\mathfrak{LB}_3^{\mathrm{lg}}(\BF_{#1})}
\newcommand{\GrLinkCplxBLg}[1]{\mathfrak{LB}_3^{\mathrm{lg}}(\BF_{#1}[x]_{\leq 1})}

\newcommand{\Comm}[2]{\left[#1, #2\right]}
\newcommand{\Id}{\mathbbm{1}}

\newcommand{\Rt}[1]{\bm{#1}}

\newcommand{\El}[2]{\begin{Bmatrix} \Rt{#1} \\ #2 \end{Bmatrix}}
\newcommand{\Eld}[3]{{\begin{pmatrix} \Rt{#1} \\ #2 \\ #3\end{pmatrix}}}

\newcommand{\ColorR}{red!25}
\newcommand{\ColorBackR}{red!50!black}
\newcommand{\ColorG}{green!25}
\newcommand{\ColorBackG}{green!50!black}
\newcommand{\ColorB}{blue!25}
\newcommand{\ColorBackB}{blue!50!black}
\newcommand{\ColorX}{black}
\newcommand{\ColorNeut}{gray!10}

\newcommand{\BoxR}{{~\colorbox{\ColorR}{RED}~}}
\newcommand{\BoxG}{{~\colorbox{\ColorG}{GREEN}~}}
\newcommand{\BoxB}{{~\colorbox{\ColorB}{BLUE}~}}
\newcommand{\BoxX}{{~\colorbox{\ColorX}{\color{white} BLACK}~}}
\newcommand{\KR}{K_{\text{RED}}}
\newcommand{\KG}{K_{\text{GREEN}}}
\newcommand{\KB}{K_{\text{BLUE}}}
\newcommand{\eR}{e_{\text{RED}}}
\newcommand{\eG}{e_{\text{GREEN}}}
\newcommand{\eB}{e_{\text{BLUE}}}

\newcommand{\CellR}[1]{\cellcolor{\ColorR}{#1}}
\newcommand{\CellG}[1]{\cellcolor{\ColorG}{#1}}
\newcommand{\CellB}[1]{\cellcolor{\ColorB}{#1}}
\newcommand{\CellX}[1]{\cellcolor{\ColorX}{\color{white} #1}}

\newcommand{\Quant}[1]{\forall #1, \quad}
\newcommand{\DQuant}[2]{\forall #1; \forall #2, \quad}

\newcommand{\asite}[1]{{\color{red} \underline{#1}}}

\newcommand{\Cnref}[1]{\Cref{#1} (\nameref*{#1})}

\newcommand{\vu}{\widetilde{u}}

\newcommand{\RelNameLeft}{\texttt{left}}
\newcommand{\RelNameMid}{\texttt{middle}}
\newcommand{\RelNameRight}{\texttt{right}}
\newcommand{\RelNameStraight}{\texttt{straight}}
\newcommand{\RelNameRev}{\texttt{reverse}}
\newcommand{\RelNameA}{\texttt{A3:}\allowbreak}
\newcommand{\RelNameBSm}{\texttt{B3-sm:}\allowbreak}
\newcommand{\RelNameBLg}{\texttt{B3-lg:}\allowbreak}
\newcommand{\RelNameAlg}{\texttt{generic:}\allowbreak}
\newcommand{\RelNamePart}{\texttt{partial:}\allowbreak}
\newcommand{\RelNameSubgp}{\texttt{subgp:}\allowbreak}
\newcommand{\RelNameHomLift}{\texttt{lift-hom:}\allowbreak}
\newcommand{\RelNameHomLiftRaw}{\texttt{raw:lift-hom:}\allowbreak}
\newcommand{\RelNameNonHomLift}{\texttt{lift-non-hom:}\allowbreak}
\newcommand{\RelNameNonHomLiftRaw}{\texttt{raw:lift-non-hom:}\allowbreak}
\newcommand{\RelNameCommutator}[2]{\texttt{commutator}\{$#1$, $#2$\}}
\newcommand{\RelNameCommute}[2]{\texttt{commute}\{$#1$ \& $#2$\}}
\newcommand{\RelNameLinearity}[1]{\texttt{linearity}\{$#1$\}}
\newcommand{\RelNameSelfCommute}[1]{\texttt{self-commute}\{$#1$\}}
\newcommand{\RelNameOrder}[2]{\texttt{reorder}\{$#1$, $#2$\}}
\newcommand{\RelNameExpr}[1]{\texttt{expression}\{$#1$\}}
\newcommand{\RelNameExprn}[3]{\texttt{expression}\{\(#1\) \(\to\) $#2$ \& $#3$\}}
\newcommand{\RelNameId}[1]{\texttt{id}\{$#1$\}}
\newcommand{\RelNameInv}[1]{\texttt{inv}\{$#1$\}}
\newcommand{\RelNameDoub}[1]{\texttt{doub}\{$#1$\}}
\newcommand{\RelNameInvDoub}[1]{\texttt{inv-doub}\{$#1$\}}
\newcommand{\RelNameInterchange}[1]{\texttt{interchange}\{$#1$\}}
\newcommand{\RelSquare}{\allowbreak\texttt{:square}}

%% file: sections/01-intro.tex
\section{Introduction}
The concept of \emph{high-dimensional expansion (HDX)} has played an important role in quite a number of exceptional recent breakthroughs: e.g., in the analysis of Markov Chains~\cite{ALOV18,ALOV18,AOV21}, in coding theory~\cite{DELLM22,PK22b}, in quantum complexity~\cite{ABN23}, and in inapproximability~\cite{BMV24}. Very informally, a ``$k$-dimensional HDX'' is a $k$-uniform hypergraph which is ``well-connected'' in a spectral or topological sense. In the case $k=1$, HDXs recover the classical notions of \emph{expander graphs}. There are, however, two commonly acknowledged difficulties in working with HDXs: \emph{a dearth of constructions}; and, \emph{a glut of inequivalent notions of HDX}.

\paragraph*{Few constructions.} We only seem to know two families of sparse, spectral HDXs: ``quotients of affine buildings'' and ``coset complexes of finite groups of Lie type''; both are algebraic in flavor.
Until recently, both types of constructions were associated only with a single family of groups: the ``$A_d$-type'' Lie groups. Ballantine~\cite{Bal00} first constructed finite quotients of affine buildings in the case~$\widetilde{A}_2$ and showed that their spectrum is \emph{Ramanujan} (i.e., its nontrivial eigenvalues are in the underlying building's spectrum).
For $\widetilde{A}_d$, $d > 2$, Ramanujan constructions are due to~\cite{CSZ03,Li04,Sar04,LSV05a,LSV05b}, with the last reference giving explicit ($\poly(n)$-time) constructions.
As for the second family, Kaufman and Oppenheim~\cite{KO18} showed how to build finite coset complexes that are sparse spectral expanders from the $A_d$-type groups $\SL_{d+1}(\BF)$.
Though not known to have the same level of spectral expansion as the Ramanujan complexes, the \cite{KO18} complexes have some advantages. For example, they have been proven \emph{strongly} explicit in the literature~\cite{OP22}, and are also more symmetric: their automorphism group acts transitively on top-dimensional faces. Also, it is possible to analyze their spectral expansion using completely elementary tools~\cite{KO18,HS19}.

Recently, both constructions have been generalized beyond $A_d$-type.  For affine building quotients, \cite{CL23,BLM24,DDL24} constructed (weakly) explicit, spectral HDXs of type~$\widetilde{C}_d$ (not known to be Ramanujan), allowing them to solve important problems in property testing and inapproximability (see below).
For coset complexes, \cite{OP22,Pra23} generalized the \cite{KO18} construction to all Chevalley groups (including the infinite $B_d$-, $C_d$-, $D_d$-type families), and this was further generalized to Kac--Moody--Steinberg groups in~\cite{GV24}.

\paragraph*{Inequivalent notions.} In the $1$-dimensional (graph) case, there are two main definitions for expansion, edge expansion and spectral expansion, between which Cheeger's inequality gives a bidirectional relation.
By contrast, for higher-dimensional simplicial complexes, there are several inequivalent notions: e.g., topological expansion, spectral expansion, coboundary expansion, cosystolic expansion. 
(Topological expansion is usually considered an end in its own right; spectral expansion is motivated by random walk mixing; and coboundary/cosystolic expansion is most closely connected with property testing applications.)
On top of this, applications of HDX often use not just ``generic HDX properties'', but also additional bespoke properties of the known constructions.
As an example, the aforementioned recent breakthroughs in agreement testing~\cite{BLM24,DDL24} and low-soundness PCPs~\cite{BMV24} needed HDXs with both good spectral expansion and a strong variant of coboundary expansion --- and only the type-$\widetilde{C}_d$ affine building quotients fit the bill.

\subsection{Our main theorem} \label{subsec:main}
Motivated by the evident need for diverse HDXs with diverse properties, we sought to understand the coboundary and cosystolic expansion of the Chevalley coset complex HDXs.  
Recent additional work of Kaufman--Oppenheim~\cite{KO21} made progress on this in the~$A_3$ case; we revisit their work and extend it also to the~$B_3$ case.\footnote{We expect our techniques will also extend to the last $3$-dimensional coset complex family, of type~$C_3$, but have not yet verified this.}
Let us first state our top-line theorem, and then discuss what is perhaps the key contribution in our paper: a new method for proving coboundary expansion.
First, we recall from~\cite{OP22} the HDX in question:
\begin{definition}[The ``$B_3$-type coset complex HDX'']
    Let $\BF_q$ be a field of order~$q$ and odd characteristic.\footnote{\label{foot:2}In \cite{OP22}, the ``base field'' was always assumed to have prime order~$p$.  However this was done for notational simplicity, and it is easy to verify that everything is just the same (after replacing ``$p$'' by ``$q$'') for base fields~$\BF_q$ of characteristic not~$2$.
    Also, though that paper's statements required characteristic at least~$5$, they noted that characteristic~$3$ was also fine for all types other than~``$G_2$''.} 
    Let $\mathfrak{F}$ be an extension of degree \mbox{$m \geq 6$}, so $|\mathfrak{F}| = n \coloneqq q^m$.
    Let $B_3(\mathfrak{F})$
    denote the universal $B_3$-type Chevalley group\footnote{$B_3(\mathfrak{F})$ is also known as $\Omega_7(\mathfrak{F})$, the commutator subgroup of $\mathsf{O}_7(\mathfrak{F})$, the group of $7 \times 7$ orthogonal matrices over~$\mathfrak{F}$.} over~$\mathfrak{F}$, of cardinality $N \coloneqq n^9(n^6-1)(n^4-1)(n^2-1) \sim n^{21}$.
    Let $H_{\Rt{\alpha}}$, $H_{\Rt{\beta}}$, $H_{\Rt{\psi}}$, $H_{\Rt{\omega}} \subseteq B_3(\mathfrak{F})$ be the four subgroups defined in~\cite{OP22} (and in \Cref{def:opcc} in this paper) of cardinality $q^{20}$, $q^{20}$, $q^{31}$, $q^{31}$ (respectively), and let $\BCplx{3}{q}{m}$ be the associated $3$-dimensional coset complex, acted upon by $B_3(\mathfrak{F})$ in a tetrahedron-transitive way.
    This complex is on $V \coloneqq (2q^{-20} + 2q^{-31})N \sim 2q^{21m-20}$ vertices,  with each vertex in $\Theta(q^{31})$ tetrahedra.
\end{definition}
\begin{remark}
    The \emph{smallest} possible expanding instantiation of the above complex ($q=19$, $m=6$) has $V > 2^{450}$.
    (Similar astronomical numbers hold for all known HDX families.)
    This illustrates the importance of \emph{strong} explicitness; we can work with that HDX implicitly, computing incidences in time ``$\poly(450)$''.
\end{remark}
\begin{notation}
    Let $\BCplxSkel{3}{q}{m}$ denote the $2$-skeleton of $\BCplx{3}{q}{m}$ (meaning we simply forget that $\BCplx{3}{q}{m}$ has tetrahedra).
\end{notation}
Thinking of $\BCplxSkel{3}{q}{m}$ as a (connected) ``graph with triangles'' on $V$ vertices, each vertex has degree at most~$q^{15}$. \cite{OP22} shows that each vertex in $\BCplxSkel{3}{q}{m}$ has a link~$L$ which is isomorphic to one of two certain graphs, $L_q^{\text{sm}}$ (on $q^4$ vertices) or $L_q^{\text{lg}}$ (on $q^{15}$ vertices), each having a random walk matrix with second-largest eigenvalue $\lambda_2 \leq \frac{1}{\sqrt{q/2}-1} < 1/2$ (provided $q > 18$). Thus $(\BCplxSkel{3}{q}{m})_m$ is a strongly explicit family of bounded-degree $2$-dimensional spectral HDXs.

\begin{theorem} \label{thm:main} (Main.)
    If $q$ is a sufficiently large power of~$5$, the $2$-dimensional HDXs $\BCplxSkel{3}{q}{m}$ are $(\eps_0,\mu_0)$-cosystolic expanders, where $\eps_0, \mu_0 > 0$ are universal constants.
\end{theorem}
\begin{corollary} \label{cor:main}
    The $2$-dimensional HDXs in \Cref{thm:main} are topological expanders: if the graph-with-triangles $\BCplxSkel{3}{q}{m}$ is (continuously) drawn in the plane, no matter how curvily, there will be a point in~$\BR^2$ contained in a constant fraction of its triangles.
\end{corollary}

Kaufman and Oppenheim~\cite{KO21} proved similar theorems for the family of ``$A_3$-type coset complexes over $\BF_q$'' they had earlier constructed~\cite{KO18}, and which we denote $(\ACplx{3}{q}{m})_{m \geq 4}$. (Their theorems held for $q$ being any sufficiently large prime.)

\subsection{Showing coboundary expansion}\label{sec:show}

Besides providing new examples of cosystolic and topological expanders with additional properties --- strong explicitness, strong symmetry --- our work provides new methods for showing \emph{coboundary} expansion.
This is relevant because establishing \Cref{thm:main} --- the cosystolic expansion of the complex $\BCplx{3}{q}{m}$ --- reduces to showing that each vertex-link in~$\BCplx{3}{q}{m}$ is has coboundary expansion at least $\beta \geq \Omega(1)$.
(This reduction follow from known ``local-to-global'' techniques \cite{EK16,DD23} and known spectral bounds \cite{OP22}.)
We caution that this reduction only works provided $q$ is sufficiently large as a function of~$\beta$.

Kaufman--Oppenheim~\cite{KO21} followed this path for their coset complex families $(\ACplx{3}{q}{m})_{m \geq 4}$.
All vertex-links in their family are isomorphic; this common link is a $2$-dimensional coset complex we'll denote by $\GrLinkCplxA{q}$. (The reason for this notation is given in \Cref{sec:unip}.)
In our case, for the $B_3$-type coset complex families $(\BCplx{3}{q}{m})_{m \geq 6}$, it was shown in~\cite{OP22} that there are \emph{two} different vertex-links (up to isomorphism). 
We call them the ``small link'' $\GrLinkCplxBSm{q}$ and the ``large link'' $\GrLinkCplxBLg{q}$. 
Thus to establish our main \Cref{thm:main}, the key task is to establish coboundary expansion of the $2$-dimensional coset complexes $\GrLinkCplxBSm{q}$ and $\GrLinkCplxBLg{q}$ (for sufficiently large~$q$).

Here we are in the setting of coset complexes of bounded diameter, just like in most prior work on coboundary expansion in sparse HDXs (even the affine building quotient HDXs).
In $2$-dimensional such settings, it is often the case that to show the ``quantitative'' property of $1$-coboundary expansion, it suffices to show the (formally weaker) ``qualitative'' result of vanishing $1$-(co)homology.
Now, vanishing $1$-homology (over $\BZ$ and therefore over $\BF_2$) is implied by the stronger \emph{homotopy} property of ``simply connectivity''; i.e., having ``no holes'' (like a triangulated sphere, and not like a triangulated torus).  
In turn, in a $2$-dimensional coset complex defined by subgroups $H_0, H_1, H_2$ and ambient group $G = \langle H_0,H_1,H_2\rangle$, the homotopical property of simple connectedness has been shown~\cite{Lan50,Bun52,Gar79,AH93} to be equivalent to a group-theoretic condition: \label{sec:diamond}
\begin{quote}
$(\diamond)$ \quad For every word~$w$ with symbols from $H_0, H_1, H_2$ that is equal to~$\Id$ in~$G$, there is a \emph{derivation} of this fact in which all the steps use relations that hold inside $H_0$, $H_1$, or $H_2$.\footnote{This is a derivation in the presentation-theoretic sense: Start with a word, then repeatedly identify a subsequence where all the elements belong to the same subgroup and apply an in-subgroup relation to this subsequence. In the language of \cite{AH93}, the existence of such derivations for every word in $G$ means that $G$ is isomorphic to the free product of $H_0,H_1,H_2$ amalgamated along their intersections.}
\end{quote}
And indeed, it is ultimately by using the group-theoretic condition~$(\diamond)$ that Kaufman--Oppenheim~\cite{KO21} established coboundary expansion for the links $\GrLinkCplxA{q}$ of their $A_3$-type coset complexes $(\ACplx{3}{q}{m})_{m \geq 4}$.
Proofs of coboundary expansion in even simpler situations (e.g.,~\cite{LMM16}) can also be established in this way.

However, for the links $\GrLinkCplxBSm{q}$ and $\GrLinkCplxBLg{q}$ of~$B_3$ coset complexes, our group-theoretic explorations strongly suggested that it is \emph{false} that they are simply connected;\footnote{We hope to rigorously prove this conjecture in the future. For now, we discuss our evidence at the end of \Cref{sec:bd}.} and it is provably false for $q=2,3,5$!\footnote{This is by a computer calculation that we describe in \Cref{sec:story:compute}.}
Thus, we need a new method.

\subsection{Our new techniques, part I: Overcoming not being simply connected}
To push through and show coboundary expansion, we evidently need to hope that the links $\GrLinkCplxBSm{q}$ and $\GrLinkCplxBLg{q}$ have vanishing $1$-cohomology over $\BF_2$, \emph{despite} them not being simply connected.
To gain any sort of confidence that this hope might be realized, one might consider doing an explicit computer calculation of the cohomology.  However this runs into an immediate difficulty: even the ``small'' link $\GrLinkCplxBSm{q}$ has $q^{20}$ triangles.  This would be feasible to study by computer only for $q = 2$ (a case that is anyway irrelevant, as we always need odd~$q$).

There is still one possible angle for computer-assisted attack, which is to look at certain smaller ``degenerate'' versions of the link complexes.
We notate these by $\LinkCplxBSm{q}$, $\LinkCplxBLg{q}$; they formally correspond to the vertex-links in the  ``$m = 1$'' case of $(\BCplx{3}{q}{m})_{m \geq 6}$.
This ``$m = 1$'' case does not actually show up in our infinite HDX families, but it still may serve as a simpler ``warmup'' for understanding $\GrLinkCplxBSm{q}$, $\GrLinkCplxBLg{q}$.  For example, $\LinkCplxBSm{q}$ has only $q^7$ triangles, and $\LinkCplxBLg{q}$ has~$q^9$.
Thus for very small~$q$, we might be able to answer (via computer calculation) whether $\LinkCplxBSm{q}$ and $\LinkCplxBLg{q}$ have vanishing $1$-cohomology over~$\BF_2$\dots

\paragraph*{Interlude: degenerate links in the $A_3$ case.} Kaufman and Oppenheim~\cite{KO21} also heavily studied the degenerate version $\LinkCplxA{q}$ of their link $\GrLinkCplxA{q}$. 
The ambient group for this coset complex $\LinkCplxA{q}$ is in fact the ``unipotent group'' for $A_3(\BF_q) \cong \SL_{4}(q)$. 
This unipotent group $\mathrm{Unip}_4(\BZ/p\BZ)$ is the well known group of $4 \times 4$ upper-unitriangular matrices over~$\BZ/p\BZ$. 
Now, whether or not the group-theoretic condition~$(\diamond)$ holds for the unipotent group $\mathrm{Unip}_4(\BZ/p\BZ)$ was the main topic in the PhD thesis of Kirchmeyer~\cite{Kir78} (he also investigated the $C_2$-analogue). Kirchmeyer showed that~$(\diamond)$ fails for $p = 2$, and (with computer assistance\footnote{SNOBOL code running on a CDC 6600.}) that it holds for $p = 3$. 
Using code of Evens~\cite{Eve78}, he also verified it for $p = 5,7,11,13$. Much later, Biss and Dasgupta~\cite{BD01}  gave a human-readable proof of~$(\diamond)$ for all odd~$p$.\footnote{We give new, possibly easier-to-interpret proof of the Biss-Dasgupta result in \Cref{sec:bd}.}  From this, Kaufman and Oppenheim immediately got that their degenerate link $\LinkCplxA{q}$ \emph{is} simply connected.  Their main work then~\cite[Sec.~7]{KO21} was to ``lift'' this group-theoretic result to show simple-connectedness of the actual vertex-link for their complexes, $\GrLinkCplxA{q}$.

\paragraph*{Back to the $B_3$ case.}
Returning to the $B_3$ case, we used computer assistance (see \Cref{sec:story:compute}) to show that for $p = 2,3$, 
$\LinkCplxBSm{p}$ and $\LinkCplxBLg{p}$ \emph{have} nontrivial $1$-homology (over~$\BF_2$); in particular,~$(\diamond)$ fails for the unipotent groups of~$B_3(\BF_2)$ and $B_3(\BF_3)$.
But happily, computer calculations showed that the $1$-homology (over $\BF_2$) of $\LinkCplxBSm{p}$ and $\LinkCplxBLg{p}$ 
\emph{does} vanish for~$p = 5$. These calculations were not completely trivial, as $5^9$ is rather large; we had to compute the $\BF_2$-rank of a $3$-sparse, $1{\small,}953{\small,}125 \times 1{\small,}171{\small,}875$ matrix.
But once they finished, we became hopeful that the $1$-cohomology in the larger, actual links
$\GrLinkCplxBSm{q}$, $\GrLinkCplxBLg{q}$ might also vanish.
But establishing this would require significantly new ``lifting'' technology, for three reasons:
\begin{itemize}
    \item It's \emph{not} be sufficient to study the $q = 5$ case of $\GrLinkCplxBSm{5}$, $\GrLinkCplxBLg{5}$, since the reduction from $\beta$-coboundary expansion in the links to global cosystolic expansion of $\BCplx{3}{q}{m}$ only works if~$q$ is sufficiently large as a function of~$\beta$.  Thus we need to ``lift'' the $\BF_5$ computer calculation to general fields $\BF_q$ with $q = 5^k$.
    \item The lifting results of Kaufman--Oppenheim~\cite{KO21} were purely group-theoretic, lifting condition~$(\diamond)$ for the ambient group of $\LinkCplxA{q}$ to condition~$(\diamond)$ for the ambient group of $\GrLinkCplxA{q}$.
    But condition~$(\diamond)$ does \emph{not} hold in our degenerate $B_3$-type links; only vanishing $\BF_2$-cohomology.  So we need to augment the group-theoretic lifting methods in a way that takes into account computer-discovered $\BF_2$-triangulations.
    \item Finally, even for the group-theoretic lifting, the additional complexity of $B_3$ vs.\ $A_3$ means we require new and more systematic techniques, beyond those used in~\cite{KO21}.
\end{itemize}

The next section gives an overview of how we implement all of this lifting technology and thereby prove the following:
\begin{theorem}\label{thm:keytech} (Key technical theorem.)
    If $q = 5^k$, $k \in \BN^+$, then $\GrLinkCplxBSm{q}$ and $\GrLinkCplxBLg{q}$ (the links of $\BCplx{3}{q}{m}$ for $m \geq 6$) are $2$-dimensional $\delta_0$-coboundary expanders, where $\delta_0 > 0$ is a universal constant.
\end{theorem}
With this in hand, we can use the reduction described in \Cref{sec:show} to obtain our main \Cref{thm:main} on cosystolic expansion of the $\BCplx{3}{5^k}{m}$
HDXs. 

\begin{remark}
    The key technical theorem of~\cite{KO21} for the $A_3$ case is the same as \Cref{thm:keytech}, except it's for the links $\GrLinkCplxA{p}$ of the HDX family $\ACplx{3}{p}{m}$ with $p$ a sufficiently large prime.\footnote{In \cite{KO21} they state that $p$ may be a prime power, but this is not completely clear because the Biss-Dasgupta result holds for the ring~$\BZ/q\BZ$, but is different for the ring~$\BF_q$.}
    One interesting byproduct of our methods is that in the $A_3$ case, we can show cosystolic expansion of $\ACplx{3}{q}{m}$ for $q$ a (sufficiently large) power of $2,5,7,11,13$, using the  results of~\cite{Kir78,Eve78} in place of~\cite{BD01}.
\end{remark}

\subsection{Our new techniques, part II:  Lifting theorems}

\newcommand{\PiBLg}{\Pi_{B_3}^{\mathrm{lg}}}

For both the ``small link'' and ``large link'' cases, the ambient group is a kind of \emph{unipotent group} defined by a set of so-called \emph{Steinberg relations}.
These relations provide a connection between algebraic interactions of elements in the unipotent group and geometric interactions of a set of vectors \emph{roots}.

\subsubsection{Defining the group $\GrUnipBLg{q}$}

Let us focus on the more complex case we study, the ``large link'' of $B_3$. Without yet getting into full details of root systems (see \Cref{sec:prelim:root-systems}), there are three linearly independent \emph{base root} vectors $\Rt{\alpha},\Rt{\beta},\Rt{\psi}$  and $9$ total \emph{root} vectors
\[
\PiBLg \coloneqq \{\Rt{\alpha}, \Rt{\beta}, \Rt{\psi}, \Rt{\alpha+\beta}, \Rt{\beta+\psi}, \Rt{\beta+2\psi}, \Rt{\alpha+\beta+\psi},\Rt{\alpha+\beta+2\psi},\Rt{\alpha+2\beta+2\psi} \}.
\]
The \emph{height} $\height(\Rt{\zeta})$ of a root $\Rt{\zeta} \in \PiBLg$ is simply the sum of its coefficients on $\Rt{\alpha},\Rt{\beta},\Rt{\psi}$   (So, the heights of roots range between $0$ and $5$.)

We adopt the nonstandard convention $[n] = \{0,\ldots,n\}$, and we adopt the convention where $\Comm{x}{y} \coloneqq xyx^{-1}y^{-1}$ is the ``commutator'' of two group elements.

The group $\GrUnipBLg{q}$, which is the ambient group for $\GrLinkCplxBLg{q}$, is defined as follows.

\begin{definition}\label{def:intro:unip-b3-large}(Informal; see \Cref{sec:unip} for the more formal definition.)
    $\GrUnipBLg{q}$ is generated by symbols of the form $\Eld{\zeta}{t}{i}$ where the root $\Rt{\zeta} \in \PiBLg$ is a \emph{type}, $t \in \BF_{q}$ is a \emph{coefficient}, and $i \in [\height(\Rt{\zeta})]$ is a \emph{degree}. These symbols satisfy the ``linearity'' relations
\[
\DQuant{i \in [\height(\Rt{\zeta})]}{t,u \in \BF_q} \Eld{\zeta}{t}{i} \Eld{\zeta}{u}{i} = \Eld{\zeta}{t+u}{i}
\]
and the ``commutation'' relations
\[
\DQuant{i \in [\height(\Rt{\zeta})], j \in [\height(\Rt{\eta})]}{t,u \in \BF_q} \Comm{\Eld{\zeta}{t}{i}}{\Eld{\eta}{u}{j}} = \prod_{\substack{a \Rt{\zeta} + b \Rt{\eta} \in \PiBLg \\ a,b \in \BN^+}} \Eld{a\zeta+b\eta}{C^{\Rt{\zeta},\Rt{\eta}}_{a,b} \cdot t^a u^b}{ai+bj},
\]
where $C^{\Rt{\zeta},\Rt{\eta}}_{a,b} \in \{\pm1,\pm2\}$ are some constants not depending on $t$, $u$, $i$, or $j$.
\end{definition}
It is known that all elements in this group can be written uniquely in the form \[
\prod_{\Rt{\zeta} \in \PiBLg} \prod_{i=0}^{\height(\Rt{\zeta})} \Eld{\zeta}{t_{\Rt{\zeta},i}}{i}
\]
for some fixed ordering of the roots $\PiBLg$ (see \Cref{thm:structure3}). Thus, the group has $q^{3\cdot2+2\cdot3+2\cdot4+1\cdot5+1\cdot6} = q^{31}$ elements.

Now consider all the commutator and linearity relations involved in the presentation of the group $\GrUnipBLg{q}$. There are $9$ total roots, and hence $\binom{9}2 + 9 = 45$ pairs of (not necessarily distinct) roots, each giving a commutator relation. Further, for each pair of roots $\Rt{\zeta}$ and $\Rt{\eta}$, there are $(\height(\Rt{\zeta})+1)(\height(\Rt{\eta})+1)$ pairs of degrees $(i,j)$ to which the relation needs to apply. Our goal is to show that all of these commutator relations given 4-tuples $(\Rt{\zeta},\Rt{\eta},i,j)$ for $\Rt{\zeta},\Rt{\eta} \in \PiBLg$, $i \in [\height(\Rt{\zeta})]$, $j \in [\height(\Rt{\eta})]$, as well as the linearity relations given 2-tuples $(\Rt{\zeta},i)$ for $\Rt{\zeta} \in \PiBLg$ and $i \in [\height(\Rt{\zeta})]$, can be proven from \emph{in-subgroup Steinberg relations} and \emph{lifted relations of constant length}. We next turn to describing these two kinds of relations.

\subsubsection{The in-subgroup relations}

Recall that the coset complex $\GrLinkCplxBLg{q}$ is defined both by this ambient group $\GrUnipBLg{q}$ and by three subgroups. We now specify these subgroups. These subgroups are written $H_{\setminus \Rt{\psi}}, H_{\setminus \Rt{\beta}}, H_{\setminus \Rt{\alpha}}$, and they are the subgroups generated by $\Rt{\alpha}$ and $\Rt{\beta}$-type elements, by $\Rt{\alpha}$ and $\Rt{\psi}$-type elements, and by $\Rt{\beta}$ and $\Rt{\psi}$-type elements, respectively. For instance, $H_{\setminus \Rt{\psi}}$ consists of elements of the form
\[
\prod_{\Rt{\zeta} \in \{\Rt{\alpha},\Rt{\beta},\Rt{\alpha+\beta} \}} \prod_{i=0}^{\height(\Rt{\zeta})} \Eld{\zeta}{t_{\Rt{\zeta},i}}{i}
\]
and therefore has size $q^{2\cdot2+1\cdot3}=q^7$.

However, the number of pairs of roots which correspond only to elements in $H_{\setminus \Rt{\psi}}$, $H_{\setminus \Rt{\beta}}$, or $H_{\setminus \Rt{\alpha}}$ are $\binom{3}2+3 = 6$, $\binom{2}2+2 = 4$, and $\binom{4}{2}+4 = 10$, respectively, for a total of $17$ pairs.\footnote{We subtract $3$ to avoid double-counting the pair $\{\Rt{\zeta},\Rt{\zeta}\}$ for a base root $\Rt{\zeta}$, since each such relation is in two subgroups.} Thus $28$ root-pairs are ``missing'' from the in-subgroup relations.

\subsubsection{The lifted relations}

Next, we define the ``lifted'' relations. There is a smaller group, $\UnipBLg{p}$, which can be formed formed by taking the definition of $\GrUnipBLg{p}$ (\Cref{def:intro:unip-b3-large}) and discarding the degree feature.\footnote{Formally, one can equate $\Eld{\zeta}ti = \Eld{\zeta}tj$ for all $i,j \in [\height(\Rt{\zeta})]$.} The elements of this group are generated by symbols $\El{\zeta}t$ for $t \in \BF_p$ and obey syntactically identical relations to \Cref{def:intro:unip-b3-large} except that the degree is deleted. Further, by picking three parameters $(i,j,k) \in [1]$ and $t,u,v \in \BF_{p^k}$, one gets a \emph{homogeneous homomorphism}
\begin{align*}
    f : \UnipBLg{p} &\to \GrUnipBLg{p^k}, \\
    \El{\zeta}{w} &\mapsto \Eld{\zeta}{w t^a u^b v^c}{ai+bj+ck}
\end{align*}
for every root $\Rt{\zeta} = a\Rt{\alpha} + b\Rt{\beta} + c\Rt{\psi}$. (It is also necessary at times to consider \emph{nonhomogeneous homomorphisms} where an element of type $\Rt{\zeta}$ in $\UnipBLg{p}$ gets mapped to a product of $\Rt{\zeta}$ elements in $\UnipBLg{p^k}$ of differing degrees.) The ``lifted'' relations in $\GrUnipBLg{p^k}$ are then the images of constant-length relations under this map $f$. In particular, the image of a Steinberg relation will always be a Steinberg relation (even for any pair of roots!) but we can only achieve $2^3 = 8$ possible degree pairs.

\subsubsection{The game plan}

The lifted commutation relations complement the in-subgroup commutation relations in a sense: The lifted relations work for every pair of roots, but only $8$ degrees, while the in-subgroup relations work only for some pairs of roots, but for all possible degrees. The main thrust of our analysis of $B_3$ is to combine these two kinds of relations to gradually build up to all possible root-degree-pairs. We aim to do this by proving implications between various root-degree pairs, which eventually can be used to generate all possible root-degree pairs.

There is one major difficulty that we have elided so far. When we study whether the group $\GrUnipBLg{\BF_{p^k}}$ is defined solely by the in-subgroup relations and the lifted relations, the \emph{generators} we are allowed to consider must themselves be restricted to elements lying in the subgroups $H_{\setminus \Rt{\alpha}}, H_{\setminus \Rt{\beta}}, H_{\setminus \Rt{\psi}}$. So, in order to even properly \emph{state} relations involving a ``missing'' root $\Rt{\zeta}$ --- i.e., a $\Rt{\zeta}$ that has nonzero coefficients for all three base roots --- we must first prove an \emph{establishing} relation for $\Rt{\zeta}$ elements, stating that the various possible ways to ``name'' a $\Rt{\zeta}$-element coincide. Otherwise, we can only lift relations using specific ``names'' for $\Rt{\zeta}$-elements.

Another difficulty is caused by the relatively brittle structure of the lifting homomorphisms. The lift of a linearity relation $\El{\zeta}y \El{\zeta}z = \El{\zeta}{y+z}$ in $\UnipBLg{p}$ to $\GrUnipBLg{p^k}$ is of the form \[ \Eld{\zeta}{t^a u^b v^c y}{ai+bj+ck} \Eld{\zeta}{t^a u^b v^c z}{ai+bj+ck} = \Eld{\zeta}{t^a u^b v^c (y+z)}{ai+bj+ck} \] (if $\Rt{\zeta} = a\Rt{\alpha} + b\Rt{\beta} + c\Rt{\psi}$, for some $t,u,v \in \BF_{p^k}$, $i,j,k \in [1]$). The ratio of the entries in the left-hand side is an element of $\BF_p \subseteq \BF_{p^k}$ and thus lifting cannot give all possible pairs of entries. Similarly, the lift of a commutation relation $\Comm{\El{\zeta}y}{\El{\zeta}z} = \Id$ is \[ \Comm{\Eld{\zeta}{t^a u^b v^c y}{ai+bj+ck}}{\Eld{\zeta}{t^a u^b v^c z}{ai+bj+ck}} = \Id, \] and therefore we can never get a relation stating that two $\Rt{\zeta}$-elements with \emph{different} degrees commute directly from lifting.

The general outline for our proof for $B_3$'s large link is as follows:

\begin{itemize}
    \item As a starting point, prove relations not involving a missing root ($\Rt{\alpha+\beta+\psi}$, $\Rt{\alpha+\beta+2\psi}$, or $\Rt{\alpha+2\beta+2\psi}$) where possible. (There is only one such relation: The relation that $\Rt{\alpha+\beta}$ and $\Rt{\beta+\psi}$ commute.)
    \item Take the lowest-height missing root --- in this case, $\Rt{\alpha+\beta+\psi}$. We have $\Rt{\alpha+\beta+\psi} = (\Rt{\alpha+\beta})+\Rt{\psi} = \Rt{\alpha}+(\Rt{\beta+\psi})$. This means that there are different ways to ``name'' $\Rt{\alpha+\beta+\psi}$ elements of a given degree $i \in [4]$. Prove that all these different ways one ``should'' be able to write a $\Rt{\alpha+\beta+\psi}$ element of a given degree are equivalent;\footnote{We mean this in a quite literal sense: We draw graphs where the vertices are all the different ways one should be able to write a $\Rt{\alpha+\beta+\psi}$ element of a given degree and the edges are equalities, and show that these graphs are connected.} therefore, we can ``establish'' a single name for all such equivalent forms. Finally, state and prove as many relations as possible about $\Rt{\alpha+\beta+\psi}$ (which do not involve other missing roots).
    \item Repeat the above for the missing root $\Rt{\alpha+\beta+2\psi}$, which can be decomposed into smaller roots as $\Rt{\alpha+\beta+2\psi} = \Rt{\alpha} + (\Rt{\beta+2\psi}) = (\Rt{\alpha+\beta}) + 2\Rt{\psi}$.
    \item  Repeat the above for the missing root $\Rt{\alpha+2\beta+2\psi}$, which can be decomposed into smaller roots as $\Rt{\alpha+2\beta+2\psi} = (\Rt{\alpha+\beta}) + (\Rt{\beta+2\psi}) = \Rt{\alpha} + 2(\Rt{\beta+\psi}) = (\Rt{\alpha+\beta+2\psi}) + \Rt{\beta}$.

    \item Prove all remaining commutation relations.
    \item Prove linearity relations (their proofs typically rely on commutation relations).
\end{itemize}
We now comment briefly on why the $B_3$ proof is much more labor intensive than the $A_3$ proof, which uses exactly the same basic model except for the underlying set of roots. There are several reasons why $B_3$ is more complicated than $A_3$:
\begin{itemize}
    \item The $A_3$ link contains only $6$ roots, while the small $B_3$ link contains $7$ roots and the large link contains $9$ roots. Thus, there are simply more relations to prove. Indeed, since the number of root-pairs grows quadratically with the number of roots, and the number of degree-pairs is the product of the heights, the sheer number of root-degree-pairs used in the commutation relations which define the group blows up roughly quartically in the number of roots. On the other hand, lifting will only ever cover $8$ degree-pairs, regardless of the underlying roots' height. Thus, lifting becomes less useful as the number of roots increases, and we have to compensate by proving implications among the relations.
    \item In the large link of $B_3$, there are $3$ missing roots (in the other two links, there is only $1$).
    \item A further complication in the large link of $B_3$ is the presence of roots which differ by a vector of even height: E.g., $\Rt{\alpha}$ and $\Rt{\alpha+2\beta+2\psi}$ differ by $\Rt{2\beta+2\psi}$. Thus, when homogeneously lifting, the ratios of the coefficients we produce between the two types of elements will always be squares. We will survive by proving linearity relations and using the fact that every finite field element is the sum of two squares. 
\end{itemize}

Finally, we mention three stylistic choices which we have made to improve readability of the proofs. Firstly, in a multiline derivation, we often highlight the subsequence of a word which we are ``actively operating'' on (i.e., applying relations to) using \asite{red underline}. Secondly, we have given all relations ``codes'' (in addition to a traditional numbering scheme), displayed when we reference them via \texttt{teletype text}, so that the reader does not have to click a hyperlink to understand which other relation is being referred to. Finally, group elements of a given type are typeset \emph{vertically} (as ``column vectors'') --- either $\El{\zeta}{p}$ for a $\Rt{\zeta}$-type element with entry $p$, or $\Eld{\zeta}{t}{i}$ for a $\Rt{\zeta}$-type element with degree $i$ and coefficient $t$. This makes it visually easier to directly compare the respective types, coefficients, and degrees in a product of elements, and reduces the number of times we have to split lines.

\subsection*{Acknowledgments}

Thanks to Richard Peng and Junzhao Yang for helpful pointers on computing ranks over $\BF_2$, and to Richard Peng and Ryan Bai \cite{BP24} for independently computing the rank of our large matrix over $\BF_2$.

%% file: sections/02-prelim.tex
\section{Preliminaries}

As mentioned, we use $[n] = \{0,\ldots,n\}$ to denote the set of numbers between $0$ and $n$ \emph{inclusive} (so $|[n]| = n+1$). We denote group identity elements as $\Id$, and define the commutator of two elements in a group $x,y \in G$ as:
\begin{equation}\label{eq:def:comm}
    \Comm{x}{y} = xyx^{-1} y^{-1}.
\end{equation}

\subsection{Simplicial complexes}
\begin{definition}
    A \emph{pure $d$-dimensional simplicial complex}\footnote{In this paper we only consider pure, finite, simplicial complexes. Thus we may sometimes drop any of the three adjectives.} on (finite) vertex set~$V$ is a downward-closed collection~$X$ of subsets of~$V$ in which all maximal sets have cardinality~$d+1$.
    It is \emph{partite} if there is a partition $V = V_0 \sqcup V_1 \sqcup \cdots \sqcup$ such that each maximal set has one element in each~$V_i$.
    We write $X(j)$ for the subcollection of sets of cardinality~$j+1$, called the \emph{faces/simplices of dimension~$j$}.  
    We will sometimes use the terms vertices, edges, triangles, and tetrahedra for the faces of dimension~$0, 1, 2, 3$ (respectively).
\end{definition}
\begin{definition}
    A \emph{weighted} simplicial complex $(X,\pi)$ is a pair where $X$ is some $d$-dimensional complex and $\pi = \pi_d$ is a probability distribution on~$X(d)$. 
    This $\pi$ induces further probability distributions $\pi_j$ on $X(j)$, $0 \leq j < d$, as follows: $\boldsymbol{\sigma} \sim \pi_j$ is defined by first drawing $\boldsymbol{\tau} \sim \pi$ and then letting $\boldsymbol{\sigma}$ be a uniformly random subset of~$\boldsymbol{\tau}$ of cardinality~$j+1$.
\end{definition}
\begin{notation}
    If we don't specify a weighting for $d$-dimensional complex~$X$, then by default $\pi$ is assumed to be the uniform distribution on $X(d)$.
\end{notation}
\begin{definition}
    For $(X,\pi)$ a $d$-dimensional weighted complex and $\sigma \in X(j)$, the \emph{link} of~$\sigma$ is the $(d-j-1)$-dimensional simplicial complex $X_\sigma = \{\tau \setminus \sigma : \tau \in X,\ \tau \supseteq \sigma\}$.
    It becomes a weighted complex $(X_\sigma, \pi^\sigma)$ under the natural weighting induced by~$\pi$, in which $\pi^{\sigma}(\gamma)$ is proportional to $\pi(\sigma \cup \gamma)$.
\end{definition}
\begin{definition}
    We say a family $(X_m)_m$ of $d$-dimensional complexes is of \emph{bounded degree} if there is some~$D$ such that every vertex link $(X_m)_\sigma$ (meaning $|\sigma| = 1$) has cardinality at most~$D$.
\end{definition}
\begin{definition}
    If $X$ is a $d$-dimensional simplicial complex and~$j \leq d$, the \emph{$j$-skeleton} of~$X$ is the $j$-dimensional complex $\bigcup_{i \leq j} X(i)$. 
    If $X$ is weighted, it also naturally inherits a weighting.
\end{definition}

\subsection{Expansion definitions}
In this section on expansion properties, $(X,\pi)$ will always denote a weighted $d$-dimensional simplicial complex.

\subsubsection{Cocycle expansion}
We define all our cohomology notions over~$\BF_2$.
\begin{definition}
    For $-1 \leq j \leq d$, a function $f : X(j) \to \BF_2$ is called a \emph{$j$-chain}\footnote{Arguably this should be called a \emph{$j$-cochain}, but since our complexes are finite there is no need to distinguish.}, and it may naturally be identified with a column vector in $\BF_2^{X(j)}$ or a subset of~$X(j)$.
\end{definition}
\begin{definition}
    The \emph{distance} between two $j$-chains $f$ and~$f'$ is defined to be 
    \begin{equation}
        \dist(f,f') = \Pr_{\boldsymbol{\sigma} \sim \pi_j}[f(\boldsymbol{\sigma}) \neq f'(\boldsymbol{\sigma})].
    \end{equation}
\end{definition}
\begin{definition}
    For $0 \leq j \leq d$, the \emph{$(j-1)$-coboundary operator} $\delta_{j-1} : \BF_2^{X(j-1)} \to \BF_2^{X(j)}$ is defined on $(k-1)$-chains~$f$ by $\delta_{j-1} f(\tau) = \sum_{\sigma \subset \tau} f(\sigma)$.  We may write $\delta f$ for brevity.
    The operator $\delta_{j-1}$ may be thought of as a matrix in $\BF_2^{X(j-1) \times X(j)}$. 
    We have $\delta_{j} \delta_{j-1} = 0$.
\end{definition}
\begin{definition}
    We will also write $\partial_j : \BF_2^{X(j)} \to \BF_2^{X(j-1)}$ for $\delta_{j-1}^\intercal$, the \emph{$j$-boundary operator}, which acts on $j$-chains as $\partial_j f(\sigma) = \sum_{\tau \supset \sigma} f(\tau)$.
\end{definition}
\begin{remark}  \label{rem:CSP}
    Let $b$ be a $j$-chain (``right-hand sides'').
    Then from a computer science perspective, one might consider each pair $(\tau, b(\tau))$ for $\tau \in X(j)$ as a $\BF_2$-affine ``constraint/check'' on an assignment $f$ to variable set $X(j-1)$; namely, the $(j+1)$-ary constraint $\sum_{\sigma \subset \tau} f(\sigma) = b(\tau)$.
    This gives a weighted CSP, where the $\tau$-constraint is weighted by $\pi_j(\tau)$.
    The associated linear system is $\delta f = b$, and the (weighted) fraction of violated constraints is $\dist(\delta f, b)$.
\end{remark}    
\begin{definition}
    The subspace of \emph{$(j-1)$-cocycles} is $Z^{j-1} = Z^{j-1}(X) = \ker \delta^{j-1}$.
    In the language of \Cref{rem:CSP}, this is all the assignments~$f$ that satisfy the $X(j)$-CSP with right-hand sides $b = 0$.
\end{definition}
\begin{definition}
    Suppose that whenever $f$ is a $j$-chain with $\dist(f, Z^{j}) > v$, it holds that $\dist(\delta f, 0) > \epsilon \cdot v$.
    Then we say that $X$ has \emph{$j$-cocycle expansion at least~$\eps$}, and we write $h^j = h^{j}(X)$ for the least possible such~$\eps$.
    If the ``$j$'' is omitted, we mean that the condition holds for all $0 \leq j < d$.
\end{definition}
    In the language of \Cref{rem:CSP}, $h^{j} \geq \eps$ means that the $X(j+1)$-CSP with right-hand sides~$0$ has the following testability property: for any assignment~$f$ to $X(j)$, if a random constraint is satisfied with probability at least $1 - \eps\cdot v$, then $f$ must be $v$-close to a satisfying assignment.
\begin{definition}
    The subspace of \emph{$j$-coboundaries} is $B^{j} = B^j(X) =  \Ima \delta^{j-1}$.
    In the language of \Cref{rem:CSP}, this is the set of all the right-hand sides~$b$ that make the $X(j)$-CSP satisfiable.
\end{definition}
Since $\delta_j \delta_{j-1} = 0$, the $j$-coboundaries are always ``trivially'' $j$-cocycles. 
This motivates the following definition:
\begin{definition}
    The \emph{$j$th cohomology space} is $H^j = H^j(X) = Z^j/B^j$.
    We say that it \emph{vanishes} if $B^j = Z^j$. 
\end{definition}
\begin{definition}
    We define the \emph{$j$-cosystole} to be
    \begin{equation}
        s^j = s^j(X) = \min\{\dist(f,0) : b \in Z^j/B^j\},
    \end{equation}
    with the convention $s^j = 1$ if the $j$th cohomology vanishes.
\end{definition}
\begin{definition}
    Suppose the $j$-cocycle expansion satisfies $h^j \geq \eps$.
    If, moreover, $s^j \geq \mu$, we say that $X$ has
    \emph{$j$-cosystolic expansion at least $(\eps,\mu)$}.
    If the $j$th cohomology in fact vanishes, $X$ is said to have \emph{$j$-coboundary expansion at least~$\eps$}.
    If the ``$j$'' is omitted, we mean that the condition holds for all $0 \leq j < d$.
\end{definition}
\begin{example}
    Suppose the $1$-skeleton of~$X$ is the (weighted) graph~$G = (V,E,\pi_1)$.  Then $Z^0$ consists of the functions $f : V \to \BF_2$ that are constant on each connected component of~$G$, and $h^0 \geq \eps$ iff each component of~$G$ has edge-expansion (Cheeger constant) at least~$\eps$.
    If $G$ is connected, we would say $X$ has $0$-coboundary at least~$\eps$ expansion; on the other hand, if all connected components have size at least~$\mu$ (under~$\pi_0$), we would say $X$ has $0$-cosystolic expansion at least $(\eps,\mu)$.
\end{example}

\subsubsection{Spectral expansion}
\begin{definition}
    Let $\sigma \in X(j)$, $j \leq d-2$, and let $G_\sigma = (V, E, \pi)$ be the $1$-skeleton of the link~$X_\sigma$.
    Then if $A_\sigma$ denotes the ($\pi_1$-weighted) adjacency matrix for~$G_\sigma$, and $D_\sigma$ denotes the diagonal degree matrix (with $v$th diagonal entry $2\pi_0(\{v\})$), we write $W_\sigma = D_\sigma^{-1} A_\sigma$ for the \emph{standard random walk matrix of~$X_\sigma$} (which has invariant distribution~$\pi_0$). 
    We will also write $\lambda_2(X_\sigma)$ for the second-largest eigenvalue of~$W_\sigma$.
\end{definition}
\begin{definition}
    For $j \leq d-1$, we say that $X$ has \emph{$j$-spectral\footnote{Other authors would call this $(j-1)$-local or $(j-1)$-spectral expansion.} expansion parameter at most~$\gamma$} if $\lambda_2(X_\sigma) \leq \gamma$ for all $|\sigma| = j$.
    If the ``$j$'' is omitted, we mean that the condition holds for all $j \leq d-2$.
\end{definition}
Ballmann and \'{S}wi\k{a}tkowski~\cite{BS97} showed the following:
\begin{theorem}
    Suppose $X$ has $1$-spectral expansion parameter  $\gamma$, and also that (the $1$-skeleton of) $X$ is connected.
    Then $\lambda_2(X_\emptyset) \leq \frac{\gamma}{1-\gamma}$; i.e., the $0$-spectral expansion parameter of~$X$ is at most $\frac{\gamma}{1-\gamma}$.
\end{theorem}
Using this inductively, one can establish~\cite{Opp18} the \emph{Trickling Down} theorem:
\begin{theorem}
    Assume (the $1$-skeleton of) $X_\sigma$ is connected for all $j$-faces $\sigma$, $j < d-1$.
    If $X$ has $(d-1)$-spectral expansion parameter $\gamma$, then it has $j$-spectral expansion parameter at most $\frac{\gamma}{1-(d-1-j)\gamma}$.
\end{theorem}
The utility of this theorem is that (under the mild constraint of connectivity) one can show \emph{global} spectral expansion just by verifying spectral expansion in the \emph{local} link graphs of the $(d-2)$-dimensional faces.

\subsubsection{Topological expansion}
Gromov~\cite{Gro10} introduced the following concept: 
\begin{definition}
    We say that $X$ has \emph{topological expansion at least~$\delta$} if, for every continuous map $f : X \to \BR^d$, there is a point $p \in \BR^d$ such that $\Pr_{\boldsymbol{\sigma} \sim \pi}[p \in \boldsymbol{\sigma}] \geq \delta$; in words, $p$ is contained in at least a $\delta$-fraction of the (images) of the $d$-faces.
\end{definition}
It is easy to see that a bounded-degree ($1$-dimensional) expander graph is an $\Omega(1)$-topological expander.
Generalizing this, Gromov showed the following (see~\cite{DKW18}): 
\begin{theorem} \label{thm:to-top}
    Suppose $X$ has $j$-cosystolic expansion at least $(\eps,\mu)$ for all $0 \leq j < d$.  
    Assume that $\pi_1(\sigma) \leq \Delta$ for all vertices $\sigma \in X(0)$, where $\Delta = \Delta(\eps, \mu) > 0$ is sufficiently small.
    Then $X$ has topological expansion at least~$\delta = \Omega_d(\mu \eps^{d+1})$.

    In particular, if $(X_m)_m$ is a growing family of bounded-degree $d$-dimensional complexes, with $j$-cosystolic expansion at least $(\Omega(1), \Omega(1))$ for all $0 \leq j < d$, then $X_m$ has topological expansion~$\Omega(1)$.
\end{theorem}

\subsection{Cosystolic/topological expansion from local conditions}
A sequence of works~\cite{KKL14b,EK16,DD23} gave results showing that cosystolic expansion --- and hence also topological expansion --- follows from ``local'' considerations. 
The following version is from~\cite{EK16}:
\begin{theorem} 
    For any $\beta > 0$, there are $\gamma, \eps, \mu > 0$ (depending also on~$d$) such that the following holds: 
    Suppose that for all $j$-faces~$\sigma$, $0 \leq j < d-1$, the complex $X_\sigma$ has coboundary expansion at least~$\beta$.
    Then provided $X$ has spectral expansion parameter at most~$\gamma$, we may conclude that $X$ has $j$-cosystolic expansion at least $(\eps,\mu)$ for all $j < d-1$ (though not necessarily for~$j=d-1$).
\end{theorem}

Suppose we wish to apply this theorem in the case $d=3$.  We would need to verify:
\begin{enumerate}
    \item Each vertex-link has $1$-coboundary expansion at least~$\beta$.
    \item Each vertex-link has $0$-coboundary expansion at least~$\beta$.
    \item Each edge-link has $0$-coboundary expansion at least~$\beta$.
\end{enumerate}
But the second and third conditions here are essentially superfluous, since $X$ is already assumed to be an excellent spectral expander.  More precisely, by the ``easy direction'' of Cheeger's inequality, the $0$-coboundary expansion of the vertex- and edge-links is at least $\frac12 - \frac12 \gamma$, which exceeds~$\beta$ provided $\gamma$ is small enough.
Thus only the first condition above is essential.
Moreover, by the Trickling Down theorem, we only need to verify spectral expansion for the edge-links.  Putting this together (exactly the same as was done in~\cite{KO21}), we conclude the following (where the quantitative dependence uses~\cite{DD23}):
\begin{theorem} \label{thm:workhorse}
    For any $\beta > 0$, there exists $\gamma > 0$ such that the following holds: 
    Suppose that $X$ is a connected $3$-dimensional complex, where all vertex-links are connected, and all edge-links have  $0$-spectral expansion parameter (i.e., second-largest eigenvalue) at most~$\gamma$.
    Moreover, assume that each vertex-link has coboundary expansion at least~$\beta$.
    Then $X$ has $1$-cosystolic expansion at least $(\Omega(\beta),\Omega(\beta))$.
\end{theorem}
Then, as a consequence of \Cref{thm:to-top}, we get (again, as in~\cite{KO21}):

\begin{corollary}   \label{cor:workhorse}
    Fix $\beta > 0$, and suppose $(X_m)_m$ is a growing family bounded-degree $3$-dimensional complexes, each satisfying the conditions of~\Cref{thm:workhorse}. Then the $2$-skeleton of~$X_m$ has topological expansion~$\Omega(\beta^4)$.
\end{corollary}

\subsection{Cones and fillings}
The main effort in this work will be to show good $1$-coboundary expansion for certain $2$-dimensional complexes.  
One way to show this is to use the ``random cones'' method of Gromov~\cite{Gro10}. 
In turn, Kaufman and Oppenheim~\cite{KO21} (following on ideas in~\cite{LMM16,KM19}) showed that in ``strongly symmetric'' complexes, it suffices to upper-bound the ``cone radius''.
(This is a generalization of the fact that in an edge-transitive graph~$G$, the Cheeger constant is at least $1/(2 \diam G)$.)
We state here an abbreviated version of their result which is sufficient for our purposes.
\begin{definition}  \label{def:1cone}
    In a connected $2$-dimensional simplicial complex~$X$, we say that the \emph{$1$-cone radius is at most~$R$} if there exist the following:
    \begin{itemize}
        \item A distinguished \emph{apex} vertex~$u$.
        \item For each vertex~$v$ in~$X$, a distinguished simple path~$P_{uv}$ from $u$ to~$v$ consisting of edges in~$X(1)$. 
        (For $v = u$ it is the empty path.)
        We identify each $P_{uv}$ with a $1$-chain.
        \item For each edge $\{v,w\} \in X(1)$, a set of triangles $\CT_{vw} \subseteq X(2)$ of cardinality at most~$R$ such that, when $\CT_{vw}$ is treated as a $2$-chain, the following $1$-chain identity holds: $\partial_2 \CT_{vw} = P_{uv} + \{v,w\} + P_{uw}$.
        We say that ``$\CT_{vw}$ is an $\BF_2$-filling of the loop formed by $P_{uv}$, $\{v,w\}$, and $P_{uw}$''.
    \end{itemize}
\end{definition}
If the $1$-cone radius of~$X$ is finite, then the $1$-cohomology vanishes.  Kaufman and Oppenheim's theorem makes this quantitative for strongly symmetric complexes:
\begin{theorem}\label{thm:cones1}
    (From \cite[Thm.~3.8]{KO21}.)
    Let $X$ be a $2$-dimensional complex, weighted by the uniform distribution on~$X(2)$. Assume that there is a group~$G$ of automorphisms of~$X$ that acts transitively on its $2$-faces.  
    Then if $X$'s $1$-cone radius is at most~$R$, it follows that $X$ has $1$-coboundary expansion at least $1/(3R)$.
\end{theorem}

\subsection{Coset complexes}
In this section we define a natural group-theoretic way of making simplicial complexes, dating back to Lann\'{e}r~\cite{Lan50}:
\begin{definition}
    Let $G$ be an ``ambient'' group and $\CH = \{H_0, H_1, \dots, H_d\}$ a set of~$d$ ``colored'' subgroups.
    We write $\CoCo(G; \CH)$ for the associated \emph{coset complex}, which is a partite $d$-dimensional simplicial on vertex set $V = G/H_0 \sqcup G/H_1 \sqcup \cdots \sqcup G/H_d$ in which the $d$-faces are those of the form $\{xH_0, xH_1, \dots, xH_d\}$ for $x \in G$.  
    (We say that face is \emph{induced} by~$x$, and all the coset complex we study have  $\bigcap_{i=0}^d H_i = \{\Id\}$, so that each $x \in G$ induces a unique $d$-face.)
\end{definition}
We collect some elementary properties of a coset complexes here (see, e.g.,~\cite{Gar79}):
\begin{proposition} \label{prop:cc}
    A coset complex  $\CoCo(G; \CH)$ satisfies the following:
    \begin{enumerate}
        \item The group $G$ acts on $\CoCo(G; \CH)$ as  automorphisms in a natural way ($g \in G$ maps the $d$-face induced by~$x$ to the $d$-face induced by~$gx$). 
        This action is transitive on the $d$-faces.
        \item Let $\sigma$ be a $j$-face; say $\sigma = \{xH_i : i \in S\}$ where $S \subseteq \{0, 1, \dots, d\}$ has cardinality~$j+1$.
        Then writing $H_S = \bigcap_{i \in S}$ (meaning $G$ if $S = \emptyset$),
        the link of~$\sigma$ is isomorphic to the coset complex~$\CoCo(H_S; \{H_S \cap H_i : i \not \in S\})$.
        \item The complex is connected iff $G = \langle H_0, H_1, \dots, H_d\rangle$. Quantitatively, every $x \in G$ can be expressed as a product of at most~$r$ elements from $\bigcup \CH$ iff the diameter of (the $1$-skeleton of) the complex is at most~$r$.
    \end{enumerate}
\end{proposition}
The last property above gives a connection between the group theory and $0$-cohomology, since  connectivity is equivalent to vanishing $0$-cohomology.

Much more nontrivially, there is a connection between the group theory and $1$-cohomology.
Specifically, it is known that if a complex is ``simply connected'', then its $1$-cohomology (over~$\BZ$ and hence over~$\BF_2$) vanishes.
We will not recall the definition of ``simply connected'' here, although it is similar in spirit to having finite $1$-cone radius.
Instead, we will just state the group-theoretic characterization that it turns out to be equivalent to, in coset complexes.
This was apparently first proven by Lann\'{e}r~\cite{Lan50}; see also~\cite{Bun52,Gar79}.
\begin{theorem} \label{thm:lanner}
    (Cf.~$(\diamond)$ from \Cref{sec:diamond}.)
    $\CoCo(G; \CH)$ is simply connected iff $G$ has as a presentation \mbox{$\langle y  \in \bigcup \CH \mid r \in R(H_0) \cup \cdots \cup R(H_d) \rangle$}, where $R(H_i)$ is all the ``in-subgroup relations'' for~$H_i$ (i.e., all true equations $ab = c$ for $a,b,c \in H_i$).
\end{theorem}
A quantitative version of this theorem, which bounds the $1$-cone radius of $\CoCo(G;\CH)$ in terms of the efficiency of the presentation, was shown by Kaufman and Oppenheim~\cite[Thm.~6.17]{KO21}. 
As discussed in \Cref{sec:diamond}, this will not be enough for us; we will to develop of further generalization of the quantitative implication, which we do in \Cref{sec:dehn}.

\subsection{Root systems, Chevalley groups, and unipotent subgroups}

\newcommand{\VecMatStyle}[1]{\mathbf{#1}}
\newcommand{\IndPlain}[1]{\VecMatStyle{e}_{#1}}
\newcommand{\IndVec}[2]{\VecMatStyle{e}^{(#1)}_{#2}}
\newcommand{\IndMat}[3]{\VecMatStyle{E}^{(#1)}_{#2,#3}}
\newcommand{\nnspan}[1]{\mathrm{span}_{\mathbb{N}}(#1)}
\newcommand{\posspan}[1]{\mathrm{span}_{\mathbb{N}^+}(#1)}

\subsubsection{Root systems}\label{sec:prelim:root-systems}
An (irreducible) \emph{root system} 
of rank $d$\footnote{We will henceforth only consider irreducible root systems of rank at least~$2$.}
is a finite set of vectors $\Phi$ lying in a $d$-dimensional real vector space satisfying certain symmetry conditions. (We will not recap them here; see, e.g.~\cite{Ozo07}.)

They are completely classified into four infinite families, $(A_d)_{d \geq 1}$, $(B_d)_{d \geq 2}$, $(C_d)_{d \geq 3}$, $(D_d)_{d \geq 4}$, plus five others ($G_2$, $F_4$, $E_6$, $E_7$, $E_8$).
In turn, root systems are used to help classify finite groups of reflections in~$\BR^d$.
An important concept for root systems is that of a base:
\begin{definition}
    A \emph{base} for rank-$d$ root system $\Phi$ is a subset $\Pi$ with $|\Pi| = d$ such that $\nnspan{\Pi}$ contains one of $\Rt{\pm \zeta}$ for every $\Rt{\zeta} \in \Phi$.
    Here, $\nnspan{P}$ (respectively, $\posspan{P}$) denotes all nonnegative (respectively, positive) linear combinations of vectors in~$P$.
    Each root system has a ``unique'' base, up isometry.
\end{definition}
\begin{definition}
    Suppose $P \subseteq \Phi$ and $\Rt{\zeta} \in \nnspan{P}$; say $\Rt{\zeta} = \sum_{\Rt{\eta} \in P} c_{\Rt{\eta}} \Rt{\eta}$.
    Then we write $\mathrm{ht}_P(\Rt{\zeta}) = \sum_{\Rt{\eta} \in P} c_{\Rt{\eta}}$ for the \emph{height} of $\Rt{\zeta}$ with respect to~$P$.
\end{definition}

The root systems relevant for our work are the rank-$3$ ones, namely $A_3$, $B_3$, and $C_3$.  In fact, since we are leaving $C_3$ for later work, we only define here $A_3$ and~$B_3$.

\paragraph{Type-A.}  The root system $A_d$ consists of all vectors in $\BR^{d+1}$ of the form $\IndPlain{i} - \IndPlain{j}$, $i \neq j \in [d]$, where $\IndPlain{i}$ denotes the $i$-th standard basis vector.
(This is rank~$d$, since all vectors are orthogonal to~$(1, 1, \dots, 1)$.)
A canonical base for $A_3$, depicted in the below \Cref{fig:root-system:a3}, is 
\begin{equation} \label{eqn:PiA3}
    \Pi_{A_3} = \{\Rt{\alpha}, \Rt{\beta}, \Rt{\gamma}\}, \text{where } \Rt{\alpha} = (1,-1,0,0),\ 
    \Rt{\beta} = (0,1,-1,0),\ 
    \Rt{\gamma} = (0,0,1,-1),
\end{equation}
and we will also name the root vector
\begin{equation}
    \Rt{\delta} = -(\Rt{\alpha} + \Rt{\beta} + \Rt{\gamma}) = (-1, 0, 0, 1).
\end{equation}
In this root system we have the simple feature that if $\Rt{\zeta},\Rt{\eta} \in A_d$ are linearly independent, then $\posspan{\Rt{\zeta},\Rt{\eta}} \cap A_d$ is either empty or $\{\Rt{\zeta}+\Rt{\eta}\}$.

\input{figures/root_systems/a3}

\paragraph{Type-B.}  The root system $B_d$ consists of all integer vectors in $\BR^{d}$ of length~$1$ (``short roots'') or~$\sqrt{2}$ (``long roots'').
A canonical base for $B_3$, depicted in the below \Cref{fig:root-system:b3}, is 
\begin{equation}    \label{eqn:PiB3}
    \Pi_{B_3} = \{\Rt{\alpha}, \Rt{\beta}, \Rt{\psi}\}, \text{where } \Rt{\alpha} = (1,-1,0),\ 
    \Rt{\beta} = (0,1,-1),\ 
    \Rt{\psi} = (0,0,1).
\end{equation}
(Note also our slight abuse of notation: $\Rt{\alpha}$, $\Rt{\beta}$ denote formally different vectors in $A_3$ vs.\ $B_3$.)
We will also name the root vector
\begin{equation}
    \Rt{\omega} = -(\Rt{\alpha} + \Rt{\beta} + \Rt{\psi}) = (-1, 0, 0).
\end{equation}
Note that $\Rt{\alpha}$ and $\Rt{\beta}$ are ``long'' and $\Rt{\psi}$ and $\Rt{\omega}$ are ``short''.

\input{figures/root_systems/b3}

Understanding $\posspan{\Rt{\zeta},\Rt{\eta}} \cap B_d$ where $\Rt{\zeta}, \Rt{\eta} \in B_d$ are linearly independent roots is more complex than it was in the $A_d$ case. Indeed, if $\Rt{\zeta}$ and $\Rt{\eta}$ are independent short roots, then $\posspan{\Rt{\zeta},\Rt{\eta}} \cap B_d = \{\Rt{\zeta}+\Rt{\eta}\}$, a long root. If $\Rt{\zeta}$ and $\Rt{\eta}$ are long, then $\posspan{\Rt{\zeta},\Rt{\eta}} \cap B_d$ is empty or $\{\Rt{\zeta}+\Rt{\eta}\}$, a long root. Finally, if $\Rt{\zeta}$ is long and $\Rt{\eta}$ is short, then $\posspan{\Rt{\zeta},\Rt{\eta}} \cap B_d$ is empty or $\{\Rt{\zeta}+\Rt{\eta},\Rt{\zeta}+2\Rt{\eta}\}$, respectively long and short roots. This makes understanding commutator relations in the corresponding Chevalley group more complex, as we will now see.

\subsubsection{Chevalley groups and unipotent subgroups}
Each root system can be combined with (almost) any finite field~$\BF$ to produce a finite simple (or nearly-simple) group called a \emph{(universal) Chevalley group}.
We define these groups abstractly via their \emph{Steinberg presentation}:

\begin{definition}[Steinberg presentation of a Chevalley group]\label{def:prelim:steinberg}
    Let $\Phi$ be a root system and $\BF$ a finite field. The corresponding (universal) Chevalley group $G_\Phi(\BF)$ is generated by elements/symbols\footnote{These elements are traditionally written $x_{\Rt{\zeta}}(t)$, but we will find our subscript-free notation more readable.} of the form $\El{\zeta}{t}$ for $\Rt{\zeta} \in \Phi$ and $t \in \BF$. We sometimes call $\Rt{\zeta}$ the ``type'' of the element and~$t$ its ``entry''.
    The elements are
    subject to the following three families of relations:
\begin{itemize}
    \item ``Linearity'': For every $\Rt{\zeta} \in \Phi$ and $t,u \in \BF$, 
    \begin{equation}    \label{eqn:linrels}
        \El{\zeta}{t} \El{\zeta}{u} = \El{\zeta}{t+u}.        
    \end{equation}
    (We remark that from that these relations imply that $\El{\zeta}{0} = \Id$ and $\El{\zeta}{t}^{-1} = \El{\zeta}{-t}$, cf. \Cref{prop:prelim:group-homo}.)
    \item ``Commutator'': For every $\Rt{\zeta} \neq -\Rt{\eta} \in \Phi$ and $t,u \in \BF$, 
    \begin{equation}    \label{eqn:conrels}
        \Comm{\El{\zeta}{t}}{\El{\eta}{u}} = \prod_{\substack{a\Rt{\zeta} + b \Rt{\eta} \in \Phi \\ a,b \in \BN^+}} \El{a\Rt{\zeta} + b \Rt{\eta}}{C^{\Rt{\zeta},\Rt{\eta}}_{a,b} \cdot t^au^b},
    \end{equation} 
    where the product is ordered according to a fixed, global total order on~$\Phi$, and where   $C^{\Rt{\zeta},\Rt{\eta}}_{a,b} \in \{\pm1,\pm2,\pm3\}$ are certain so-called \emph{Chevalley constants} whose values (depending only on the root ordering, not on $t,u$) we will discuss below.
    
    \item ``Diagonal'': For every $\Rt{\zeta} \in \Phi$ and $t,u \neq 0\in \BF$, $h_{\Rt{\zeta}}(t) h_{\Rt{\zeta}}(u) = h_{\Rt{\zeta}}(tu)$, where $h_{\Rt{\zeta}}(t) := g_{\Rt{\zeta}}(t) g_{\Rt{\zeta}}(-1)$ and $g_{\Rt{\zeta}}(t) := \El{\zeta}{t} \El{-\zeta}{-t^{-1}} \El{\zeta}{t}$.
\end{itemize}
    In words: the linearity relations state that multiplying elements of the same type adds their entries; the commutator relations say that commuting $\Rt{\zeta}$- and $\Rt{\eta}$-type elements produces a product of $\Rt{\xi}$-type elements over all $\Rt{\xi}$ in $\posspan{\{\Rt{\zeta},\Rt{\eta}\}} \cap \Phi$.
\end{definition}
\begin{remark}
    We give short shrift to the ``diagonal'' relations, as they will not be very relevant for our work. This is because we mainly study ``unipotent'' subgroups of $G_\Phi(\BF)$, which never simultaneously contain elements of both types~$\Rt{\pm \zeta}$. 

    Note also that the ``linearity'' and ``commutator'' relations would also make sense if the entries came from a commutative ring (without division), rather than a field.
\end{remark}
\begin{example}
    The group $G_{A_d}(\BF)$ is isomorphic to $\SL_d(\BF)$.
    For $1 \leq i\neq j \leq d$, this isomorphism maps $\El{\IndPlain{i} - \IndPlain{j}}{t}$ to the $(d+1)\times(d+1)$ elementary matrix with entry~$t$ in the $(i,j)$ position.
    From this fact, one can infer the Chevalley constants (which are in $\{\pm 1\}$).
    For more details, see \Cref{app:Ad}.
\end{example}
\begin{example}
    The group $G_{B_d}(\BF)$ is isomorphic to $\Omega_{2d+1}(\BF)$, the commutator subgroup of the orthogonal group~$\mathsf{O}_{2d+1}(\BF)$.
    For more details of the concrete realization of $G_{B_3}(\BF)$ as matrices in $\BF^{7 \times 7}$ (including Chevalley constants), see \Cref{app:Bd}.
\end{example}

%% file: figures/root_systems/a3.tex
\usetikzlibrary{3d}

\begin{figure}
\tdplotsetmaincoords{5}{170}
\begin{subfigure}[b]{0.45\textwidth}
\centering
    \begin{tikzpicture}[->,scale=2,tdplot_main_coords]
\draw[red]  (0,0,0) -- (1.4142135623731,0,0) node[black,pos=1.1] {$\Rt{\alpha}$}; 
\draw[red]  (0,0,0) -- (-0.707106781186548,1.22474487139159,0) node[black,pos=1.1] {$\Rt{\beta}$}; 
\draw[red]  (0,0,0) -- (0,-0.816496580927726,1.15470053837925) node[black,pos=1.1] {$\Rt{\gamma}$}; 
\draw[red]  (0,0,0) -- (0.707106781186548,1.22474487139159,0); 
\draw[red]  (0,0,0) -- (-0.707106781186548,0.408248290463863,1.15470053837925); 
\draw[red]  (0,0,0) -- (0.707106781186548,0.408248290463863,1.15470053837925); 
\draw[red]  (0,0,0) -- (-1.4142135623731,0,0); 
\draw[red]  (0,0,0) -- (0.707106781186548,-1.22474487139159,0); 
\draw[red]  (0,0,0) -- (0,0.816496580927726,-1.15470053837925); 
\draw[red]  (0,0,0) -- (-0.707106781186548,-1.22474487139159,0); 
\draw[red]  (0,0,0) -- (0.707106781186548,-0.408248290463863,-1.15470053837925); 
\draw[red]  (0,0,0) -- (-0.707106781186548,-0.408248290463863,-1.15470053837925)node[black,pos=1.1] {$\Rt{\delta}$}; 
    \end{tikzpicture}
    \caption{The root system $A_3 \subseteq \BR^4$ under an orthogonal projection into $\BR^3$. We also label the  our canonical base roots $\Rt{\alpha}$, $\Rt{\beta}$, $\Rt{\gamma}$, and their negative sum $\Rt{\delta} = -(\Rt{\alpha}+\Rt{\beta}+\Rt{\gamma})$.}
\end{subfigure}%
\hfill
\begin{subfigure}[b]{0.45\textwidth}
\centering
    \begin{tikzpicture}[->,scale=2,tdplot_main_coords]
\draw[red,very thick]  (0,0,0) -- (1.4142135623731,0,0) node[black,pos=1.1] {$\Rt{\alpha}$}; 
\draw[red,very thick]  (0,0,0) -- (-0.707106781186548,1.22474487139159,0) node[black,pos=1.1] {$\Rt{\beta}$}; 
\draw[red,very thick]  (0,0,0) -- (0,-0.816496580927726,1.15470053837925) node[black,pos=1.1] {$\Rt{\gamma}$}; 
\draw[red]  (0,0,0) -- (0.707106781186548,1.22474487139159,0) node[black,pos=1.1] {$\Rt{\alpha+\beta}$}; 
\draw[red]  (0,0,0) -- (-0.707106781186548,0.408248290463863,1.15470053837925) node[black,pos=1.3] {$\Rt{\beta+\gamma}$}; 
\draw[red]  (0,0,0) -- (0.707106781186548,0.408248290463863,1.15470053837925) node[black,pos=1.3] {$\Rt{\alpha+\beta+\gamma}$}; 
\draw[gray,dashed]  (0,0,0) -- (-1.4142135623731,0,0); 
\draw[gray,dashed]  (0,0,0) -- (0.707106781186548,-1.22474487139159,0); 
\draw[gray,dashed]  (0,0,0) -- (0,0.816496580927726,-1.15470053837925); 
\draw[gray,dashed]  (0,0,0) -- (-0.707106781186548,-1.22474487139159,0); 
\draw[gray,dashed]  (0,0,0) -- (0.707106781186548,-0.408248290463863,-1.15470053837925); 
\draw[gray,dashed]  (0,0,0) -- (-0.707106781186548,-0.408248290463863,-1.15470053837925); 
    \end{tikzpicture}
    \caption{The ``link of $\Rt{\delta}$'' in $A_3$: $\Rt{\alpha}$, $\Rt{\beta}$, and $\Rt{\gamma}$ are drawn as solid lines; all roots which are nonnegative integer combinations of these are drawn as colored thin lines; and the remaining roots are dashed and gray.}
\end{subfigure}%
\caption{The root system $A_3$ and a specific ``link'' within it.}\label{fig:root-system:a3}
\end{figure}
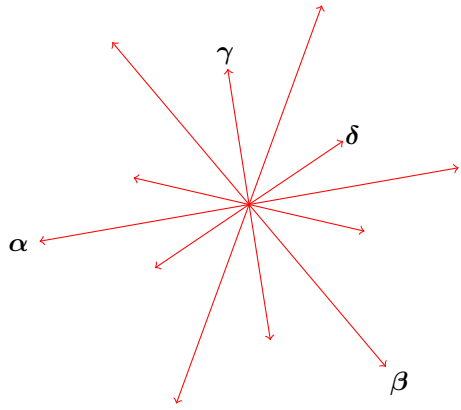
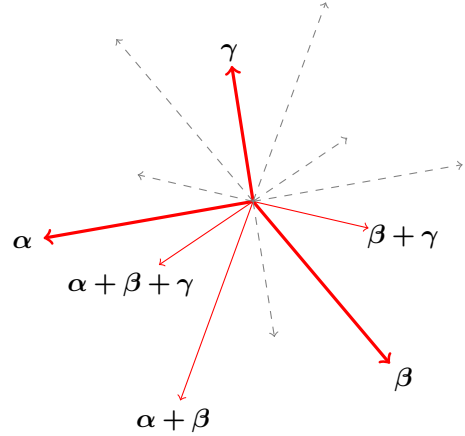

%% file: figures/root_systems/b3.tex
\begin{figure}
\tdplotsetmaincoords{70}{125}
\begin{subfigure}[b]{0.45\textwidth}
\centering
    \begin{tikzpicture}[->,scale=2,tdplot_main_coords]
        \draw[red]  (0,0,0) -- (+1,+1,0 );
        \draw[red]  (0,0,0) -- (+1,-1,0 ) node[black,pos=1.1] {$\Rt{\alpha}$};
        \draw[red]  (0,0,0) -- (-1,+1,0 );
        \draw[red]  (0,0,0) -- (-1,-1,0 );
        \draw[red]  (0,0,0) -- (0 ,+1,+1);
        \draw[red]  (0,0,0) -- (0 ,+1,-1) node[black,pos=1.1] {$\Rt{\beta}$};
        \draw[red]  (0,0,0) -- (0 ,-1,+1);
        \draw[red]  (0,0,0) -- (0, -1,-1);
        \draw[red]  (0,0,0) -- (+1,0, +1);
        \draw[red]  (0,0,0) -- (+1,0, -1);
        \draw[red]  (0,0,0) -- (-1,0, +1);
        \draw[red]  (0,0,0) -- (-1,0, -1);
        \draw[blue] (0,0,0) -- (+1,0, 0 );
        \draw[blue] (0,0,0) -- (-1,0, 0 ) node[black,pos=1.15] {$\Rt{\omega}$};
        \draw[blue] (0,0,0) -- (0 ,+1,0 );
        \draw[blue] (0,0,0) -- (0, -1,0 );
        \draw[blue] (0,0,0) -- (0, 0, +1) node[black,pos=1.15] {$\Rt{\psi}$};
        \draw[blue] (0,0,0) -- (0, 0, -1);
    \end{tikzpicture}
    \caption{The root system $B_3 \subseteq \BR^3$. We distinguish ``long'' and ``short'' roots as blue and red, respectively. We also label our canonical base roots $\Rt{\alpha}$, $\Rt{\beta}$, $\Rt{\psi}$, and their negative sum $\Rt{\omega} = -(\Rt{\alpha}+\Rt{\beta}+\Rt{\psi})$.}
\end{subfigure}%
\hfill
\begin{subfigure}[b]{0.45\textwidth}
\centering
    \begin{tikzpicture}[->,scale=2,tdplot_main_coords]
        \draw[red]  (0,0,0) -- (+1,+1,0 ) node[black,pos=1.2] {$\Rt{\alpha+2\beta+2\psi}$};
        \draw[red,very thick]  (0,0,0) -- (+1,-1,0 ) node[black,pos=1.1] {$\Rt{\alpha}$};
        \draw[gray,dashed]  (0,0,0) -- (-1,+1,0 );
        \draw[gray,dashed]  (0,0,0) -- (-1,-1,0 );
        \draw[red]  (0,0,0) -- (0 ,+1,+1) node[black,pos=1.15] {$\Rt{\beta+2\psi}$};
        \draw[red,very thick]  (0,0,0) -- (0 ,+1,-1) node[black,pos=1.1] {$\Rt{\beta}$};
        \draw[gray,dashed]  (0,0,0) -- (0 ,-1,+1);
        \draw[gray,dashed]  (0,0,0) -- (0, -1,-1);
        \draw[red]  (0,0,0) -- (+1,0, +1) node[black,pos=1.1] {$\Rt{\alpha+\beta+2\psi}$};
        \draw[red]  (0,0,0) -- (+1,0, -1) node[black,pos=1.1] {$\Rt{\alpha+\beta}$};
        \draw[gray,dashed]  (0,0,0) -- (-1,0, +1);
        \draw[gray,dashed]  (0,0,0) -- (-1,0, -1);
        \draw[blue] (0,0,0) -- (+1,0, 0 ) node[black,pos=1.3] {$\Rt{\alpha+\beta+\psi}$};
        \draw[gray,dashed] (0,0,0) -- (-1,0, 0 );
        \draw[blue] (0,0,0) -- (0 ,+1,0 ) node[black,pos=1.3] {$\Rt{\beta+\psi}$};
        \draw[gray,dashed] (0,0,0) -- (0, -1,0 );
        \draw[blue,very thick] (0,0,0) -- (0, 0, +1) node[black,pos=1.15] {$\Rt{\psi}$};
        \draw[gray,dashed] (0,0,0) -- (0, 0, -1);
    \end{tikzpicture}
    \caption{The ``link of $\Rt{\omega}$'' in $B_3$: $\Rt{\alpha}$ (long), $\Rt{\beta}$ (long), and $\Rt{\psi}$ (short) are drawn as solid lines; all roots which are nonnegative integer combinations of these are drawn as colored thin lines; and the remaining roots are dashed and gray.}
\end{subfigure}%
\hfill
\begin{subfigure}[b]{0.45\textwidth}
\centering
    \begin{tikzpicture}[->,scale=2,tdplot_main_coords]
        \draw[gray,dashed]  (0,0,0) -- (+1,+1,0 );
        \draw[gray,dashed]  (0,0,0) -- (+1,-1,0 );
        \draw[red]  (0,0,0) -- (-1,+1,0 ) node[black,pos=1.35] {$\Rt{\beta+\psi+\omega}$};
        \draw[gray,dashed]  (0,0,0) -- (-1,-1,0 );
        \draw[red]  (0,0,0) -- (0 ,+1,+1) node[black,pos=1.15] {$\Rt{\beta+2\psi}$};
        \draw[red,very thick]  (0,0,0) -- (0 ,+1,-1) node[black,pos=1.1] {$\Rt{\beta}$};
        \draw[gray,dashed]  (0,0,0) -- (0 ,-1,+1);
        \draw[gray,dashed]  (0,0,0) -- (0, -1,-1);
        \draw[gray,dashed]  (0,0,0) -- (+1,0, +1);
        \draw[gray,dashed]  (0,0,0) -- (+1,0, -1);
        \draw[red]  (0,0,0) -- (-1,0, +1) node[black,pos=1.1] {$\Rt{\psi+\omega}$};
        \draw[gray,dashed]  (0,0,0) -- (-1,0, -1);
        \draw[gray,dashed] (0,0,0) -- (+1,0, 0 );
        \draw[blue,very thick] (0,0,0) -- (-1,0, 0 ) node[black,pos=1.15] {$\Rt{\omega}$};
        \draw[blue] (0,0,0) -- (0 ,+1,0 ) node[black,pos=1.3] {$\Rt{\beta+\psi}$};
        \draw[gray,dashed] (0,0,0) -- (0, -1,0 );
        \draw[blue,very thick] (0,0,0) -- (0, 0, +1) node[black,pos=1.15] {$\Rt{\psi}$};
        \draw[gray,dashed] (0,0,0) -- (0, 0, -1);
    \end{tikzpicture}
    \caption{The ``link of $\Rt{\alpha}$'' in $B_3$: $\Rt{\beta}$ (long), $\Rt{\psi}$ (short), and $\Rt{\omega}$ (short) are drawn as thick lines; all roots which are nonnegative integer combinations of these are drawn as colored thin lines; and the remaining roots are dashed and gray.}
\end{subfigure}%
\hfill
\caption{The root system $B_3$ and two specific ``links'' within it.}\label{fig:root-system:b3}
\end{figure}
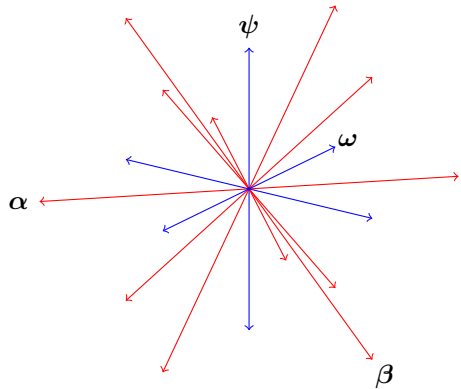
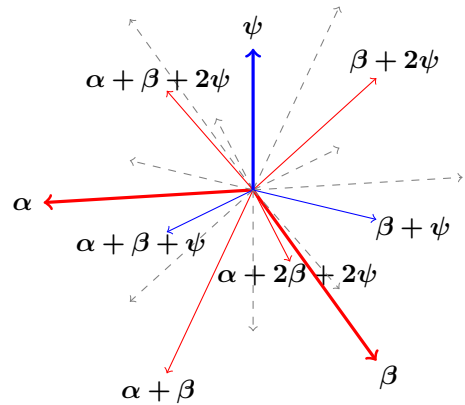
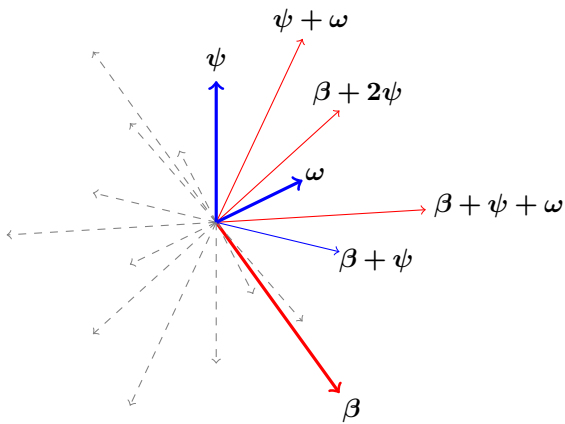

%% file: sections/03-chevalley.tex
\section{Chevalley coset complexes and their properties}
\subsection{Unipotent subgroups} \label{sec:unip}
Before defining the Chevalley coset complexes, it will be helpful to define and study a generalization of the unipotent subgroups of Chevalley groups, where we restrict what entries are allowed for the generators.

\subsubsection{Three presentations of the unipotent subgroups}

Throughout this section, $\Phi$ denotes a root system. For a set $I \subset \Phi$, $I^+ \subset \Phi$ denotes the set of roots in $\Phi$ which are nonnegative integer combinations of roots in $I$.

\begin{definition}  \label{def:unip}
    Let $I \subset \Phi$ be linearly independent, let $\BF$ be a finite field, and let $L \subseteq \BF$.
    We define the corresponding \emph{unipotent subgroup} $U_I(L)$ to be the subgroup of $G_\Phi(\BF)$ generated by all elements $\El{\zeta}{t}$ where $\Rt{\zeta} \in I$ and $t \in L$.
    (This notation $U_I(L)$ does not show dependence on $\Phi$ or $\BF$, but these will always be clear from context.)
\end{definition}
\begin{remark}
    In the preceding definition, 
    when $I$ is a base for~$\Phi$ and $L = \BF$, the resulting $U_I(L)$ is usually termed \emph{the} unipotent subgroup of $G_{\Phi}(\BF)$.  
    (Different bases lead to isomorphic subgroups.)
    As an example, \emph{the} unipotent subgroup of $G_{A_d}(\BF)$, namely $U_I(\BF)$ with $I = \{\IndPlain{i} - \IndPlain{i+1} : 1 \leq i < d\}$,  is isomorphic to the group of upper-triangular matrices in $\BF^{(d+1)\times(d+1)}$ with $1$'s on the diagonal.
\end{remark}
We will be particular interested in unipotent subgroups of the form $U_I(\BF_q[x]_{\leq 1})$; these were first studied by~\cite{KO18} in the $\Phi = A_d$ case, and by~\cite{OP22} in the general~$\Phi$ case.  
Here we are using the following notation:
\begin{notation}
    Let $\BF_q$ be a finite field, and $\BF_{q^m} \cong \BF_q[x]/(p(x))$ an extension of degree~$d$.  
    We write $\BF_q[x]^{(m)}_{\leq h}$ for the subset of $\BF_{q^m}$ given by polynomials in~$x$ over~$\BF_q$ of degree at most~$h$.
    We will usually drop the superscript~$(m)$ in this notation, as in a sense the set does not depend on~$m$ once $m > h$.
    (On the other hand, when $m \leq h+1$, this set is all of~$\BF_{q^m}$.)
\end{notation}

Concerning these unipotent subgroups, we have the very important structure theorem below (which follows from \cite[Prop.~3.14]{OP22}, a key component of which is \cite[Lem.~17]{Ste16}), which we slightly strengthen to give quantitative details on \emph{how} the presentation for~$H$ derives words:

\begin{theorem}[Polynomial presentation] \label{thm:structure}
    Let $q$ be an odd prime power and let $\BF_{q^m} \cong \BF_q[x]/(p(x))$, where $p(x)$ is irreducible over~$\BF_q$ of degree~$d$.  
    Let $I \subset \Phi \neq G_2$ be linearly independent. \begin{itemize}
        \item \emph{Generation:} $U_I(\BF_q[x]_{\leq 1}^{(m)})$ is generated by the symbols $\El{\zeta}{f}$ where $f \in \BF_q[x]^{(m)}_{\leq \height_I(\zeta)}$.
        \item \emph{Efficient presentation:} There is a function $\kappa : \BN \to \BN$\footnote{The function $\kappa$ has growth rate $\kappa(r) = O(r \log r)$, though we won't need this.}
        depending \emph{only} on $I, \Phi$ --- and not on~$q$ or~$m$ --- such that every trivial word of length at most~$r$ in the symbols $\El{\zeta}{f}$ ($\Rt{\zeta} \in I^+$, $f \in \BF_q[x]_{\leq 1}^{(m)}$) can be reduced to~$\Id$ by the use of at most $\kappa(r)$ applications of the linearity relations \Cref{eqn:linrels} and the commutator relations \Cref{eqn:conrels} (the first batch are commutator relations, the second batch are linearity relations).
    \end{itemize}    
\end{theorem}

\begin{proof}
    This follows simply by inspecting the proof of \cite[Prop.~3.14]{OP22} (which in turn follows Carter's proof~\cite[Thm.~5.3.3]{Car89} of \cite[Lem.~17]{Ste16}).
    To briefly sketch the proof, given any trivial word~$w$ of length~$r$ over the symbols $\El{\zeta}{f(x)}$, one repeatedly finds consecutive pairs of symbols that are misordered with respect to~$\prec$, and reorders them using the commutator relation. 
    This may increase the length of the word by~$O(1)$, but decreases the number of misorderings by~$1$.  One repeats this ($O(r \log r)$ times) until the word's symbols are ordered.  Then one repeatedly uses linearity relations ($O(r \log r)$ times in total) to merge consecutive symbols of the same type.
    The final result is a word of the form in \Cref{eqn:prodme}.
    But now all the symbols in this word must have entry~$0$, as the word itself equals~$\Id$, and \Cref{thm:structure}'s uniqueness statement implies $\prod_{\Rt{\zeta} \in I^+} \El{\zeta}{0}$ is the unique representation of~$\Id$.
\end{proof}

\begin{remark}
    Fix a total ordering~$\prec$ on the roots in~$I^+$ that is height-respecting; i.e., $\Rt{\zeta} \prec \Rt{\eta} \implies \height_{I}(\Rt{\zeta}) \leq \height_{I}(\Rt{\zeta})$. The elements of $U_I(\BF_q[x]_{\leq 1}^{(m)})$ can be written uniquely as
    \begin{equation}    \label{eqn:prodme}
        \prod_{\Rt{\zeta} \in I^+} \El{\zeta}{f_{\Rt{\zeta}}(x)},
    \end{equation}
    where the product is taken in the $\prec$ ordering, and 
    where $f_{\Rt{\zeta}}(x)$ has degree at most $\height_{I}(\Rt{\zeta})$.  (More precisely, $f_{\Rt{\zeta}}(x) \in 
    \BF_q[x]_{\leq \height_{I}(\Rt{\zeta})}^{(m)}$.)
\end{remark}

\begin{remark}
    Note from the presentation of~$H$ given above that as soon as $m$ exceeds the maximum height any root, the irreducible~$p(x)$ plays no role, as no entry ever has degree as large as~$p(x)$'s.  
    In other words, we could just as well think of the ``entries'' as formal polynomials in~$\BF_q[x]$, rather than as living inside the field~$\BF_{q^m}$.
    Thus for sufficiently large~$m$, the groups $H = U_I(\BF_q[x]_{\leq 1}^{(m)})$ \emph{do not actually depend on~$m$} (e.g., their size depends only on~$I$ and~$q$).  
    In the case of $\Phi = B_3$, it suffices to have $m \geq 6$, as the maximum height of any root is~$5$. In the case of $\Phi = A_3$, it suffices to have $m \geq 4$, as the maximum height of any root is~$3$.

    Conversely, we will also take an interest in ``base'' versions of these unipotent groups, when the entries are simply from~$\BF_q$.  Formally, this is captured by the $m = 1$, particularly when one takes $p(x) = x$.
\end{remark}

We now make two further important modifications to this presentation:

\begin{definition}
    For the presentation discussed in \Cref{thm:structure}, let us introduce the alias \[ \Eld{\zeta}{t}{i} := \El{\zeta}{tx^i}. \] Following \cite{KO21}, we call these the \emph{pure degree} symbols. 
    Observe that the linearity and commutator relations that have only pure degree symbols on the left-hand side also only have pure degree symbols on the right-hand side.
    We will term this subset of relations the \emph{pure degree Steinberg relations}.
\end{definition}

By linearity (\Cref{eqn:linrels}), every element in $U_I(\BF_q[x]^{(m)}_{\leq 1})$ of the form $\El{\zeta}{f}$ is a product of elements of the form $\Eld{\zeta}{t}{i}$.

\begin{theorem}[Pure degree presentation with aliases]\label{thm:structure4}
    Let $q$ be an odd prime power and let $\BF_{q^m} \cong \BF_q[x]/(p(x))$, where $p(x)$ is irreducible over~$\BF_q$ of degree~$d$. Let $I \subset \Phi \neq G_2$ be linearly independent and $I \subset J \subset I^+$. Suppose $m \geq \max_{\Rt{\zeta} \in I^+} \height_I(\Rt{\zeta})$.
    \begin{itemize}
        \item \emph{Generation:} $U_I(\BF_q[x]_{\leq 1}^{(m)})$ is generated by the pure-degree symbols $\Eld{\zeta}{t}{i}$ where $\Rt{\zeta} \in J$, $t \in \BF_q$, and $i \in [\height_I(\Rt{\zeta})]$.
        \item \emph{Efficient presentation:} There is a function $\kappa : \BN \to \BN$
        depending \emph{only} on $I,J, \Phi$ --- and not on~$q$ or~$m$ --- such that the following holds. For each $\Rt{\zeta} \in I^+ \setminus J, t \in \BF_q, i \in [\height_I(\Rt{\zeta})]$, define a new \emph{alias} symbol $\Eld{\zeta}{t}{i}$ to stand for some fixed product of elements $\Eld{\eta}{u}{j}$, $\eta \in I, u \in \BF_q, j \in [\height_I(\Rt{\eta})]$ multiplying to $\Eld{\zeta}{t}{i}$ in the group $U_I(\BF_q[x]_{\leq 1}^{(m)})$. Then every trivial word of length at most~$r$ in the symbols $\Eld{\zeta}{t}{i}$, $\Rt{\zeta} \in J,t \in \BF_q,i \in [\height_i(\Rt{\zeta})]$ can be reduced to~$\Id$ by the use of at most $\kappa(r)$ applications of the  linearity relations \Cref{eqn:linrels} and the commutator relations \Cref{eqn:conrels} applied only to the pure-degree symbol where aliases are expanded (the first batch are commutator relations, the second batch are linearity relations).
    \end{itemize}    
\end{theorem}

In the setting we care about below, $I$ will always have size $3$ and $J$ will be the set of all roots in $I^+$ which are nonnegative integer combinations of \emph{two} roots in $I$.

\subsubsection{Lifting}

Another important notion for us will be homomorphisms between $U_I(\BF)$ and $U_I(\tilde{\BF}[x]_{\leq 1})$ where $\tilde{\BF} \supset \BF$ is a field extension. We will later use these homomorphisms to ``lift'' (computer-generated) ``fillings'' of loops between the corresponding coset complexes. The homomorphisms were essentially studied in the $A_3$ case in \cite[Lemma 7.13]{KO21}.

\begin{theorem}
\label{thm:chev:general-lift}
    Let $I \subset \Phi$ be linearly independent, let $\BF$ be a finite field, and $\tilde{\BF} \supseteq \BF$ a field extension. Consider the two unipotent subgroups $U_I(\BF)$ and $U_I(\tilde{\BF}[x]_{\leq 1})$. Suppose we have field elements $t_{\Rt{\zeta},b} \in \tilde{\BF}$ for $\Rt{\zeta} \in I$, $b \in \{0,1\}$.
    
    Define a map $f : U_I(\BF) \to U_I(\tilde{\BF}[x]_{\leq 1})$ by specifying that 
    for
    $\Rt{\eta} = \sum_{\Rt{\zeta} \in I} c_{\Rt{\zeta}} \Rt{\zeta}$,
    \[
    f \El{\eta}{u} = \El{\eta}{u \prod_{\Rt{\zeta} \in I} (t_{\Rt{\zeta},1} x + t_{\Rt{\zeta},0})^{c_{\Rt{\zeta}}}}.
    \]
    Then this map is a homomorphism.
\end{theorem}
\begin{proof}
    It suffices to check that the $f$-image of every relation in the Steinberg presentation of $U_I(\BF)$ (the ``$m = 1$ case'' of \Cref{thm:structure}) is true within $U_I(\tilde{\BF}[x]_{\leq 1})$ (the general-$m$ of \Cref{thm:structure}).
    (This is sometimes called van Dyck's Theorem.  In fact, each $f$-image will be a Steinberg relation in the latter group.)
    It is easy to verify this for the linearity relations, so it remains to verify it for the commutation relations.
    
    In $U_I(\BF)$, commutation relations are of the form \[
    \Comm{\El{\eta}{t}}{\El{\theta}{u}} = \prod_{\substack{a \Rt{\eta} + b \Rt{\theta} \in \Phi \\ a,b \in \BN^+}} \El{a \eta + b \theta}{C^{\Rt{\eta},\Rt{\theta}} \cdot t^a u^b}.
    \] Suppose the roots $\Rt{\eta}$ and $\Rt{\theta}$ have expansions $\Rt{\eta} = \sum_{\Rt{\zeta} \in I} c_{\Rt{\zeta}} \Rt{\zeta}$ and $\Rt{\theta} = \sum_{\Rt{\zeta} \in I} d_{\Rt{\zeta}} \Rt{\zeta}$. Then $\Rt{a \eta + b \theta} = \sum_{\Rt{\zeta} \in I} (ac_{\Rt{\zeta}} + bd_{\Rt{\zeta}} )$. So by our definition of $f$,
    \begin{align*}
    f \El{\eta}{t} &= \El{\eta}{t \prod_{\Rt{\zeta} \in I} (t_{\Rt{\zeta},1} x + t_{\Rt{\zeta},0})^{c_{\Rt{\zeta}}}}, \\
    f \El{\theta}{u} &= \El{\theta}{u \prod_{\Rt{\zeta} \in I} (t_{\Rt{\zeta},1} x + t_{\Rt{\zeta},0})^{d_{\Rt{\zeta}}}}, \\
    f \El{a \eta + b \theta}{C^{\Rt{\eta},\Rt{\theta}} \cdot t^a u^b} &= \El{a\eta + b\theta}{C^{\Rt{\eta},\Rt{\theta}} \cdot t^a u^b \prod_{\Rt{\zeta} \in I} (t_{\Rt{\zeta},1} x + t_{\Rt{\zeta},0})^{ac_{\Rt{\zeta}} + bd_{\Rt{\zeta}}} }.
    \end{align*}
    Applying the commutator relation in $U_I(\tilde{\BF}[x]_{\leq1})$ gives:
    \begin{align*}
        f \Comm{\El{\eta}{t}}{\El{\theta}{u}} &= \Comm{\El{\eta}{t \prod_{\Rt{\zeta} \in I} (t_{\Rt{\zeta},1} x + t_{\Rt{\zeta},0})^{c_{\Rt{\zeta}}}}}{\El{\theta}{u \prod_{\Rt{\zeta} \in I} (t_{\Rt{\zeta},1} x + t_{\Rt{\zeta},0})^{d_{\Rt{\zeta}}}}} \\
        &= \prod_{\substack{a \Rt{\eta} + b \Rt{\theta} \in \Phi \\ a,b \in \BN^+}} \El{a\eta + b\theta}{C^{\Rt{\eta},\Rt{\theta}} \cdot \left(t \prod_{\Rt{\zeta} \in I} (t_{\Rt{\zeta},1} x + t_{\Rt{\zeta},0})^{c_{\Rt{\zeta}}}\right)^a \left(u \prod_{\Rt{\zeta} \in I} (t_{\Rt{\zeta},1} x + t_{\Rt{\zeta},0})^{d_{\Rt{\zeta}}}\right)^b} \\
        &= \prod_{\substack{a \Rt{\eta} + b \Rt{\theta} \in \Phi \\ a,b \in \BN^+}} \El{a\eta + b\theta}{C^{\Rt{\eta},\Rt{\theta}} t^a u^b \cdot \prod_{\Rt{\zeta} \in I} (t_{\Rt{\zeta},1} x + t_{\Rt{\zeta},0})^{ac_{\Rt{\zeta}} + bd_{\Rt{\zeta}}}} \\
        &= \prod_{\substack{a \Rt{\eta} + b \Rt{\theta} \in \Phi \\ a,b \in \BN^+}} f \El{a \eta + b \theta}{C^{\Rt{\eta},\Rt{\theta}} \cdot t^a u^b},
    \end{align*}
    and so the image of a commutator relation is indeed a (commutator) relation.
\end{proof}

Note that using the pure degree symbols, the image of an element $\El{\eta}{t}$ in \Cref{thm:chev:general-lift} is the multinomial expansion:
\begin{equation}\label{eq:chev:nonhom}
 \prod_{b_{\Rt{\zeta}} \in [1] : \Rt{\zeta} \in I} \Eld{\eta}{t \prod_{\Rt{\zeta} \in I} t_{\Rt{\zeta},b_{\Rt{\zeta}}}^{x_{\Rt{\zeta}}}}{\sum_{\Rt{\zeta} \in I} x_{\Rt{\zeta}} b_{\Rt{\zeta}}}.
\end{equation}
We call this general form a \emph{nonhomogeneous lift}, in contrast to one useful special case which simplifies the notation considerably is what we call the \emph{homogeneous lift}, where for every $\Rt{\zeta}$, either $t_{\Rt{\zeta},1}$ or $t_{\Rt{\zeta},0}$ is zero:
\begin{corollary}[Homogeneous lifting]
    Let $I \subset \Phi$ be linearly independent, let $\BF$ be a finite field, and $\tilde{\BF} \supseteq \BF$ a field extension. Consider the two unipotent subgroups $U_I(\BF)$ and $U_I(\tilde{\BF}[x]_{\leq 1})$. Suppose we have $(t_{\Rt{\zeta}} \in \tilde{\BF})_{\Rt{\zeta} \in I}$ and $(b_{\Rt{\zeta}} \in [1])_{\Rt{\zeta} \in I}$. There is a (unique) \emph{lift} homomorphism $f : U_I(\BF) \to U_I(\tilde{\BF}[x]_{\leq 1})$ such that for every $\Rt{\zeta} \in I$, \[
    f \El{\Rt{\zeta}}{1} = \Eld{\Rt{\zeta}}{ut_{\Rt{\zeta}}}{b_{\Rt{\zeta}}}.
    \]
    The image of an element of type $\Rt{\eta} = \sum_{\Rt{\zeta} \in I} c_{\Rt{\zeta}} \Rt{\zeta}$ is:
    \[
    f \El{\eta}{u} = \Eld{\eta}{u \prod_{\Rt{\zeta} \in I} t_{\Rt{\zeta}}}{\sum_{\Rt{\zeta} \in I} b_{\Rt{\zeta}}}.
    \]
\end{corollary}

\subsection{The Chevalley coset complexes}
Kaufman--Oppenheim~\cite{KO18} showed how to construct bounded-degree HDXs from coset complexes over $G_{A_d}(\BF) \cong \SL_{d+1}(\BF)$, and O'Donnell--Pratt~\cite{OP22} generalized their work to all Chevalley groups. (See~\cite{GV24} for further generalizations.)
Let us recap the main construction and theorem from~\cite{OP22}:
\begin{definition}  \label{def:opcc}
    (\cite{OP22}.) Let $\BF_q$ be a finite field of characteristic at least~$3$,\footnote{See \Cref{foot:2}.} 
    and let $\BF_{q^m} \cong \BF_q[x]/(p(x))$, where $p(x)$ is an irreducible of degree~$m$.
    Let $\Phi$ be a rank-$d$ root system, and $\Pi \subset \Phi$ a base.
    Define the ``special'' root set $\CS = \Pi \cup \{\Rt{\zeta}\}$, where $\Rt{\zeta} = -\sum_{\Rt{\eta} \in \Pi} \Rt{\eta} \in \Phi$.
    Now for each $I \subsetneq \CS$, we define a unipotent subgroup of the type discussed in \Cref{thm:structure}:
    \begin{equation}
       H_{I} = H_I(m)
       \coloneqq
       U_{I}(\BF_q[x]^{(m)}_{\leq 1}),
       \quad 
       \text{with the shorthand }
       H_{\setminus \Rt{\eta}} \coloneqq H_{\CS\setminus \{\Rt{\eta}\}}.
    \end{equation}
    
    Finally, we define the $d$-dimensional coset complex
    \begin{equation}
        \mathfrak{K}\Phi_q(m) = \CoCo\Bigl(G_\Phi(\BF_{q^m});\{H_{\setminus\Rt{\eta}} : \Rt{\eta} \in \CS\}\Bigr).
    \end{equation}
\end{definition}

We will need to record a couple of facts about these complexes, before recalling the main theorem about them:
\begin{proposition}\label{prop:link-ko}
    (\cite[Thm.~3.18]{OP22}.) In $\mathfrak{K}\Phi_q(m)$, suppose vertex~$\sigma$ is a coset of $H_{\setminus \Rt{\eta}}$, $\Rt{\eta} \in S$.  Then the link of~$\sigma$ is isomorphic to the coset complex
    $\CoCo\Bigl(H_{\setminus \Rt{\eta}}; \{H_{ \setminus \{\Rt{\eta}, \Rt{\zeta}\}} : \Rt{\zeta} \in S, \Rt{\zeta} \neq \Rt{\eta}\}\Bigr)$.
\end{proposition}
\begin{notation}
    Let us introduce some shorthands for the unipotent groups we are interested in, referring
    back to \Cref{sec:prelim:root-systems}.
    In the $A_3$ case, all the unipotent groups are isomorphic, and we will write $U_{A_3}(\cdot)$ for the case of $I = \{\Rt{\alpha}, \Rt{\beta}, \Rt{\gamma}\}$.
    In the $B_3$ case, there are two isomorphic subgroups of interest.  We will write $U_{B_3}^{\mathrm{sm}}(\cdot)$ for the case of deleting a long root; say, $I = \{\Rt{\beta}, \Rt{\psi}, \Rt{\omega}\}$. We will write $U_{B_3}^{\mathrm{lg}}(\cdot)$ for the case of deleting a short root; say, $I = \{\Rt{\alpha}, \Rt{\beta}, \Rt{\psi}\}$.  
\end{notation}

\begin{proposition}\label{prop:diameter}
    Fix $d$, the rank of~$\Phi$.  Then the diameter of each vertex-link described in \Cref{prop:link-ko} is bounded by a universal constant~$r_0 = r_0(d)$.
\end{proposition}
\begin{proof}
    To see this, one simply has to inspect the proof~\cite[Cor.~3.19]{OP22} that these vertex-links are connected.
    That proof uses the group-theoretic condition for connectedness from \Cref{prop:cc}, and one can employ its quantitative form to establish the proposition.
    Specifically, one needs to observe that the proof of~\cite[Cor.~3.19]{OP22} actually shows  $r_0$-bounded generation of $H_{\setminus \{\Rt{\eta}\}}$ from $\bigcup\{H_{S \setminus \{\Rt{\eta}, \Rt{\zeta}\}} : \Rt{\zeta} \in S, \Rt{\zeta} \neq \Rt{\eta}\}$ for some~$r_0$ depending only on~$d$ (and not on~$q$ or~$m$).
    The key point is that inspecting the proof of the central~\cite[Lem.~3.13]{OP22}, its generation uses a number of elements that depends only at worst on the number of roots in~$\Phi$ and the maximum ``height'' of any root in~$\Phi$, all of which are bounded by a function only of~$d$.
\end{proof}

(In the case $d=3$, it is not hard to show $r_0 \leq 20$.)

We now give the main theorem from~\cite{OP22} about these coset complexes (excluding the case ``$G_2$'' for simplicity; see~\cite{Pra23} for its treatment).
\begin{theorem} \label{thm:op}
    (\cite[Thm.~3.6, Cor.~3.7]{OP22}.) Let $\Phi \neq G_2$ be of rank~$d$. Then the coset complex (family) from \Cref{def:opcc} is strongly explicit family, on $q^{\Theta(m)}$ vertices, of bounded degree $D = q^{O(1)}$, and has $j$-spectral expansion parameter at most $\frac{1}{\sqrt{q/2} - (d-1-j)}$.
\end{theorem}

For the $3$-dimensional complexes in these families, we may hope to bound the $1$-cone radius of their vertex-links.
This would let us establish good $1$-coboundary expansion via \Cref{thm:cones1}, as we show in the next proposition (which just verifies some simple conditions).
\begin{proposition} \label{prop:1cone}
    Let $\mathfrak{K}$ be a $3$-dimensional coset complex as in \Cref{def:opcc}.  
    If the $1$-cone radius of each vertex-link $\mathfrak{K}_\sigma$ at most~$R_0$, then $\mathfrak{K}$ has $1$-coboundary expansion at least~$1/(3R_0)$.
\end{proposition}
\begin{proof}
    Given \Cref{thm:cones1}, there are only two small things to check.
    First, we need that when $\mathfrak{K}(3)$ is (naturally) weighted by the uniform distribution, the induced weighting on the $2$-faces of each vertex-link $\mathfrak{K}_\sigma$ is also uniform.
    This will follow if we can show that each $2$-face in $\mathfrak{K}$ is contained in the same number of $3$-faces.
    Now a $2$-face in $\mathfrak{K}$ is of the form $\{xH_{\setminus \Rt{\zeta_1}}, xH_{\setminus \Rt{\zeta_2}}, xH_{\setminus \Rt{\zeta_3}}\}$ for $\Rt{\zeta_1}, \Rt{\zeta_2}, \Rt{\zeta_3}$ being three of the four roots in~$S$.
    It's not hard to see that the  $3$-faces this is contained in are in $1$-$1$ correspondence with the elements of $H' \coloneqq H_{\setminus \Rt{\zeta_1}} \cap H_{\setminus \Rt{\zeta_2}} \cap H_{\setminus \Rt{\zeta_3}}$.
    And from the definition of $H_{\setminus \Rt{\zeta}}$ and \Cref{thm:structure}, one sees that $H' = U_{\{\Rt{\zeta_4}\}}(\BF_q[x]_{\leq 1})$, where $\Rt{\zeta_4}$ is the fourth root in~$S$. Thus  $|H'| = q^2$, independent of the starting $2$-face, as desired.

     The second small thing to check is that for each vertex-link $\mathfrak{K}_\sigma$, there is a group of automorphisms $G$ that acts transitively on its $2$-faces.  But this is automatic from \Cref{prop:cc}, since $\mathfrak{K}_\sigma$ is a coset complex.
\end{proof}
Now the following proposition gives us a concrete method to bound $1$-cone radii.
(Recall we say a loop~$L$ is ``$\BF_2$-filled'' by a collection of triangles if the $\partial_2$-boundary of the $2$-chain formed by the triangles equals the $1$-chain formed by~$L$.)
\begin{proposition} \label{prop:last-junk}
    Let $\mathfrak{K}$ be a $3$-dimensional coset complex as in \Cref{def:opcc}, and $\mathfrak{K}_\sigma$ some vertex-link.  
    Suppose that each loop~$L$ in (the $1$-skeleton of)~$\mathfrak{K}_\sigma$ of length at most~$r_1 \coloneqq 2r_0+1$ can be $\BF_2$-filled by at most~$R_0$ triangles, where $r_0 = r_0(3)$ is the universal constant from \Cref{prop:diameter}.  
    Then $\mathfrak{K}_\sigma$ has $1$-coboundary expansion at least~$1/(3R_0)$.
\end{proposition}
\begin{proof}
    By \Cref{prop:diameter}, the diameter of every vertex-link is at most~$r_0$.
    Thus, all the paths $P_{uv}$ arising in \Cref{def:1cone} need not have length more than~$r_0$; hence the loops formed by $P_{uv}$, $\{v,w\}$, $P_{uw}$ need not have length more than~$r_1$.
    Thus our hypothesis gives us  $\BF_2$-fillings of all such loops with at most~$R_0$ triangles, thus bounding the $1$-cone radius of each vertex-link by~$R_0$.  Now we are done by \Cref{prop:1cone}.
\end{proof}

We are now ready to put everything together; by virtue of \Cref{thm:op} (which in particular shows that the spectral expansion parameter can be made arbitrarily small by taking~$q$ sufficiently large), \Cref{prop:last-junk},  \Cref{thm:workhorse}, and \Cref{cor:workhorse}, we get:
\begin{theorem}[Inception Theorem] \label{thm:inception}
    Let $\Phi$ be a root system of rank~$3$. 
    Let $R_0$ be any constant, and suppose $q$ is an odd prime power that is sufficiently large (as a function of~$R_0$).
    Further suppose that the coset complex 
    $\CoCo\Bigl(H_{\setminus \{\Rt{\eta}\}}; \{H_{S \setminus \{\Rt{\eta}, \Rt{\zeta}\}} : \Rt{\zeta} \in S, \Rt{\zeta} \neq \Rt{\eta}\}\Bigr)$
    described in \Cref{prop:link-ko} and \Cref{thm:structure} (which depends only on~$\Phi$, $q$)
    has the property that every loop of length at most~$r_1$ (the universal constant from \Cref{prop:last-junk}) can be $\BF_2$-filled by at most~$R_0$ triangles.
    
    Then the bounded-degree $3$-dimensional coset complex $(\mathfrak{K}\Phi_q(m))_{m \geq 6}$ (which has the HDX properties described in \Cref{thm:op}) has $1$-cosystolic expansion at least $(\Omega(1/R_0), \Omega(1/R_0))$; and, its $2$-skeleton has topological expansion~$\Omega(1/R_0^4)$.
\end{theorem}
We refer to this as the ``Inception Theorem'' since it is, in a sense, where the main work \emph{starts}.  
This theorem is essentially due to Kaufman--Oppenheim~\cite{KO21}; we have only made some minor adaptations to present it at a level of generality suitable for studying all rank-$3$ root systems.

Perhaps the main work of Kaufman and Oppenheim (see \cite[Sec.~7]{KO21}) was to establish the filling property needed by the Inception Theorem in the $\Phi = A_3$ case.  They did this purely group-theoretically, relying heavily on a result of Biss and Dasgupta~\cite{BD01} concerning presentations of the unipotent group (of $\SL_{4}(\BF)$).  
As discussed further in \Cref{sec:bd}, this strategy does not seem to work in the $\Phi = B_3$ case.
Thus we will need an additional tool to construct $\BF_2$-fillings, the subject of the next section.

%% file: sections/04-triangulations.tex
\section{Triangulations and derivations in coset complexes} \label{sec:dehn}
As we will see, in a coset complex $\CoCo(G; \{H_0, \dots, H_d\})$, to each closed loop of edges one can associate a word over symbols from $H_0, \dots, H_d$ that is trivial in~$G$.
Recall \Cref{thm:lanner}, which implies that if every such word can be proven to equal~$\Id$ using only relations that hold within~$H_i$'s, then the complex is simply connected.  In turn, this implies that every closed loop is $\BF_2$-fillable by triangles.  
Naturally, this suggests that one should be able to turn derivations of $w = \Id$ into triangulations of loops corresponding to~$w$.
In this section, we not only give a direct and quantitative proof of this (known) fact, but we augment it to account for cases when there are trivial words~$w$ in~$G$ that \emph{cannot} be shown to equal~$\Id$ just from ``in-subgroup'' relations.
What we show is that if any such short word~$w$ \emph{happens} to have a small $\BF_2$-filling within the complex, then we can treat~$w$ as a ``new'' relation in a subsequent group-theoretic quest to $\BF_2$-triangulate all closed loops.

\subsection{Triangulations of loops in coset complexes}

In this section only, it will be convenient to work with coset \emph{multi}complexes, which are a generalization of simplicial complexes in which the faces are multisets rather than sets.

\begin{definition}
    Let $G$ be a group and $\CH$ a set of subgroups, $|\CH| = d+1$.
    We write $\CoCoCo(G; \CH)$ for the associated \emph{coset multicomplex}, in which the maximal multisets are all those $(d+1)$-multisets of the form $\{xH_0, xH_2, \dots, xH_d\}$ for $H_0, H_1, \dots, H_d \in \CH$ not necessarily distinct.
    If we only include these maximal multisets with $H_0, \dots, H_d$ distinct (and hence encompassing all of~$\CH$), we would recover the usual coset complex, $\CoCo(G;\CH)$.
    We also naturally generalize the notions of chains and (co-)boundary to such complexes.
\end{definition}

\begin{remark}
    Throughout this section we will assume our multicomplexes have $d \geq 2$; indeed, we will ultimately only need the case $d=2$.
\end{remark}
\begin{definition}\label{def:free}
    In the setting of preceding definition we think of each $H_i$ as a ``color'', and for any $x \in H_i \in \CH$, we introduce the symbol~$\col{x}{H_i}$, called a ``colored element''. 

    We write $F_{\CH}$ for the free group generated by all the colored elements, together with the additional relations $\col{x}{H_i}^{-1} = \col{x^{-1}}{H_i}$.
    Finally, we write $\phi : F_{\CH} \to G$ for the  homomorphism that maps the word $\col{x_0}{H_0}\col{x_1}{H_1}\cdots\col{x_{\ell-1}}{H_{\ell-1}}$ to $x_0 x_1 \cdots x_{\ell-1}$. 
    (Here $x_j \in H_j \in \CH$.)
\end{definition}

\begin{definition}
    Let $\hat{w} = \col{x_0}{H_0}\col{x_1}{H_1}\cdots\col{x_{\ell-1}}{H_{\ell-1}} \in F_{\CH}$. 
    (We typically use a $\hat{\text{hat}}$ to denote a word in~$F_{\CH}$.)
    We associate to it the following walk $\CL(\hat{w})$ of length~$t$ in (the skeleton of) $\CoCoCo(G;\CH)$:
    \begin{equation}
        \Id H_0 
        \to x_0H_1 
        \to x_0x_1H_2 
        \to \cdots 
        \to x_0x_1\cdots x_{\ell-2}H_{\ell-1}
        \to x_0 x_1 \cdots x_{\ell-1} H_0 = \phi(\hat{w}) H_0,
    \end{equation}
    which is a closed walk if $\phi(\hat{w}) = \Id$.
    (Above, the presence of the edge between $x_0\cdots x_{i-1} H_i$ and $x_0 \cdots x_i H_{i+1}$ in the skeleton of $\CoCoCo(G;\CH)$ is ``witnessed'' by the element $x_0 \cdots x_{i}$, noting that $x_0\cdots x_{i-1} H_i = x_0\cdots x_{i} H_i$ because $x_i \in H_i$.)
\end{definition}
\begin{remark}
    Consideration of the above definition lets one easily conclude Fact~3 from \Cref{prop:cc}, that $\CoCo(G;\CH)$ is connected iff $G$ is generated by the subgroups in $\CH$ iff $\phi$ is onto.
\end{remark}
\begin{notation}
    Let $L$ be any walk in (the skeleton of) $\CoCoCo(G;\CH)$.  
    We write $[L]$ for the $1$-chain given by the $\BF_2$-sum of~$e$ over all  (undirected) edges~$e$ in~$L$ (where an edge is a $1$-face aka cardinality-$2$-multiset in the multicomplex).
    In particular, an edge~$e$ (which may be a self-loop) occurs with coefficient~$1 \in \BF_2$ in~$[L]$ iff $e$ is used an odd number of times in~$L$.
\end{notation}
\begin{fact}
    If $L^{-1}$ denotes the reversal of~$L$, we have $[L^{-1}] = [L]$.
    Moreover, if $L$ is a closed walk and $L'$ denotes a cyclic shift of~$L$, we still have $[L'] = [L]$.
\end{fact}
\begin{definition}
    If $L$ is a closed walk in $\CoCoCo(G;\CH)$, we write $\triangle(L)$ for the cardinality of the smallest set $S$ of triangles (cardinality-$3$-multisets in $\CoCoCo(G;\CH)$) such that, when $S$ is viewed as a $2$-chain, we have $\partial_2 S = [L]$ (i.e., $\BF_2$-summing the $3|S|$ boundary edges yields~$[L]$).
    We write $\triangle(L) = \infty$ if there is no such~$S$.
    We extend this notation by writing $\triangle(\hat{w})$ in place of $\triangle(\CL(\hat{w}))$ whenever $\hat{w}$ is a word in $F_{\CH}$ satisfying $\phi(\hat{w}) = \Id$.
\end{definition}
\begin{definition}
\label{def:translate}
    Given any multiset $E = \{xH_1, \dots, x H_e\}$ in the multicomplex $\CoCoCo(G;\CH)$, and any $y \in G$, the ``translated'' multiset $yE = \{y H_1, \dots, y H_e\}$ is also in $\CoCoCo(G;\CH)$.
    We may naturally extend this notation to translations $yL$ of closed walks~$L$, or to translations $yS$ of $2$-chains of triangles~$S$.
\end{definition}
The following facts are straightforward:
\begin{fact}    \label{fact:id}
    For any $H_i$, the length-$1$ word $\hat{w} = \col{\Id}{H_i} \in F_{\CH}$ has 
    $\triangle(w) = 1$, since $[\CL(w)]$ is the self-loop at $\Id H_i$, which is the boundary of the single (degenerate) triangle $\{\Id H_i, \Id H_i, \Id H_i\}$.
\end{fact}
\begin{fact}    \label{fact:recolor}
    For any $s \in H_i \cap H_j$, the ``recoloring'' word $\hat{w} = \col{s}{H_i}\col{s^{-1}}{H_j}$
    has $[\CL(\hat{w})] = 0$ and hence 
    $\triangle(\hat{w}) = 0$.
\end{fact}
\begin{fact}    \label{fact:in-subgroup}
    Let $\hat{w} = \col{x_0}{H}\col{x_1}{H}\cdots\col{x_{t-1}}{H}$ be an ``in-subgroup'' word with $\phi(\hat{w}) = \Id$, where $H\in \CH$.
    Then $\CL(\hat{w})$ consists of~$t$ self-loops at $\Id H$, and hence 
    $\triangle(\hat{w}) \leq 1$.
\end{fact}
\begin{fact}    \label{fact:translate}
    Let $L$ be a closed walk and let $S$ be a $2$-chain with $\partial_3 S = [L]$.  
    Then for any translation $y \in G$ we have that $\partial_2 (yS) = [yL]$.
    Hence $\triangle(yL) = \triangle(L)$.
\end{fact}
\begin{fact}    \label{fact:cyclic}
    Let $\hat{w} = \col{x_0}{H_0}\col{x_1}{H_1}\cdots\col{x_{t-1}}{H_{t-1}}$ in $F_{\CH}$ with $\phi(\hat{w}) = \Id \in G$.
    Then $\CL(\hat{w}^{-1})$ is the reverse of~$\CL(\hat{w})$; hence $[\CL(\hat{w}^{-1})] = [\CL(\hat{w})]$ and $\triangle(\hat{w}^{-1}) = \triangle(\hat{w})$.
    
    Moreover, if $\hat{w}^i = \col{x_i}{H_i}\col{x_{i+1}}{H_{i+1}}\cdots\col{x_{t-1}}{H_{t-1}}\col{x_0}{H_0}\cdots\col{x_{i-1}}{H_{i-1}}$ is some cyclic shift of~$\hat{w}$, then $x_0x_1\cdots x_{i-1}\CL(\hat{w}^i) = \CL(\hat{w})$, and hence $\triangle(\hat{w}^i) = \triangle(\hat{w})$.
\end{fact}
The following proposition is important; informally, it says that if ``$xy=\Id$'' has a small triangulation, and ``$y=z$'' has a small triangulation, then we can deduce ``$xz=\Id$'' has a small triangulation.
\begin{proposition} \label{prop:stitch}
    Let $\hat{x},\hat{y},\hat{z}$ be words in $F_{\CH}$ with $\phi(\hat{x}\hat{y}) =
    \phi(\hat{y}\hat{z}^{-1}) = \Id$, hence $\phi(\hat{x}\hat{z}) = \Id$.
    Then $\triangle(\hat{x}\hat{z}) \leq \triangle(\hat{x}\hat{y}) + \triangle(\hat{y}\hat{z}^{-1}) + 2$.
\end{proposition}
\begin{proof}
    Let $S_1$ (respectively,~$S_2$) be a $2$-chain  of triangles with boundary $[\CL(\hat{x}\hat{y})]$ (respectively, $[\phi(\hat{x})\CL(\hat{y}\hat{z}^{-1})]$) and cardinality $\triangle(\hat{x}\hat{y})$ (respectively, $\triangle(\hat{y}\hat{z}^{-1})$, using \Cref{fact:translate}); here we may assume $\triangle(\hat{x}\hat{y}), \triangle(\hat{y}\hat{z}^{-1}) < \infty$ else there is nothing to prove.
    We will identify a set $S^*$ of two triangles such that $[\CL(\hat{x}\hat{y})] + [\phi(\hat{x})\CL(\hat{y}\hat{z}^{-1})] + \partial_2S^* = [\CL(\hat{x}\hat{z})]$ (as $1$-chains over~$\BF_2$).
    Then the $2$-chain $S_1 + S_2 + S^*$ has boundary $[\CL(\hat{x}\hat{z})]$, from which the proposition follows.
    So it suffices to show that the $1$-chain
    \begin{equation}    \label{eqn:stitch}
        [\CL(\hat{x}\hat{y})] + [\CL(\hat{x}\hat{z})] + [\phi(\hat{x})\CL(\hat{y}\hat{z}^{-1})]
    \end{equation}
    is the boundary of some two triangles in $\CoCoCo(G;\CH)$.
    
    We will assume $\hat{x},\hat{y},\hat{z}$ have length at least~$1$, as the case when one or more has length~$0$ is easier.
    Let $\hat{x} = \col{x_0}{A_0} \cdots \col{x_{j-1}}{A_{j-1}}$, where $x_i \in A_i \in \CH$, and similarly write $\hat{y} = \col{y_0}{B_0} \cdots \col{y_{k-1}}{B_{k-1}}$ and $\hat{z} = \col{z_0}{C_0} \cdots \col{z_{\ell-1}}{C_{\ell-1}}$.
    Then one may check that we have the following closed walks:
    \begin{gather}
        \CL(\hat{x}\hat{y}) = \underbrace{\Id A_0 \to x_0 A_1 \to \cdots \to \phi(\hat{x})A_{j-1}}_{P_0} \to \underbrace{\phi(\hat{x})B_0 \to 
        \phi(\hat{x})y_0B_1 \to 
        \cdots \to \Id B_{k-1}}_{P_1}\ (\to \Id A_0) \\
        \CL(\hat{x}\hat{z}) = \underbrace{\Id A_0 \to x_0 A_1 \to \cdots \to \phi(\hat{x})A_{j-1}}_{P_0} \to \underbrace{\phi(\hat{x})C_0 \to 
        \phi(\hat{x})z_0C_1 \to 
        \cdots \to \Id C_{\ell-1}}_{P_2}\ (\to \Id A_0) \\
        \phi(\hat{x})\CL(\hat{y}\hat{z}^{-1}) = \underbrace{\phi(\hat{x})B_0 \to \phi(\hat{x})y_0B_1 \to \cdots \to \Id B_{k-1}}_{P_1} \to \underbrace{\Id C_{\ell-1} \to z_{\ell-1}^{-1}C_{\ell-2} \to \cdots \to \phi(\hat{x}) C_0}_{P_2^{-1}} \ (\to \phi(\hat{x})B_0)
    \end{gather}
    From this we can compute that the $1$-chain in \Cref{eqn:stitch} simplifies significantly, with the $P_0$, $P_1$, and $P_2$ portions canceling (over~$\BF_2$).
    The six uncanceled edges are
    \begin{align}
        \phi(\hat{x})A_{j-1} &\to \phi(\hat{x})B_0, & 
        \Id B_{k-1} &\to \Id A_0,\\
        \phi(\hat{x})A_{j-1} &\to \phi(\hat{x})C_0, & 
        \Id C_{\ell-1} &\to \Id A_0,\\
        \phi(\hat{x}) C_0 &\to \phi(\hat{x})B_0, & \Id B_{k-1} &\to \Id C_{\ell-1},
    \end{align}
    which indeed form (the boundary of) two possibly degenerate triangles, as claimed.
\end{proof}
Using \Cref{fact:cyclic}, we immediately conclude:
\begin{corollary}   \label{cor:stitch}
    Let $\hat{p},\hat{q},\hat{u},\hat{v}$ be words in $F_{\CH}$ with $\phi(\hat{p}\hat{u}\hat{q}) = \phi(\hat{u}\hat{v}^{-1}) = \Id$,
    hence $\phi(\hat{p}\hat{v}\hat{q}) = \Id$.
    Then $\triangle(\hat{p}\hat{v}\hat{q}) \leq \triangle(\hat{p}\hat{u}\hat{q}) + \triangle(\hat{u}\hat{v}^{-1}) + 2$.
\end{corollary}
Finally, by repeatedly using the above with the ``recoloring'' words from \Cref{fact:recolor}, we obtain the following useful lemma:
\begin{lemma}   \label{lem:recoloring}
    Let $\hat{w} = \col{x_0}{A_0} \cdots \col{x_{\ell-1}}{A_{\ell-1}}$ (with $x_i \in A_i \in \CH$) have $\phi(w) = \Id$.
    Let $\hat{w}'$ be a valid ``recoloring'', $\hat{w}' = \col{x_0}{B_0} \cdots \col{x_{\ell-1}}{B_{\ell-1}}$ (with $x_i \in B_i \in \CH$).
    Then $\triangle(\hat{w}') \leq \triangle(\hat{w}) + 2|\hat{w}|$.
\end{lemma}

\subsection{Derivations imply triangulations}
\begin{definition}  \label{def:notationy}
    Let $G$ be a group, let $\CH$ be a set of subgroups, and let $S = \bigcup \CH$, regarded as a set of generating symbols (as for a group presentation).
    For integers $\ell \geq 3$, $t \geq 1$, we define the \emph{$(\ell,t)$-bounded relators} $R_{\ell,t}$ to be the set of words $r = x_0 x_1 \cdots x_{\ell'-1}$ over~$S$ of length $\ell' \leq \ell$ such that: (i)~$r = \Id$ when the symbols are interpreted as group elements in~$G$; (ii)~there is a ``coloring'' $\hat{r} = \col{x_0}{H_0}\col{x_1}{H_1} \cdots \col{x_{\ell'-1}}{H_{\ell'-1}}$ (with $x_i \in H_i \in \CH$) satisfying $\triangle(\hat{r}) \leq t$.
\end{definition}
\begin{remark} \label{rem:in-subgroup}
    By \Cref{fact:in-subgroup}, $R_{\ell,t}$ always contains all trivial ``in-subgroup'' words of length at most~$\ell$; i.e., all $w$ where $x_0, \dots, x_{\ell'-1}$ are in a single $H \in \CH$, and where $x_0 \cdots x_{\ell'-1} = \Id$ when the symbols $x_i$ are interpreted as elements in~$H$.
    In particular, it includes all length-$3$ words of the form $st\Id$ where $s,t$ are group-theoretic inverses.
\end{remark}
\begin{remark}
    Given $G, \CH, S, R_{\ell,t}$ as above, the reader may find it helpful in what follows to keep in mind the group presentation $\CP_{\ell,t} = \langle S \mid R_{\ell,t}\rangle$, and the associated Dehn function.
    Strictly speaking, we won't quite refer to these concepts for a few minor technical reasons: we prefer not to have formal inverse symbols; we prefer to think of reducing trivial words to~$\Id$ rather than to the empty word; and, we prefer to freely allow cyclic shifts and inverses.
\end{remark}

\begin{definition}
    We define an equivalence relation~$\sim$ on words over~$S$ by specifying that $v \sim u$ if $v$ is a cyclic shift of either~$u$ or~$u^{-1}$, where (as expected) $u^{-1}$ denotes the word formed by reversing~$u$ and replacing each symbol with the symbol for its group-theoretic inverse.
    (The semantics of $v \sim u$ is ``$v = \Id$ iff $u = \Id$''.)
    We define an \emph{equation} to be a string ``$y=z$'', where $y,z$ are words, and we also identify this equation with the word $yz^{-1}$. 
    Finally, we say equation ``$y=z$'' is \emph{derived from equation ``$w=x$'' via relator $r \in R_{\ell,t}$} if (for words $p,q,u,v$) 
    \begin{equation}    \label{eqn:reduce}
        \text{``$w=x$''} \sim puq, \quad 
        \text{``$y=z$''} \sim pvq, \quad 
        r \sim uv^{-1}.
    \end{equation}
    In other words, up to equivalences, $x'=y'$ can be obtained from $x=y$ by substitution of an equation equivalent to~$r$.
    (We may also describe the above situation by saying that word $pvq$ is derived from word $puq$ via~$r$.)
    Note that it is equivalent to say that``$w=x$'' is derived from ``$y=z$'' via~$r$.
\end{definition}    
\begin{fact}    \label{fact:preserve-triv}
    If ``$x=y$'' is true when the symbols are interpreted in~$G$, and ``$y=z$'' is derived via~$r \in R_{\ell,t}$, then ``$y = z$'' is also true in~$G$.
\end{fact}
\begin{remark}
    When studying the Dehn function with respect to a presentation, one typically also considers ``free reduction'' (replacing $ss^{-1}$ by the empty word, $s \in S$) and its opposite, ``free expansion''.
    But these are already included in the preceding formulation via two derivation steps.\footnote{Except in the case where free reduction reduces a word to the empty word. However, we will only ever consider reducing to~$\Id$.}
    For example, to implement free reduction on a word of the form $wss^{-1}x$ where at least one of $w,x$ is nonempty --- say~$x = x_1 \cdots x_m$, $m \geq 1$ --- we can first reduce $ss^{-1}$ to~$\Id$ per \Cref{rem:in-subgroup}.
    We can then perform ``$\Id$-deletion'' by reducing $\Id x_1$ to $x_1$ (again by \Cref{rem:in-subgroup}, since ``$\Id x_1 = x_1$'' is ``in-subgroup'' for any subgroup containing~$x_1$).
\end{remark}
\begin{definition}
    Let $w$ be a word over~$S$.
    We write $\delta_{\ell,t}(w)$ for the least number of derivation steps, via relators from $R_{\ell,t}$, that it takes to reduce $w$ to the word~$\Id$.
    (Or, we write $\delta_{\ell,t}(w) = \infty$ if this is not possible.)
    If $w \sim \text{``$x=y$''}$, we also write this as $\delta_{\ell,t}(x=y)$, and refer to it as the least number of  steps required to 
    derive ``$x=y$'' via~$R_{\ell,t}$.
\end{definition}

Finally, we have:
\begin{theorem} \label{thm:relative-dehn}
    Let $\hat{w}$ be a word in $F_{\CH}$ with $\phi(\hat{w}) = \Id$ and let $w$ be the word over~$S$ formed by forgetting the colors of~$\hat{w}$.

    Then $\triangle(\hat{w}) \leq (t+4\ell+2)\cdot \delta_{\ell,t}(w)+1$.
\end{theorem}
\begin{proof}
    Let $\delta = \delta_{\ell,t}(w) < \infty$ (else there is nothing to prove) and let $w_0, w_1, \dots, w_\delta$ denote a sequence of words with $w_0 = \Id$ and $w_\delta = w$ in which each $w_{i}$ is derived from $w_{i-1}$ via some relator in $R_{\ell,t}$.
    We make the simple observation that 
    \begin{equation}    \label{eqn:len-bound}
        |w_\delta| \leq \ell \delta,
    \end{equation}
    since any application of an $R_{\ell,t}$-relator 
    can increase length by at most~$\ell$ (in fact, $\ell-1$).
    We also note from \Cref{fact:preserve-triv} and induction that each $w_i$ equals~$\Id$ when its symbols are interpreted in~$G$.
    Thus for any coloring $\hat{w}_i \in F_{\CH}$ of~$w_i$, we have $\phi(\hat{w}_i) = \Id$ and hence $\triangle(\hat{w}_i)$ is well defined (albeit it might be~$\infty$).

    Our main goal will be to inductively define colorings $\hat{w}_0, \hat{w}_1, \dots, \hat{w}_\delta$ in $F_{\CH}$ (starting from an arbitrary coloring $\hat{w}_0$ of $w_0 = \Id$) and show
    \begin{equation}    \label{ineq:induct}
        \triangle(\hat{w}_{i}) \leq \triangle(\hat{w}_{i-1}) + t+2\ell + 2.
    \end{equation}
    This establishes $\triangle(\hat{w}_\delta) \leq (t+2\ell+2)\cdot \delta + 1$ by induction (note that $\triangle(\hat{w}_0) = 1$ since $[\hat{w}_0]$ is a single self-loop).
    Then the proof is completed by observing that although the coloring~$\hat{w}$ of~$w$ given in the theorem might not equal the final coloring~$\hat{w}_\delta$ we produced, we have 
    \begin{equation}
        \triangle(\hat{w}) \leq \triangle(\hat{w}_\delta) + 2|w| \leq \triangle(\hat{w}_\delta) + 2\ell \delta 
    \end{equation}
    by \Cref{lem:recoloring}.
   
    To define $\hat{w}_i$ and establish \Cref{ineq:induct}, suppose as in \Cref{eqn:reduce} that
    \begin{equation}
        w_{i-1} \sim puq, \quad 
        w_i \sim pvq, \quad 
        r \sim uv^{-1} \in R_{\ell,t},
    \end{equation}
    and all these words are~$\Id$ when interpreted in~$G$.
    By definition of $R_{\ell,t}$ and \Cref{fact:cyclic}, we may infer that there is a coloring $u' = \col{u_0}{C_0} \col{u_1}{C_1} \cdots \col{u_{|u|-1}}{C_{|u|-1}} \in F_{\CH}$ of~$u$ and similarly a coloring $v'$ of~$v$ such that $r' = (u')(v')^{-1} \in F_{\CH}$ has $\triangle(r') \leq t$.
    Note that the coloring of the $u$-symbols in~$r'$ might not agree with the coloring of the $u$-symbols in $\hat{w}_i$.
    However, if we let $\hat{r}$ be a recoloring of~$r'$ in which the $u$-colors \emph{do} agree with those in~$\hat{w}_{i-1}$, we may conclude from \Cref{lem:recoloring} that \begin{equation}
        \triangle(\hat{r}) \leq \triangle(r') + 2|r| \leq t +2\ell.
    \end{equation}
    We may now naturally define the coloring $\hat{w}_i$ for $w_i$ by using the $p$- and $q$-colors from $\hat{w}_{i-1}$, the $u$-colors appearing in both $\hat{w}_{i-1}$ and $\hat{r}$, and the $v$-colors from $\hat{r}$.
    Now finally applying \Cref{cor:stitch} yields \Cref{ineq:induct}, completing the proof.
\end{proof}

\subsection{Back to the plain coset complex}
We now translate \Cref{thm:relative-dehn} into a statement about statement about loops (closed walks) $L$ in the \emph{usual} coset complex $\CoCo(G;\CH)$.  Using notation from the previous section:
\begin{corollary} \label{cor:dehn}
    Let $L$ be a closed loop in $\CoCo(G;\CH)$ of length~$r$.
    Then there is a word~$w$ over $S$ of length~$r$ such that $L$ can be~$\BF_2$-filled by at most $(t+4\ell+2)\cdot \delta_{\ell,t}(w)+1$ triangles in~$\CoCo(G;\CH)$.
\end{corollary}
\begin{proof}
    Write the loop $L$ as $C_0 \to C_1 \to \cdots \to C_{r-1} \to C_0$, where the $C_i$'s are cosets (of the subgroups in~$\CH$).
    Since the plain coset complex $\CoCo(G;\CH)$ is partite, without self-loops, it follows that consecutive vertices in loop~$L$ are of different ``colors'' (i.e., they are not cosets of the same subgroup).
    We may select elements $g_0, g_1, \dots, g_{r-1} \in G$ that ``witness'' each edge in~$L$; so $g_{i} \in C_{i} \cap C_{i+1}$ (indices taken mod~$r$).  Now if $x_{i} = g_{i-1}^{-1}g_{i}$, then $x_i \in H_i \in \CH$, where $H_i$ is the subgroup of which~$C_i$ is a coset; and, $x_0 x_1 \cdots x_{r-1} = \Id$.  This word $x_0 x_1 \cdots x_{r-1}$ will be the required~$w$.
    Moreover, it is easy to see that translating $L$ by $g_{r-1}^{-1}$ (in the sense of \Cref{def:translate}) transforms it into the form $\CL(\widehat{w})$ for colored word~$\widehat{w} = \col{x_0}{H_0}\col{x_1}{H_1}\cdots\col{x_{r-1}}{H_{r-1}} \in F_{\CH}$, which (again) has no two consecutive colored elements of the same color, and also has $\phi(\widehat{w}) = w = \Id$.  

    Now suppose we apply \Cref{thm:relative-dehn} to this $\widehat{w}$. 
    We conclude that there is an $\BF_2$-filling \emph{in $\CoCoCo(G;\CH)$} of $\CL{\widehat{w}}$ using some set~$T$ of at most $(t+4\ell+2) \cdot \delta_{\ell,t}(w)+1$ (possibly degenerate) triangles.  
    But now suppose we simply discard any degenerate triangles from~$T$, forming~$T'$ (where a degenerate triangle means a cardinality-$3$-multiset with $2$~or~$3$ repeated vertices).  Since the boundary (in $\CoCoCo(G;\CH)$) of a degenerate triangle is always a self-loop, it follows that $T'$ has the same boundary in $\CoCoCo(G;\CH)$ that~$T$ has --- namely~$\CL{\widehat{w}}$ --- except possibly at self-loops.  But $\CL{\widehat{w}}$ \emph{has} no self-loops, since $\widehat{w}$ has no two consecutive same colors; and the boundary of~$T'$ also has no self-loops, since $T'$ only has non-degenerate triangles.
    Thus the boundary of~$T'$ is precisely~$\CL(\widehat{w})$, and this \emph{holds within the plain coset complex $\CoCo(G;\CH)$}.
    So we have a proper~$\BF_2$-filling $T'$ in $\CoCo(G;\CH)$ of $\CL(\widehat{w})$ using at most $(t+4\ell+2) \cdot \delta_{\ell,t}(w)+1$ triangles, and this may be translated back (as in \Cref{fact:translate}) to give an equal-size filling of the original~$L$.
\end{proof}
Finally, one more summatory corollary:
\begin{corollary} \label{cor:dehn2}
    Let $\CoCo(G;\CH)$ be a coset complex.
    Using the notation from \Cref{def:notationy}, suppose that every trivial word over~$S$ of length at most~$r_1$ can be reduced to~$\Id$ using at most~$\delta$ derivation steps, via relators from~$R_{\ell,t}$.  Then every loop of length at most~$r_1$ can be~$\BF_2$-filled by at most $(t+4\ell+2)\cdot \delta +1$ triangles.
\end{corollary}

%% file: sections/05-lifting.tex
\section{Lifting, and derivations for the Chevalley HDXs}

We now discuss the final notions that we use to give a proof of \Cref{thm:keytech}.

\subsection{Homomorphisms between groups give triangulation bounds}

In this subsection, we develop some machinery for upper-bounding the triangulation size of a relation in a coset (multi)complex of a group $\tilde{G}$ via a triangulation in another complex of a group $G$ mapping homomorphically onto $\tilde{G}$. 
(We mention that we have written this section for coset multicomplexes, but the results/proofs are identical for plain complexes.)

Consider following the general setup. Let $G, \tilde{G}$ be groups and $f : G \to \tilde{G}$ a group homomorphism. Let $\CH$ denote a set of subgroups of $G$. For each $H \in \CH$, let $\tilde{H} \subseteq \tilde{G}$ denote a corresponding subgroup such that $f(H) \subseteq \tilde{H}$. Correspondingly, let $\tilde{\CH} = (\tilde{H})_{H \in \CH}$. (We emphasize that these groups and subgroups are arbitrary; they need not be the particular groups we consider in the rest of this paper.) Recall the definitions of the groups $F_\CH$ and $F_{\tilde{\CH}}$ from the preceding section (\Cref{def:free}) and their respective homomorphisms to $G$ and $\tilde{G}$; we denote these $\phi : F_\CH \to G$ and $\tilde{\phi} : F_{\tilde{\CH}} \to \tilde{G}$.

\paragraph*{Lifting trivial words in the free groups.} Let $\hat{x} = \binom{x_0}{H_0} \cdots \binom{x_{\ell-1}}{H_{\ell-1}}$ be a colored word in $F_\CH$ such that $\phi(\binom{x_0}{H_0} \cdots \binom{x_{\ell-1}}{H_{\ell-1}}) = \Id$ (in $G$). Correspondingly, since $f(x_i) \in f(H_i) \subseteq \tilde{H}_i$ for each $i$ we have that $\tilde{x} \coloneqq \binom{f(x_0)}{\tilde{H}_0} \cdots \binom{f(x_{\ell-1})}{\tilde{H}_{\ell-1}}$ is a word in $F_{\tilde{\CH}}$. Hence also in $\tilde{G}$ we have \[ \tilde{\phi}(\tilde{x}) = f(x_0) \cdots f(x_{\ell-1}) = f(x_0 \cdots x_{\ell-1}) = f(\Id) = \Id. \]

Our goal is now to prove the following theorem:

\begin{theorem}\label{thm:lifting:tri}
With the notation of the preceding paragraph, let $X^\circ = \CoCo^\circ(G,\CH)$ and $\tilde{X}^\circ = \CoCo^\circ(\tilde{G}, \tilde{\CH})$ denote the corresponding coset multicomplexes.
Let $\hat{x}$ be a word in $F_\CH$ such that $\phi(\hat{x}) = \Id$ in $G$. Consider the corresponding word $\tilde{x}$ in $F_{\tilde{\CH}}$, recalling that $\tilde{\phi}(\tilde{x}) = \Id$ in $\tilde{G}$. Then
$
\tilde{\triangle} (\tilde{x}) \leq \triangle (\hat{x}).
$
\end{theorem}

To do this, we need to further develop a few notions.

\paragraph*{Lifting faces.} Call a multiset $\{xH_0,\ldots,xH_i\}$ for $x \in G$ an \emph{$i$-multiface}. Let $X^\circ(i)$ and $\tilde{X}^\circ(i)$ denote the respective sets of $i$-multifaces. We show how the homomorphism $f$ \emph{extends to a map of multifaces} between the multicomplexes $X^\circ$ and $\tilde{X}^\circ$. 

For $0 \leq i < |\CH|$, we define
\[
f_i : X^\circ(i) \to \tilde{X}^\circ(i)
\]
in the natural way: $f_i(\{xH_0,\ldots,xH_i\}) = \{f_i(x) \tilde{H}_0,\ldots,f_i(x) \tilde{H}_i\}$. To check that this map is well-defined, suppose that $\{xH_0,\ldots,xH_i\} = \{x'H_0,\ldots,x'H_i\}$. Then for each $0 \leq j \leq i$, $x H_j = x' H_j$, therefore $x^{-1} x' \in H_j$, therefore $f_i(x)^{-1} f_i(x') = f_i(x^{-1} x') \in f_i(H_j) \subseteq \tilde{H}_j$ and so $f_i(x) \tilde{H}_j = f_i(x') \tilde{H}_j$. Having defined $f_i$ on faces, it extends naturally to a map on chains $f_i : \BF_2^{X^\circ(i)} \to \BF_2^{\tilde{X}^\circ(i)}$, which has the property that $|f_i(T)| \leq |T|$ for all $T \in X^\circ(i)$.

For $0 \leq i < |\CH|$, let $\partial_i : \BF_2^{X^\circ(i)} \to \BF_2^{X^\circ(i-1)}$ denote the boundary operator for $i$-multifaces in $X^\circ$ and similarly for $\tilde{\partial}_i$ in $\tilde{X}^\circ$. Our key claim is the following: For every $i < |\CH|$,
\begin{equation} \label{eqn:comdiag}
f_{i-1} \partial_i = \tilde{\partial}_i f_i.
\end{equation}
Note that these are both functions $\BF_2^{X^\circ(i)} \to \BF_2^{\tilde{X}^\circ(i-1)}$. By linearity, it suffices to verify the above identity when applied to a single face $\{xH_0,\ldots,xH_i\}$. The right-hand side maps this face to $\{f(x) \tilde{H}_0,\ldots, f(x) \tilde{H}_i\}$ and then to $\sum_{j=0}^i \1_{\{f(x) \tilde{H}_0,\ldots, f(x) \tilde{H}_i\} \setminus \{f(x) \tilde{H}_j\}}$. The left-hand side maps this face to $\sum_{j=0}^i \1_{\{x H_0,\ldots, x H_i\} \setminus \{x H_j\}}$ and then to $\sum_{j=0}^i \1_{\{f(x) \tilde{H}_0,\ldots, f(x) \tilde{H}_i\} \setminus \{f(x) \tilde{H}_j\}}$, as desired.

We now have:

\begin{proof}[Proof of \Cref{thm:lifting:tri}]
    Say we are given $\hat{x}, \tilde{x}$ as in the theorem statement.
     Recall that $\CL(\hat{x}), \CL(\tilde{x})$ are the closed walks
    \[
    \1 H_0 \to x_0 H_1 \to \cdots \to x_0 \cdots x_{\ell-2} H_{\ell-1} \to \1 H_0 \text{ and } \1 \tilde{H}_0 \to f(x_0) \tilde{H}_1 \to \cdots \to f(x_0) \cdots f(x_{\ell-2}) \tilde{H}_{\ell-1} \to \1 \tilde{H}_0
    \]
    in $X^\circ$ and $\tilde{X}^\circ$, respectively. As we see from the preceding paragraph, $f_1$ maps the edges of the walk $\CL(\hat{x})$ to those of the walk $\CL(\tilde{x})$ (using that $f$ is a homomorphism).
    That is, $f_1[\CL(\hat{x})] = [\CL(\tilde{x})]$ as $1$-chains in $\BF_2^{\tilde{X}^\circ(1)}$.
    Now  let $T \in \BF_2^{X^\circ(2)}$ be an $\BF_2$-filling of $[\CL(\hat{x})]$ with $|T|\leq \triangle(\hat{x})$.
    Since $\partial_2 T = [\CL(\hat{x})]$, we deduce that
    \begin{equation} \tilde{\partial}_2 (f_2 T) = f_{1} \partial_2 T = f_{1} [\CL(\hat{x})] = [\CL(\tilde{x})], \end{equation} where we used \Cref{eqn:comdiag}
    Hence $f_2 T \in \BF_2^{\tilde{X}^\circ(2)}$ is an $\BF_2$-filling of $[\CL(\tilde{x})]$, and its size is at most $T$'s size.
\end{proof}

\subsection{Statements of sufficient theorems for \Cref{thm:main}}

Recall that our notations $\LinkCplxA{q}$, $\LinkCplxBSm{q}$, and $\LinkCplxBLg{q}$ denote, respectively, the link complex in $A_3$, the small link complex in $B_3$, and the large link complex in $B_3$. These are coset complexes of the unipotent groups $\UnipA{q}$, $\UnipBSm{q}$, and $\UnipBLg{q}$, respectively. We also have the graded equivalents $\GrLinkCplxA{q}$, $\GrLinkCplxBSm{q}$, and $\GrLinkCplxBLg{q}$, which are coset complexes of the graded unipotent groups $\GrUnipA{q}$, $\GrUnipBSm{q}$, and $\GrUnipBLg{q}$, respectively.

The key theorems we use to prove \Cref{thm:main} are the following:

\begin{theorem}[Base bound: Computer rank bound for \emph{fixed} $B_3$ links]\label{thm:story:b3:base}
    $\LinkCplxBSm{5}$ and $\LinkCplxBLg{5}$ have vanishing $1$-(co)homology (over $\BF_2$).
\end{theorem}

\begin{theorem}[Strong lifting theorem for $B_3$ links]\label{thm:story:b3:lift}
    There exists $N,M \in \BN$ such that for every odd prime (power) $p$ \emph{and every $k \geq 1$}:
    \begin{enumerate}
        \item Every pure-degree Steinberg relation in $\GrUnipBSm{p^k}$ can be derived in $\leq N$ steps from (i) in-subgroup relations of length $\leq M$ and (ii) lifts of pure-degree relations in $\UnipBSm{p}$ of length $\leq M$. 
        \item Every pure-degree Steinberg relation in $\GrUnipBLg{p^k}$ can be derived in $\leq N$ steps from (i) in-subgroup relations of length $\leq M$ and (ii) lifts of pure-degree relations in $\UnipBLg{p}$ of length $\leq M$. 
    \end{enumerate}
\end{theorem}

Here, the \emph{pure-degree Steinberg relations} are the linearity and commutator relations mentioned in \Cref{thm:structure4}, applied to the symbols $\Eld{\zeta}{t}{i}$, $\Rt{\zeta} \in I^+, t \in \BF_q, i \in [\height_I(\Rt{\zeta})]$ --- these symbols are generators if $\Rt{\zeta} \in J$ and aliases for products of generators if $\Rt{\zeta} \not\in J$, where $J$ is the set of roots which are nonnegative integer combinations of two roots in $I$.

These theorems are proven in \Cref{sec:story:compute} and \Cref{sec:b3} below. We also prove a strengthened version of a lifting theorem from \cite{KO21} as a warmup to the proof of \Cref{thm:story:b3:lift}:

\begin{theorem}[Strong lifting theorem for $A_3$ links]\label{thm:story:a3:lift}
    There exists $N,M \in \BN$ such that for every odd prime power $q$, every pure-degree Steinberg relation in $\GrUnipA{q}$ can be derived in $\leq N$ steps from (i) in-subgroup relations of length $\leq M$ and (ii) lifts of relations in $\UnipA{A}$ of length $\leq M$.
\end{theorem}

We prove this theorem in \Cref{sec:a3} below. Finally, we mention a theorem which was proven already in~\cite{KO21}:

\begin{theorem}[Base theorem for $A_3$ links]\label{thm:story:a3:base}
    There exists $N,M \in \BN$ such that for every odd prime $p$, every Steinberg relation in $\UnipA{q}$ can be derived in $\leq N$ steps from in-subgroup relations of length $\leq M$.
\end{theorem}

\subsection{Proof of \Cref{thm:main}}

Finally, we prove \Cref{thm:main} (assuming \Cref{thm:story:b3:base,thm:story:b3:lift}, which we prove in subsequent sections), the cosystolic expansion bound for $(\BCplx{3}{5^k}{m})_{m \geq 6}$ for sufficiently large $k$. An identical proof replacing \Cref{thm:story:b3:base} with \Cref{thm:story:a3:base}, \Cref{thm:story:b3:lift} with \Cref{thm:story:a3:lift}, and $5$ with $q$ would recover the results of \cite{KO21} for the family $(\ACplx{3}{q}{m})_{m \geq 4}$.

\begin{proof}[Proof of \Cref{thm:main}]
Our goal is to combine \Cref{cor:dehn2}, \Cref{thm:lifting:tri}, \Cref{thm:story:b3:base}, \Cref{thm:story:b3:lift}, and Inception \Cref{thm:inception} in order to prove \Cref{thm:main}.

Let $k \in \BN^+$. Recall that in the corresponding complex $\BCplx{3}{5^k}{m}$, $m \geq 6$, every link is isomorphic to either $\GrLinkCplxBLg{5^k}$ or $\GrLinkCplxBSm{5^k}$. So, to apply \Cref{thm:inception}, we need to check that in both of these link complexes, every loop of length at most $r_1$ (the constant from \Cref{prop:last-junk}) can be $\BF_2$-filled by at most $R_0$ triangles (a constant of our choosing).

Consider the larger link complex $\GrLinkCplxBLg{5^k}$ without loss of generality, whose ambient group is $\GrUnipBLg{5^k}$. We are now interested in taking words~$w$ of length at most~$r_1$  that are trivial in the ambient group, and \emph{attempting} to derive $w = \Id$ using only relations that hold within the colored subgroups (``in-subgroup'' relations).  (Note that the words~$w$ themselves will only use symbols from the colored subgroups.)

Recall our notation (\Cref{def:notationy}) that $R_{\ell,t}$ denotes the set of all relations in the ambient group $\GrUnipBLg{q}$ of length $\leq \ell$ $\BF_2$-fillable using at most $t$ triangles (every relation is fillable because of \Cref{thm:story:b3:base}). Let $\ell_0$ be a constant parameter to be fixed later.

Let $t_0$ denote the maximum $\BF_2$-filling size of any loop in the base complex $\LinkCplxBLg{q}$. Let $R^{\mathrm{lift}}_{\ell_0}$ denote the set of all colored relations in the ambient group $\GrUnipBLg{q}$ which are lifts (in the sense of any homomorphism in \Cref{thm:chev:general-lift}) of relations of length $\leq \ell_0$ in the ungraded group $\UnipBLg{q}$. By \Cref{thm:chev:general-lift}, $R^{\mathrm{lift}}_{\ell_0} \subseteq R_{\ell_0,t_0}$.

Let $R^{\mathrm{subgp}}_{\ell_0}$ denote the set of all (colored) relations in the ambient group $\GrUnipBLg{q}$ of length $\leq \ell_0$ which hold in one of the three designated link subgroups. By \Cref{rem:in-subgroup}, $R^{\mathrm{subgp}}_{\ell_0} \subseteq R_{\ell,1} \subseteq R_{\ell,t}$.

We now claim that $R^{\mathrm{lift}}_{\ell_0}$ and $R^{\mathrm{subgp}}_{\ell_0}$ together are enough to derive any colored relation of length at most $r_1$ in $\GrLinkCplxBLg{5^k}$ in at most~$\delta$ steps. Given this, the premise of  \Cref{cor:dehn2} is fulfilled and we conclude that every loop of length at most~$r_1$ in the complex can be $\BF_2$-filled by at most $R_0 \coloneqq (t_0 + 4\ell_0 + 2)\cdot \delta+1$ triangles.
Finally, provided that this all works with $\ell_0$, $t_0$, and $\delta$ being absolute constants that do not depend on~$q$, we can take $q$ to be sufficiently large and obtain the conclusion of the Inception~\Cref{thm:inception}.

It remains to check our claim. This claim is \emph{almost} the statement of \Cref{thm:story:b3:lift}, except that this theorem is only stated for pure degree Steinberg relations. So, it suffices to reduce to these relations.

\paragraph*{Reduction to pure degree Steinberg relations.}
Let us outline the plan for deriving relations $w = \Id$ of length at most~$r_1$ within the ambient group, using only in-subgroup and lifted relations.
The first step will be to expand $w$ to an equivalent~$w'$ (of length $O(r_1)$) which uses only pure degree symbols.
This is straightforward, by using the linearity relations to write
\begin{equation}
    \El{\zeta}{t_0 + t_1x + \cdots + t_hx^h} = \Eld{\zeta}{t_0}{0}\Eld{\zeta}{t_1}{1} \cdots \Eld{\zeta}{t_h}{h}.
\end{equation}
Next, \Cref{thm:structure4} gives a method for deriving $w' = \Id$ in $\kappa(O(r_1)) = O(r_1 \log r_1)$ steps using only pure degree Steinberg relations.
We thus see that it suffices to be able to derive each \emph{pure degree Steinberg} relation in $N$ steps, using only in-subgroup and lifted relations of length at most~$M$ (where $N,M$ are absolute constants not depending on~$q$) --- because then we can set $\ell_0 = M$, $\delta = O(r_1 \log r_1) \cdot N$.
\end{proof}

%% file: sections/06-story.tex
\section{Computational analysis for $B_3$}\label{sec:story:compute}

The computational analysis (proof of \Cref{thm:story:b3:base}) is divided into a ``pipeline'' of two separate parts:

\begin{enumerate}
    \item Calculating the $\BF_2$-incidence matrix of Chevalley link groups over finite fields.
    \item Calculating the rank of an arbitrary sparse $\BF_2$ matrix.
\end{enumerate}

For the latter we used some prebuilt tools. We address these two parts in two separate subsections.

\subsection{Generating the incidence matrices}

We wrote a Sage~\cite{sagemath} script to generate the incidence matrices of $B_3$ (and $A_3$, $C_3$) complexes over arbitrarily fields. (This script will be uploaded on Github for the published version of this paper.)

The script takes as input several parameters:
\begin{itemize}
    \item The field size $q$.
    \item The root system ($A_3$, $B_3$, or $C_3$).
    \item Three independent roots $\Rt{\zeta},\Rt{\eta},\Rt{\theta}$ to use as a positive basis.
\end{itemize}
It defines the following data types:
\begin{itemize}
    \item \texttt{MatrixGroup}: A group of $n \times n$ matrices over a finite field $\BF$. Consists of a finite set of matrices. Constructed from a set of generators.
    \item \texttt{MatrixGroupCoset}: A left coset of a \texttt{MatrixGroup} by a matrix.
    \item $A_3$, $B_3$, and $C_3$ root systems. We give procedures for mapping each $(\Rt{\zeta},t)$ pair, where $\Rt{\zeta}$ is a root and $t$ a field element, to a matrix over the field, and for finding the set of roots in the nonnegative span of a given set of roots. Combined with the previous item, this lets us calculate the Steinberg groups for these root systems and their link subgroups.
\end{itemize}
The script proceeds to enumerate the elements of $G$, the entire Steinberg link group generated by $\{\Rt{\zeta},\Rt{\eta},\Rt{\theta}\}$, as well as what we call the ``$\BoxR$'', ``$\BoxG$'', and ``$\BoxB$'' subgroups $\KR, \KG, \KB \subseteq G$, generated by the pairs of roots $\{\Rt{\zeta},\Rt{\eta}\},\{\Rt{\zeta},\Rt{\theta}\},\{\Rt{\eta},\Rt{\theta}\}$, respectively. As a sanity check, it compares the sizes of $G,\KR,\KB,\KG$ to $q^e,q^{\eR},q^{\eG},q^{\eB}$, respectively, where $e,\eR,\eG,\eB$ are respectively the sizes of the nonnegative spans of $\{\Rt{\zeta},\Rt{\eta},\Rt{\theta}\}, \{\Rt{\zeta},\Rt{\eta}\},\{\Rt{\zeta},\Rt{\theta}\},\{\Rt{\eta},\Rt{\theta}\}$.

The next step is writing down the vertices, edges, and triangles of the complex. We begin with the vertices. This is equivalent to enumerating all the cosets of $G/\KR,G/\KG,G/\KB$ and constructing dictionaries mapping an element of $G$ to its cosets (i.e., mapping $x$ to $\KR x,\KG x,\KB x$. We again confirm that we have enumerates $q^{e-\eR},q^{e-\eG},q^{e-\eB}$ cosets respectively.

Now, we enumerate all edges and triangles. We pass over all $x \in G$ and add $(\KR x,\KG x,\KB x)$ as a triangle and $(\KR x,\KG x),(\KR x,\KB x),(\KG x,\KB x)$ as $\BoxR$-$\BoxG$, $\BoxR$-$\BoxB$, and $\BoxG$-$\BoxB$ edges, respectively. Again, we confirm that the numbers of triangles and edges are correct.

Finally, we output the triangle-edge incidence matrix (and the edge-vertex matrix for good measure) in the \texttt{sms} (\underline{s}parse \underline{m}atrix \underline{s}torage) format in order to subsequently calculate their rank.

\begin{remark}
    Recall that given some fixed ordering of the roots in a positive subset, every link group element can be expressed uniquely as a product of elements of each type in the link. If the link has $e$ roots, then we go from encoding group elements as $n \times n$ matrices over $\BF_q$ ($n=7$ in the case of $B_3$) to length-$t$ vectors over $\BF_q$. This could be the basis for a more memory-efficient program to output the triangle-edge incidence matrix. This program also has the nice property that the set of triangles (a.k.a. group elements) is simply the product set $\BF_q^e$. Further, e.g. the \BoxR-\BoxB edges correspond to cosets of the subgroup $\KR \cap \KB \cong \BF_q$ and so may be represented by length-$(e-1)$ vectors. Indeed, even calculating the \BoxR-\BoxB edge incident to some triangle is just some polynomial mapping $\BF_q^e \to \BF_q^{e-1}$. However, we chose to keep the current matrix group approach as it is more ``canonical'' (does not require fixing orderings of the roots) and as computing the rank, not outputting the incidence matrix, is the bottleneck step.
\end{remark}

\subsection{Computing the rank over $\BF_2$ of the matrices}

In the ``large link'' of $B_3$ over $\BF_q$ there are $q^9$ triangles, $3q^8$ edges ($q^8$ for each possible color-pair), and $q^7+q^6+q^5$ vertices. Since the $1$-skeleton of the complex is connected, the edge-vertex incidence matrix always has $\BF_2$-rank $q^7+q^6+q^5 - 1$ (equiv., corank $1$). Hence the $1$-(co)homology over $\BF_2$ vanishes iff the triangle-edge incidence matrix has $\BF_2$-rank $3q^8 - q^7+q^6+q^5 + 1$ (equiv., corank $q^7+q^6+q^5 - 1$). This is the incidence matrix is a $q^9 \times 3q^8$ matrix over $\BF_2$ with three $1$'s per row.

\begin{table}
\begin{centering}
\begin{tabular}{c|c|c|c|c|c|c}
    $q$ & \# triangles & \# edges & \# vertices & Needed rank & Calculated rank & Vanishes? \\
    & $q^9$ & $3q^8$ & $q^7+q^6+q^5 - 1$ & $3q^8 - q^7+q^6+q^5 + 1$ & & \\ \hline \hline
    2 & 512 & 768 & 224 & 545 & 127 & \xmark \\ \hline
    3 & 19,683 & 19,683 & 3,149 & 16,525 & 16,525 & \cmark \\ \hline
    5 & 1,953,125 & 1,171,875 & 96,875 & 1,075,001 & 1,075,001 & \cmark \\ \hline
    7 & 40,353,607 & 17,294,403 & 957,999 & 16,336,405 & ? & ?
\end{tabular}
\caption{The size of the ``large link'' of $B_3$ for some small values of $q$, the rank of the triangle-edge incidence matrix needed for the $1$-(co)homology over $\BF_2$ to vanish, the rank we calculated, and whether the (co)homology vanishes.}\label{tab:ranks:b3-large}
\end{centering}
\end{table}

\begin{table}
\begin{centering}
\begin{tabular}{c|c|c|c|c|c|c}
    $q$ & \# triangles & \# edges & \# vertices & Needed rank & Calculated rank & Vanishes? \\
    & $q^7$ & $3q^6$ & $q^5+q^4+q^3 - 1$ & $3q^6 - q^5+q^4+q^3 + 1$ & & \\ \hline \hline
    2 & 128 & 192 & 56 & 138 & 15 & \xmark \\ \hline
    3 & 2,187 & 2,187 & 351 & 1,837 & 1,825 & \xmark \\ \hline
    5 & 78,125 & 46,875 & 3,875 & 43,001 & 43,001 & \cmark \\ \hline
    7 & 823,543 & 352,947 & 19,551 & 333,397 & 333,397 & \cmark
\end{tabular}
\caption{A similar table for the ``small link'' of $B_3$.}\label{tab:ranks:b3-small}
\end{centering}
\end{table}

For highly sparse matrices such as these, using Gaussian elimination to calculate the rank can be tricky: In an $m \times n$ matrix with $O(1)$ nonzero entries per row, there are only $O(m)$ total nonzero entries, but an unlucky choice of pivots for Gaussian elimination can cause \emph{fill-in}, making the resulting matrix dense, with $\Omega(m^2)$ nonzero entries. When $m \sim 10^6$ and especially when $m \sim 10^7$, with our computing resources we can only afford $\Theta(m)$ memory (megabytes), not $\Theta(m^2)$ memory (terabytes).

However, we did manage to perform the rank calculation in the $q=5$ case. Recall, the $q=5$ triangle-edge incidence matrix is $1,953,125 \times 1,171,875$, so this rank computation was quite expensive. We used the software package \texttt{linbox} \cite{linbox}, which implements sparse Gaussian elimination over finite fields. This computation already took roughly 10 hours to run on our personal computers; $q=7$ would be impractical, given the huge blowup in time and memory usage. We also received independent confirmation of the rank from an optimized program for calculating ranks of sparse matrices over $\BF_2$ due to Ryan Bai and Richard Peng~\cite{BP24}.\footnote{This program also uses sparse Gaussian elimination. The idea behind their script is to perform Gaussian elimination on the matrix in sparse representation until it becomes sufficiently dense, and then to switch to a dense matrix representation which uses a bitarray. Even with a simple pivoting heuristic --- eliminate the rows with the smallest number of nonzeros --- this seemed to run about 10x faster than \texttt{linbox}.}

%% file: sections/07-A3.tex
\section{Lifting for $A_3$}\label{sec:a3}

\newcommand{\LiftScalars}{\BF_{p^k}}

In this section, we prove the lifting theorem for $A_3$ (\Cref{thm:story:a3:lift}), lifting from $\UnipA{q}$ to $\GrUnipA{q}$. This theorem was known in the work of \cite{KO21}.

All relations stated in this section hold in the graded link group $\GrUnipA{p^k}$ and are proven via constant-length in-subgroup and lifted relations in a number of steps independent of $p$ or $k$.

We emphasize again the nonstandard notational convention $[n] = \{0, 1, \dots, n\}$.

\subsection{Link structure}

We first discuss the structure of the link of $\Rt{\omega}$ in $A_3$ in preparation for analyzing the group $\GrUnipA{\BF_q}$. 

\begin{table}[h!]
\centering
\begin{tabular}{c|c|c|c}
    \textbf{Color} & \textbf{Base} & \textbf{Positive span} & \textbf{Total \#} \\ \hline
    & $\{\Rt{\alpha},\Rt{\beta},\Rt{\gamma}\}$ & $\{\Rt{\alpha+\beta},\Rt{\beta+\gamma},\Rt{\alpha+\beta+\gamma}\}$ & 6 \\
    \BoxR & $\{\Rt{\alpha},\Rt{\beta}\}$ & $\{\Rt{\alpha+\beta}\}$ & 3 \\
    \BoxG & $\{\Rt{\alpha},\Rt{\gamma}\}$ & $\emptyset$ & 2 \\
    \BoxB & $\{\Rt{\beta},\Rt{\gamma}\}$ & $\{\Rt{\beta+\gamma}\}$ & 3 \\
\end{tabular}
\caption{The roots spanned by $\Rt{\alpha},\Rt{\beta},\Rt{\gamma}$ in $A_3$ as well as the subsets spanned by pairs.}
\end{table}

The link of the base roots $\{\Rt{\alpha},\Rt{\beta},\Rt{\gamma}\}$ contains $6$ roots; $5$ of these roots (all except the unique ``missing'' root $\Rt{\alpha+\beta+\gamma}$) are contained in the span of some subset of two of the base roots. Correspondingly, in the Steinberg group $\UnipA{q}$, there are $\binom{6}2 = 15$ commutation relations and $6$ linear relations. Together, the \BoxR, \BoxG, and \BoxB subgroups cover $\binom{3}2 + \binom{2}2 + \binom{3}2 = 7$ commutation relations and $5$ linear relations. The ``missing'' $15-7=8$ commutation relations from the subgroups can be divided into the following classes: (i) the commutation relation $\{\Rt{\alpha+\beta},\Rt{\beta+\gamma}\}$ (this pair has empty span), (ii) the commutation relations $\{\Rt{\alpha},\Rt{\beta+\gamma}\}$ and $\{\Rt{\alpha+\beta},\Rt{\gamma}\}$ (these pairs both span $\Rt{\alpha+\beta+\gamma}$), and (iii) the commutation relations $\{\Rt{\alpha+\beta+\gamma},\Rt{\zeta}\}$ for $\Rt{\zeta} \in \{\Rt{\alpha},\Rt{\beta},\Rt{\gamma},\Rt{\alpha+\beta},\Rt{\beta+\gamma}\}$ (these pairs all have empty span). The only commutation relation not involving a ``missing'' root is the relation that $\Rt{\alpha+\beta}$ and $\Rt{\beta+\gamma}$ commute. We are also ``missing'' the linear relation for $\Rt{\alpha+\beta+\gamma}$.

\subsection{Assumed relations}

We now discuss the actual relations which we assume in the \emph{graded} Steinberg group $\GrUnipA{p}$, in order to prove the pure-degree Steinberg relations.

\paragraph*{In-subgroup relations.} We start by enumerating the in-subgroup Steinberg relations, for completeness:

\begin{tcolorbox}[colback=\ColorR, colframe=\ColorBackR, title={Commutation relations spanned by $\Rt{\alpha}$ and $\Rt{\beta}$}]
\begin{relation}[\RelNameA\RelNameSubgp\RelNameCommutator{\Rt{\alpha}}{\Rt{\beta}}]\label{rel:a3:sub:comm:alpha:beta}
\[
    \DQuant{i,j \in [1]}{t,u \in \LiftScalars} \Comm{\Eld{\alpha}{t}i}{\Eld{\beta}{u}j} = \Eld{\alpha+\beta}{tu}{i+j}.
\]
\end{relation}
\begin{relation}[\RelNameA\RelNameSubgp\RelNameCommutator{\Rt{\alpha}}{\Rt{\alpha+\beta}}]\label{rel:a3:sub:comm:alpha:alpha+beta}
\[
    \DQuant{i \in [1], j \in [2]}{t,u \in \LiftScalars} \Comm{\Eld{\alpha}{t}i}{\Eld{\alpha+\beta}{u}j} = \Id.
\]
\end{relation}
\begin{relation}[\RelNameA\RelNameSubgp\RelNameCommutator{\Rt{\beta}}{\Rt{\alpha+\beta}}]\label{rel:a3:sub:comm:beta:alpha+beta}
\[
    \DQuant{i \in [1], j \in [2]}{t,u \in \LiftScalars} \Comm{\Eld{\beta}{t}i}{\Eld{\alpha+\beta}{u}j} = \Id.
\]
\end{relation}
\end{tcolorbox}

\begin{tcolorbox}[colback=\ColorG, colframe=\ColorBackG, title={Commutation relation spanned by $\Rt{\alpha}$ and $\Rt{\gamma}$}]
\begin{relation}[\RelNameA\RelNameSubgp\RelNameCommutator{\Rt{\alpha}}{\Rt{\gamma}}]\label{rel:a3:sub:comm:alpha:gamma}
\[
    \DQuant{i,j \in [1]}{t,u \in \LiftScalars} \Comm{\Eld{\alpha}{t}i}{\Eld{\gamma}{u}j} = \Id.
\]
\end{relation}
\end{tcolorbox}

\begin{tcolorbox}[colback=\ColorB, colframe=\ColorBackB, title={Commutation relations spanned by $\Rt{\beta}$ and $\Rt{\gamma}$}]
\begin{relation}[\RelNameA\RelNameSubgp\RelNameCommutator{\Rt{\beta}}{\Rt{\gamma}}]\label{rel:a3:sub:comm:beta:gamma}
\[
    \DQuant{i,j \in [1]}{t,u \in \LiftScalars} \Comm{\Eld{\beta}{t}i}{\Eld{\gamma}{u}j} = \Eld{\beta+\gamma}{tu}{i+j}.
\]
\end{relation}
\begin{relation}[\RelNameA\RelNameSubgp\RelNameCommutator{\Rt{\beta}}{\Rt{\beta+\gamma}}]\label{rel:a3:sub:comm:beta:beta+gamma}
\[
    \DQuant{i \in [1], j \in [2]}{t,u \in \LiftScalars} \Comm{\Eld{\beta}{t}i}{\Eld{\beta+\gamma}{u}j} = \Id.
\]
\end{relation}
\begin{relation}[\RelNameA\RelNameSubgp\RelNameCommutator{\Rt{\gamma}}{\Rt{\beta+\gamma}}]\label{rel:a3:sub:comm:gamma:beta+gamma}
\[
    \DQuant{i \in [1], j \in [2]}{t,u \in \LiftScalars} \Comm{\Eld{\gamma}{t}i}{\Eld{\beta+\gamma}{u}j} = \Id.
\]
\end{relation}
\end{tcolorbox}

\begin{tcolorbox}[breakable,colback=\ColorNeut,title={Linearity relations}]
\begin{relation}[\RelNameA\RelNameSubgp\RelNameLinearity{\Rt{\alpha}}]\label{rel:a3:sub:lin:alpha}
\[
\DQuant{i \in [1]}{t,u \in \LiftScalars} \Eld{\alpha}{t}{i} \Eld{\alpha}{u}{i} = \Eld{\alpha}{t+u}{i}.
\]
\end{relation}
\begin{relation}[\RelNameA\RelNameSubgp\RelNameLinearity{\Rt{\beta}}]\label{rel:a3:sub:lin:beta}
\[
\DQuant{i \in [1]}{t,u \in \LiftScalars} \Eld{\beta}{t}{i} \Eld{\beta}{u}{i} = \Eld{\beta}{t+u}{i}.
\]
\end{relation}
\begin{relation}[\RelNameA\RelNameSubgp\RelNameLinearity{\Rt{\gamma}}]\label{rel:a3:sub:lin:gamma}
\[
\DQuant{i \in [1]}{t,u \in \LiftScalars} \Eld{\gamma}{t}{i} \Eld{\gamma}{u}{i} = \Eld{\gamma}{t+u}{i}.
\]
\end{relation}
\begin{relation}[\RelNameA\RelNameSubgp\RelNameLinearity{\Rt{\alpha+\beta}}]\label{rel:a3:sub:lin:alpha+beta}
\[
\DQuant{i \in [2]}{t,u \in \LiftScalars} \Eld{\alpha+\beta}{t}{i} \Eld{\alpha+\beta}{u}{i} = \Eld{\alpha+\beta}{t+u}{i}.
\]
\end{relation}
\begin{relation}[\RelNameA\RelNameSubgp\RelNameLinearity{\Rt{\beta+\gamma}}]\label{rel:a3:sub:lin:beta+gamma}
\[
\DQuant{i \in [2]}{t,u \in \LiftScalars} \Eld{\beta+\gamma}{t}{i} \Eld{\beta+\gamma}{u}{i} = \Eld{\beta+\gamma}{t+u}{i}.
\]
\end{relation}
\end{tcolorbox}

\begin{tcolorbox}[breakable,colback=\ColorNeut,title={Self-commutator relations}]
\begin{relation}[\RelNameA\RelNameSubgp\RelNameSelfCommute{\Rt{\alpha}}]\label{rel:a3:sub:comm:self:alpha}
\[
\DQuant{i,j \in [1]}{t,u \in \LiftScalars} \Comm{\Eld{\alpha}{t}{i}}{\Eld{\alpha}{u}{j}} = \Id.
\]
\end{relation}
\begin{relation}[\RelNameA\RelNameSubgp\RelNameSelfCommute{\Rt{\beta}}]\label{rel:a3:sub:comm:self:beta}
\[
\DQuant{i,j \in [1]}{t,u \in \LiftScalars} \Comm{\Eld{\beta}{t}{i}}{\Eld{\beta}{u}{j}} = \Id.
\]
\end{relation}
\begin{relation}[\RelNameA\RelNameSubgp\RelNameSelfCommute{\Rt{\gamma}}]\label{rel:a3:sub:comm:self:gamma}
\[
\DQuant{i,j \in [1]}{t,u \in \LiftScalars} \Comm{\Eld{\gamma}{t}{i}}{\Eld{\gamma}{u}{j}} = \Id.
\]
\end{relation}
\begin{relation}[\RelNameA\RelNameSubgp\RelNameSelfCommute{\Rt{\alpha+\beta}}]\label{rel:a3:sub:comm:self:alpha+beta}
\[
\DQuant{i,j \in [2]}{t,u \in \LiftScalars} \Comm{\Eld{\alpha+\beta}{t}{i}}{\Eld{\alpha+\beta}{u}{j}} = \Id.
\]
\end{relation}
\begin{relation}[\RelNameA\RelNameSubgp\RelNameSelfCommute{\Rt{\beta+\gamma}}]\label{rel:a3:sub:comm:self:beta+gamma}
\[
\DQuant{i,j \in [2]}{t,u \in \LiftScalars} \Comm{\Eld{\beta+\gamma}{t}{i}}{\Eld{\beta+\gamma}{u}{j}} = \Id.
\]
\end{relation}
\end{tcolorbox}

\paragraph*{Lifted relations.} We also use a single lifted relation, the nonhomogeneous lift (cf. \Cref{eq:chev:nonhom}) of the relation $\Comm{\El{\alpha+\beta}{1}}{\El{\beta+\gamma}{1}} = 1$ in $\UnipA{q}$.

\begin{relation}[\RelNameA\RelNameNonHomLiftRaw\RelNameCommute{\Rt{\alpha+\beta}}{\Rt{\beta+\gamma}}]\label{rel:a3:comm:raw:lift:alpha+beta:beta+gamma}
    \begin{multline*}
        \Quant{t_1,t_0,u_1,u_0,v_1,v_0 \in \LiftScalars} \\
        \Comm{\Eld{\alpha+\beta}{t_1u_1}{2} \Eld{\alpha+\beta}{t_1u_0+t_0u_1}{1} \Eld{\alpha+\beta}{t_0u_0}{0}}{\Eld{\beta+\gamma}{u_1v_1}{2} \Eld{\beta+\gamma}{u_1v_0+u_0v_1}{1} \Eld{\beta+\gamma}{u_0v_0}{0}} = \Id.
    \end{multline*}
\end{relation}

\subsection{Basic identities}

We turn to proving useful properties about the group $\GrUnipA{p^k}$ which will be useful in the sequel.

\input{figures/a3/alpha+beta+gamma}

Applying the linearity relations (\Cnref{rel:a3:sub:lin:alpha}, \Cnref{rel:a3:sub:lin:beta}, \Cnref{rel:a3:sub:lin:gamma}, \Cnref{rel:a3:sub:lin:alpha+beta}, and \Cnref{rel:a3:sub:lin:beta+gamma}) with \Cref{prop:prelim:group-homo}, we get the identity relations:

\begin{tcolorbox}[breakable,colback=\ColorNeut,title={Identity relations}]
\begin{relation}[\RelNameA\RelNameSubgp\RelNameId{\Rt{\alpha}}]\label{rel:a3:sub:id:alpha}
\[
\Quant{i \in [1]} \Eld{\alpha}{0}{i} = \Id.
\]
\end{relation}
\begin{relation}[\RelNameA\RelNameSubgp\RelNameId{\Rt{\beta}}]\label{rel:a3:sub:id:beta}
\[
\Quant{i \in [1]} \Eld{\beta}{0}{i} = \Id.
\]
\end{relation}
\begin{relation}[\RelNameA\RelNameSubgp\RelNameId{\Rt{\gamma}}]\label{rel:a3:sub:id:gamma}
\[
\Quant{i \in [1]} \Eld{\gamma}{0}{i} = \Id.
\]
\end{relation}
\begin{relation}[\RelNameA\RelNameSubgp\RelNameId{\Rt{\alpha+\beta}}]\label{rel:a3:sub:id:alpha+beta}
\[
\Quant{i \in [2]} \Eld{\alpha+\beta}{0}{i} = \Id.
\]
\end{relation}
\begin{relation}[\RelNameA\RelNameSubgp\RelNameId{\Rt{\alpha}}]\label{rel:a3:sub:id:beta+gamma}
\[
\Quant{i \in [2]} \Eld{\beta+\gamma}{0}{i} = \Id.
\]
\end{relation}
\end{tcolorbox}
and the inverse relations:
\begin{tcolorbox}[breakable,colback=\ColorNeut,title={Inverse relations}]
\begin{relation}[\RelNameA\RelNameSubgp\RelNameInv{\Rt{\alpha}}]\label{rel:a3:sub:inv:alpha}
\[
\DQuant{i \in [1]}{t \in \LiftScalars} \Eld{\alpha}{t}{i} \Eld{\alpha}{-t}{i} = \Id.
\]
\end{relation}
\begin{relation}[\RelNameA\RelNameSubgp\RelNameInv{\Rt{\beta}}]\label{rel:a3:sub:inv:beta}
\[
\DQuant{i \in [1]}{t \in \LiftScalars} \Eld{\beta}{t}{i} \Eld{\beta}{-t}{i} = \Id.
\]
\end{relation}
\begin{relation}[\RelNameA\RelNameSubgp\RelNameInv{\Rt{\gamma}}]\label{rel:a3:sub:inv:gamma}
\[
\DQuant{i \in [1]}{t \in \LiftScalars} \Eld{\gamma}{t}{i} \Eld{\gamma}{-t}{i} = \Id.
\]
\end{relation}
\begin{relation}[\RelNameA\RelNameSubgp\RelNameInv{\Rt{\alpha+\beta}}]\label{rel:a3:sub:inv:alpha+beta}
\[
\DQuant{i \in [2]}{t \in \LiftScalars} \Eld{\alpha+\beta}{t}{i} \Eld{\alpha+\beta}{-t}{i} = \Id.
\]
\end{relation}
\begin{relation}[\RelNameA\RelNameSubgp\RelNameInv{\Rt{\alpha}}]\label{rel:a3:sub:inv:beta+gamma}
\[
\DQuant{i \in [2]}{t \in \LiftScalars} \Eld{\beta+\gamma}{t}{i} \Eld{\beta+\gamma}{-t}{i} = \Id.
\]
\end{relation}
\end{tcolorbox}

We also state some other useful simple relations. From \Cnref{rel:a3:sub:comm:alpha:beta}, \Cnref{rel:a3:sub:inv:alpha}, and \Cnref{rel:a3:sub:inv:beta}, we can express $\Rt{\alpha+\beta}$ elements as products of $\Rt{\alpha}$ and $\Rt{\beta}$ elements:

\begin{relation}[\RelNameA\RelNameExpr{\Rt{\alpha+\beta}}]\label{rel:a3:expr:alpha+beta}
    \[
    \DQuant{i,j \in [1]}{t,u \in \LiftScalars} \Eld{\alpha+\beta}{tu}{i+j} = \Eld{\alpha}{t}{i} \Eld{\beta}{u}j \Eld{\alpha}{-t}{i} \Eld{\beta}{-u}j.
    \]
\end{relation}

Similarly, from \Cnref{rel:a3:sub:comm:beta:gamma}, \Cnref{rel:a3:sub:inv:beta}, and \Cnref{rel:a3:sub:inv:gamma}:

\begin{relation}[\RelNameA\RelNameExpr{\Rt{\beta+\gamma}}]\label{rel:a3:expr:beta+gamma}
    \[
    \DQuant{i,j \in [1]}{t,u \in \LiftScalars} \Eld{\beta+\gamma}{tu}{i+j} = \Eld{\beta}{t}{i} \Eld{\gamma}{u}j \Eld{\beta}{-t}{i} \Eld{\gamma}{-u}j.
    \]
\end{relation}

Another useful relation derived from \Cnref{eq:comm:left:str} and \Cnref{rel:a3:sub:comm:alpha:beta} lets us reorder $\Rt{\alpha}$ and $\Rt{\beta}$ elements:

\begin{relation}[\RelNameA\RelNameOrder{\Rt{\alpha}}{\Rt{\beta}}]\label{rel:a3:order:alpha:beta}
    \[
    \DQuant{i,j \in [1]}{t,u \in \LiftScalars} \Eld{\alpha}{t}i \Eld{\beta}{u}j = \Eld{\alpha+\beta}{tu}{i+j} \Eld{\beta}{u}j \Eld{\alpha}{t}i. 
    \]
\end{relation}

And similarly from \Cnref{rel:a3:sub:comm:beta:gamma} for $\Rt{\beta}$ and $\Rt{\gamma}$ elements:

\begin{relation}[\RelNameA\RelNameOrder{\Rt{\beta}}{\Rt{\gamma}}]\label{rel:a3:order:beta:gamma}
    \[
    \DQuant{i,j \in [1]}{t,u \in \LiftScalars} \Eld{\beta}{t}i \Eld{\gamma}{u}j = \Eld{\beta+\gamma}{tu}{i+j} \Eld{\gamma}{u}j \Eld{\beta}{t}i. 
    \]
\end{relation}

\subsection{$\alpha+\beta$ and $\beta+\gamma$ commute}

Our first order of business is to show that $\Rt{\alpha+\beta}$ and $\Rt{\beta+\gamma}$ elements commute; this is the only commutation relation not involving the missing root $\Rt{\alpha+\beta+\gamma}$. We first prove:

\begin{relation}[\RelNameA\RelNameHomLift\RelNameCommute{\Rt{\alpha+\beta}}{\Rt{\beta+\gamma}}]\label{rel:a3:comm:lift:alpha+beta:beta+gamma:homog}
    \[
    \DQuant{i,j,k \in [1]}{t,u \in \LiftScalars} \Comm{\Eld{\alpha+\beta}{t}{i+j}}{\Eld{\beta+\gamma}{u}{j+k}} = \Id.
    \]
\end{relation}

\begin{proof}
    Apply \Cnref{rel:a3:comm:raw:lift:alpha+beta:beta+gamma} with $t_i = t$, $t_{1-i} = 0$, $u_1=u_0=1$, $v_i = u$, $v_{1-i} = 0$.
\end{proof}

\begin{relation}[\RelNameA\RelNamePart\RelNameCommute{\Rt{\alpha+\beta}}{\Rt{\beta+\gamma}}]\label{rel:a3:comm:lift:alpha+beta:beta+gamma:missing}
    \begin{align*}
        \Quant{t,u \in \LiftScalars} \Comm{\Eld{\alpha+\beta}{t}{2}}{\Eld{\beta+\gamma}{u}{0}} &= \Id, \\
        \Quant{t,u \in \LiftScalars} \Comm{\Eld{\alpha+\beta}{t}{0}}{\Eld{\beta+\gamma}{u}{2}} &= \Id.
    \end{align*}
\end{relation}

\begin{proof}
    We prove the first equality, the proof for the second is symmetric. We use \Cnref{rel:a3:comm:raw:lift:alpha+beta:beta+gamma} with $t_1 = t, t_0 = 1, u_1 = 1, u_0 = 1, v_1 = 0, v_0 = u$. This gives
    \begin{equation}\label{eq:a3:comm:lift:alpha+beta:beta+gamma}
    \Comm{\Eld{\alpha+\beta}{t}{2} \Eld{\alpha+\beta}{t+1}{1} \Eld{\alpha+\beta}{1}{0}}{\Eld{\beta+\gamma}{u}{1} \Eld{\beta+\gamma}{u}{0}} = \Id.
    \end{equation}
    Now we observe that for every $t,u \in \LiftScalars$, and every $(i,j) \in [2] \times [1] \setminus \{(2,0)\}$, the previous proposition (\Cnref{rel:a3:comm:lift:alpha+beta:beta+gamma:homog}) gives \[
    \Comm{\Eld{\alpha+\beta}{\pm t}{i}}{\Eld{\beta+\gamma}{\pm u}{j}} = \Id.
    \]
    Using \Cnref{rel:a3:sub:comm:self:alpha+beta} and \Cnref{rel:a3:sub:comm:self:beta+gamma}, we conclude that all pairs of (possibly non-distinct) elements \emph{except} $\Eld{\alpha+\beta}{\pm t}2$ and $\Eld{\beta+\gamma}{\pm u}{0}$ present in \Cref{eq:a3:comm:lift:alpha+beta:beta+gamma} commute and therefore $\Eld{\alpha+\beta}{t}{2}$ and $\Eld{\beta+\gamma}{u}{0}$ indeed commute.
\end{proof}

\begin{relation}[\RelNameA\RelNameCommute{\Rt{\alpha+\beta}}{\Rt{\beta+\gamma}}]\label{rel:a3:comm:alpha+beta:beta+gamma}
    \[
    \DQuant{i,j \in [2]}{t,u \in \LiftScalars} \Comm{\Eld{\alpha+\beta}{t}{i}}{\Eld{\beta+\gamma}{u}{j}} = \Id.
    \]
\end{relation}

\begin{proof}
    Together, \Cnref{rel:a3:comm:lift:alpha+beta:beta+gamma:homog} and \Cnref{rel:a3:comm:lift:alpha+beta:beta+gamma:missing} cover all degree-pairs in $[2] \times [2]$.
\end{proof}

\subsection{Establishing $\alpha+\beta+\gamma$}

Our next goal is to establish a syntactic definition of $\Rt{\alpha+\beta+\gamma}$ elements. Towards this, we prove the following relation, allowing us to interchange two commutators --- each of which, we have in the back of our minds, should create a $\Rt{\alpha+\beta+\gamma}$ element.

\begin{relation}[\RelNameA\RelNameInterchange{\Rt{\alpha+\beta+\gamma}}]\label{rel:a3:interchange:alpha+beta+gamma}
\[ \DQuant{i,j,k \in [1]}{t,u,v \in \LiftScalars} \Comm{\Eld{\alpha}{t}i}{\Eld{\beta+\gamma}{uv}{j+k}} = \Comm{\Eld{\alpha+\beta}{tu}{i+j}}{\Eld{\gamma}{v}{k}}. \]
\end{relation}

\begin{proof}
    We write a product of $\Rt{\alpha}$ and $\Rt{\beta+\gamma}$ elements:
    \begin{align*}
        & \Eld{\alpha}{t}{i} \asite{\Eld{\beta+\gamma}{uv}{j+k}} \\
        \intertext{Expand the $\Rt{\beta+\gamma}$ element as a product of $\Rt{\beta}$ and $\Rt{\gamma}$ elements (\Cnref{rel:a3:expr:beta+gamma}):}
        &= \asite{\Eld{\alpha}{t}{i} \Eld{\beta}{u}{j}} \Eld{\gamma}{v}{k} \Eld{\beta}{-u}{j} \Eld{\gamma}{-v}{k} \\
        \intertext{Commute the $\Rt{\alpha}$ element on the left fully to the right, introducing $\Rt{\alpha+\beta}$ commutators with $\Rt{\beta}$ elements (\Cnref{rel:a3:order:alpha:beta}]) and no commutators with $\Rt{\gamma}$ elements (\Cnref{rel:a3:sub:comm:alpha:gamma}):}
        &= \Eld{\alpha+\beta}{tu}{i+j} \asite{\Eld{\beta}{u}{j} \Eld{\gamma}{v}{k}} \Eld{\alpha+\beta}{-tu}{i+j} \Eld{\beta}{-u}{j} \Eld{\gamma}{-v}{k} \Eld{\alpha}{t}{i}. \\
        \intertext{Now, we move one $\Rt{\beta}$ element to the right across a $\Rt{\gamma}$ element (\Cnref{rel:a3:order:beta:gamma}):}
        &= \Eld{\alpha+\beta}{tu}{i+j} \Eld{\beta+\gamma}{uv}{j+k} \Eld{\gamma}{v}{k} \asite{\Eld{\beta}{u}{j} \Eld{\alpha+\beta}{-tu}{i+j} \Eld{\beta}{-u}{j}} \Eld{\gamma}{-v}{k}\Eld{\alpha}{t}{i} \\
        \intertext{at which point we can commute $\Rt{\beta}$ elements across the $\Rt{\alpha+\beta}$ element (\Cnref{rel:a3:sub:comm:beta:alpha+beta}) and then cancel them (\Cnref{rel:a3:sub:inv:beta}):}
        &= \Eld{\alpha+\beta}{tu}{i+j} \asite{\Eld{\beta+\gamma}{uv}{j+k} \Eld{\gamma}{v}{k} \Eld{\alpha+\beta}{-tu}{i+j} \Eld{\gamma}{-v}{k}} \Eld{\alpha}{t}{i}.  \\
        \intertext{Finally, we move the $\Rt{\beta+\gamma}$ element to the right, immediately before the $\Rt{\alpha}$ element, since $\Rt{\beta+\gamma}$ elements commutes with $\Rt{\gamma}$ elements (\Cnref{rel:a3:sub:comm:gamma:beta+gamma}) and $\Rt{\alpha+\beta}$ elements (\Cnref{rel:a3:comm:alpha+beta:beta+gamma}):}
        &= \asite{\Eld{\alpha+\beta}{tu}{i+j} \Eld{\gamma}{v}{k} \Eld{\alpha+\beta}{-tu}{i+j}  \Eld{\gamma}{-v}{k}} \Eld{\beta+\gamma}{uv}{j+k} \Eld{\alpha}{t}{i}  \\
        \intertext{Which reduces via \Cnref{rel:a3:sub:inv:alpha+beta} and \Cnref{rel:a3:sub:inv:gamma} to:}
        &= \Comm{\Eld{\alpha+\beta}{tu}{i+j}}{\Eld{\gamma}{v}{k}} \Eld{\beta+\gamma}{uv}{j+k} \Eld{\alpha}{t}{i}
    \end{align*}
    as desired.
\end{proof}

Given this relation, we can now \emph{define} $\Rt{\alpha+\beta+\gamma}$ elements and give the two commutator relations which create them in one fell swoop!

\begin{proposition}[Existence of $\Rt{\alpha+\beta+\gamma}$]\label{prop:a3:est:alpha+beta+gamma}
    There exist elements $\Eld{\alpha+\beta+\gamma}{t}{i}$ for $t \in \LiftScalars$ and $i \in [3]$ such that the following hold.
    \begin{gather*}
        \DQuant{i \in [1], j \in [2]}{t,u \in \LiftScalars} \Eld{\alpha+\beta+\gamma}{tu}{i+j} = \Comm{\Eld{\alpha}{t}{i}}{\Eld{\beta+\gamma}{u}{j}}, \\
        \DQuant{i \in [2], j \in [1]}{t,u \in \LiftScalars} \Eld{\alpha+\beta+\gamma}{tu}{i+j} = \Comm{\Eld{\alpha+\beta}{t}{i}}{\Eld{\gamma}{u}{j}}.
    \end{gather*}
\end{proposition}

\begin{proof}
    Follows from \Cref{prop:prelim:equal-conn}, \Cnref{rel:a3:interchange:alpha+beta+gamma}, and the fact that the corresponding graph (\Cref{fig:a3:alpha+beta+gamma}) is connected.
\end{proof}

We remark that any of the expansions of $\Eld{\alpha+\beta+\gamma}{t}{i}$ given in the above theorem may be taken as an ``alias'' for this element in the sense of \Cref{thm:structure4}; the same will be true for the $B_3$ expansion theorems below.

We now establish some useful relations for decomposing $\Rt{\alpha+\beta+\gamma}$ elements. By \Cnref{prop:a3:est:alpha+beta+gamma}, \Cnref{rel:a3:sub:inv:alpha}, and \Cnref{rel:a3:sub:inv:beta+gamma}:

\begin{relation}[\RelNameA\RelNameExprn{\Rt{\alpha+\beta+\gamma}}{\Rt{\alpha}}{\Rt{\beta+\gamma}}]\label{rel:a3:expr:alpha+beta+gamma:alpha:beta+gamma}
    \[
    \DQuant{i \in [1], j \in [2]}{t,u \in \LiftScalars} \Eld{\alpha+\beta+\gamma}{tu}{i+j} = \Eld{\alpha}{t}{i} \Eld{\beta+\gamma}{u}{j} \Eld{\alpha}{-t}{i} \Eld{\beta+\gamma}{-u}{j}.
    \]
\end{relation}

By \Cnref{prop:a3:est:alpha+beta+gamma}, \Cnref{rel:a3:sub:inv:alpha+beta}, and \Cnref{rel:a3:sub:inv:gamma}:

\begin{relation}[\RelNameA\RelNameExprn{\Rt{\alpha+\beta+\gamma}}{\Rt{\alpha+\beta}}{\Rt{\gamma}}]\label{rel:a3:expr:alpha+beta+gamma:alpha+beta:gamma}
    \[
    \DQuant{i \in [2], j \in [1]}{t,u \in \LiftScalars} \Eld{\alpha+\beta+\gamma}{tu}{i+j} = \Eld{\alpha+\beta}{t}{i} \Eld{\gamma}{u}{j} \Eld{\alpha+\beta}{-t}{i} \Eld{\gamma}{-u}{j}.
    \]
\end{relation}

\subsection{Remaining commutators}

Now, it remains to prove (a) that $\Rt{\alpha+\beta+\gamma}$ elements commutes with all remaining elements and (b) that $\Rt{\alpha+\beta+\gamma}$ elements themselves satisfy a linearity relation.

\begin{relation}[\RelNameA\RelNameCommute{\Rt{\alpha}}{\Rt{\alpha+\beta+\gamma}}]\label{rel:a3:comm:alpha:alpha+beta+gamma}
    \[ \DQuant{i \in [1], j \in [3]}{t,u \in \LiftScalars} \Comm{\Eld{\alpha}{t}i}{\Eld{\alpha+\beta+\gamma}{u}j} = \Id. \]
\end{relation}

\begin{proof}
    By \Cnref{rel:a3:expr:alpha+beta+gamma:alpha+beta:gamma}, $\Rt{\alpha+\beta+\gamma}$ elements can be expressed as products of $\Rt{\alpha+\beta}$ and $\Rt{\gamma}$ elements. $\Rt{\alpha}$ elements commute with $\Rt{\alpha+\beta}$ elements (\Cnref{rel:a3:sub:comm:alpha:alpha+beta}) and $\Rt{\gamma}$ elements (\Cnref{rel:a3:sub:comm:alpha:gamma}).
\end{proof}

Symmetrically, we have:

\begin{relation}[\RelNameA\RelNameCommute{\Rt{\gamma}}{\Rt{\alpha+\beta+\gamma}}]\label{rel:a3:comm:gamma:alpha+beta+gamma}
    \[ \DQuant{i \in [1], j \in [3]}{t,u \in \LiftScalars} \Comm{\Eld{\gamma}{t}i}{\Eld{\alpha+\beta+\gamma}{u}j} = \Id. \]
\end{relation}

\begin{proof}
    By \Cnref{rel:a3:expr:alpha+beta+gamma:alpha:beta+gamma}, $\Rt{\alpha+\beta+\gamma}$ elements can be expressed as products of $\Rt{\alpha}$ and $\Rt{\beta+\gamma}$ elements. $\Rt{\gamma}$ elements commute with $\Rt{\alpha}$ elements (\Cnref{rel:a3:sub:comm:alpha:gamma}) and $\Rt{\beta+\gamma}$ elements (\Cnref{rel:a3:sub:comm:gamma:beta+gamma}).
\end{proof}

With slightly more effort, we also have:

\begin{relation}[\RelNameA\RelNameCommute{\Rt{\beta}}{\Rt{\alpha+\beta+\gamma}}]\label{rel:a3:comm:beta:alpha+beta+gamma}
    \[ \DQuant{i \in [1], j \in [3]}{t,u \in \LiftScalars} \Comm{\Eld{\beta}{t}i}{\Eld{\alpha+\beta+\gamma}{u}j} = \Id. \]
\end{relation}

\begin{proof}
We write a product of $\Rt{\beta}$ and $\Rt{\alpha+\beta+\gamma}$ elements:
    \begin{align*}
        & \Eld{\beta}{t}i \asite{\Eld{\alpha+\beta+\gamma}{u}j} \\
        \intertext{Decomposing $j = j_1+j_2$ for $j_1 \in [2], j_2 \in [1]$ arbitrarily, and expanding the $\Rt{\alpha+\beta+\gamma}$ element into a product of $\Rt{\alpha+\beta}$ and $\Rt{\gamma}$ elements (\Cnref{rel:a3:expr:alpha+beta+gamma:alpha+beta:gamma}):}
        &= \asite{\Eld{\beta}{t}i \Eld{\alpha+\beta}{u}{j_1}} \Eld{\gamma}{1}{j_2} \Eld{\alpha+\beta}{-u}{j_1} \Eld{\gamma}{-1}{j_2} \\
        \intertext{Moving the $\Rt{\beta}$ element from the left fully to the right, creating no commutators with $\Rt{\alpha+\beta}$ elements (\Cnref{rel:a3:sub:comm:beta:alpha+beta}) and $\Rt{\beta+\gamma}$ commutators with $\Rt{\gamma}$ elements (\Cnref{rel:a3:order:beta:gamma}):}
        &= \Eld{\alpha+\beta}{u}{j_1} \Eld{\beta+\gamma}{t}{i+j_2} \Eld{\gamma}{1}{j_2} \Eld{\alpha+\beta}{-t}{j_1} \Eld{\beta+\gamma}{-u}{i+j_2} \Eld{\gamma}{-1}{j_2} \Eld{\beta}{t}{i}. \\
        \intertext{Now we can commute the $\Rt{\beta+\gamma}$ elements across the $\Rt{\alpha+\beta}$ element (\Cnref{rel:a3:comm:alpha+beta:beta+gamma}), thereby cancelling them (\Cnref{rel:a3:sub:inv:beta+gamma}):}
        &= \Eld{\alpha+\beta}{u}{j_1} \Eld{\gamma}{1}{j_2} \Eld{\alpha+\beta}{-u}{j_1} \Eld{\gamma}{-1}{j_2} \Eld{\beta}{t}{i} \\
        \intertext{Reducing back into an $\Rt{\alpha+\beta+\gamma}$ element (\Cnref{rel:a3:expr:alpha+beta+gamma:alpha+beta:gamma}):}
        &= \Eld{\alpha+\beta+\gamma}{u}{j} \Eld{\beta}{t}{i},
    \end{align*}
    as desired.
\end{proof}

Now we get several easy corollaries:

\begin{relation}[\RelNameA\RelNameCommute{\Rt{\alpha+\beta}}{\Rt{\alpha+\beta+\gamma}}]\label{rel:a3:comm:alpha+beta:alpha+beta+gamma}
    \[ \DQuant{i \in [2], j \in [3]}{t,u \in \LiftScalars}  \Comm{\Eld{\alpha+\beta}{t}i}{\Eld{\alpha+\beta+\gamma}{u}j} = \Id. \]
\end{relation}

\begin{proof}
    $\Rt{\alpha+\beta}$ elements can be expressed as products of $\Rt{\alpha}$ and $\Rt{\beta}$ elements by \Cnref{rel:a3:expr:alpha+beta}. These types of elements commute with $\Rt{\alpha+\beta+\gamma}$ elements by \Cnref{rel:a3:comm:alpha:alpha+beta+gamma} and \Cnref{rel:a3:comm:beta:alpha+beta+gamma}, respectively.
\end{proof}

\begin{relation}[\RelNameA\RelNameCommute{\Rt{\beta+\gamma}}{\Rt{\alpha+\beta+\gamma}}]\label{rel:a3:comm:beta+gamma:alpha+beta+gamma}
    \[ \DQuant{i \in [2], j \in [3]}{t,u \in \LiftScalars}  \Comm{\Eld{\beta+\gamma}{t}i}{\Eld{\alpha+\beta+\gamma}{u}j} = \Id. \]
\end{relation}

\begin{proof}
    $\Rt{\beta+\gamma}$ elements can be expressed as products of $\Rt{\beta}$ and $\Rt{\gamma}$ elements by \Cnref{rel:a3:expr:beta+gamma}. These types of elements commute with $\Rt{\alpha+\beta+\gamma}$ elements by \Cnref{rel:a3:comm:beta:alpha+beta+gamma} and \Cnref{rel:a3:comm:gamma:alpha+beta+gamma}, respectively.
\end{proof}

\begin{relation}[\RelNameA\RelNameSelfCommute{\Rt{\alpha+\beta+\gamma}}]\label{rel:a3:comm:self:alpha+beta+gamma}
    \[ \DQuant{i,j \in [3]}{t,u \in \LiftScalars}  \Comm{\Eld{\alpha+\beta+\gamma}{t}i}{\Eld{\alpha+\beta+\gamma}{u}j} = \Id. \]
\end{relation}

\begin{proof}
    $\Rt{\alpha+\beta+\gamma}$ elements can be expressed as products of $\Rt{\alpha+\beta}$ and $\Rt{\gamma}$ elements by \Cnref{rel:a3:expr:alpha+beta+gamma:alpha+beta:gamma}. These types of elements commute with $\Rt{\alpha+\beta+\gamma}$ elements by \Cnref{rel:a3:comm:alpha+beta:alpha+beta+gamma} and \Cnref{rel:a3:comm:gamma:alpha+beta+gamma}, respectively.
\end{proof}

\subsection{Linearity for $\alpha+\beta+\gamma$}

Finally, we prove linearity for $\Rt{\alpha+\beta+\gamma}$ elements.

\begin{relation}[\RelNameA\RelNameLinearity{\Rt{\alpha+\beta+\gamma}}]\label{rel:a3:lin:alpha+beta+gamma}
    \[ \DQuant{i \in [3]}{t,u \in \LiftScalars} \Eld{\alpha+\beta+\gamma}{t}i \Eld{\alpha+\beta+\gamma}{u}i = \Eld{\alpha+\beta+\gamma}{t+u}i. \]
\end{relation}

\begin{proof}
    We write a product of $\Rt{\alpha+\beta+\gamma}$ elements:
    \begin{align*}
        & \asite{\Eld{\alpha+\beta+\gamma}{t}i} \Eld{\alpha+\beta+\gamma}{u}i \\
        \intertext{Decomposing $i = i_1 + i_2$ for $i_1 \in [2],i_2 \in [1]$ arbitrarily and expanding one $\Rt{\alpha+\beta+\gamma}$ into a product of $\Rt{\alpha+\beta}$ and $\Rt{\gamma}$ elements (\Cnref{rel:a3:expr:alpha+beta+gamma:alpha+beta:gamma}):}
        &= \Eld{\alpha+\beta}{t}{i_1} \Eld{\gamma}{1}{i_2} \Eld{\alpha+\beta}{-t}{i_1} \Eld{\gamma}{-1}{i_2} \asite{\Eld{\alpha+\beta+\gamma}{u}i}. \\
        \intertext{Moving the other $\Rt{\alpha+\beta+\gamma}$ element on the right \emph{partially} to the left since it commutes with $\Rt{\gamma}$ elements (\Cnref{rel:a3:comm:gamma:alpha+beta+gamma}) and $\Rt{\alpha+\beta}$ elements (\Cnref{rel:a3:comm:alpha+beta:alpha+beta+gamma}):}
        &= \Eld{\alpha+\beta}{t}{i_1} \asite{\Eld{\alpha+\beta+\gamma}{u}i} \Eld{\gamma}{1}{i_2} \Eld{\alpha+\beta}{-t}{i_1} \Eld{\gamma}{-1}{i_2} \\
        \intertext{Expanding the other $\Rt{\alpha+\beta+\gamma}$ into a similar product of $\Rt{\alpha+\beta}$ and $\Rt{\gamma}$ elements (\Cnref{rel:a3:expr:alpha+beta+gamma:alpha+beta:gamma})}
        &= \asite{\Eld{\alpha+\beta}{t}{i_1} \Eld{\alpha+\beta}{u}{i_1}} \Eld{\gamma}{1}{i_2} \Eld{\alpha+\beta}{-u}{i_1} \asite{\Eld{\gamma}{-1}{i_2} \Eld{\gamma}{1}{i_2}} \Eld{\alpha+\beta}{-t}{i_1} \Eld{\gamma}{-1}{i_2} \\
        \intertext{Applying linearity to the two consecutive $\Rt{\alpha+\beta}$ elements (\Cnref{rel:a3:sub:lin:alpha+beta}) and cancelling the two consecutive $\Rt{\gamma}$ elements (\Cnref{rel:a3:sub:inv:gamma}):}
        &= \Eld{\alpha+\beta}{t+u}{i_1} \Eld{\gamma}{1}{i_2} \asite{\Eld{\alpha+\beta}{-u}{i_1} \Eld{\alpha+\beta}{-t}{i_1}} \Eld{\gamma}{-1}{i_2} \\
        \intertext{Applying linearity to the two newly consecutive $\Rt{\alpha+\beta}$ elements (\Cnref{rel:a3:sub:lin:alpha+beta}):}
        &= \asite{\Eld{\alpha+\beta}{t+u}{i_1} \Eld{\gamma}{1}{i_2} \Eld{\alpha+\beta}{-(t+u)}{i_1} \Eld{\gamma}{-1}{i_2}}  \\
        \intertext{Reducing into a single $\Rt{\alpha+\beta+\gamma}$ element (\Cnref{rel:a3:expr:alpha+beta+gamma:alpha+beta:gamma}):}
        &= \Eld{\alpha+\beta+\gamma}{t+u}{i}.
    \end{align*}
    as desired.
\end{proof}

%% file: figures/a3/alpha+beta+gamma.tex
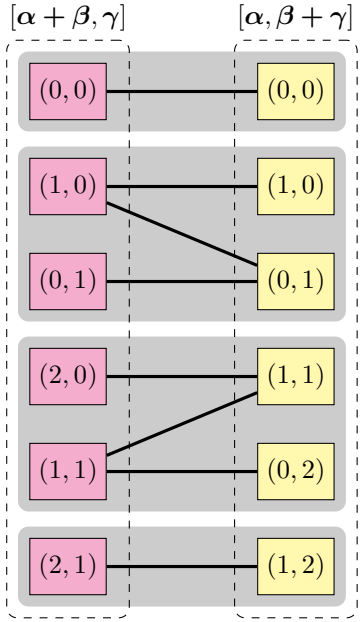
\begin{figure}
    \centering
    \begin{tikzpicture}
\node[algnode0] (v000) {$(0,0)$};
\node[algnode0] (v010) [below=of v000] {$(1,0)$};
\node[algnode0] (v001) [below=of v010] {$(0,1)$};
\node[algnode0] (v020) [below=of v001] {$(2,0)$};
\node[algnode0] (v011) [below=of v020] {$(1,1)$};
\node[algnode0] (v021) [below=of v011] {$(2,1)$};
\node[algnode1] (v100) [right=of v000] {$(0,0)$};
\node[algnode1] (v110) [below=of v100] {$(1,0)$};
\node[algnode1] (v101) [below=of v110] {$(0,1)$};
\node[algnode1] (v111) [below=of v101] {$(1,1)$};
\node[algnode1] (v102) [below=of v111] {$(0,2)$};
\node[algnode1] (v112) [below=of v102] {$(1,2)$};
\draw[algedge] (v011) -- (v111);
\draw[algedge] (v021) -- (v112);
\draw[algedge] (v000) -- (v100);
\draw[algedge] (v010) -- (v110);
\draw[algedge] (v011) -- (v102);
\draw[algedge] (v001) -- (v101);
\draw[algedge] (v010) -- (v101);
\draw[algedge] (v020) -- (v111);
\begin{scope}[on background layer]
    \node[algbackh, fit={(v000) (v100)}] {};
    \node[algbackh, fit={(v010) (v101)}] {};
    \node[algbackh, fit={(v020) (v102)}] {};
    \node[algbackh, fit={(v021) (v112)}] {};
    \node[algback, fit={(v000) (v021)}] (b0) {};
    \node[above=0 of b0] {$\Comm{\Rt{\alpha+\beta}}{\Rt{\gamma}}$};
    \node[algback, fit={(v100) (v112)}] (b1) {};
    \node[above=0 of b1] {$\Comm{\Rt{\alpha}}{\Rt{\beta+\gamma}}$};
\end{scope}
\end{tikzpicture}
\caption{Establishing $\Rt{\alpha+\beta+\gamma}$: A bipartite graph with left vertex-set $[2] \times [1]$ and right vertex-set $[1] \times [2]$, with an edge $(i,j) \sim (k,\ell)$ iff \Cnref{rel:a3:interchange:alpha+beta+gamma} states that for all $t,u,v \in R$, $\Comm{\Eld{\alpha+\beta}{tu}{i}}{\Eld{\gamma}{v}{j}} = \Comm{\Eld{\alpha}{t}{k}}{\Eld{\beta+\gamma}{2uv}{\ell}}$. (Hence, there are edges $(i+j,k) \sim (i,j+k)$ for every $i,j,k \in [1]$.) Additionally, grey blocks partition the vertices based on the sum of coordinates in $[3]$. In this case, the blocks also correspond to connected components in the graph.}\label{fig:a3:alpha+beta+gamma}
\end{figure}

%% file: sections/08-B3.tex
\section{Lifting for $B_3$}\label{sec:b3}

We next turn to proving \Cref{thm:story:b3:lift}, the lifting theorem for $B_3$ links.

\input{sections/08a-B3-small}
\input{sections/08b-B3-large}

%% file: sections/08a-B3-small.tex
\subsection{Link of $\alpha$}\label{sec:b3:small}

We first prove \Cref{thm:story:b3:lift} for the ``small'' link of $B_3$, i.e., a lifting theorem from $\UnipBSm{p}$ to $\GrUnipBSm{p^k}$. All relations stated in this section hold in the graded link group $\GrUnipBSm{p^k}$ and are proven via constant-length in-subgroup and lifted relations in a number of steps independent of $p$ or $k$.

\subsubsection{Link structure}

\begin{table}[h!]
\centering
\begin{tabular}{c|c|c|c}
    \textbf{Color} & \textbf{Base} & \textbf{Positive span} & \textbf{Total \#} \\ \hline
    & $\{\Rt{\beta},\Rt{\psi},\Rt{\omega}\}$ & $\{\Rt{\beta+\psi},\Rt{\psi+\omega},\Rt{\beta+2\psi},\Rt{\beta+\psi+\omega}\}$ & 7 \\
    \BoxR & $\{\Rt{\beta},\Rt{\psi}\}$ & $\{\Rt{\beta+\psi},\Rt{\beta+2\psi}\}$ & 4 \\
    \BoxG & $\{\Rt{\beta},\Rt{\omega}\}$ & $\emptyset$ & 2 \\
    \BoxB & $\{\Rt{\psi},\Rt{\omega}\}$ & $\{\Rt{\psi+\omega}\}$ & 3 \\
\end{tabular}
\caption{The roots spanned by $\Rt{\beta},\Rt{\psi},\Rt{\omega}$ in $B_3$ as well as the subsets spanned by pairs.}
\end{table}

The link of $\Rt{\alpha}$ is the subset of $B_3$ generated by the basis $\Rt{\beta},\Rt{\psi},\Rt{\omega}$. The link contains $7$ vectors, of which $6$ (all but the missing root $\Rt{\beta+\psi+\omega}$) are contained in the span of some subset of two basis vectors. In the corresponding Steinberg group $\UnipBSm{q}$, there are $\binom{7}{2} = 21$ total commutation relations and $7$ linear relations. The \BoxR, \BoxG, and \BoxB subgroups together cover $\binom{4}2 + \binom{2}2 + \binom{3}2 = 10$ commutation relations and $6$ linear relations, leaving us with a total of $11$ commutation relations and $1$ linear relation to prove. There are two commutation relations not involving a ``missing'' root: $\Rt{\beta+\psi}$ and $\Rt{\psi+\omega}$ commute and $\Rt{\beta+2\psi}$ and $\Rt{\omega}$ commute.

\begin{remark}
    The sequence of manipulations shows we only need weak information from lifting: The relation that $\Rt{\beta+\psi}$ and $\Rt{\psi+\omega}$ elements commute. Indeed, even in the unlifted case, it shows that this non-in-subgroup relation, together with the in-subgroup relations, gives the structure of the entire group (one can get this proof by ignoring all degree markers). Interestingly, we do \emph{not} need to lift the relation that $\Rt{\beta+2\psi}$ and $\Rt{\omega}$ elements commute: Instead, we can circle back at the end and prove this relation directly (\Cnref{rel:b3-small:comm:beta+2psi:omega} below).
\end{remark}

\subsubsection{Assumed relations}

\paragraph{In-subgroup relations}

We enumerate the following in-subgroup relations in $\GrUnipBSm{p^k}$:

\begin{tcolorbox}[colback=\ColorR, colframe=\ColorBackR, title={Commutation relations spanned by $\Rt{\beta}$ and $\Rt{\psi}$}]
\begin{relation}[\RelNameBSm\RelNameSubgp\RelNameCommutator{\Rt{\beta}}{\Rt{\psi}}]\label{rel:b3-small:sub:comm:beta:psi}
    \[ \DQuant{i, j \in [1]}{t,u \in \LiftScalars} \Comm{\Eld{\beta}{t}{i}}{\Eld{\psi}{u}{j}} = \Eld{\beta+\psi}{tu}{i+j} \Eld{\beta+2\psi}{tu^2}{i+2j}. \]
\end{relation}
\begin{relation}[\RelNameBSm\RelNameSubgp\RelNameCommute{\Rt{\beta}}{\Rt{\beta+\psi}}]\label{rel:b3-small:sub:comm:beta:beta+psi}
    \[ \DQuant{i \in [1], j \in [2]}{t,u \in \LiftScalars} \Comm{\Eld{\beta}{t}{i}}{\Eld{\beta+\psi}{u}{j}} = \Id. \]
\end{relation}
\begin{relation}[\RelNameBSm\RelNameSubgp\RelNameCommute{\Rt{\beta}}{\Rt{\beta+2\psi}}]\label{rel:b3-small:sub:comm:beta:beta+2psi}
    \[ \DQuant{i \in [1], j \in [3]}{t,u \in \LiftScalars} \Comm{\Eld{\beta}{t}{i}}{\Eld{\beta+2\psi}{u}{j}} = \Id. \]
\end{relation}
\begin{relation}[\RelNameBSm\RelNameSubgp\RelNameCommutator{\Rt{\psi}}{\Rt{\beta+\psi}}]\label{rel:b3-small:sub:comm:psi:beta+psi}
    \[ \DQuant{i \in [1], j \in [2]}{t,u \in \LiftScalars} \Comm{\Eld{\psi}{t}{i}}{\Eld{\beta+\psi}{u}{j}} = \Eld{\beta+2\psi}{2tu}{i+j}. \]
\end{relation}
\begin{relation}[\RelNameBSm\RelNameSubgp\RelNameCommute{\Rt{\psi}}{\Rt{\beta+2\psi}}]\label{rel:b3-small:sub:comm:psi:beta+2psi}
    \[ \DQuant{i \in [1], j \in [3]}{t,u \in \LiftScalars} \Comm{\Eld{\psi}{t}{i}}{\Eld{\beta+2\psi}{u}{j}} = \Id. \]
\end{relation}
\begin{relation}[\RelNameBSm\RelNameSubgp\RelNameCommute{\Rt{\beta+\psi}}{\Rt{\beta+2\psi}}]\label{rel:b3-small:sub:comm:beta+psi:beta+2psi}
    \[ \DQuant{i \in [2], j \in [3]}{t,u \in \LiftScalars} \Comm{\Eld{\beta+\psi}{t}{i}}{\Eld{\beta+2\psi}{u}{j}} = \Id. \]
\end{relation}
\end{tcolorbox}

\begin{tcolorbox}[colback=\ColorG, colframe=\ColorBackG, title={Commutation relation spanned by $\Rt{\beta}$ and $\Rt{\omega}$}]
\begin{relation}[\RelNameBSm\RelNameSubgp\RelNameCommute{\Rt{\beta}}{\Rt{\omega}}]\label{rel:b3-small:sub:comm:beta:omega}
    \[ \DQuant{i, j \in [1]}{t,u \in \LiftScalars} \Comm{\Eld{\beta}{t}{i}}{\Eld{\omega}{u}{j}} = \Id. \]
\end{relation}
\end{tcolorbox}

\begin{tcolorbox}[colback=\ColorB, colframe=\ColorBackB, title={Commutation relations spanned by $\Rt{\psi}$ and $\Rt{\omega}$}]
\begin{relation}[\RelNameBSm\RelNameSubgp\RelNameCommutator{\Rt{\psi}}{\Rt{\omega}}]\label{rel:b3-small:sub:comm:psi:omega}
    \[ \DQuant{i, j \in [1]}{t,u \in \LiftScalars} \Comm{\Eld{\psi}{t}{i}}{\Eld{\omega}{u}{j}} = \Eld{\psi+\omega}{2tu}{i+j}. \]
\end{relation}
\begin{relation}[\RelNameBSm\RelNameSubgp\RelNameCommute{\Rt{\psi}}{\Rt{\psi+\omega}}]\label{rel:b3-small:sub:comm:psi:psi+omega}
    \[ \DQuant{i \in [1], j \in [2]}{t,u \in \LiftScalars} \Comm{\Eld{\psi}{t}{i}}{\Eld{\psi+\omega}{u}{j}} = \Id. \]
\end{relation}
\begin{relation}[\RelNameBSm\RelNameSubgp\RelNameCommute{\Rt{\omega}}{\Rt{\psi+\omega}}]\label{rel:b3-small:sub:comm:omega:psi+omega}
    \[ \DQuant{i \in [1], j \in [2]}{t,u \in \LiftScalars} \Comm{\Eld{\omega}{t}{i}}{\Eld{\psi+\omega}{u}{j}} = \Id. \]
\end{relation}
\end{tcolorbox}

\begin{tcolorbox}[breakable,colback=\ColorNeut,title={Linearity relations}]
\begin{relation}[\RelNameBSm\RelNameSubgp\RelNameLinearity{\Rt{\beta}}]\label{rel:b3-small:sub:lin:beta}
    \[ \DQuant{i \in [1]}{t,u \in \LiftScalars} \Eld{\beta}{t}{i} \Eld{\beta}{u}{i} = \Eld{\beta}{t+u}{i}. \]
\end{relation}
\begin{relation}[\RelNameBSm\RelNameSubgp\RelNameLinearity{\Rt{\psi}}]\label{rel:b3-small:sub:lin:psi}
    \[ \DQuant{i \in [1]}{t,u \in \LiftScalars} \Eld{\psi}{t}{i} \Eld{\psi}{u}{i} = \Eld{\psi}{t+u}{i}. \]
\end{relation}
\begin{relation}[\RelNameBSm\RelNameSubgp\RelNameLinearity{\Rt{\omega}}]\label{rel:b3-small:sub:lin:omega}
    \[ \DQuant{i \in [1]}{t,u \in \LiftScalars} \Eld{\omega}{t}{i} \Eld{\omega}{u}{i} = \Eld{\omega}{t+u}{i}. \]
\end{relation}
\begin{relation}[\RelNameBSm\RelNameSubgp\RelNameLinearity{\Rt{\beta+\psi}}]\label{rel:b3-small:sub:lin:beta+psi}
    \[ \DQuant{i \in [2]}{t,u \in \LiftScalars} \Eld{\beta+\psi}{t}{i} \Eld{\beta+\psi}{u}{i} = \Eld{\beta+\psi}{t+u}{i}. \]
\end{relation}
\begin{relation}[\RelNameBSm\RelNameSubgp\RelNameLinearity{\Rt{\psi+\omega}}]\label{rel:b3-small:sub:lin:psi+omega}
    \[ \DQuant{i \in [2]}{t,u \in \LiftScalars} \Eld{\psi+\omega}{t}{i} \Eld{\psi+\omega}{u}{i} = \Eld{\psi+\omega}{t+u}{i}. \]
\end{relation}
\begin{relation}[\RelNameBSm\RelNameSubgp\RelNameLinearity{\Rt{\beta+2\psi}}]\label{rel:b3-small:sub:lin:beta+2psi}
    \[ \DQuant{i \in [3]}{t,u \in \LiftScalars} \Eld{\beta+2\psi}{t}{i} \Eld{\beta+2\psi}{u}{i} = \Eld{\beta+2\psi}{t+u}{i}. \]
\end{relation}
\end{tcolorbox}

\begin{tcolorbox}[breakable,colback=\ColorNeut,title={Self-commutator relations}]
\begin{relation}[\RelNameBSm\RelNameSubgp\RelNameSelfCommute{\Rt{\beta}}]\label{rel:b3-small:sub:comm:self:beta}
    \[ \DQuant{i,j \in [1]}{t,u \in \LiftScalars} \Comm{\Eld{\beta}{t}{i}}{\Eld{\beta}{u}{j}} = \Id. \]
\end{relation}
\begin{relation}[\RelNameBSm\RelNameSubgp\RelNameSelfCommute{\Rt{\psi}}]\label{rel:b3-small:sub:comm:self:psi}
    \[ \DQuant{i,j \in [1]}{t,u \in \LiftScalars} \Comm{\Eld{\psi}{t}{i}}{\Eld{\psi}{u}{j}} = \Id. \]
\end{relation}
\begin{relation}[\RelNameBSm\RelNameSubgp\RelNameSelfCommute{\Rt{\omega}}]\label{rel:b3-small:comm:self:omega}
    \[ \DQuant{i,j \in [1]}{t,u \in \LiftScalars} \Comm{\Eld{\omega}{t}{i}}{\Eld{\omega}{u}{j}} = \Id. \]
\end{relation}
\begin{relation}[\RelNameBSm\RelNameSubgp\RelNameSelfCommute{\Rt{\beta+\psi}}]\label{rel:b3-small:sub:comm:self:beta+psi}
    \[ \DQuant{i,j \in [2]}{t,u \in \LiftScalars} \Comm{\Eld{\beta+\psi}{t}{i}}{\Eld{\beta+\psi}{u}{j}} = \Id. \]
\end{relation}
\begin{relation}[\RelNameBSm\RelNameSubgp\RelNameSelfCommute{\Rt{\psi+\omega}}]\label{rel:b3-small:sub:comm:self:psi+omega}
    \[ \DQuant{i,j \in [2]}{t,u \in \LiftScalars} \Comm{\Eld{\psi+\omega}{t}{i}}{\Eld{\psi+\omega}{u}{j}} = \Id. \]
\end{relation}
\begin{relation}[\RelNameBSm\RelNameSubgp\RelNameSelfCommute{\Rt{\beta+2\psi}}]\label{rel:b3-small:sub:comm:self:beta+2psi}
    \[ \DQuant{i,j \in [3]}{t,u \in \LiftScalars} \Comm{\Eld{\beta+2\psi}{t}{i}}{\Eld{\beta+2\psi}{u}{j}} = \Id. \]
\end{relation}
\end{tcolorbox}

\paragraph{Lifted relations} We again only use a single lifted relation, the nonhomogeneous lift (cf. \Cref{eq:chev:nonhom}) of the relation $\Comm{\El{\beta+\psi}{1}}{\El{\psi+\omega}{1}} = 1$ in $\UnipBSm{q}$.

\begin{relation}[\RelNameBSm\RelNameNonHomLiftRaw\RelNameCommute{\Rt{\beta+\psi}}{\Rt{\psi+\omega}}]\label{rel:b3-small:comm:raw:lift:beta+psi:psi+omega}
    \begin{multline*}
        \Quant{t_1,t_0,u_1,u_0,v_1,v_0 \in \LiftScalars} \\
        \Comm{\Eld{\beta+\psi}{t_1u_1}{2} \Eld{\beta+\psi}{t_1u_0+t_0u_1}{1} \Eld{\beta+\psi}{t_0u_0}{0}}{\Eld{\psi+\omega}{u_1v_1}{2} \Eld{\psi+\omega}{u_1v_0+u_0v_1}{1} \Eld{\psi+\omega}{u_0v_0}{0}} = \Id.
    \end{multline*}
\end{relation}

\subsubsection{Additional identities}

Again, from the linearity relations (\Cnref{rel:b3-large:sub:lin:alpha}, \Cnref{rel:b3-large:sub:lin:beta}, \Cnref{rel:b3-large:sub:lin:psi}, \Cnref{rel:b3-large:sub:lin:alpha+beta}, \Cnref{rel:b3-large:sub:lin:beta+psi}, \Cnref{rel:b3-large:sub:lin:beta+2psi}) and \Cref{prop:prelim:group-homo}, we get respective identity relations:

\begin{tcolorbox}[breakable,colback=\ColorNeut,title={Identity relations}]
\begin{relation}[\RelNameBSm\RelNameSubgp\RelNameId{\Rt{\beta}}]\label{rel:b3-small:sub:id:beta}
    \[ \Quant{i \in [1]} \Eld{\beta}{0}{i} = \Id. \]
\end{relation}
\begin{relation}[\RelNameBSm\RelNameSubgp\RelNameId{\Rt{\psi}}]\label{rel:b3-small:sub:id:psi}
    \[ \Quant{i \in [1]} \Eld{\psi}{0}{i} = \Id. \]
\end{relation}
\begin{relation}[\RelNameBSm\RelNameSubgp\RelNameId{\Rt{\omega}}]\label{rel:b3-small:sub:id:omega}
    \[ \Quant{i \in [1]} \Eld{\omega}{0}{i} = \Id. \]
\end{relation}
\begin{relation}[\RelNameBSm\RelNameSubgp\RelNameId{\Rt{\beta+\psi}}]\label{rel:b3-small:sub:id:beta+psi}
    \[ \Quant{i \in [2]} \Eld{\beta+\psi}{0}{i} = \Id. \]
\end{relation}
\begin{relation}[\RelNameBSm\RelNameSubgp\RelNameId{\Rt{\psi+\omega}}]\label{rel:b3-small:sub:id:psi+omega}
    \[ \Quant{i \in [2]} \Eld{\psi+\omega}{0}{i} = \Id. \]
\end{relation}
\begin{relation}[\RelNameBSm\RelNameSubgp\RelNameId{\Rt{\beta+2\psi}}]\label{rel:b3-small:sub:id:beta+2psi}
    \[ \Quant{i \in [3]} \Eld{\beta+2\psi}{0}{i} = \Id. \]
\end{relation}
\end{tcolorbox}

and the inverse relations:

\begin{tcolorbox}[breakable,colback=\ColorNeut,title={Inverse relations}]
\begin{relation}[\RelNameBSm\RelNameSubgp\RelNameInv{\Rt{\beta}}]\label{rel:b3-small:sub:inv:beta}
    \[ \DQuant{i \in [1]}{t \in \LiftScalars} \Eld{\beta}{t}{i} \Eld{\beta}{-t}{i} = \Id. \]
\end{relation}
\begin{relation}[\RelNameBSm\RelNameSubgp\RelNameInv{\Rt{\psi}}]\label{rel:b3-small:sub:inv:psi}
    \[ \DQuant{i \in [1]}{t \in \LiftScalars} \Eld{\psi}{t}{i} \Eld{\psi}{-t}{i} = \Id. \]
\end{relation}
\begin{relation}[\RelNameBSm\RelNameSubgp\RelNameInv{\Rt{\omega}}]\label{rel:b3-small:sub:inv:omega}
    \[ \DQuant{i \in [1]}{t \in \LiftScalars} \Eld{\omega}{t}{i} \Eld{\omega}{-t}{i} = \Id. \]
\end{relation}
\begin{relation}[\RelNameBSm\RelNameSubgp\RelNameInv{\Rt{\beta+\psi}}]\label{rel:b3-small:sub:inv:beta+psi}
    \[ \DQuant{i \in [2]}{t \in \LiftScalars} \Eld{\beta+\psi}{t}{i} \Eld{\beta+\psi}{-t}{i} = \Id. \]
\end{relation}
\begin{relation}[\RelNameBSm\RelNameSubgp\RelNameInv{\Rt{\psi+\omega}}]\label{rel:b3-small:sub:inv:psi+omega}
    \[ \DQuant{i \in [2]}{t \in \LiftScalars} \Eld{\psi+\omega}{t}{i} \Eld{\psi+\omega}{-t}{i} = \Id. \]
\end{relation}
\begin{relation}[\RelNameBSm\RelNameSubgp\RelNameInv{\Rt{\beta+2\psi}}]\label{rel:b3-small:sub:inv:beta+2psi}
    \[ \DQuant{i \in [3]}{t \in \LiftScalars} \Eld{\beta+2\psi}{t}{i} \Eld{\beta+2\psi}{-t}{i} = \Id. \]
\end{relation}
\end{tcolorbox}

We also use some expansions:

\begin{relation}[\RelNameBSm\RelNameSubgp\RelNameExpr{\Rt{\beta+\psi}}]\label{rel:b3-small:sub:expr:beta+psi}
    \[ \DQuant{i, j \in [1]}{t,u \in \LiftScalars} \Eld{\beta+\psi}{2tu}{i+j} = \Eld{\psi}{-u}{j} \Eld{\beta}{t}{i} \Eld{\psi}{2u}{j} \Eld{\beta}{-t}{i} \Eld{\psi}{-u}{j}. \]
\end{relation}

\begin{proof}
    Start with \Cnref{rel:b3-small:sub:comm:beta:psi}:
    \begin{align*}
    \Comm{\Eld{\beta}{t}{i}}{\Eld{\psi}{2u}{j}} &= 
    \Eld{\beta+\psi}{2tu}{i+j} \Eld{\beta+2\psi}{4tu^2}{i+2j} \\
    \intertext{Commute $\Rt{\beta+2\psi}$ and $\Rt{\beta+\psi}$ (\Cnref{rel:b3-small:sub:comm:beta+psi:beta+2psi}):}
    &= \Eld{\beta+2\psi}{4tu^2}{i+2j} \Eld{\beta+\psi}{2tu}{i+j} \\
    \intertext{Replace $\Rt{\beta+2\psi}$ with a commutator of $\Rt{\psi}$ and $\Rt{\beta+\psi}$ elements (\Cnref{rel:b3-small:sub:comm:psi:beta+psi}):}
    &= \Comm{\Eld{\psi}{u}{j}}{\Eld{\beta+\psi}{2tu}{i+j}} \Eld{\beta+\psi}{2tu}{i+j} \\
    \intertext{Expand the commutator and cancel with the $\Rt{\beta+\psi}$ element on the right (\Cnref{rel:b3-small:sub:inv:beta+psi} and \Cnref{rel:b3-small:sub:inv:psi}):}
    &= \Eld{\psi}{u}{j} \Eld{\beta+\psi}{2tu}{i+j} \Eld{\psi}{-u}{j}.
    \end{align*}
    Expanding the commutator on the LHS and conjugating by $\Rt{\psi}$ elements (\Cnref{rel:b3-small:sub:inv:psi} and \Cnref{rel:b3-small:sub:inv:beta}) gives the relation.
\end{proof}

From \Cnref{rel:b3-small:sub:comm:psi:beta+psi}, \Cnref{rel:b3-small:sub:inv:psi}, and \Cnref{rel:b3-small:sub:inv:beta+psi}:
\begin{relation}[\RelNameBSm\RelNameSubgp\RelNameExpr{\Rt{\beta+2\psi}}]\label{rel:b3-small:sub:expr:beta+2psi}
    \[ \DQuant{i \in [1], j \in [2]}{t,u \in \LiftScalars} \Eld{\beta+2\psi}{2tu}{i+j} = \Eld{\psi}{t}{i} \Eld{\beta+\psi}{u}{j} \Eld{\psi}{-t}{i} \Eld{\beta+\psi}{-u}{j}. \]
\end{relation}

From \Cnref{rel:b3-small:sub:comm:psi:omega}, \Cnref{eq:comm:mid:str}, \Cnref{eq:comm:right:str}, \Cnref{rel:b3-small:sub:inv:psi}, and \Cnref{rel:b3-small:sub:inv:omega}:

\begin{relation}[\RelNameBSm\RelNameSubgp\RelNameOrder{\Rt{\psi}}{\Rt{\omega}}]\label{rel:b3-small:order:psi:omega}
    \[ \DQuant{i, j \in [1]}{t,u \in \LiftScalars} \Eld{\psi}{t}{i} \Eld{\omega}{u}{j} = \Eld{\omega}{u}{j} \asite{\Eld{\psi+\omega}{2tu}{i+j}} \Eld{\psi}{t}{i} = \Eld{\omega}{u}{j} \Eld{\psi}{t}{i} \asite{\Eld{\psi+\omega}{2tu}{i+j}}. \]
\end{relation}

From \Cnref{rel:b3-small:sub:comm:beta:psi}, \Cnref{eq:comm:mid:str}, \Cnref{eq:comm:right:str}, \Cnref{rel:b3-small:sub:comm:beta+psi:beta+2psi}, \Cnref{rel:b3-small:sub:inv:beta}, and \Cnref{rel:b3-small:sub:inv:psi}:

\begin{relation}[\RelNameBSm\RelNameSubgp\RelNameOrder{\Rt{\beta}}{\Rt{\psi}}]\label{rel:b3-small:order:beta:psi}
    \[ \DQuant{i, j \in [1]}{t,u \in \LiftScalars} \Eld{\beta}{t}{i} \Eld{\psi}{u}{j} = \Eld{\psi}{u}{j} \asite{\Eld{\beta+\psi}{tu}{i+j} \Eld{\beta+2\psi}{-tu^2}{i+2j}} \Eld{\beta}{t}{i} = \Eld{\psi}{u}{j} \Eld{\beta}{t}{i} \asite{\Eld{\beta+2\psi}{-tu^2}{i+2j} \Eld{\beta+\psi}{tu}{i+j}}. \]
\end{relation}
From \Cnref{rel:b3-small:sub:comm:beta:psi}, \Cnref{eq:comm:left:inv}, \Cnref{rel:b3-small:sub:comm:beta+psi:beta+2psi}, \Cnref{rel:b3-small:sub:inv:beta+psi}, and \Cnref{rel:b3-small:sub:inv:beta+2psi}:
\begin{relation}[\RelNameBSm\RelNameSubgp\RelNameOrder{\Rt{\psi}}{\Rt{\beta}}]\label{rel:b3-small:order:psi:beta}
    \[ \DQuant{i, j \in [1]}{t,u \in \LiftScalars} \Eld{\psi}{u}{j} \Eld{\beta}{t}{i} = \asite{\Eld{\beta+\psi}{-tu}{i+j} \Eld{\beta+2\psi}{-tu^2}{i+2j}} \Eld{\beta}{t}{i} \Eld{\psi}{u}{j}. \]
\end{relation}

\subsubsection{$\beta+\psi$ and $\psi+\omega$ commute}

\begin{relation}[\RelNameBSm\RelNameCommute{\Rt{\beta+\psi}}{\Rt{\psi+\omega}}]\label{rel:b3-small:comm:beta+psi:psi+omega}
    \[ \DQuant{i,j \in [2]}{t,u \in \LiftScalars} \Comm{\Eld{\beta+\psi}{t}{i}}{\Eld{\psi+\omega}{u}{j}} = \Id. \]
\end{relation}

\begin{proof}
    Analogous to the proof of \Cnref{rel:a3:comm:alpha+beta:beta+gamma} using homogeneous and nonhomogeneous lifting and \Cnref{rel:b3-small:comm:raw:lift:beta+psi:psi+omega}.
\end{proof}

\subsubsection{Establishing $\beta+\psi+\omega$}

\begin{relation}[\RelNameBSm\RelNameCommute{\Rt{\beta+2\psi}}{\Rt{\psi+\omega}}]\label{rel:b3-small:comm:beta+2psi:psi+omega}
    \[ \DQuant{i \in [3], j \in [2]}{t,u \in \LiftScalars} \Comm{\Eld{\beta+2\psi}{t}{i}}{\Eld{\psi+\omega}{u}{j}} = \Id. \]
\end{relation}

\begin{proof}
By \Cnref{rel:b3-small:sub:expr:beta+2psi}, $\Rt{\beta+2\psi}$ elements can be expressed as products of $\Rt{\psi}$ and $\Rt{\beta+\psi}$ elements. $\Rt{\psi+\omega}$ elements commute with these types of elements by \Cnref{rel:b3-small:sub:comm:psi:psi+omega} and \Cnref{rel:b3-small:comm:beta+psi:psi+omega}, respectively.
\end{proof}

\begin{relation}[\RelNameBSm\RelNameInterchange{\Rt{\beta+\psi+\omega}}]\label{rel:b3-small:interchange:beta+psi+omega}
    \[ \DQuant{i,j,k \in [1]}{t,u,v \in \LiftScalars} \Comm{\Eld{\beta+\psi}{tu}{i+j}}{\Eld{\omega}{v}{k}} = \Comm{\Eld{\beta}{t}{i}}{\Eld{\psi+\omega}{2uv}{j+k}}. \]
\end{relation}

\begin{proof}
    We write a product of $\Rt{\beta+\psi}$ and $\Rt{\omega}$ elements:
    \begin{align*}
        & \asite{\Eld{\beta+\psi}{tu}{i+j}} \Eld{\omega}{v}{k} \\
        \intertext{Expand the $\Rt{\beta+\psi}$ element into a product of $\Rt{\psi}$ and $\Rt{\beta}$ elements (\Cnref{rel:b3-small:sub:expr:beta+psi}):}
        &= \Eld{\psi}{-u/2}{j} \Eld{\beta}{t}{i} \Eld{\psi}{u}{j} \Eld{\beta}{-t}{i} \Eld{\psi}{-u/2}{j} \asite{\Eld{\omega}{v}{k}} \\
        \intertext{Moving the $\Rt{\omega}$ element from the right fully to the left creates $\Rt{\psi+\omega}$ commutators with $\Rt{\psi}$ elements (\Cnref{rel:b3-small:order:psi:omega}) and no commutators with $\Rt{\beta}$ elements (\Cnref{rel:b3-small:sub:comm:beta:omega}):}
        &= \Eld{\omega}{v}{k} \asite{\Eld{\psi+\omega}{-uv}{j+k}} \Eld{\psi}{-u/2}{j}  \Eld{\beta}{t}{i} \Eld{\psi}{u}{j} \Eld{\psi+\omega}{2uv}{j+k}\Eld{\beta}{-t}{i} \Eld{\psi}{-u/2}{j}\Eld{\psi+\omega}{-uv}{j+k}. \\
        \intertext{Now observe that $\Rt{\psi+\omega}$ elements commute with $\Rt{\omega}$ and $\Rt{\beta+\psi}$ elements (\Cnref{rel:b3-small:sub:comm:psi:psi+omega} and \Cnref{rel:b3-small:comm:beta+psi:psi+omega}, respectively), and therefore we can move the indicated $\Rt{\psi+\omega}$ element cyclically over to the right\footnote{I.e., we move it to the left, then left-multiply by its inverse, creating its inverse on the left of the \emph{other} side of the equation (the product of $\Rt{\beta+\psi}$ and $\Rt{\omega}$ elements), move this inverse over to the right of the other side, then multiply it back onto the current side on the right.}, merging it with the $\Rt{\psi+\omega}$ element there (\Cref{rel:b3-small:sub:lin:psi+omega}):}
        &= \Eld{\omega}{v}{k} \asite{\Eld{\psi}{-u/2}{j} \Eld{\beta}{t}{i}} \Eld{\psi}{u}{j} \Eld{\psi+\omega}{2uv}{j+k}\asite{\Eld{\beta}{-t}{i} \Eld{\psi}{-u/2}{j}} \Eld{\psi+\omega}{-2uv}{j+k} \\
        \intertext{Moving the $\Rt{\psi}$ elements together across $\Rt{\beta}$ elements, creating $\Rt{\beta+\psi}$ and $\Rt{\beta+2\psi}$ commutators (\Cnref{rel:b3-small:order:beta:psi} and \Cnref{rel:b3-small:order:psi:beta}):}
        &= \Eld{\omega}{v}{k} \Eld{\beta+\psi}{tu/2}{i+j} \Eld{\beta+2\psi}{-tu^2/4}{i+2j} \Eld{\beta}{t}{i} \asite{\Eld{\psi}{-u/2}{j} \Eld{\psi}{u}{j} \Eld{\psi+\omega}{2uv}{j+k}\Eld{\psi}{-u/2}{j}} \Eld{\beta}{-t}{i}  \Eld{\beta+2\psi}{tu^2/4}{i+2j}\Eld{\beta+\psi}{tu/2}{i+j} \Eld{\psi+\omega}{-2uv}{j+k} \\
        \intertext{Commuting the $\Rt{\psi}$ elements together across the $\Rt{\psi+\omega}$ elements (\Cnref{rel:b3-small:sub:comm:psi:psi+omega}) and cancelling them (\Cnref{rel:b3-small:sub:inv:psi}):}
        &= \Eld{\omega}{v}{k} \Eld{\beta+\psi}{tu/2}{i+j} \asite{\Eld{\beta+2\psi}{-tu^2/4}{i+2j}} \Eld{\beta}{t}{i} \Eld{\psi+\omega}{2uv}{j+k}\Eld{\beta}{-t}{i} \asite{\Eld{\beta+2\psi}{tu^2/4}{i+2j}} \Eld{\beta+\psi}{tu/2}{i+j} \Eld{\psi+\omega}{-2uv}{j+k}
        \intertext{$\Rt{\beta+2\psi}$ elements commute with $\Rt{\beta}$ elements (\Cnref{rel:b3-small:sub:comm:beta:beta+2psi}) and $\Rt{\psi+\omega}$ elements (\Cnref{rel:b3-small:comm:beta+2psi:psi+omega}), so we can cancel them (\Cnref{rel:b3-small:sub:inv:beta+2psi}):}
        &= \Eld{\omega}{v}{k} \asite{\Eld{\beta+\psi}{tu/2}{i+j}} \Eld{\beta}{t}{i} \Eld{\psi+\omega}{2uv}{j+k} \Eld{\beta}{-t}{i} \asite{\Eld{\beta+\psi}{tu/2}{i+j}} \Eld{\psi+\omega}{-2uv}{j+k}.
        \intertext{Finally, we collect $\Rt{\beta+\psi}$ elements on the right since they commute with $\Rt{\beta}$ and $\Rt{\psi+\omega}$ elements (\Cnref{rel:b3-small:sub:comm:beta:beta+psi} and \Cnref{rel:b3-small:comm:beta+psi:psi+omega}, respectively) and add them (\Cnref{rel:b3-small:sub:lin:beta+psi}):}
        &= \asite{\Eld{\omega}{v}{k}} \Eld{\beta}{t}{i} \Eld{\psi+\omega}{2uv}{j+k} \Eld{\beta}{-t}{i} \Eld{\psi+\omega}{-2uv}{j+k} \Eld{\beta+\psi}{tu}{i+j} \\
        \intertext{And similarly bring $\Rt{\omega}$ element to the penultimate right position since it commutes with $\Rt{\beta}$ and $\Rt{\psi+\omega}$ elements (\Cnref{rel:b3-small:sub:comm:beta:omega} and \Cnref{rel:b3-small:sub:comm:omega:psi+omega}, respectively):}
        &= \asite{\Eld{\beta}{t}{i} \Eld{\psi+\omega}{2uv}{j+k}\Eld{\beta}{-t}{i} \Eld{\psi+\omega}{-2uv}{j+k}} \Eld{\omega}{v}{k} \Eld{\beta+\psi}{tu}{i+j} \\
        \intertext{And reduce into $\Comm{\Rt{\beta}}{\Rt{\psi+\omega}}$ form using \Cnref{rel:b3-small:sub:inv:beta} and \Cnref{rel:b3-small:sub:inv:psi+omega}:}
        &= \Comm{\Eld{\beta}{t}{i}}{\Eld{\psi+\omega}{2uv}{j+k}} \Eld{\omega}{v}{k} \Eld{\beta+\psi}{tu}{i+j}.
    \end{align*}
\end{proof}

\input{figures/b3-small/beta+psi+omega}

\begin{proposition}[Establishing $\Rt{\beta+\psi+\omega}$]\label{prop:b3-small:est:beta+psi+omega}
    There exist elements $\Eld{\beta+\psi+\omega}{t}{i}$ for $t \in \LiftScalars$ and $i \in [3]$ such that the following hold.
    \begin{gather*}
        \DQuant{i \in [2], j \in [1]}{t,u \in \LiftScalars} \Eld{\beta+\psi+\omega}{2tu}{i+j} = \Comm{\Eld{\beta+\psi}{t}{i}}{\Eld{\omega}{u}{j}}, \\
        \DQuant{i \in [1], j \in [2]}{t,u \in \LiftScalars} \Eld{\beta+\psi+\omega}{tu}{i+j} = \Comm{\Eld{\beta}{t}{i}}{\Eld{\psi+\omega}{u}{j}}.
    \end{gather*}
\end{proposition}

\begin{proof}
    Uses \Cnref{prop:prelim:equal-conn} and \Cnref{rel:b3-small:interchange:beta+psi+omega} (i.e., the blocks in \Cref{fig:b3-small:def:beta+psi+omega} are connected).
\end{proof}

By \Cnref{rel:b3-small:sub:inv:beta} and \Cnref{rel:b3-small:sub:inv:psi+omega}, this implies:
\begin{relation}[\RelNameBSm\RelNameExprn{\Rt{\beta+\psi+\omega}}{\Rt{\beta}}{\Rt{\psi+\omega}}]\label{rel:b3-small:expr:beta+psi+omega:beta:psi+omega}
    \[ \DQuant{i \in [2], j \in [1]}{t,u \in \LiftScalars} \Eld{\beta+\psi+\omega}{2tu}{i+j} = \Eld{\beta+\psi}{t}{i} \Eld{\omega}{u}{j} \Eld{\beta+\psi}{-t}{i} \Eld{\omega}{-u}{j}. \]
\end{relation}
and by \Cnref{rel:b3-small:sub:inv:beta+psi} and \Cnref{rel:b3-small:sub:inv:omega}:
\begin{relation}[\RelNameBSm\RelNameExprn{\Rt{\beta+\psi+\omega}}{\Rt{\beta+\psi}}{\Rt{\omega}}]\label{rel:b3-small:expr:beta+psi+omega:beta+psi:omega}
    \[ \DQuant{i \in [1], j \in [2]}{t,u \in \LiftScalars} \Eld{\beta+\psi+\omega}{tu}{i+j} = \Eld{\beta}{t}{i} \Eld{\psi+\omega}{u}{j} \Eld{\beta}{-t}{i} \Eld{\psi+\omega}{-u}{j}. \]
\end{relation}

\subsubsection{Remaining commutation relations}

We now turn to exploring properties of $\Rt{\beta+\psi+\omega}$.

\begin{relation}[\RelNameBSm\RelNameCommute{\Rt{\beta+\psi+\omega}}{\Rt{\omega}}]\label{rel:b3-small:comm:beta+psi+omega:omega}
    \[ \DQuant{i \in [3],j \in [1]}{t,u \in \LiftScalars} \Comm{\Eld{\beta+\psi+\omega}{t}{i}}{\Eld{\omega}{u}{j}} = \Id. \]
\end{relation}

\begin{proof}
    By \Cnref{rel:b3-small:expr:beta+psi+omega:beta:psi+omega}, $\Rt{\beta+\psi+\omega}$ elements can be expressed as products of $\Rt{\beta}$ and $\Rt{\psi+\omega}$ elements. $\Rt{\omega}$ elements commute with these two types of elements by \Cnref{rel:b3-small:sub:comm:beta:omega} and \Cnref{rel:b3-small:sub:comm:omega:psi+omega}, respectively.
\end{proof}

\begin{relation}[\RelNameBSm\RelNameCommute{\Rt{\beta+\psi+\omega}}{\Rt{\beta}}]\label{rel:b3-small:comm:beta+psi+omega:beta}
    \[ \DQuant{i \in [3], j \in [1]}{t,u \in \LiftScalars} \Comm{\Eld{\beta+\psi+\omega}{t}{i}}{\Eld{\beta}{u}{j}} = \Id. \]
\end{relation}

\begin{proof}
    Symmetric to the previous proof: By \Cnref{rel:b3-small:expr:beta+psi+omega:beta+psi:omega}, $\Rt{\beta+\psi+\omega}$ elements can be expressed as products of $\Rt{\beta+\psi}$ and $\Rt{\omega}$ elements. $\Rt{\beta}$ elements commute with these two types of elements by \Cnref{rel:b3-small:sub:comm:beta:beta+psi} and \Cnref{rel:b3-small:sub:comm:beta:omega}, respectively.
\end{proof}

\begin{relation}[\RelNameBSm\RelNameCommute{\Rt{\beta+\psi+\omega}}{\Rt{\psi}}]\label{rel:b3-small:comm:beta+psi+omega:psi}
    \[ \DQuant{i \in [3], j \in [1]}{t,u \in \LiftScalars} \Comm{\Eld{\beta+\psi+\omega}{t}{i}}{\Eld{\psi}{u}{j}} = \Id. \]
\end{relation}

\begin{proof}
    We write a product of $\Rt{\beta+\psi+\omega}$ and $\Rt{\psi}$ elements:
    \begin{align*}
        &\asite{\Eld{\beta+\psi+\omega}{t}{i}} \Eld{\psi}{u}{j} \\
        \intertext{Expanding the $\Rt{\beta+\psi+\omega}$ element into a commutator of $\Rt{\beta}$ and $\Rt{\psi+\omega}$ elements (\Cnref{rel:b3-small:expr:beta+psi+omega:beta:psi+omega}):}
        &= \Eld{\beta}{t}{i_1} \Eld{\psi+\omega}{1/2}{i_2} \Eld{\beta}{-t}{i_1} \Eld{\psi+\omega}{-1/2}{i_2} \asite{\Eld{\psi}{u}{j}} \\
        \intertext{Decompose $i = i_1 + i_2$ for $i_1 \in [1],i_2 \in [2]$ arbitrarily. Moving the $\Rt{\psi}$ element on the right fully to the left creates no commutators with $\Rt{\psi+\omega}$ elements (\Cnref{rel:b3-small:sub:comm:psi:psi+omega}) and $\Rt{\beta+\psi}$ and $\Rt{\beta+2\psi}$ commutators with $\Rt{\beta}$ elements (\Cnref{rel:b3-small:order:beta:psi}):}
        &= \Eld{\psi}{u}{j} \Eld{\beta}{t}{i_1} \Eld{\beta+2\psi}{-tu^2}{i_1+2j} \asite{\Eld{\beta+\psi}{tu}{i_1+j}}  \Eld{\psi+\omega}{1/2}{i_2} \asite{\Eld{\beta+\psi}{-tu}{i_1+j}} \Eld{\beta+2\psi}{tu^2}{i_1+2j} \Eld{\beta}{-t}{i_1} \Eld{\psi+\omega}{-1/2}{i_2}.
        \intertext{Now, $\Rt{\beta+\psi}$ elements commute with $\Rt{\psi+\omega}$ elements by \Cnref{rel:b3-small:comm:beta+psi:psi+omega}, so we can cancel them (\Cnref{rel:b3-small:sub:inv:beta+psi}):}
        &= \Eld{\psi}{u}{j} \Eld{\beta}{t}{i_1} \asite{\Eld{\beta+2\psi}{-tu^2}{i_1+2j}} \Eld{\psi+\omega}{1/2}{i_2} \asite{\Eld{\beta+2\psi}{tu^2}{i_1+2j}} \Eld{\beta}{-t}{i_1} \Eld{\psi+\omega}{-1/2}{i_2};
        \intertext{again, $\Rt{\beta+2\psi}$ elements commute with $\Rt{\psi+\omega}$ elements by \Cnref{rel:b3-small:comm:beta+2psi:psi+omega}, so we can cancel them (\Cnref{rel:b3-small:sub:inv:beta+2psi}):}
        &= \Eld{\psi}{u}{j} \asite{\Eld{\beta}{t}{i_1} \Eld{\psi+\omega}{1/2}{i_2} \Eld{\beta}{-t}{i_1}  \Eld{\psi+\omega}{-1/2}{i_2}} \\
        \intertext{Reducing back into a $\Rt{\beta+\psi+\omega}$ element (\Cnref{rel:b3-small:expr:beta+psi+omega:beta:psi+omega}):}
        &= \Eld{\psi}{u}{j} \Eld{\beta+\psi+\omega}{t}{i},
    \end{align*}
    as desired.
\end{proof}

\begin{relation}[\RelNameBSm\RelNameCommute{\Rt{\beta+\psi+\omega}}{\Rt{\psi+\omega}}]\label{rel:b3-small:comm:beta+psi+omega:psi+omega}
    \[
    \DQuant{i \in [3], j \in [2]}{t,u \in \LiftScalars} \Comm{\Eld{\beta+\psi+\omega}{t}{i}}{\Eld{\psi+\omega}{t}{j}} = \Id.
    \]
\end{relation}

\begin{proof}
    $\Rt{\psi+\omega}$ elements may be expressed as products of $\Rt{\psi}$ and $\Rt{\omega}$ elements (\Cnref{rel:b3-small:sub:comm:psi:omega}, \Cnref{rel:b3-small:sub:inv:psi}, and \Cnref{rel:b3-small:sub:inv:omega}) and $\Rt{\beta+\psi+\omega}$ elements commute with both types of elements by \Cnref{rel:b3-small:comm:beta+psi+omega:psi} and \Cnref{rel:b3-small:comm:beta+psi+omega:omega}, respectively.
\end{proof}

\begin{relation}[\RelNameBSm\RelNameCommute{\Rt{\beta+\psi+\omega}}{\Rt{\beta+\psi}}]\label{rel:b3-small:comm:beta+psi+omega:beta+psi}
    \[
    \DQuant{i \in [3], j \in [2]}{t,u \in \LiftScalars} \Comm{\Eld{\beta+\psi+\omega}{t}{i}}{\Eld{\beta+\psi}{t}{j}} = \Id.
    \]
\end{relation}

\begin{proof}
    $\Rt{\beta+\psi}$ elements may be expressed as products of $\Rt{\beta}$ and $\Rt{\psi}$ elements (\Cnref{rel:b3-small:sub:expr:beta+psi}) and $\Rt{\beta+\psi+\omega}$ elements commute with both types of elements by \Cnref{rel:b3-small:comm:beta+psi+omega:beta} and \Cnref{rel:b3-small:comm:beta+psi+omega:psi}, respectively.
\end{proof}

\begin{relation}[\RelNameBSm\RelNameCommute{\Rt{\beta+\psi+\omega}}{\Rt{\beta+2\psi}}]\label{rel:b3-small:comm:beta+psi+omega:beta+2psi}
    \[
    \DQuant{i \in [3], j \in [3]}{t,u \in \LiftScalars} \Comm{\Eld{\beta+\psi+\omega}{t}{i}}{\Eld{\beta+2\psi}{t}{j}} = \Id.
    \]
\end{relation}

\begin{proof}
    $\Rt{\beta+2\psi}$ elements may be expressed as products of $\Rt{\beta+\psi}$ and $\Rt{\psi}$ elements (\Cnref{rel:b3-small:sub:expr:beta+2psi}) and $\Rt{\beta+\psi+\omega}$ elements commute with both types of elements by \Cnref{rel:b3-small:comm:beta+psi+omega:beta+psi} and \Cnref{rel:b3-small:comm:beta+psi+omega:psi}, respectively.
\end{proof}

\begin{relation}[\RelNameBSm\RelNameSelfCommute{\Rt{\beta+\psi+\omega}}]\label{rel:b3-small:comm:self:beta+psi+omega}
    \[
    \DQuant{i,j \in [3]}{t,u \in \LiftScalars} \Comm{\Eld{\beta+\psi+\omega}{t}{i}}{\Eld{\beta+\psi+\omega}{t}{j}} = \Id.
    \]
\end{relation}

\begin{proof}
    $\Rt{\beta+\psi+\omega}$ elements may be expressed as products of $\Rt{\beta+\psi}$ and $\Rt{\omega}$ elements (\Cnref{rel:b3-small:expr:beta+psi+omega:beta+psi:omega}) and $\Rt{\beta+\psi+\omega}$ elements commute with both types of elements by \Cnref{rel:b3-small:comm:beta+psi+omega:beta+psi} and \Cnref{rel:b3-small:comm:beta+psi+omega:omega}, respectively.
\end{proof}

\begin{relation}[\RelNameBSm\RelNameLinearity{\Rt{\beta+\psi+\omega}}]\label{rel:b3-small:lin:beta+psi+omega}
    \[ \DQuant{i \in [3]}{t,u \in \LiftScalars} \Eld{\beta+\psi+\omega}{t}{i} \Eld{\beta+\psi+\omega}{u}{i} = \Eld{\beta+\psi+\omega}{t+u}{i+j}. \]
\end{relation}

\begin{proof}
We write a product of $\Rt{\beta+\psi+\omega}$ elements of the same degree:
\begin{align*}
    & \asite{\Eld{\beta+\psi+\omega}{t}{i}} \Eld{\beta+\psi+\omega}{u}{i} \\
    \intertext{Decompose arbitrarily $i = i_1+i_2$ for $i_1 \in [1], i_2 \in [2]$ and expand one $\Rt{\beta+\psi+\omega}$ element into a product of $\Rt{\beta}$ and $\Rt{\psi+\omega}$ elements (\Cnref{rel:b3-small:expr:beta+psi+omega:beta:psi+omega}):}
    &= \Eld{\beta}{t}{i_1} \Eld{\psi+\omega}{1}{i_2} \Eld{\beta}{-t}{i_1} \Eld{\psi+\omega}{-1}{i_2} \asite{\Eld{\beta+\psi+\omega}{u}{i}} \\
    \intertext{Commuting the remaining $\Rt{\beta+\psi+\omega}$ element on the right \emph{partially} to the left, creating no commutators with $\Rt{\beta}$ or $\Rt{\psi+\omega}$ elements (\Cnref{rel:b3-small:comm:beta+psi+omega:beta} and \Cnref{rel:b3-small:comm:beta+psi+omega:psi+omega}, respectively):}
    &= \Eld{\beta}{t}{i_1} \asite{\Eld{\beta+\psi+\omega}{u}{i}} \Eld{\psi+\omega}{1}{i_2} \Eld{\beta}{-t}{i_1} \Eld{\psi+\omega}{-1}{i_2} \\
    \intertext{Now, we expand the $\Rt{\beta+\psi+\omega}$ into a similar product of $\Rt{\beta}$ and $\Rt{\psi+\omega}$ elements (\Cnref{rel:b3-small:expr:beta+psi+omega:beta:psi+omega}):}
    &= \Eld{\beta}{t}{i_1} \Eld{\beta}{u}{i_1} \Eld{\psi+\omega}{1}{i_2} \Eld{\beta}{-u}{i_1} \asite{\Eld{\psi+\omega}{-1}{i_2} \Eld{\psi+\omega}{1}{i_2}} \Eld{\beta}{-t}{i_1} \Eld{\psi+\omega}{-1}{i_2} \\
    \intertext{Cancelling the adjacent $\Rt{\psi+\omega}$ elements (\Cnref{rel:b3-small:sub:inv:psi+omega}):}
    &= \asite{\Eld{\beta}{t}{i_1} \Eld{\beta}{u}{i_1}} \Eld{\psi+\omega}{1}{i_2} \asite{\Eld{\beta}{-u}{i_1} \Eld{\beta}{-t}{i_1}} \Eld{\psi+\omega}{-1}{i_2} \\
    \intertext{Using linearity for adjacent $\Rt{\beta}$ elements (\Cnref{rel:b3-small:sub:lin:beta}):}
    &= \asite{\Eld{\beta}{t+u}{i_1} \Eld{\psi+\omega}{1}{i_2} \Eld{\beta}{-(t+u)}{i_1} \Eld{\psi+\omega}{-1}{i_2}} \\
    \intertext{Reducing back into a $\Rt{\beta+\psi+\omega}$ element (\Cnref{rel:b3-small:expr:beta+psi+omega:beta:psi+omega}):}
    &= \Eld{\beta+\psi+\omega}{t+u}{i}.
\end{align*}
\end{proof}

Note that \Cnref{rel:b3-small:lin:beta+psi+omega} implies the same-degree case of the previous proposition (\Cnref{rel:b3-small:comm:self:beta+psi+omega}) --- that is, that $\Rt{\beta+\psi+\omega}$ elements of the same degree commute --- since addition over $R$ is commutative!

By \Cref{prop:prelim:group-homo}, \Cnref{rel:b3-small:lin:beta+psi+omega} implies:

\begin{relation}[\RelNameBSm\RelNameInv{\Rt{\beta}}]\label{rel:b3-small:inv:beta+psi+omega}
    \[ \DQuant{i \in [3]}{t \in \LiftScalars} \Eld{\beta+\psi+\omega}{t}{i} \Eld{\beta+\psi+\omega}{-t}{i} = \Id. \]
\end{relation}

Now, from \Cnref{eq:comm:mid:str}, \Cnref{eq:comm:right:str}, \Cnref{prop:b3-small:est:beta+psi+omega}, \Cnref{rel:b3-small:inv:beta+psi+omega}, \Cnref{rel:b3-small:sub:inv:omega}, and \Cnref{rel:b3-small:sub:inv:beta+psi}, we have:

\begin{relation}[\RelNameBSm\RelNameOrder{\Rt{\beta+\psi}}{\Rt{\omega}}]\label{rel:b3-small:order:beta+psi:omega}
    \[ \DQuant{i \in [2], j \in [1]}{t,u \in \LiftScalars} \Eld{\beta+\psi}{t}{i} \Eld{\omega}{u}{j} =  \Eld{\omega}{u}{j} \asite{\Eld{\beta+\psi+\omega}{2tu}{i+j}} \Eld{\beta+\psi}{t}{i} =  \Eld{\omega}{u}{j}\Eld{\beta+\psi}{t}{i}\asite{\Eld{\beta+\psi+\omega}{2tu}{i+j}}. \]
\end{relation}

\begin{relation}[\RelNameBSm\RelNameCommute{\Rt{\beta+2\psi}}{\Rt{\omega}}]\label{rel:b3-small:comm:beta+2psi:omega}
    \[ \DQuant{i \in [3], j \in [1]}{t,u \in \LiftScalars} \Comm{\Eld{\beta+2\psi}{t}{i}}{\Eld{\omega}{u}{j}} = I. \]
\end{relation}

\begin{proof}
    We write a product of $\Rt{\omega}$ and $\Rt{\beta+2\psi}$ elements:
    \begin{align*}
        & \asite{\Eld{\beta+2\psi}{t}{j}} \Eld{\omega}{u}{i} \\
        \intertext{Expanding the $\Rt{\beta+2\psi}$ element into a commutator of $\Rt{\beta+\psi}$ and $\Rt{\psi}$ elements (\Cnref{rel:b3-small:sub:expr:beta+2psi}):}
        &= \Eld{\psi}{1/2}{j_1} \Eld{\beta+\psi}{t}{j_2} \Eld{\psi}{-1/2}{j_1} \Eld{\beta+\psi}{-t}{j_2} \asite{\Eld{\omega}{u}{i}} \\
        \intertext{Moving the $\Rt{\omega}$ element on the right fully to the left, creating $\Rt{\omega+\psi}$ commutators with $\Rt{\psi}$ elements and $\Rt{\beta+\psi+\omega}$ commutators with $\Rt{\beta+\psi}$ elements (\Cnref{rel:b3-small:order:psi:omega} and \Cnref{rel:b3-small:order:beta+psi:omega}):}
        &= \Eld{\omega}{u}{i} \Eld{\psi}{1/2}{j_1} \Eld{\psi+\omega}{u}{i+j_1} \Eld{\beta+\psi}{t}{j_2} \asite{\Eld{\beta+\psi+\omega}{2tu}{i+j_2}} \Eld{\psi+\omega}{-u}{i+j_1} \Eld{\psi}{-1/2}{j_1} \asite{\Eld{\beta+\psi+\omega}{-2tu}{i+j_2}} \Eld{\beta+\psi}{-t}{j_2}. \\
        \intertext{Now, we observe that $\Rt{\beta+\psi+\omega}$ elements commute with $\Rt{\psi+\omega}$ and $\Rt{\psi}$ elements (\Cnref{rel:b3-small:comm:beta+psi+omega:psi+omega} and \Cnref{rel:b3-small:comm:beta+psi+omega:psi}, respectively), so we can cancel the $\Rt{\beta+\psi+\omega}$ elements (\Cnref{rel:b3-small:inv:beta+psi+omega}):}
        &=  \Eld{\omega}{u}{i} \asite{\Eld{\psi+\omega}{u}{i+j_1}} \Eld{\psi}{1/2}{j_1} \Eld{\beta+\psi}{t}{j_2} \asite{\Eld{\psi+\omega}{-u}{i+j_1}} \Eld{\psi}{-1/2}{j_1} \Eld{\beta+\psi}{-t}{j_2}  \\
        \intertext{and similarly $\Rt{\psi+\omega}$ elements commute with $\Rt{\psi}$ and $\Rt{\beta+\psi}$ elements  (\Cnref{rel:b3-small:sub:comm:psi:psi+omega} and \Cnref{rel:b3-small:comm:beta+psi:psi+omega}, respectively) so we may cancel them (\Cnref{rel:b3-small:sub:inv:psi+omega}):}
        &= \Eld{\omega}{u}{i} \asite{\Eld{\psi}{1/2}{j_1} \Eld{\beta+\psi}{t}{j_2} \Eld{\psi}{-1/2}{j_1} \Eld{\beta+\psi}{-t}{j_2}}  \\
        \intertext{and reducing back into a single $\Rt{\beta+2\psi}$ element (\Cnref{rel:b3-small:sub:expr:beta+2psi}):}
        &= \Eld{\omega}{u}{i} \Eld{\beta+2\psi}{t}{j},
    \end{align*}
    as desired.
\end{proof}

%% file: figures/b3-small/beta+psi+omega.tex
\begin{figure}
    \centering
    \begin{tikzpicture}
\node[algnode0] (v000) {$(0,0)$};
\node[algnode0] (v010) [below=of v000] {$(1,0)$};
\node[algnode0] (v001) [below=of v010] {$(0,1)$};
\node[algnode0] (v020) [below=of v001] {$(2,0)$};
\node[algnode0] (v011) [below=of v020] {$(1,1)$};
\node[algnode0] (v021) [below=of v011] {$(2,1)$};
\node[algnode1] (v100) [right=of v000] {$(0,0)$};
\node[algnode1] (v110) [below=of v100] {$(1,0)$};
\node[algnode1] (v101) [below=of v110] {$(0,1)$};
\node[algnode1] (v111) [below=of v101] {$(1,1)$};
\node[algnode1] (v102) [below=of v111] {$(0,2)$};
\node[algnode1] (v112) [below=of v102] {$(1,2)$};
\draw[algedge] (v011) -- (v111);
\draw[algedge] (v021) -- (v112);
\draw[algedge] (v000) -- (v100);
\draw[algedge] (v010) -- (v110);
\draw[algedge] (v011) -- (v102);
\draw[algedge] (v001) -- (v101);
\draw[algedge] (v010) -- (v101);
\draw[algedge] (v020) -- (v111);
\begin{scope}[on background layer]
    \node[algbackh, fit={(v000) (v100)}] {};
    \node[algbackh, fit={(v010) (v101)}] {};
    \node[algbackh, fit={(v020) (v102)}] {};
    \node[algbackh, fit={(v021) (v112)}] {};
    \node[algback, fit={(v000) (v021)}] (b0) {};
    \node[above=0 of b0] {$\Comm{\Rt{\beta+\psi}}{\Rt{\omega}}$};
    \node[algback, fit={(v100) (v112)}] (b1) {};
    \node[above=0 of b1] {$\Comm{\Rt{\beta}}{\Rt{\psi+\omega}}$};
\end{scope}
\end{tikzpicture}
\caption{Establishing $\Rt{\beta+\psi+\omega}$: A bipartite graph with left vertex-set $[2] \times [1]$ and right vertex-set $[1] \times [2]$, with an edge $(i,j) \sim (k,\ell)$ iff \Cnref{rel:b3-small:interchange:beta+psi+omega} states that for all $t,u,v \in R$, $\Comm{\Eld{\beta+\psi}{tu}{i}}{\Eld{\omega}{v}{j}} = \Comm{\Eld{\beta}{t}{k}}{\Eld{\psi+\omega}{2uv}{\ell}}$. (Hence, there are edges $(i+j,k) \sim (i,j+k)$ for every $i,j,k \in [1]$.) Additionally, grey blocks partition the vertices based on the sum of coordinates in $[3]$. In this case, the blocks also correspond to connected components in the graph.}\label{fig:b3-small:def:beta+psi+omega}
\end{figure}

%% file: sections/08b-B3-large.tex
\subsection{Link of $\omega$}\label{sec:b3:large}

Finally, we prove \Cref{thm:story:b3:lift} for the ``big'' link of $B_3$, i.e., a lifting theorem from $\UnipBLg{p}$ to $\GrUnipBLg{p^k}$ --- this is the bulk of the work in our paper.  All relations stated in this section hold in the graded link group $\GrUnipBLg{p^k}$ and are proven via constant-length in-subgroup and lifted relations in a number of steps independent of $p$ or $k$.

\newcommand{\liftC}{(t',u',v')}

\newcommand{\newedgeA}{$\mathtt{A}$}
\newcommand{\newedgeB}{$\mathtt{B}$}
\newcommand{\newedgeC}{$\mathtt{C}$}
\newcommand{\newedgeD}{$\mathtt{D}$}

As we proceed, the (subjectively) most ``exciting'' propositions and relations are tagged with \excl. Also, since all the relations we mention are symmetric to replacing every element $\Eld{\zeta}{t}{i}$ with $\Eld{\zeta}{t}{\mathsf{ht}(\zeta)}$, we often omit one of the two ``paired'' relations for brevity.

\subsubsection{Link structure}

\begin{table}[h!]
\centering
\begin{tabular}{c|c|c|c}
    \textbf{Color} & \textbf{Base} & \textbf{Positive span} & \textbf{Total \#} \\ \hline
    & $\{\Rt{\alpha},\Rt{\beta},\Rt{\psi}\}$ & $\{\Rt{\alpha+\beta},\Rt{\beta+\psi},\Rt{\beta+2\psi},\Rt{\alpha+\beta+\psi},\Rt{\alpha+\beta+2\psi},\Rt{\alpha+2\beta+2\psi}\}$ & 9 \\
    \BoxR & $\{\Rt{\alpha},\Rt{\beta}\}$ & $\{\Rt{\alpha+\beta}\}$ & 3 \\
    \BoxG & $\{\Rt{\alpha},\Rt{\psi}\}$ & $\emptyset$ & 2 \\
    \BoxB & $\{\Rt{\beta},\Rt{\psi}\}$ & $\{\Rt{\beta+\psi},\Rt{\beta+2\psi}\}$ & 4 \\
\end{tabular}
\caption{The roots spanned by $\Rt{\alpha},\Rt{\beta},\Rt{\psi}$ in $B_3$ as well as the subsets spanned by pairs.}
\end{table}

There are $9$ roots meaning $\binom{9}2 = 36$ commutation relations. We cover $\binom{3}2+\binom{2}2+\binom{4}2=10$ of them, so we sadly have to prove $26$. We cover $6$ linear relations and have to prove $3$.

\subsubsection{Assumed relations}

\paragraph{In-subgroup relations.} We enumerate the following in-subgroup relations in $\GrUnipBLg{p^k}:$

\begin{tcolorbox}[colback=\ColorR,colframe=\ColorBackR,title={Commutation relations spanned by $\Rt{\alpha}$ and $\Rt{\beta}$}]
\begin{relation}[\RelNameBLg\RelNameSubgp\RelNameCommutator{\Rt{\alpha}}{\Rt{\beta}}]\label{rel:b3-large:sub:comm:alpha:beta}
\[ \DQuant{i,j \in [1]}{t,u \in \LiftScalars} \Comm{\Eld{\alpha}ti}{\Eld{\beta}uj} = \Eld{\alpha+\beta}{tu}{i+j}. \]
\end{relation}
\begin{relation}[\RelNameBLg\RelNameSubgp\RelNameCommute{\Rt{\alpha}}{\Rt{\alpha+\beta}}]\label{rel:b3-large:sub:comm:alpha:alpha+beta}
\[ \DQuant{i \in [1], j \in [2]}{t,u \in \LiftScalars}\Comm{\Eld{\alpha}ti}{\Eld{\alpha+\beta}uj} = \Id. \]
\end{relation}
\begin{relation}[\RelNameBLg\RelNameSubgp\RelNameCommute{\Rt{\beta}}{\Rt{\alpha+\beta}}]\label{rel:b3-large:sub:comm:beta:alpha+beta}
\[ \DQuant{i \in [1], j \in [2]}{t,u \in \LiftScalars}\Comm{\Eld{\beta}ti}{\Eld{\alpha+\beta}uj} = \Id. \] 
\end{relation}
\end{tcolorbox}

\begin{tcolorbox}[colback=\ColorG,colframe=\ColorBackG,title={Commutation relations spanned by $\Rt{\alpha}$ and $\Rt{\psi}$}]
\begin{relation}[\RelNameBLg\RelNameSubgp\RelNameCommute{\Rt{\alpha}}{\Rt{\psi}}]\label{rel:b3-large:sub:comm:alpha:psi}
\[    \DQuant{i \in [1], j \in [1]}{t,u \in \LiftScalars} \Comm{\Eld{\alpha}ti}{\Eld{\psi}uj} = \Id. \]
\end{relation}
\end{tcolorbox}

\begin{tcolorbox}[colback=\ColorB,colframe=\ColorBackB,title={Commutation relations spanned by $\Rt{\beta}$ and $\Rt{\psi}$}]
\begin{relation}[\RelNameBLg\RelNameSubgp\RelNameCommutator{\Rt{\beta}}{\Rt{\psi}}]\label{rel:b3-large:sub:comm:beta:psi}
    \[ \DQuant{i,j \in [1]}{t,u \in \LiftScalars} \Comm{\Eld{\beta}{t}i}{\Eld{\psi}{u}j} = \Eld{\beta+\psi}{tu}{i+j} \Eld{\beta+2\psi}{tu^2}{i+2j}. \]
\end{relation}
\begin{relation}[\RelNameBLg\RelNameSubgp\RelNameCommute{\Rt{\beta}}{\Rt{\beta+\psi}}]\label{rel:b3-large:sub:comm:beta:beta+psi}
    \[ \DQuant{i \in [1], j \in [2]}{t,u \in \LiftScalars} \Comm{\Eld{\beta}{t}i}{\Eld{\beta+\psi}{u}j} = \Id. \]
\end{relation}
\begin{relation}[\RelNameBLg\RelNameSubgp\RelNameCommute{\Rt{\beta}}{\Rt{\beta+2\psi}}]\label{rel:b3-large:sub:comm:beta:beta+2psi}
    \[ \DQuant{i \in [1], j \in [3]}{t,u \in \LiftScalars} \Comm{\Eld{\beta}{t}i}{\Eld{\beta+2\psi}{u}j} = \Id. \]
\end{relation}
\begin{relation}[\RelNameBLg\RelNameSubgp\RelNameCommutator{\Rt{\psi}}{\Rt{\beta+\psi}}]\label{rel:b3-large:sub:comm:psi:beta+psi}
    \[ \DQuant{i \in [1], j \in [2]}{t,u \in \LiftScalars} \Comm{\Eld{\psi}{t}i}{\Eld{\beta+\psi}{u}j} = \Eld{\beta+2\psi}{2tu}{i+j}. \]
\end{relation}
\begin{relation}[\RelNameBLg\RelNameSubgp\RelNameCommute{\Rt{\psi}}{\Rt{\beta+2\psi}}]\label{rel:b3-large:sub:comm:psi:beta+2psi}
    \[ \DQuant{i \in [1], j \in [3]}{t,u \in \LiftScalars} \Comm{\Eld{\psi}{t}i}{\Eld{\beta+2\psi}{u}j} = \Id. \]
\end{relation}
\begin{relation}[\RelNameBLg\RelNameSubgp\RelNameCommute{\Rt{\beta+\psi}}{\Rt{\beta+2\psi}}]\label{rel:b3-large:sub:comm:beta+psi:beta+2psi}
    \[ \DQuant{i \in [2], j \in [3]}{t,u \in \LiftScalars} \Comm{\Eld{\beta+\psi}{t}i}{\Eld{\beta+2\psi}{u}j} = \Id. \]
\end{relation}
\end{tcolorbox}

\begin{tcolorbox}[breakable,colback=\ColorNeut,title={Linearity relations}]
\begin{relation}[\RelNameBLg\RelNameSubgp\RelNameLinearity{\Rt{\alpha}}]\label{rel:b3-large:sub:lin:alpha}
    \[ \DQuant{i \in [1]}{t,u \in \LiftScalars} \Eld{\alpha}{t}{i} \Eld{\alpha}{u}{i} = \Eld{\alpha}{t+u}{i}. \]
\end{relation}
\begin{relation}[\RelNameBLg\RelNameSubgp\RelNameLinearity{\Rt{\beta}}]\label{rel:b3-large:sub:lin:beta}
    \[ \DQuant{i \in [1]}{t,u \in \LiftScalars} \Eld{\beta}{t}{i} \Eld{\beta}{u}{i} = \Eld{\alpha}{t+u}{i}. \]
\end{relation}
\begin{relation}[\RelNameBLg\RelNameSubgp\RelNameLinearity{\Rt{\psi}}]\label{rel:b3-large:sub:lin:psi}
    \[ \DQuant{i \in [1]}{t,u \in \LiftScalars} \Eld{\psi}{t}{i} \Eld{\psi}{u}{i} = \Eld{\alpha}{t+u}{i}. \]
\end{relation}
\begin{relation}[\RelNameBLg\RelNameSubgp\RelNameLinearity{\Rt{\alpha+\beta}}]\label{rel:b3-large:sub:lin:alpha+beta}
    \[ \DQuant{i \in [2]}{t,u \in \LiftScalars} \Eld{\alpha+\beta}{t}{i} \Eld{\alpha+\beta}{u}{i} = \Eld{\alpha+\beta}{t+u}{i}. \]
\end{relation}
\begin{relation}[\RelNameBLg\RelNameSubgp\RelNameLinearity{\Rt{\beta+\psi}}]\label{rel:b3-large:sub:lin:beta+psi}
    \[ \DQuant{i \in [2]}{t,u \in \LiftScalars} \Eld{\beta+\psi}{t}{i} \Eld{\beta+\psi}{u}{i} = \Eld{\beta+\psi}{t+u}{i}. \]
\end{relation}
\begin{relation}[\RelNameBLg\RelNameSubgp\RelNameLinearity{\Rt{\beta+2\psi}}]\label{rel:b3-large:sub:lin:beta+2psi}
    \[ \DQuant{i \in [3]}{t,u \in \LiftScalars} \Eld{\beta+2\psi}{t}{i} \Eld{\beta+2\psi}{u}{i} = \Eld{\beta+2\psi}{t+u}{i}. \]
\end{relation}
\end{tcolorbox}

\begin{tcolorbox}[breakable,colback=\ColorNeut,title={Self-commutator relations}]
\begin{relation}[\RelNameBLg\RelNameSubgp\RelNameSelfCommute{\Rt{\alpha}}]\label{rel:b3-large:sub:comm:self:alpha}
    \[ \DQuant{i,j \in [1]}{t \in \LiftScalars} \Comm{\Eld{\alpha}{t}{i}}{\Eld{\alpha}{u}{j}} = \Id. \]
\end{relation}
\begin{relation}[\RelNameBLg\RelNameSubgp\RelNameSelfCommute{\Rt{\beta}}]\label{rel:b3-large:sub:comm:self:beta}
    \[ \DQuant{i,j \in [1]}{t \in \LiftScalars} \Comm{\Eld{\beta}{t}{i}}{\Eld{\beta}{u}{j}} = \Id. \]
\end{relation}
\begin{relation}[\RelNameBLg\RelNameSubgp\RelNameSelfCommute{\Rt{\psi}}]\label{rel:b3-large:sub:comm:self:psi}
    \[ \DQuant{i,j \in [1]}{t \in \LiftScalars} \Comm{\Eld{\psi}{t}{i}}{\Eld{\psi}{u}{j}} = \Id. \]
\end{relation}
\begin{relation}[\RelNameBLg\RelNameSubgp\RelNameSelfCommute{\Rt{\alpha+\beta}}]\label{rel:b3-large:sub:comm:self:alpha+beta}
    \[ \DQuant{i,j \in [2]}{t \in \LiftScalars} \Comm{\Eld{\alpha+\beta}{t}{i}}{\Eld{\alpha+\beta}{u}{j}} = \Id. \]
\end{relation}
\begin{relation}[\RelNameBLg\RelNameSubgp\RelNameSelfCommute{\Rt{\beta+\psi}}]\label{rel:b3-large:sub:comm:self:beta+psi}
    \[ \DQuant{i,j \in [2]}{t \in \LiftScalars} \Comm{\Eld{\beta+\psi}{t}{i}}{\Eld{\beta+\psi}{u}{j}} = \Id. \]
\end{relation}
\begin{relation}[\RelNameBLg\RelNameSubgp\RelNameSelfCommute{\Rt{\beta+2\psi}}]\label{rel:b3-large:sub:comm:self:beta+2psi}
    \[ \DQuant{i,j \in [3]}{t \in \LiftScalars} \Comm{\Eld{\beta+2\psi}{t}{i}}{\Eld{\beta+2\psi}{u}{j}} = \Id. \]
\end{relation}
\end{tcolorbox}

\paragraph{Relations from lifting} We also enumerate some lifted relations. Firstly, we have the nonhomogeneous lift (cf. \Cref{eq:chev:nonhom}) of the relation $\Comm{\El{\alpha+\beta}{1}}{\El{\beta+\psi}{1}} = 1$ in $\UnipBLg{q}$.

\begin{relation}[\RelNameBLg\RelNameNonHomLiftRaw\RelNameCommute{\Rt{\alpha+\beta}}{\Rt{\beta+\psi}}]\label{rel:b3-large:comm:raw:lift:alpha+beta:beta+psi}
    \begin{multline*}
        \Quant{t_1,t_0,u_1,u_0,v_1,v_0 \in \LiftScalars} \\
        \Comm{\Eld{\alpha+\beta}{t_1u_1}{2} \Eld{\alpha+\beta}{t_1u_0+t_0u_1}{1} \Eld{\alpha+\beta}{t_0u_0}{0}}{\Eld{\beta+\psi}{u_1v_1}{2} \Eld{\beta+\psi}{u_1v_0+u_0v_1}{1} \Eld{\beta+\psi}{u_0v_0}{0}} = \Id.
    \end{multline*}
\end{relation}

and the nonhomogeneous lift of the relation $\Comm{\El{\alpha}{1}}{\Comm{\El{\alpha}{1}}{\El{\beta+\psi}{1}}} = 1$ in $\UnipBLg{q}$.

\begin{relation}[\RelNameBLg\RelNameNonHomLiftRaw\RelNameCommute{\Rt{\alpha}}{\Rt{\alpha+2\beta+2\psi}}]\label{rel:b3-large:comm:raw:lift:alpha:alpha+2beta+2psi:square}
    \begin{multline*}
    \Quant{t_1,t_0,u_1,u_0,v_1,v_0 \in \LiftScalars} \\
    \Comm{\Eld{\alpha}{t_1}{1}\Eld{\alpha}{t_0}{0}}{\Comm{\Eld{\alpha+\beta}{t_1u_1}{2} \Eld{\alpha+\beta}{t_1u_0+t_0u_1}{2} \Eld{\alpha+\beta}{t_0u_0}{0}}{\Eld{\beta+2\psi}{t_1u_1^2}{3} \Eld{\beta+2\psi}{t_0u_1^2 + 2t_1u_0u_1}{2} \Eld{\beta+2\psi}{t_1u_0^2 + 2t_0u_0u_1}{1} \Eld{\beta+2\psi}{t_0u_0^2}{0}}} = \Id.
    \end{multline*}
\end{relation}

We also need numerous homogeneously lifted relations. These are all the lifts of respective $\UnipBLg{q}$ relations given by deleting the degrees and replacing $t_1,t_0,u_1,u_0,v_1,v_0$ with $1$.

\begin{relation}[\RelNameBLg\RelNameHomLiftRaw\RelNameInterchange{\Rt{\alpha+\beta+\psi}}]\label{rel:b3-large:raw:lift:alpha+beta+psi}
    \begin{multline*}
        \DQuant{i,j,k \in [1]}{t,u,v \in \LiftScalars} \\ \Eld{\psi}{-v/2}{k} \Eld{\alpha+\beta}{tu}{i+j} \Eld{\psi}{v}{k} \Eld{\alpha+\beta}{-tu}{i+j} \Eld{\psi}{-v/2}{k} = \Eld{\beta+\psi}{-uv/2}{j+k} \Eld{\alpha}{t}{i} \Eld{\beta+\psi}{uv}{j+k} \Eld{\alpha}{-t}{i} \Eld{\beta+\psi}{-uv/2}{j+k}.
    \end{multline*}
\end{relation}

\begin{relation}[\RelNameBLg\RelNameHomLiftRaw\RelNameDoub{\Rt{\alpha+\beta+\psi}}]\label{rel:b3-large:raw:lift:doub:alpha+beta+psi}
    \begin{gather*}
    \DQuant{i,j,k \in [1]}{t,u,v \in \LiftScalars} \Eld{\psi}{-v/2}{k} \Eld{\alpha+\beta}{tu}{i+j} \Eld{\psi}{v}{k} \Eld{\alpha+\beta}{-tu}{i+j} \Eld{\psi}{-v/2}{k} \cdot \\
    \Eld{\psi}{-v/2}{k} \Eld{\alpha+\beta}{tu}{i+j} \Eld{\psi}{v}{k} \Eld{\alpha+\beta}{-tu}{i+j} \Eld{\psi}{-v/2}{k} = \Eld{\psi}{-v}{k} \Eld{\alpha+\beta}{tu}{i+j} \Eld{\psi}{2v}{k} \Eld{\alpha+\beta}{-tu}{i+j} \Eld{\psi}{-v}{k}
    \end{gather*}
\end{relation}

\begin{relation}[\RelNameBLg\RelNameHomLiftRaw\RelNameInterchange{\Rt{\alpha+\beta+2\psi}}]\label{rel:b3-large:raw:lift:alpha+beta+2psi}
    \[
    \DQuant{i,j,k \in [1]}{t,u,v \in \LiftScalars} \Comm{\Eld{\psi}{-v/2}{k} \Eld{\alpha+\beta}{tu}{i+j} \Eld{\psi}{v}{k} \Eld{\alpha+\beta}{-tu}{i+j} \Eld{\psi}{-v/2}{k}}{\Eld{\psi}{v}{k}} = \Comm{\Eld{\alpha}{t}{i}}{\Eld{\beta+2\psi}{-2uv^2}{j+2k}}.
    \]
\end{relation}

\begin{relation}[\RelNameBLg\RelNameHomLiftRaw\RelNameCommutator{\Rt{\beta+\psi}}{\Comm{\Rt{\alpha}}{\Rt{\beta+2\psi}}}]\label{rel:b3-large:comm:raw:lift:beta+psi:alpha:beta+2psi}
        \[
    \DQuant{i,j,k\in [1]}{t,u,v \in \LiftScalars} \Comm{\Eld{\beta+\psi}{uv}{j+k}}{\Comm{\Eld{\alpha}{t}{i}}{\Eld{\beta+2\psi}{uv^2}{j+2k}}} = \Id.
    \]
\end{relation}

\begin{relation}[\RelNameBLg\RelNameHomLiftRaw\RelNameInvDoub{\Comm{\Rt{\alpha}}{\Rt{\beta+2\psi}}}]\label{rel:b3-large:inv-doub:raw:lift:alpha:beta+2psi}
    \begin{gather*}
    \DQuant{i,j,k \in [1]}{t,u,v \in \LiftScalars} \Comm{\Eld{\alpha}{t}{i}}{\Eld{\beta+2\psi}{uv^2}{j+2k}} = \Comm{\Eld{\alpha}{-t}{i}}{\Eld{\beta+2\psi}{-uv^2}{j+2k}}, \\
    \DQuant{i \in [1], j \in [3]}{t,u,v \in \LiftScalars} \Comm{\Eld{\alpha}{t}{i}}{\Eld{\beta+2\psi}{uv^2}{j+2k}} \Comm{\Eld{\alpha}{-t}{i}}{\Eld{\beta+2\psi}{uv^2}{j+2k}} = \Id, \\
    \DQuant{i \in [1], j \in [3]}{t,u \in \LiftScalars} \Comm{\Eld{\alpha}{t}{i}}{\Eld{\beta+2\psi}{uv^2}{j+2k}} \Comm{\Eld{\alpha}{t}{i}}{\Eld{\beta+2\psi}{uv^2}{j+2k}} = \Comm{\Eld{\alpha}{2t}{i}}{\Eld{\beta+2\psi}{uv^2}{j+2k}}.
    \end{gather*}
\end{relation}

\begin{relation}[\RelNameBLg\RelNameHomLiftRaw\RelNameCommutator{\Rt{\beta+2\psi}}{\Rt{\alpha+\beta+\psi}}]\label{rel:b3-large:comm:raw:lift:beta+2psi:alpha+beta+psi}
        \[
    \DQuant{i,j,k \in [1]}{t,u,v \in \LiftScalars} \Comm{\Eld{\beta+2\psi}{uv^2}{j+2k}}{\Eld{\psi}{-v/2}{k} \Eld{\alpha+\beta}{tu}{i+j} \Eld{\psi}{v}{k} \Eld{\alpha+\beta}{-tu}{i+j} \Eld{\psi}{-v/2}{k}} = \Id.
    \]
\end{relation}

\begin{relation}[\RelNameBLg\RelNameHomLiftRaw\RelNameInterchange{\Rt{\alpha+2\beta+2\psi}}]\label{rel:b3-large:raw:lift:alpha+2beta+2psi}
    \begin{multline*}
    \DQuant{i,j,k \in [1]}{t,u,v \in \LiftScalars} \\
    \Comm{\Eld{\alpha+\beta}{tu}{i+j}}{\Eld{\beta+2\psi}{2uv^2}{j+2k}} = \Comm{\Eld{\psi}{-v/2}{k} \Eld{\alpha+\beta}{tu}{i+j} \Eld{\psi}{v}{k} \Eld{\alpha+\beta}{-tu}{i+j} \Eld{\psi}{-v/2}{k}}{\Eld{\beta+\psi}{uv}{j+k}}
    \end{multline*}
    and
    \begin{multline*}
    \DQuant{i,j,k \in [1]}{t,u,v \in \LiftScalars} \\
    \Comm{\Eld{\psi}{-v/2}{k} \Eld{\alpha+\beta}{tu}{i+j} \Eld{\psi}{v}{k} \Eld{\alpha+\beta}{-tu}{i+j} \Eld{\psi}{-v/2}{k}}{\Eld{\beta+\psi}{uv}{j+k}} = \Comm{\Comm{\Eld{\alpha}{t}{i}}{\Eld{\beta+2\psi}{uv^2}{j+2k}}}{\Eld{\beta}{u}{j}}.
    \end{multline*}
\end{relation}

\begin{relation}[\RelNameBLg\RelNameHomLiftRaw\RelNameCommutator{\Rt{\psi}}{\Comm{\Rt{\alpha+\beta}}{\Rt{\beta+2\psi}}}]\label{rel:b3-large:comm:raw:lift:psi:alpha+beta:beta+2psi}
\[
    \DQuant{i,j,k \in [1]}{t,u,v \in \LiftScalars} \Comm{\Eld{\psi}{v}{k}}{\Comm{\Eld{\alpha+\beta}{tu}{i+j}}{\Eld{\beta+2\psi}{uv^2}{j+2k}}} = \Id.
\]
\end{relation}

\begin{relation}[\RelNameBLg\RelNameHomLiftRaw\RelNameCommutator{\Rt{\alpha+\beta}}{\Comm{\Rt{\alpha+\beta}}{\Rt{\beta+2\psi}}}]\label{rel:b3-large:comm:raw:lift:alpha+beta:alpha+beta:beta+2psi}
\[
\DQuant{i,j,k \in [1]}{t,u,v \in \LiftScalars, s \in \{\pm 1\}} \Comm{\Eld{\alpha+\beta}{tu}{i+j}}{\Comm{\Eld{\alpha+\beta}{s tu}{i+j}}{\Eld{\beta+2\psi}{uv^2}{j+2k}}} = \Id.
\]
\end{relation}

\begin{relation}[\RelNameBLg\RelNameHomLiftRaw\RelNameInvDoub{\Comm{\Rt{\alpha+\beta}}{\Rt{\beta+2\psi}}}]\label{rel:b3-large:inv-doub:raw:lift:alpha+beta:beta+2psi}
    \begin{gather*}
        \DQuant{i,j,k \in [1]}{t,u,v \in \LiftScalars} \Comm{\Eld{\alpha+\beta}{tu}{i+j}}{\Eld{\beta+2\psi}{uv^2}{j+2k}}  = \Comm{\Eld{\alpha+\beta}{-tu}{i+j}}{\Eld{\beta+2\psi}{-uv^2}{j+2k}} , \\
        \DQuant{i,j,k \in [1]}{t,u,v \in \LiftScalars} \Comm{\Eld{\alpha+\beta}{tu}{i+j}}{\Eld{\beta+2\psi}{uv^2}{j+2k}} \Comm{\Eld{\alpha+\beta}{-tu}{i+j}}{\Eld{\beta+2\psi}{uv^2}{j+2k}} = \Id, \\
        \DQuant{i,j,k \in [1]}{t,u,v \in \LiftScalars} \Comm{\Eld{\alpha+\beta}{tu}{i+j}}{\Eld{\beta+2\psi}{uv^2}{j+2k}} \Comm{\Eld{\alpha+\beta}{tu}{i+j}}{\Eld{\beta+2\psi}{uv^2}{j+2k}} = \Comm{\Eld{\alpha+\beta}{2tu}{i+j}}{\Eld{\beta+2\psi}{uv^2}{j+2k}}.
    \end{gather*}
\end{relation}

\begin{relation}[\RelNameBLg\RelNameHomLiftRaw\RelNameInvDoub{\Comm{\Rt{\beta}}{\Rt{\alpha+\beta+2\psi}}}]\label{rel:b3-large:inv-doub:raw:lift:beta:alpha+beta+2psi}
    \begin{gather*}
        \DQuant{i,j,k \in [1]}{t,u,v \in \LiftScalars} \Comm{\Eld{\beta}{t}{i}}{\Comm{\Eld{\alpha}{t}{i}}{\Eld{\beta+2\psi}{uv^2}{j+2k}}} = \Comm{\Eld{\beta}{-t}{i}}{\Comm{\Eld{\alpha}{-t}{i}}{\Eld{\beta+2\psi}{uv^2}{j+2k}}}, \\
        \DQuant{i,j,k \in [1]}{t,u,v \in \LiftScalars} \Comm{\Eld{\beta}{t}{i}}{\Comm{\Eld{\alpha}{t}{i}}{\Eld{\beta+2\psi}{uv^2}{j+2k}}}\Comm{\Eld{\beta}{-t}{i}}{\Comm{\Eld{\alpha}{t}{i}}{\Eld{\beta+2\psi}{uv^2}{j+2k}}} = \Id, \\
    \end{gather*}
    and
    \begin{multline*}
        \DQuant{i,j,k \in [1]}{t,u,v \in \LiftScalars} \\ \Comm{\Eld{\beta}{t}{i}}{\Comm{\Eld{\alpha}{t}{i}}{\Eld{\beta+2\psi}{uv^2}{j+2k}}} \Comm{\Eld{\beta}{t}{i}}{\Comm{\Eld{\alpha}{t}{i}}{\Eld{\beta+2\psi}{uv^2}{j+2k}}}= \Comm{\Eld{\beta}{2t}{i}}{\Comm{\Eld{\alpha}{t}{i}}{\Eld{\beta+2\psi}{uv^2}{j+2k}}}.
    \end{multline*}
\end{relation}

\begin{relation}[\RelNameBLg\RelNameHomLiftRaw\RelNameCommutator{\Rt{\beta+\psi}}{\Rt{\alpha+\beta+2\psi}}]\label{rel:b3-large:comm:raw:lift:beta+psi:alpha+beta+2psi}
        \[
    \DQuant{i,j,k \in [1]}{t,u,v \in \LiftScalars} \Comm{\Eld{\beta+\psi}{uv}{j+k}}{\Comm{\Eld{\alpha}{t}{i}}{\Eld{\beta+2\psi}{uv^2}{j+2k}}} = \Id.
    \]
\end{relation}

\begin{relation}[\RelNameBLg\RelNameHomLiftRaw\RelNameCommutator{\Rt{\beta+2\psi}}{\Rt{\alpha+\beta+2\psi}}]\label{rel:b3-large:comm:raw:lift:beta+2psi:alpha+beta+2psi}
        \[
    \DQuant{i,j,k \in [1]}{t,u,v \in \LiftScalars} \Comm{\Eld{\beta+2\psi}{uv^2}{j+2k}}{\Comm{\Eld{\alpha}{t}{i}}{\Eld{\beta+2\psi}{uv^2}{j+2k}}} = \Id.
    \]
\end{relation}

\subsubsection{Additional identities}

As usual, from the linearity relations (\Cnref{rel:b3-large:sub:lin:alpha}, \Cnref{rel:b3-large:sub:lin:beta}, \Cnref{rel:b3-large:sub:lin:psi}, \Cnref{rel:b3-large:sub:lin:alpha+beta}, \Cnref{rel:b3-large:sub:lin:beta+psi}, \Cnref{rel:b3-large:sub:lin:beta+2psi}) and \Cref{prop:prelim:group-homo}, we get respective identity relations:

\begin{tcolorbox}[breakable,colback=\ColorNeut,title={Identity relations}]
\begin{relation}[\RelNameBLg\RelNameId{\Rt{\alpha}}]\label{rel:b3-large:sub:id:alpha}
    \[ \Quant{i \in [1]} \Eld{\alpha}{0}{i} = \Id. \]
\end{relation}
\begin{relation}[\RelNameBLg\RelNameId{\Rt{\beta}}]\label{rel:b3-large:sub:id:beta}
    \[ \Quant{i \in [1]} \Eld{\beta}{0}{i} = \Id. \]
\end{relation}
\begin{relation}[\RelNameBLg\RelNameId{\Rt{\psi}}]\label{rel:b3-large:sub:id:psi}
    \[ \Quant{i \in [1]} \Eld{\psi}{0}{i} = \Id. \]
\end{relation}
\begin{relation}[\RelNameBLg\RelNameId{\Rt{\alpha+\beta}}]\label{rel:b3-large:sub:id:alpha+beta}
    \[ \Quant{i \in [2]} \Eld{\alpha+\beta}{0}{i} = \Id. \]
\end{relation}
\begin{relation}[\RelNameBLg\RelNameId{\Rt{\beta+\psi}}]\label{rel:b3-large:sub:id:beta+psi}
    \[ \Quant{i \in [2]} \Eld{\beta+\psi}{0}{i} = \Id. \]
\end{relation}
\begin{relation}[\RelNameBLg\RelNameId{\Rt{\beta+2\psi}}]\label{rel:b3-large:sub:id:beta+2psi}
    \[ \Quant{i \in [3]} \Eld{\beta+2\psi}{0}{i} = \Id. \]
\end{relation}
\end{tcolorbox}

and inverse relations:

\begin{tcolorbox}[breakable,colback=\ColorNeut,title={Inverse relations}]
\begin{relation}[\RelNameBLg\RelNameSubgp\RelNameInv{\Rt{\alpha}}]\label{rel:b3-large:sub:inv:alpha}
    \[ \DQuant{i \in [1]}{t \in \LiftScalars} \Eld{\alpha}{t}{i} \Eld{\alpha}{-t}{i} = \Id. \]
\end{relation}
\begin{relation}[\RelNameBLg\RelNameSubgp\RelNameInv{\Rt{\beta}}]\label{rel:b3-large:sub:inv:beta}
    \[ \DQuant{i \in [1]}{t \in \LiftScalars} \Eld{\beta}{t}{i} \Eld{\beta}{-t}{i} = \Id. \]
\end{relation}
\begin{relation}[\RelNameBLg\RelNameSubgp\RelNameInv{\Rt{\psi}}]\label{rel:b3-large:sub:inv:psi}
    \[ \DQuant{i \in [1]}{t \in \LiftScalars} \Eld{\psi}{t}{i} \Eld{\psi}{-t}{i} = \Id. \]
\end{relation}
\begin{relation}[\RelNameBLg\RelNameSubgp\RelNameInv{\Rt{\alpha+\beta}}]\label{rel:b3-large:sub:inv:alpha+beta}
    \[ \DQuant{i \in [2]}{t \in \LiftScalars} \Eld{\alpha+\beta}{t}{i} \Eld{\alpha+\beta}{-t}{i} = \Id. \]
\end{relation}
\begin{relation}[\RelNameBLg\RelNameSubgp\RelNameInv{\Rt{\beta+\psi}}]\label{rel:b3-large:sub:inv:beta+psi}
    \[ \DQuant{i \in [2]}{t \in \LiftScalars} \Eld{\beta+\psi}{t}{i} \Eld{\beta+\psi}{-t}{i} = \Id. \]
\end{relation}
\begin{relation}[\RelNameBLg\RelNameSubgp\RelNameInv{\Rt{\beta+2\psi}}]\label{rel:b3-large:sub:inv:beta+2psi}
    \[ \DQuant{i \in [3]}{t \in \LiftScalars} \Eld{\beta+2\psi}{t}{i} \Eld{\beta+2\psi}{-t}{i} = \Id. \]
\end{relation}
\end{tcolorbox}

We also have the standard expression with the same proof as \Cref{rel:b3-small:sub:expr:beta+psi}:
\begin{relation}[\RelNameBLg\RelNameExpr{\Rt{\beta+\psi}}]\label{rel:b3-large:sub:expr:beta+psi}
    \[ \DQuant{i,j \in [1]}{t,u \in \LiftScalars} \Eld{\beta+\psi}{tu}{i+j} = \Eld{\psi}{-t/2}{i} \Eld{\beta}{u}{j} \Eld{\psi}{t}{i} \Eld{\beta}{-u}{j} \Eld{\psi}{-t/2}{i}. \]
\end{relation}

Similarly from \Cnref{rel:b3-large:sub:inv:alpha}, \Cnref{rel:b3-large:sub:inv:beta}, and \Cnref{rel:b3-large:sub:comm:alpha:beta}:

\begin{relation}[\RelNameBLg\RelNameExpr{\Rt{\alpha+\beta}}]\label{rel:b3-large:sub:expr:alpha+beta}
    \[ \DQuant{i,j \in [1]}{t,u \in \LiftScalars} \Eld{\alpha+\beta}{tu}{i+j} = \Eld{\alpha}{t}{i} \Eld{\beta}{u}{j} \Eld{\alpha}{-t}{i} \Eld{\beta}{-u}{j}. \]
\end{relation}

From \Cnref{rel:b3-large:sub:inv:psi}, \Cnref{rel:b3-large:sub:inv:beta+psi}, and \Cnref{rel:b3-large:sub:comm:psi:beta+psi}:

\begin{relation}[\RelNameBLg\RelNameExpr{\Rt{\beta+2\psi}}]\label{rel:b3-large:sub:expr:beta+2psi}
    \[ \DQuant{i,j \in [1]}{t,u \in \LiftScalars} \Eld{\beta+2\psi}{2tu}{i+j} = \Eld{\psi}{t}{i} \Eld{\beta+\psi}{u}{j} \Eld{\psi}{-t}{i} \Eld{\beta+\psi}{-u}{j}. \]
\end{relation}

We also enumerate order-change relations.

From \Cnref{rel:b3-large:sub:comm:alpha:beta}, \Cnref{eq:comm:left:inv}, and \Cnref{eq:comm:right:inv}, we deduce:
\begin{relation}[\RelNameBLg\RelNameOrder{\Rt{\beta}}{\Rt{\alpha}}]\label{rel:b3-large:order:beta:alpha}
    \[ \Eld{\beta}{u}{j} \Eld{\alpha}{t}{i} = \asite{\Eld{\alpha+\beta}{-tu}{i+j}} \Eld{\alpha}{t}{i} \Eld{\beta}{u}{j}. \]
\end{relation}

From \Cnref{rel:b3-large:sub:comm:beta:psi}, \Cnref{rel:b3-large:sub:comm:beta+psi:beta+2psi}, \Cnref{eq:comm:left:inv}, \Cnref{eq:comm:mid:inv}, \Cnref{eq:comm:right:inv}, \Cnref{rel:b3-large:sub:inv:beta}, \Cnref{rel:b3-large:sub:inv:psi}, \Cnref{rel:b3-large:sub:inv:beta+psi}, and \Cnref{rel:b3-large:sub:inv:beta+2psi}, we have:

\begin{relation}[\RelNameBLg\RelNameOrder{\Rt{\psi}}{\Rt{\beta}}]\label{rel:b3-large:order:psi:beta}
\begin{multline*}
    \Quant{t,u \in \LiftScalars} \Eld{\psi}{u}{j} \Eld{\beta}{t}{i} = \Eld{\beta}{t}{i} \Eld{\psi}{u}{j} \asite{\Eld{\beta+\psi}{-tu}{i+j} \Eld{\beta+2\psi}{tu^2}{i+2j}} = \asite{\Eld{\beta+2\psi}{-tu^2}{i+2j} \Eld{\beta+\psi}{-tu}{i+j}} \Eld{\beta}{t}{i} \Eld{\psi}{u}{j} \\
     =  \Eld{\beta}{t}{i} \asite{\Eld{\beta+2\psi}{-tu^2}{i+2j} \Eld{\beta+\psi}{-tu}{i+j}} \Eld{\psi}{u}{j} = \Eld{\beta}{t}{i} \asite{\Eld{\beta+\psi}{-tu}{i+j} \Eld{\beta+2\psi}{-tu^2}{i+2j}} \Eld{\psi}{u}{j}.
\end{multline*}
\end{relation}

From \Cnref{rel:b3-large:sub:comm:psi:beta+psi}, \Cnref{eq:comm:mid:str}, \Cnref{eq:comm:right:str}, \Cnref{rel:b3-large:sub:inv:beta+psi}, and \Cnref{rel:b3-large:sub:inv:beta+2psi}:
\begin{relation}[\RelNameBLg\RelNameOrder{\Rt{\psi}}{\Rt{\beta+\psi}}]\label{rel:b3-large:order:psi:beta+psi}
    \[ \DQuant{i \in [1],j\in[2]}{t,u \in \LiftScalars} \Eld{\psi}{t}{i} \Eld{\beta+\psi}{u}{j} = \Eld{\beta+\psi}{u}{j} \asite{\Eld{\beta+2\psi}{2tu}{i+j}} \Eld{\psi}{t}{i} = \Eld{\beta+\psi}{u}{j} \Eld{\psi}{t}{i} \asite{\Eld{\beta+2\psi}{2tu}{i+j}}. \]
\end{relation}
Similarly from \Cnref{rel:b3-large:sub:comm:psi:beta+psi}, \Cnref{eq:comm:mid:inv}, \Cnref{eq:comm:right:inv}, \Cnref{rel:b3-large:sub:inv:psi}, \Cnref{rel:b3-large:sub:inv:beta+psi}, and \Cnref{rel:b3-large:sub:inv:beta+2psi}:
\begin{relation}[\RelNameBLg\RelNameOrder{\Rt{\beta+\psi}}{\Rt{\psi}}]\label{rel:b3-large:order:beta+psi:psi}
    \[ \DQuant{i \in [1],j\in[2]} {t,u \in \LiftScalars} \Eld{\beta+\psi}{u}{j} \Eld{\psi}{t}{i} = \Eld{\psi}{t}{i} \asite{\Eld{\beta+2\psi}{-2tu}{i+j}} \Eld{\beta+\psi}{u}{j} = \Eld{\psi}{t}{i} \Eld{\beta+\psi}{u}{j} \asite{\Eld{\beta+2\psi}{-2tu}{i+j}}. \]
\end{relation}

\begin{table}
\centering
\begin{tabular}{c|c|c|c||c|c|c|c|c|c|}
\hline
Root & $\Rt{\alpha}$ & $\Rt{\beta}$ & $\Rt{\psi}$ & $\Rt{\alpha+\beta}$ & $\Rt{\beta+\psi}$ & $\Rt{\beta+2\psi}$ & $\Rt{\alpha+\beta+\psi}$ & $\Rt{\alpha+\beta+2\psi}$ & $\Rt{\alpha+2\beta+2\psi}$ \\ \hline
Combination & $i$ & $j$ & $k$ & $i+j$ & $j+k$ & $j+2k$ & $i+j+k$ & $i+j+2k$ & $i+2j+2k$ \\ \hline
Height & $1$ & $1$ & $1$ & $2$ & $2$ & $3$ & $3$ & $4$ & $5$ \\ \hline
\hline
& $0$ & $0$ & $0$ & $0$ & $0$ & $0$ & $0$ & $0$ & $0$ \\ \cline{2-10}
& $0$ & $0$ & $1$ & $0$ & $1$ & $2$ & $1$ & $2$ & $2$ \\ \cline{2-10}
& $0$ & $1$ & $0$ & $1$ & $1$ & $1$ & $1$ & $1$ & $2$ \\ \cline{2-10}
& $0$ & $1$ & $1$ & $1$ & $2$ & $3$ & $2$ & $3$ & $4$ \\ \cline{2-10}
& $1$ & $0$ & $0$ & $1$ & $0$ & $0$ & $1$ & $1$ & $1$ \\ \cline{2-10}
& $1$ & $0$ & $1$ & $1$ & $1$ & $2$ & $2$ & $3$ & $3$ \\ \cline{2-10}
& $1$ & $1$ & $0$ & $2$ & $1$ & $1$ & $2$ & $2$ & $3$ \\ \cline{2-10}
& $1$ & $1$ & $1$ & $2$ & $2$ & $3$ & $3$ & $4$ & $5$ \\ \cline{2-10}
\end{tabular}
\caption{Table of certain nonnegative linear combinations of three Boolean variables $i,j,k \in [1]$. There are eight rows, indexed by triples $(i,j,k) \in [1]^3$. Each column contains a linear combination of the variables. This combination has a ``height'', which is its maximum value (i.e., $i=j=k=1$). The column also corresponds to a root in the link. This is because the combinations arise when we study lifting relations from the base group to the lifted group: If $\Rt{\alpha},\Rt{\beta},\Rt{\psi}$ elements are lifted to monomials of degrees $i,j,k$, respectively, then the other six roots will be lifted to monomials of degree listed in the corresponding column. We are typically interested in \emph{pairs} of columns and the extent to which the eight listed rows \emph{cover} the set of all possible pairs of entries. For instance, consider the pair of columns $\Rt{\alpha}$ and $\Rt{\beta+2\psi}$: there are $|[1] \times [3]| = 2 \cdot 4 = 8$ possible pairs of entries which could occur, and all $8$ do occur. On the other hand, consider the pair of columns $\Rt{\alpha+\beta}$ and $\Rt{\beta+\psi}$: there are $|[2] \times [2]| = 3 \cdot 3 = 9$ possible pairs of entries which could occur, but only $7$ do occur! (Note that $(1,1)$ occurs twice, for $(i,j,k) = (1,0,1)$ and $(i,j,k) = (0,1,0)$.)}\label{tab:b3-large:homog}
\end{table}

\subsubsection{$\alpha+\beta$ and $\beta+\psi$ commute}

\begin{relation}[\RelNameBLg\RelNameCommute{\Rt{\alpha+\beta}}{\Rt{\beta+\psi}}]\label{rel:b3-large:comm:alpha+beta:beta+psi}
\[
\DQuant{i,j \in [2]}{t,u \in \LiftScalars} \Comm{\Eld{\alpha+\beta}{t}i}{\Eld{\beta+\psi}{u}j} = \Id.
\]
\end{relation}

\begin{proof}
    Analogous to the proof of \Cnref{rel:a3:comm:alpha+beta:beta+gamma} (and \Cnref{rel:b3-small:comm:beta+psi:psi+omega}) using homogeneous and nonhomogeneous lifting and \Cnref{rel:b3-large:comm:raw:lift:alpha+beta:beta+psi}.
\end{proof}

\subsubsection{Establishing $\alpha+\beta+\psi$}

\input{figures/b3-large/alpha+beta+psi}

\begin{proposition}[Establishing $\Rt{\alpha+\beta+\psi}$]\label{prop:b3-large:est:alpha+beta+psi}
    There exist elements $\Eld{\alpha+\beta+\psi}{t}{i}$ for $t \in \LiftScalars$ and $i \in [3]$ such that the following hold:
    \begin{gather*}
        \DQuant{i \in [2], j \in [1]}{t,u \in \LiftScalars} \Eld{\alpha+\beta+\psi}{tu}{i+j} = 
        \Eld{\psi}{-u/2}{j} \Eld{\alpha+\beta}{t}{i} \Eld{\psi}{u}{j} \Eld{\alpha+\beta}{-t}{i} \Eld{\psi}{-u/2}{j}, \\
        \DQuant{i \in [1],j \in [2]}{t,u \in \LiftScalars} \Eld{\alpha+\beta+\psi}{tu}{i+j} = \Eld{\beta+\psi}{-u/2}{j} \Eld{\alpha}{t}{i} \Eld{\beta+\psi}{u}{j} \Eld{\alpha}{-t}{i} \Eld{\beta+\psi}{-u/2}{j}.
    \end{gather*}
\end{proposition}

\begin{proof}
    Uses \Cref{prop:prelim:equal-conn} and \Cnref{rel:b3-large:raw:lift:alpha+beta+psi}) (i.e., the blocks in \Cref{fig:b3-large:def:alpha+beta+psi} are connected).
\end{proof}

\begin{relation}[\RelNameBLg\RelNameCommute{\Rt{\alpha}}{\Rt{\alpha+\beta+\psi}}]\label{rel:b3-large:comm:alpha:alpha+beta+psi}
        \[
    \DQuant{i \in [1], j \in [3]}{t,u \in \LiftScalars} \Comm{\Eld{\alpha}{t}{i}}{\Eld{\alpha+\beta+\psi}{u}{j}} = \Id.
    \]
\end{relation}

\begin{proof}
    $\Rt{\alpha+\beta+\psi}$ elements can be expressed as products of $\Rt{\alpha+\beta}$ elements and $\Rt{\psi}$ elements (\Cnref{prop:b3-large:est:alpha+beta+psi}). $\Rt{\alpha}$ elements commute with $\Rt{\alpha+\beta}$ elements (\Cnref{rel:b3-large:sub:comm:alpha:alpha+beta}) and $\Rt{\psi}$ elements (\Cnref{rel:b3-large:sub:comm:alpha:psi}).
\end{proof}

\begin{relation}[\RelNameBLg\RelNameCommute{\Rt{\alpha+\beta}}{\Rt{\alpha+\beta+\psi}}]\label{rel:b3-large:comm:alpha+beta:alpha+beta+psi}
        \[
    \DQuant{i \in [2], j \in [3]}{t,u \in \LiftScalars} \Comm{\Eld{\alpha+\beta}{t}{i}}{\Eld{\alpha+\beta+\psi}{u}{j}} = \Id.
    \]
\end{relation}

\begin{proof}
    $\Rt{\alpha+\beta+\psi}$ elements can be expressed of products of $\Rt{\alpha}$ elements and $\Rt{\beta+\psi}$ elements (\Cnref{prop:b3-large:est:alpha+beta+psi}). $\Rt{\alpha+\beta}$ elements commute with $\Rt{\alpha}$ elements (\Cnref{rel:b3-large:sub:comm:alpha:alpha+beta}) and $\Rt{\beta+\psi}$ elements (\Cnref{rel:b3-large:comm:alpha+beta:beta+psi}).
\end{proof}

\begin{relation}[\RelNameBLg\RelNameCommute{\Rt{\beta}}{\Rt{\alpha+\beta+\psi}}]\label{rel:b3-large:comm:beta:alpha+beta+psi}
        \[
    \DQuant{i \in [1], j \in [3]}{t,u \in \LiftScalars} \Comm{\Eld{\beta}{t}{i}}{\Eld{\alpha+\beta+\psi}{u}{j}} = \Id.
    \]
\end{relation}

\begin{proof}
    Decompose (arbitrarily) $j = j_1+j_2$ for $j_1 \in [1],j_2 \in [2]$. Now, we write a product of $\Rt{\beta}$ and $\Rt{\alpha+\beta+\psi}$ elements and compute:
    \begin{align*}
        & \Eld{\beta}{t}{i} \asite{\Eld{\alpha+\beta+\psi}{u}{j}} \\
        \intertext{Expanding the $\Rt{\alpha+\beta+\psi}$ element into a product of $\Rt{\alpha}$ and $\Rt{\beta+\psi}$ elements (\Cnref{prop:b3-large:est:alpha+beta+psi}):}
        &= \asite{\Eld{\beta}{t}{i}} \Eld{\alpha}{-u/2}{j_1} \Eld{\beta+\psi}{1}{j_2} \Eld{\alpha}{u}{j_1} \Eld{\beta+\psi}{-1}{j_2} \Eld{\alpha}{-u/2}{j_1} \\
        \intertext{Now, we move the $\Rt{\beta}$ element on the left to the right, which creates $\Rt{\alpha+\beta}$ commutators with $\Rt{\alpha}$ elements (\Cnref{rel:b3-large:order:beta:alpha}) and no commutators with $\Rt{\beta+\psi}$ elements (\Cnref{rel:b3-large:sub:comm:beta:beta+psi}):}
        &= \asite{\Eld{\alpha+\beta}{tu/2}{i+j_1}} \Eld{\alpha}{-u/2}{j_1} \Eld{\beta+\psi}{1}{j_2} \asite{\Eld{\alpha+\beta}{-tu}{i+j_1}} \Eld{\alpha}{u}{j_1} \Eld{\beta+\psi}{-1}{j_2} \asite{\Eld{\alpha+\beta}{tu/2}{i+j_1}} \Eld{\alpha}{-u/2}{j_1} \Eld{\beta}{t}{i}.
        \intertext{Finally, we recall that $\Rt{\alpha+\beta}$ elements commute with $\Rt{\alpha}$ elements (\Cnref{rel:b3-large:sub:comm:alpha:alpha+beta}) and $\Rt{\beta+\psi}$ elements (\Cnref{rel:b3-large:comm:alpha+beta:beta+psi}). Therefore, we can draw them together and cancel them (\Cnref{rel:b3-large:sub:inv:alpha+beta}), leaving:}
        &= \asite{\Eld{\alpha}{-u/2}{j_1} \Eld{\beta+\psi}{1}{j_2} \Eld{\alpha}{u}{j_1} \Eld{\beta+\psi}{-1}{j_2} \Eld{\alpha}{-u/2}{j_1}} \Eld{\beta}{t}{i} \\
        \intertext{Reducing back into an $\Rt{\alpha+\beta+\psi}$ element (\Cnref{prop:b3-large:est:alpha+beta+psi}):}
        &= \Eld{\alpha+\beta+\psi}{t}{j} \Eld{\beta}{t}{i}.
    \end{align*}
\end{proof}

\begin{relation}[\RelNameBLg\RelNameHomLift\RelNameInvDoub{\Rt{\alpha+\beta+\psi}}]\label{rel:b3-large:inv-doub:alpha+beta+psi}
    \begin{gather*}
    \DQuant{i \in [3]}{t \in \LiftScalars} \Eld{\alpha+\beta+\psi}{t}{i} \Eld{\alpha+\beta+\psi}{-t}{i} = \Id, \\
    \DQuant{i \in [3]}{t \in \LiftScalars} \Eld{\alpha+\beta+\psi}{t}{i} \Eld{\alpha+\beta+\psi}{t}{i} = \Eld{\alpha+\beta+\psi}{2t}{i}.
    \end{gather*}
\end{relation}

\begin{proof}
    For the first equation, we use \Cnref{prop:b3-large:est:alpha+beta+psi}, decomposing $i = i_1 + i_2$ for $i_1 \in [2], i_2 \in [1]$, so that \[
    \Eld{\alpha+\beta+\psi}{t}{i} = \Eld{\psi}{-1/2}{i_2} \Eld{\alpha+\beta}{t}{i_1} \Eld{\psi}{1}{i_2} \Eld{\alpha+\beta}{-t}{i_1} \Eld{\psi}{-1/2}{i_2} \] and \[ \Eld{\alpha+\beta+\psi}{-t}{i} = \Eld{\psi}{1/2}{i_2} \Eld{\alpha+\beta}{t}{i_1} \Eld{\psi}{-1}{i_2} \Eld{\alpha+\beta}{-t}{i_1} \Eld{\psi}{1/2}{i_2};
    \]
    by \Cnref{rel:b3-large:sub:inv:alpha+beta} and \Cnref{rel:b3-large:sub:inv:psi}, these two elements multiply to $\Id$. For the second equation, we use \Cnref{prop:b3-large:est:alpha+beta+psi} and \Cnref{rel:b3-large:raw:lift:doub:alpha+beta+psi}.
\end{proof}

\begin{relation}[\RelNameBLg\RelNameAlg\RelNameCommutator{\Rt{\alpha+\beta}}{\Rt{\psi}}]\label{rel:b3-large:rough:comm:alpha+beta:psi}
    \begin{align*}
    \DQuant{i \in [2],j\in[1]}{t,u \in \LiftScalars} \Comm{\Eld{\alpha+\beta}{t}{i}}{\Eld{\psi}{u}{j}} &= \Eld{\alpha+\beta+\psi}{tu}{i+j} \Comm{\Eld{\alpha+\beta+\psi}{-tu}{i+j}}{\Eld{\psi}{u/2}{j}} \\
    &= \Comm{\Eld{\alpha+\beta+\psi}{tu}{i+j}}{\Eld{\psi}{u/2}{j}}^{-1} \Eld{\alpha+\beta+\psi}{tu}{i+j} .
    \end{align*}
\end{relation}

\begin{proof}
    Set \[ x = \Eld{\alpha+\beta}{t}{i}, y = \Eld{\psi}{u/2}{j}, \] so that by \Cnref{rel:b3-large:sub:lin:psi} and \Cnref{rel:b3-large:sub:inv:psi}, \[ x^{-1} = \Eld{\psi}{-u}{j}, y^2 = \Eld{\psi}{u}{j}, y^{-1} = \Eld{\psi}{-u/2}{j}. \] Thus, \[ x \star y = \Eld{\alpha+\beta+\psi}{tu}{i+j}, (x \star y)^{-1} = \Eld{\alpha+\beta+\psi}{-tu}{i+j} \] by \Cnref{prop:b3-large:est:alpha+beta+psi} and \Cnref{rel:b3-large:raw:lift:doub:alpha+beta+psi}. Finally, we use \Cnref{eq:comm:conj} to deduce $\Comm{x}{y^2} = (x \star y) \Comm{(x \star y)^{-1}}{y} = \Comm{x\star y}{y}^{-1} (x \star y)$ which is exactly the required statement.
\end{proof}

Similarly, we have:

\begin{relation}[\RelNameBLg\RelNameAlg\RelNameCommutator{\Rt{\alpha}}{\Rt{\beta+\psi}}]\label{rel:b3-large:rough:comm:alpha:beta+psi}
    \begin{align*}
        \DQuant{i \in [1], j \in [2]}{t,u \in \LiftScalars} \Comm{\Eld{\alpha}{t}{i}}{\Eld{\beta+\psi}{u}{j}} &= \Eld{\alpha+\beta+\psi}{tu}{i+j} \Comm{\Eld{\alpha+\beta+\psi}{-tu}{i+j}}{\Eld{\beta+\psi}{u/2}{j}} \\
        &= \Comm{\Eld{\alpha+\beta+\psi}{tu}{i+j}}{\Eld{\beta+\psi}{u/2}{j}}^{-1} \Eld{\alpha+\beta+\psi}{tu}{i+j} .
    \end{align*}
\end{relation}

\subsubsection{Establishing $\alpha+\beta+2\psi$}

\paragraph{Facts from lifting}

\begin{relation}[\RelNameBLg\RelNameHomLift\RelNameInterchange{\Rt{\alpha+\beta+2\psi}}]\label{rel:b3-large:lift:alpha+beta+2psi}
    \[
    \DQuant{i,j,k \in [1]}{t,u,v \in \LiftScalars} \Comm{\Eld{\alpha+\beta+\psi}{tu}{i+j+k}}{\Eld{\psi}{v}{k}} = \Comm{\Eld{\alpha}{t}{i}}{\Eld{\beta+2\psi}{-2uv}{j+2k}}.
    \]
\end{relation}

\begin{proof}
     Use \Cref{rel:b3-large:raw:lift:alpha+beta+2psi} with $\liftC=(t,\frac{u}v,v)$; the coefficients on the $\Rt{\alpha+\beta+\psi}$, $\Rt{\psi}$, $\Rt{\alpha}$, and $\Rt{\beta+2\psi}$ elements are \[ t \left(\frac{u}{v}\right) v = tu, v, t, -2 \left(\frac{u}v\right) v^2 = -2uv. \] Then, apply \Cnref{prop:b3-large:est:alpha+beta+psi}.
\end{proof}

\begin{relation}[\RelNameBLg\RelNameHomLift\RelNameCommutator{\Rt{\beta+\psi}}{\Comm{\Rt{\alpha}}{\Rt{\beta+2\psi}}}]\label{rel:b3-large:comm:lift:beta+psi:alpha:beta+2psi}
    \[
    \DQuant{i,j,k\in [1]}{t,u,v \in \LiftScalars} \Comm{\Eld{\beta+\psi}{t}{j+k}}{\Comm{\Eld{\alpha}{u}{i}}{\Eld{\beta+2\psi}{v}{j+2k}}} = \Id.
    \]
\end{relation}

\begin{proof}
    If $u = 0$, then $\Eld{\alpha}{u}{i} = \Id$ (\Cnref{rel:b3-large:sub:id:alpha}) and we have nothing to prove. If $v = 0$, then $\Eld{\beta+2\psi}{v}{j+2k} = \Id$ (\Cnref{rel:b3-large:sub:id:beta+2psi}) and we have nothing to prove. Otherwise, we use \Cnref{rel:b3-large:comm:raw:lift:beta+psi:alpha:beta+2psi} with $\liftC=(u,\frac{t^2}v,\frac{v}t)$; the coefficients on the $\Rt{\beta+\psi}$, $\Rt{\alpha}$, $\Rt{\beta+2\psi}$ are \[
    \left(\frac{t^2}v\right) \left(\frac{v}t\right) = t, u, \left(\frac{t^2}v\right) \left(\frac{v}t\right)^2 = v,
    \]
    respectively.
\end{proof}

\begin{relation}[\RelNameBLg\RelNameHomLift\RelNameInvDoub{\Comm{\Rt{\alpha}}{\Rt{\beta+2\psi}}}]\label{rel:b3-large:inv-doub:lift:alpha:beta+2psi}
    \begin{gather*}
    \DQuant{i \in [1], j \in [3]}{t,u \in \LiftScalars} \Comm{\Eld{\alpha}{t}{i}}{\Eld{\beta+2\psi}{u}{j}} = \Comm{\Eld{\alpha}{-t}{i}}{\Eld{\beta+2\psi}{-u}{j}}, \\
    \DQuant{i \in [1], j \in [3]}{t,u \in \LiftScalars} \Comm{\Eld{\alpha}{t}{i}}{\Eld{\beta+2\psi}{u}{j}} \Comm{\Eld{\alpha}{t}{i}}{\Eld{\beta+2\psi}{-u}{j}} = \Id, \\
    \DQuant{i \in [1], j \in [3]}{t,u \in \LiftScalars} \Comm{\Eld{\alpha}{t}{i}}{\Eld{\beta+2\psi}{u}{j}} \Comm{\Eld{\alpha}{t}{i}}{\Eld{\beta+2\psi}{u}{j}} = \Comm{\Eld{\alpha}{t}{i}}{\Eld{\beta+2\psi}{2u}{j}}.
    \end{gather*}
\end{relation}

\begin{proof}
    Decompose $j = j_1 + 2j_2$ for $j_1 \in [1], j_2 \in [1]$ and apply \Cnref{rel:b3-large:inv-doub:raw:lift:alpha:beta+2psi} with $\liftC=(t,u,1)$.
\end{proof}

\begin{relation}[\RelNameBLg\RelNameHomLift\RelNameCommutator{\Rt{\beta+2\psi}}{\Rt{\alpha+\beta+\psi}}]\label{rel:b3-large:comm:lift:beta+2psi:alpha+beta+psi}
        \[
    \DQuant{i,j,k \in [1]}{t,u \in \LiftScalars} \Comm{\Eld{\beta+2\psi}{t}{j+2k}}{\Eld{\alpha+\beta+\psi}{u}{i+j+k}} = \Id.
    \]
\end{relation}

\begin{proof}
    If $t = 0$, then $\Eld{\beta+2\psi}{t}{j+2k} = \Id$ (\Cnref{rel:b3-large:sub:id:beta+2psi}) and we have nothing to prove. Otherwise, apply \Cnref{rel:b3-large:comm:raw:lift:beta+2psi:alpha+beta+psi} with $\liftC=(\frac{u}t,t,1)$ and then \Cnref{prop:b3-large:est:alpha+beta+psi}.
\end{proof}

\input{figures/b3-large/alpha+beta+2psi}

\paragraph{Analysis}

\begin{relation}[\RelNameBLg\RelNameCommute{\Rt{\psi}}{\Comm{\Rt{\alpha}}{\Rt{\beta+2\psi}}}]\label{rel:b3-large:rough:comm:psi:alpha:beta+2psi}
        \[
    \DQuant{i \in [1], j \in [1], k \in [3]}{t,u,v \in \LiftScalars} \Comm{\Eld{\psi}{t}{i}}{\Comm{\Eld{\alpha}{u}{j}}{\Eld{\beta+2\psi}{v}{k}}} = \Id.
    \]
\end{relation}

\begin{proof}
    The sets of $\Rt{\alpha}$ and $\Rt{\beta+2\psi}$ elements are both closed under inversion (\Cnref{rel:b3-large:sub:inv:alpha} and \Cnref{rel:b3-large:sub:inv:beta+2psi})). $\Rt{\psi}$ elements commute with both $\Rt{\alpha}$ elements (\Cnref{rel:b3-large:sub:comm:alpha:psi}) and $\Rt{\beta+2\psi}$ elements (\Cnref{rel:b3-large:sub:comm:psi:beta+2psi}).
\end{proof}

\begin{relation}[\RelNameBLg\RelNameCommute{\Rt{\alpha}}{\Comm{\Rt{\alpha+\beta+\psi}}{\Rt{\psi}}}]\label{rel:b3-large:rough:comm:alpha:alpha+beta+psi:psi}
    \[
    \DQuant{i \in [1], j \in [3], k \in [1]}{t,u,v \in \LiftScalars} \Comm{\Eld{\alpha}{t}{i}}{\Comm{\Eld{\alpha+\beta+\psi}{u}{j}}{\Eld{\psi}{v}{k}}} = \Id.
    \]
\end{relation}

\begin{proof}
    The sets of $\Rt{\psi}$ and $\Rt{\alpha+\beta+\psi}$ elements are both closed under inversion (\Cnref{rel:b3-large:sub:inv:psi} and \Cnref{rel:b3-large:inv-doub:alpha+beta+psi}). $\Rt{\alpha}$ elements commute with $\Rt{\alpha+\beta+\psi}$ elements (\Cnref{rel:b3-large:comm:alpha:alpha+beta+psi}) and with $\Rt{\psi}$ elements (\Cnref{rel:b3-large:sub:comm:alpha:psi}).
\end{proof}

\begin{relation}[\RelNameBLg\RelNameCommute{\Rt{\alpha}}{\Comm{\Rt{\alpha}}{\Rt{\beta+2\psi}}}]\label{rel:b3-large:rough:comm:alpha:alpha:beta+2psi}
        \[
    \DQuant{i \in [1], j \in [1], k \in [3]}{t,u,v \in \LiftScalars} \Comm{\Eld{\alpha}{t}{i}}{\Comm{\Eld{\alpha}{u}{j}}{\Eld{\beta+2\psi}{v}{k}}} = \Id.
    \]
\end{relation}

\begin{proof}
    Using \Cnref{rel:b3-large:lift:alpha+beta+2psi}, we can rewrite \[ \Comm{\Eld{\alpha}{u}{j}}{\Eld{\beta+2\psi}{v}{k}} = \Comm{\Eld{\alpha+\beta+\psi}{u}{i' + j'}}{\Eld{\psi}{v}{k'}} \] where $i' = j$ and $k = j' + 2k'$ for $j',k' \in [1]$. (Note that we are using the fact that $\{j' + 2k' : j',k' \in [1]\}$ covers $[3]$. Graphically, this is equivalent to the statement that the right-hand side is covered by edges in \Cref{fig:b3-large:def:alpha+beta+2psi}.) Now, we can apply the previous proposition (\Cnref{rel:b3-large:rough:comm:alpha:alpha:beta+2psi}).
\end{proof}

\paragraph{``\newedgeA'' edges}

\begin{proposition}[Sufficient conditions for commutator of $\Rt{\alpha+\beta+\psi}$ and $\Rt{\psi}$]\label{rel:b3-large:comm:edge:alpha+beta+psi:psi}
Let $i \in [1],j \in [2], k \in [1]$. Suppose that:
\[
\Quant{t,u,v \in \LiftScalars} \Comm{\Eld{\beta+\psi}{t}{j}}{\Comm{\Eld{\alpha}{u}{i}}{\Eld{\beta+2\psi}{v}{j+k}}} = \Id.
\]
Then:
\[
\Quant{t,u,v \in \LiftScalars} \Comm{\Eld{\alpha+\beta+\psi}{tu}{i+j}}{\Eld{\psi}{v}{k}} = \Comm{\Eld{\alpha}{t}{i}}{\Eld{\beta+2\psi}{-2uv}{j+k}}.
\]
\end{proposition}

\begin{proof}
    We need that by \Cnref{eq:comm:right:str}, \Cnref{rel:b3-large:sub:inv:alpha}, and \Cnref{rel:b3-large:sub:inv:beta+2psi}:
    \begin{equation}\label{subeq:b3-large:comm:edge:alpha+beta+psi:psi:2}
        \Eld{\alpha}{-t}{i} \Eld{\beta+2\psi}{uv}{j+k} = \Eld{\beta+2\psi}{uv}{j+k} \Eld{\alpha}{-t}{i} \Comm{\Eld{\alpha}{t}{i}}{\Eld{\beta+2\psi}{-uv}{j+k}},
    \end{equation}
    and by \Cnref{eq:comm:left:inv} and \Cnref{rel:b3-large:inv-doub:lift:alpha:beta+2psi}:
    \begin{multline}\label{subeq:b3-large:comm:edge:alpha+beta+psi:psi:3}
        \Eld{\beta+2\psi}{uv}{j+k} \Eld{\alpha}{t}{i} = \Comm{\Eld{\alpha}{t}{i}}{\Eld{\beta+2\psi}{uv}{j+k}}^{-1} \Eld{\alpha}{t}{i} \Eld{\beta+2\psi}{uv}{j+k} = \Comm{\Eld{\alpha}{t}{i}}{\Eld{\beta+2\psi}{-uv}{j+k}} \Eld{\alpha}{t}{i} \Eld{\beta+2\psi}{uv}{j+k}.
    \end{multline}

    Now, we write a product of $\Rt{\alpha+\beta+\psi}$ and $\Rt{\psi}$ elements:
    \begin{align*}
    & \asite{\Eld{\alpha+\beta+\psi}{tu}{i+j}} \Eld{\psi}{v}{k} \\
    \intertext{Expanding the $\Rt{\alpha+\beta+\psi}$ element into a product of $\Rt{\alpha}$ and $\Rt{\beta+\psi}$ elements (\Cnref{prop:b3-large:est:alpha+beta+psi}):}
    &= \Eld{\beta+\psi}{-u/2}{j} \Eld{\alpha}{t}{i} \Eld{\beta+\psi}{u}{j} \Eld{\alpha}{-t}{i} \Eld{\beta+\psi}{-u/2}{j} \asite{\Eld{\psi}{v}{k}}. \\
    \intertext{Commuting the $\Rt{\psi}$ element on the right fully to the left creates no commutators with the $\Rt{\alpha}$ elements (\Cnref{rel:b3-large:sub:comm:alpha:psi}) and $\Rt{\beta+2\psi}$ commutators with the $\Rt{\beta+\psi}$ elements (\Cnref{rel:b3-large:order:beta+psi:psi}):}
    &= \Eld{\psi}{v}{k} \Eld{\beta+\psi}{-u/2}{j} \asite{\Eld{\beta+2\psi}{uv}{j+k} \Eld{\alpha}{t}{i}} \Eld{\beta+2\psi}{-2uv}{j+k} \Eld{\beta+\psi}{u}{j} \asite{\Eld{\alpha}{-t}{i} \Eld{\beta+2\psi}{uv}{j+k}} \Eld{\beta+\psi}{-u/2}{j}. \\
    \intertext{Moving the $\Rt{\beta+2\psi}$ elements closer together by introducing generic commutators (\Cref{subeq:b3-large:comm:edge:alpha+beta+psi:psi:2,subeq:b3-large:comm:edge:alpha+beta+psi:psi:3}):}
    &= \Eld{\psi}{v}{k} \Eld{\beta+\psi}{-u/2}{j} \Comm{\Eld{\beta+2\psi}{uv}{j+k}}{\Eld{\alpha}{t}{i}} \Eld{\alpha}{t}{i} \asite{\Eld{\beta+2\psi}{uv}{j+k} \Eld{\beta+2\psi}{-2uv}{j+k}} \\
    & \hspace{1in} \cdot \asite{\Eld{\beta+\psi}{u}{j} \Eld{\beta+2\psi}{uv}{j+k}} \Eld{\alpha}{-t}{i} \Comm{\Eld{\alpha}{t}{i}}{\Eld{\beta+2\psi}{-uv}{j+k}} \Eld{\beta+\psi}{-u/2}{j}. \\
    \intertext{Now we commute the $\Rt{\beta+2\psi}$ elements across the central $\Rt{\beta+\psi}$ element (\Cnref{rel:b3-large:sub:comm:beta+psi:beta+2psi}) and therefore cancel them (\Cnref{rel:b3-large:sub:inv:beta+2psi}):}
    &= \Eld{\psi}{v}{k} \Eld{\beta+\psi}{-u/2}{j} \asite{\Comm{\Eld{\alpha}{t}{i}}{\Eld{\beta+2\psi}{-uv}{j+k}}} \Eld{\alpha}{t}{i} \Eld{\beta+\psi}{u}{j} \Eld{\alpha}{-t}{i} \asite{\Comm{\Eld{\alpha}{t}{i}}{\Eld{\beta+2\psi}{-uv}{j+k}}} \Eld{\beta+\psi}{-u/2}{j}. \\
    \intertext{Finally, we commute the $\Comm{\Rt{\alpha}}{\Rt{\beta+2\psi}}$ elements to the left using the assumption for $\Rt{\beta+\psi}$ elements, \Cnref{rel:b3-large:rough:comm:alpha:alpha:beta+2psi} for $\Rt{\alpha}$ elements, and \Cnref{rel:b3-large:rough:comm:psi:alpha:beta+2psi} for $\Rt{\psi}$ elements:}
    &= \asite{\Comm{\Eld{\alpha}{t}{i}}{\Eld{\beta+2\psi}{-uv}{j+k}} \Comm{\Eld{\alpha}{t}{i}}{\Eld{\beta+2\psi}{-uv}{j+k}}} \Eld{\psi}{v}{k} \asite{\Eld{\beta+\psi}{-u/2}{j} \Eld{\alpha}{t}{i} \Eld{\beta+\psi}{u}{j} \Eld{\alpha}{-t}{i}  \Eld{\beta+\psi}{-u/2}{j}} \\
    \intertext{Reducing the doubled $\Comm{\Rt{\alpha}}{\Rt{\beta+2\psi}}$ and $\Rt{\alpha+\beta+\psi}$ (\Cnref{rel:b3-large:inv-doub:lift:alpha:beta+2psi} and \Cnref{prop:b3-large:est:alpha+beta+psi}):}
    &= \Comm{\Eld{\alpha}{t}{i}}{\Eld{\beta+2\psi}{-2uv}{j+k}} \Eld{\psi}{v}{k} \Eld{\alpha+\beta+\psi}{tu}{i+j},
    \end{align*}
    as desired.
\end{proof}

\begin{relation}[\RelNameBLg\RelNamePart\newedgeA:\RelNameInterchange{\Rt{\alpha+\beta+2\psi}}]\label{cor:b3-large:comm:alpha+beta+psi:psi:A}
    \begin{align*}
        \Quant{t,u,v \in \LiftScalars} \Comm{\Eld{\alpha+\beta+\psi}{tu}{0}}{\Eld{\psi}{v}{1}} &= \Comm{\Eld{\alpha}{t}{0}}{\Eld{\beta+2\psi}{-2uv}{1}}.
    \end{align*}
\end{relation}

\begin{proof}
        We apply the previous proposition (\Cnref{rel:b3-large:comm:edge:alpha+beta+psi:psi}) with  $i=0, j = 0, k=1$. This gives the desideratum if we can show that:
        \[
        \Quant{t,u,v \in \LiftScalars} \Comm{\Eld{\beta+\psi}{t}{0}}{\Comm{\Eld{\alpha}{u}{0}}{\Eld{\beta+2\psi}{v}{1}}} = \Id.
        \]
        Since
        \[
        \Quant{u,v \in \LiftScalars} \Comm{\Eld{\alpha}{u}{0}}{\Eld{\beta+2\psi}{v}{1}} = \Comm{\Eld{\alpha}{u}{1}}{\Eld{\beta+2\psi}{v}{0}}
        \]
        by \Cnref{rel:b3-large:lift:alpha+beta+2psi} (cf. \Cref{fig:b3-large:def:alpha+beta+2psi}), this follows from homogeneous lifting (\Cnref{rel:b3-large:comm:lift:beta+psi:alpha:beta+2psi} with $i=1,j=0,k=0$).
\end{proof}

\paragraph{The ``\newedgeB'' edge}

To get the ``\newedgeB'' edge, we want to use \Cnref{rel:b3-large:comm:edge:alpha+beta+psi:psi} with $i=0, j = 2, k = 0$. To do this, we need the condition:
\[
\Quant{t,u,v \in \LiftScalars} \Comm{\Eld{\beta+\psi}{t}2}{\Comm{\Eld{\alpha}{u}{0}}{\Eld{\beta+2\psi}{v}{2}}} = \Id.
\]
This doesn't follow directly from lifting (i.e., \Cnref{rel:b3-large:comm:lift:beta+psi:alpha:beta+2psi}). Instead, we use:

\begin{proposition}[Sufficient conditions for commutator of $\Rt{\beta+\psi}$ and $\Comm{\Rt{\alpha}}{\Rt{\beta+2\psi}}$]\label{rel:b3-large:comm:case:beta+psi:alpha:beta+2psi}
    Let $i\in[2],j\in[1],k\in[3]$. Suppose that:
    \[
    \Quant{t,u \in \LiftScalars} \Comm{\Eld{\alpha+\beta+\psi}{t}{i+j}}{\Eld{\beta+2\psi}{u}{k}} = \Id.
    \]
    Then:
    \[
    \Quant{t,u,v \in \LiftScalars} \Comm{\Eld{\beta+\psi}{t}{i}}{\Comm{\Eld{\alpha}{u}{j}}{\Eld{\beta+2\psi}{v}{k}}} = \Id.
    \]
\end{proposition}

\begin{proof}
    From \Cnref{eq:comm:mid:inv} and \Cnref{eq:comm:left:inv}, we have:
    \begin{equation}\label{subeq:b3-large:comm:case:beta+psi:alpha:beta+2psi:1}
        \Eld{\beta+\psi}{t}{i} \Eld{\alpha}{\vu}{j} = \Eld{\alpha}{\vu}{j} \asite{\Comm{\Eld{\alpha}{-\vu}{j}}{\Eld{\beta+\psi}{t}{i}}} \Eld{\beta+\psi}{t}{i} = \asite{\Comm{\Eld{\alpha}{\vu}{j}}{\Eld{\beta+\psi}{t}{i}}^{-1}} \Eld{\alpha}{\vu}{j} \Eld{\beta+\psi}{t}{i}.
    \end{equation}

    Now, we write a product of $\Rt{\beta+\psi}$ and $\Comm{\Rt{\alpha}}{\Rt{\beta+2\psi}}$ elements:
    \begin{align*}
        &\Eld{\beta+\psi}{t}{i} \asite{\Comm{\Eld{\alpha}{u}{j}}{\Eld{\beta+2\psi}{v}{k}}} \\
        \intertext{Expanding the commutator with \Cnref{rel:b3-large:sub:inv:alpha} and \Cnref{rel:b3-large:sub:inv:beta+2psi}:}
        &= \asite{\Eld{\beta+\psi}{t}{i}} \Eld{\alpha}{u}{j} \Eld{\beta+2\psi}{v}{k} \Eld{\alpha}{-u}{j} \Eld{\beta+2\psi}{-v}{k}. \\
        \intertext{Now moving the $\Rt{\beta+\psi}$ fully from left to right creates generic commutators with the $\Rt{\alpha}$ elements (\Cref{subeq:b3-large:comm:case:beta+psi:alpha:beta+2psi:1}) and no commutators with the $\Rt{\beta+2\psi}$ elements (\Cnref{rel:b3-large:sub:comm:beta+psi:beta+2psi}):}
        &= \Eld{\alpha}{u}{j} \asite{\Comm{\Eld{\alpha}{-u}{j}}{\Eld{\beta+\psi}{t}{i}} \Eld{\beta+2\psi}{v}{k} \Comm{\Eld{\alpha}{-u}{j}}{\Eld{\beta+\psi}{t}{i}}^{-1}} \Eld{\alpha}{-u}{j} \Eld{\beta+2\psi}{-v}{k} \Eld{\beta+\psi}{t}{i}. \\
    \end{align*}
    Thus, it suffices to show
    \[
    \Comm{\Eld{\alpha}{-u}{j}}{\Eld{\beta+\psi}{t}{i}} \Eld{\beta+2\psi}{v}{k} \Comm{\Eld{\alpha}{-u}{j}}{\Eld{\beta+\psi}{t}{i}}^{-1} = \Eld{\beta+2\psi}{v}{k}.
    \]
    This is rearranges to the statement that $\Comm{\Eld{\alpha}{-u}{j}}{\Eld{\beta+\psi}{t}{i}}$ and $\Eld{\beta+2\psi}{v}{k}$ commute. By \Cnref{rel:b3-large:rough:comm:alpha:beta+psi}, \[
    \Comm{\Eld{\alpha}{-u}{j}}{\Eld{\beta+\psi}{t}{i}} = \Eld{\alpha+\beta+\psi}{-tu}{i+j} \Comm{\Eld{\alpha+\beta+\psi}{tu}{i+j}}{\Eld{\beta+\psi}{t/2}{j}}.
    \]
    Indeed, the $\Rt{\beta+2\psi}$ element commutes with the $\Rt{\alpha+\beta+\psi}$ elements by the assumption and the $\Rt{\beta+\psi}$ elements by \Cnref{rel:b3-large:sub:comm:beta+psi:beta+2psi}.
\end{proof}

\begin{relation}[\RelNameBLg\RelNamePart\RelNameCommute{\Rt{\beta+\psi}}{\Comm{\Rt{\alpha}}{\Rt{\beta+2\psi}}}]\label{rel:b3-large:comm:case:beta+psi:alpha:beta+2psi:202}
    \[
    \Quant{t,u,v \in \LiftScalars} \Comm{\Eld{\beta+\psi}{t}2}{\Comm{\Eld{\alpha}{u}{0}}{\Eld{\beta+2\psi}{v}{2}}} = \Id.
    \]
\end{relation}

\begin{proof}
    We apply the previous proposition (\Cnref{rel:b3-large:comm:edge:alpha+beta+psi:psi}) with $i=2,j=0,k=2$. For this, we need:
    \[
    \Comm{\Eld{\alpha+\beta+\psi}{t}{2}}{\Eld{\beta+2\psi}{u}{2}} = \Id
    \]
    which follows from homogeneous lifting (\Cnref{rel:b3-large:comm:lift:beta+2psi:alpha+beta+psi}), setting $i=1,j=0,k=1$.
\end{proof}

\begin{relation}[\RelNameBLg\RelNamePart\newedgeB:\RelNameInterchange{\Rt{\alpha+\beta+2\psi}}]\label{cor:b3-large:comm:alpha+beta+psi:psi:B}
    \[
    \Quant{t,u,v \in \LiftScalars} \Comm{\Eld{\alpha+\beta+\psi}{tu}{2}}{\Eld{\psi}{v}{0}} = \Comm{\Eld{\alpha}{t}{0}}{\Eld{\beta+2\psi}{-2uv}{2}}.
    \]
\end{relation}

\begin{proof}
    As stated above, we apply \Cnref{rel:b3-large:comm:case:beta+psi:alpha:beta+2psi:202} with $i=0,j=2,k=0$; the necessary condition is provided by \Cnref{rel:b3-large:comm:case:beta+psi:alpha:beta+2psi:202}.
\end{proof}

\paragraph{Conclusions}

\begin{bigproposition}[Establishing $\Rt{\alpha+\beta+2\psi}$]\label{prop:b3-large:est:alpha+beta+2psi}
    There exist elements $\Eld{\alpha+\beta+2\psi}{t}{i}$ for $t \in \LiftScalars$ and $i \in [4]$ such that the following hold:
    \begin{gather*}
        \DQuant{i \in [1], j \in [3]}{t,u \in \LiftScalars} \Comm{\Eld{\alpha}{t}{i}}{\Eld{\beta+2\psi}{u}{j}} = \Eld{\alpha+\beta+2\psi}{tu}{i+j}, \\
        \DQuant{i \in [1], j \in [3]}{t,u \in \LiftScalars} \Comm{\Eld{\alpha+\beta+\psi}{u}{i}}{\Eld{\psi}{t}{j}} = \Eld{\alpha+\beta+2\psi}{-2tu}{i+j}.
    \end{gather*}
\end{bigproposition}

\begin{proof}
    Uses \Cref{prop:prelim:equal-conn} and \Cnref{rel:b3-large:lift:alpha+beta+2psi}, \Cnref{cor:b3-large:comm:alpha+beta+psi:psi:A}, and \Cnref{cor:b3-large:comm:alpha+beta+psi:psi:B} (i.e., the blocks in \Cref{fig:b3-large:def:alpha+beta+2psi} are connected).
\end{proof}

\begin{relation}[\RelNameBLg\RelNameCommute{\Rt{\alpha}}{\Rt{\alpha+\beta+2\psi}}]\label{rel:b3-large:comm:alpha:alpha+beta+2psi}
        \[
    \DQuant{i \in [1], j \in [4]}{t,u \in \LiftScalars} \Comm{\Eld{\alpha}{t}{i}}{\Eld{\alpha+\beta+2\psi}{u}{j}} = \Id.
    \]
\end{relation}

\begin{proof}
Follows from \Cnref{rel:b3-large:rough:comm:alpha:alpha+beta+psi:psi} and \Cnref{prop:b3-large:est:alpha+beta+2psi}.
\end{proof}

\begin{relation}[\RelNameBLg\RelNameCommute{\Rt{\psi}}{\Rt{\alpha+\beta+2\psi}}]\label{rel:b3-large:comm:psi:alpha+beta+2psi}
        \[
    \DQuant{i \in [1], j \in [4]}{t,u \in \LiftScalars} \Comm{\Eld{\psi}{t}{i}}{\Eld{\alpha+\beta+2\psi}{u}{j}} = \Id.
    \]
\end{relation}

\begin{proof}
    Follows from \Cnref{rel:b3-large:rough:comm:psi:alpha:beta+2psi} and \Cnref{prop:b3-large:est:alpha+beta+2psi}.
\end{proof}

\begin{relation}[\RelNameBLg\RelNameInvDoub{\Rt{\alpha+\beta+2\psi}}]\label{rel:b3-large:inv-doub:alpha+beta+2psi}
    \begin{gather*}
    \DQuant{i \in [4]}{t \in \LiftScalars} \Eld{\alpha+\beta+2\psi}{t}{i} \Eld{\alpha+\beta+2\psi}{-t}{i} = \Id, \\
    \DQuant{i \in [4]}{t \in \LiftScalars} \Eld{\alpha+\beta+2\psi}{t}{i} \Eld{\alpha+\beta+2\psi}{t}{i} = \Eld{\alpha+\beta+2\psi}{2t}{i}.
    \end{gather*}
\end{relation}

\begin{proof}
    Follows from \Cnref{rel:b3-large:inv-doub:lift:alpha:beta+2psi} and \Cnref{prop:b3-large:est:alpha+beta+2psi}.
\end{proof}

\begin{relation}[\RelNameBLg\RelNameCommutator{\Rt{\alpha+\beta}}{\Rt{\psi}}]\label{rel:b3-large:comm:alpha+beta:psi}
        \[
    \DQuant{i \in [2], j \in [1]}{t,u \in \LiftScalars} \Comm{\Eld{\alpha+\beta}{t}{i}}{\Eld{\psi}{u}{j}} = \Eld{\alpha+\beta+\psi}{tu}{i+j} \Eld{\alpha+\beta+2\psi}{tu^2}{i+2j} = \Eld{\alpha+\beta+2\psi}{tu^2}{i+2j} \Eld{\alpha+\beta+\psi}{tu}{i+j}.
    \]
\end{relation}

\begin{proof}
    Follows from \Cnref{rel:b3-large:rough:comm:alpha+beta:psi} and \Cnref{prop:b3-large:est:alpha+beta+2psi}.
\end{proof}

\paragraph{Additional identities}

As a corollary of \Cnref{rel:b3-large:inv-doub:alpha+beta+2psi}, \Cnref{prop:b3-large:est:alpha+beta+2psi}, \Cnref{rel:b3-large:sub:inv:alpha}, \Cnref{rel:b3-large:sub:inv:psi}, \Cnref{rel:b3-large:sub:inv:beta+2psi}, and \Cnref{rel:b3-large:inv-doub:alpha+beta+psi}, we get the expansion equations:
\begin{relation}[\RelNameBLg\RelNameExprn{\Rt{\alpha+\beta+2\psi}}{\Rt{\alpha}}{\Rt{\beta+2\psi}}]\label{rel:b3-large:expr:alpha+beta+2psi:alpha:beta+2psi}
    \begin{multline*}
        \DQuant{i \in [1], j \in [3]}{t,u \in \LiftScalars} \Eld{\alpha+\beta+2\psi}{tu}{i+j} = \Eld{\alpha}{t}{i} \Eld{\beta+2\psi}{u}{j} \Eld{\alpha}{-t}{i} \Eld{\beta+2\psi}{-u}{j} \\
        = \Eld{\beta+2\psi}{-u}{j} \Eld{\alpha}{t}{i} \Eld{\beta+2\psi}{u}{j} \Eld{\alpha}{-t}{i}.
    \end{multline*}
\end{relation}
and
\begin{relation}[\RelNameBLg\RelNameExprn{\Rt{\alpha+\beta+2\psi}}{\Rt{\psi}}{\Rt{\alpha+\beta+\psi}}]\label{rel:b3-large:expr:alpha+beta+2psi:alpha+beta+psi:psi}
    \begin{multline*}
        \DQuant{i \in [1], j \in [3]}{t,u \in \LiftScalars} \Eld{\alpha+\beta+2\psi}{-2tu}{i+j} = \Eld{\alpha+\beta+\psi}{u}{i} \Eld{\psi}{t}{j}  \Eld{\alpha+\beta+\psi}{-u}{i} \Eld{\psi}{-t}{j} \\
        = \Eld{\psi}{-t}{j}  \Eld{\alpha+\beta+\psi}{u}{i} \Eld{\psi}{t}{j}  \Eld{\alpha+\beta+\psi}{-u}{i}.
    \end{multline*}
\end{relation}

From \Cnref{prop:b3-large:est:alpha+beta+2psi}, \Cnref{eq:comm:left:str}, \Cnref{eq:comm:mid:str}, \Cnref{eq:comm:right:str}, \Cnref{rel:b3-large:sub:inv:alpha}, \Cnref{rel:b3-large:inv-doub:alpha+beta+2psi}, \Cnref{rel:b3-large:sub:inv:beta+2psi} we have:
\begin{relation}[\RelNameBLg\RelNameOrder{\Rt{\alpha}}{\Rt{\beta+2\psi}}]\label{rel:b3-large:order:alpha:beta+2psi}
\begin{multline*}
    \Quant{t,u \in \LiftScalars} \Eld{\alpha}{t}{i} \Eld{\beta+2\psi}{u}{j} = \asite{\Eld{\alpha+\beta+2\psi}{tu}{i+j}} \Eld{\beta+2\psi}{u}{j} \Eld{\alpha}{t}{i} \\ = \Eld{\beta+2\psi}{u}{j} \asite{\Eld{\alpha+\beta+2\psi}{tu}{i+j}} \Eld{\alpha}{t}{i} = \Eld{\beta+2\psi}{u}{j} \Eld{\alpha}{t}{i} \asite{\Eld{\alpha+\beta+2\psi}{tu}{i+j}}
\end{multline*}
\end{relation}

From \Cnref{prop:b3-large:est:alpha+beta+2psi}, \Cnref{rel:b3-large:expr:alpha+beta+2psi:alpha+beta+psi:psi}, \Cnref{eq:comm:mid:inv}, \Cnref{eq:comm:right:inv}, \Cnref{rel:b3-large:sub:inv:psi}, \Cnref{rel:b3-large:inv-doub:alpha+beta+psi}, and \Cnref{rel:b3-large:inv-doub:alpha+beta+2psi}:
\begin{relation}[\RelNameBLg\RelNameOrder{\Rt{\psi}}{\Rt{\alpha+\beta+\psi}}]\label{rel:b3-large:order:psi:alpha+beta+psi}
    \[ \Quant{t,u \in \LiftScalars} \Eld{\psi}{t}{i} \Eld{\alpha+\beta+\psi}{u}{j} = \Eld{\alpha+\beta+\psi}{u}{j} \asite{\Eld{\alpha+\beta+2\psi}{2tu}{i+j}} \Eld{\psi}{t}{i} = \Eld{\alpha+\beta+\psi}{u}{j}  \Eld{\psi}{t}{i} \asite{\Eld{\alpha+\beta+2\psi}{2tu}{i+j}}. \]
\end{relation}
From \Cnref{prop:b3-large:est:alpha+beta+2psi}, \Cnref{eq:comm:left:str}, \Cnref{eq:comm:mid:str}, \Cnref{rel:b3-large:sub:inv:psi}, and \Cnref{rel:b3-large:inv-doub:alpha+beta+psi}:
\begin{relation}[\RelNameBLg\RelNameOrder{\Rt{\alpha+\beta+\psi}}{\Rt{\psi}}]\label{rel:b3-large:order:alpha+beta+psi:psi}
    \[ \Quant{t,u \in \LiftScalars} \Eld{\alpha+\beta+\psi}{u}{j} \Eld{\psi}{t}{i} = \asite{\Eld{\alpha+\beta+2\psi}{-2tu}{i+j}} \Eld{\psi}{t}{i} \Eld{\alpha+\beta+\psi}{u}{j} = \Eld{\psi}{t}{i} \asite{\Eld{\alpha+\beta+2\psi}{-2tu}{i+j}} \Eld{\alpha+\beta+\psi}{u}{j}. \]
\end{relation}

From \Cnref{eq:comm:left:str}, \Cnref{eq:comm:right:str}, \Cnref{eq:comm:mid:str}, \Cnref{rel:b3-large:comm:alpha+beta:psi}, \Cnref{rel:b3-large:sub:inv:alpha+beta}, \Cnref{rel:b3-large:sub:inv:psi}, \Cnref{rel:b3-large:inv-doub:alpha+beta+psi}, and \Cnref{rel:b3-large:inv-doub:alpha+beta+2psi}, we have:
\begin{relation}[\RelNameBLg\RelNameOrder{\Rt{\alpha+\beta}}{\Rt{\psi}}]\label{rel:b3-large:order:alpha+beta:psi}
\begin{multline*}
    \Eld{\alpha+\beta}{t}{i} \Eld{\psi}{u}{j} = \Eld{\psi}{u}{j} \Eld{\alpha+\beta}{t}{i} \asite{\Eld{\alpha+\beta+\psi}{tu}{i+j} \Eld{\alpha+\beta+2\psi}{-tu^2}{i+2j}} = \Eld{\psi}{u}{j} \asite{\Eld{\alpha+\beta+\psi}{tu}{i+j} \Eld{\alpha+\beta+2\psi}{-tu^2}{i+2j}} \Eld{\alpha+\beta}{t}{i} \\
    = \Eld{\psi}{u}{j} \asite{\Eld{\alpha+\beta+2\psi}{-tu^2}{i+2j} \Eld{\alpha+\beta+\psi}{tu}{i+j}} \Eld{\alpha+\beta}{t}{i} = \asite{\Eld{\alpha+\beta+2\psi}{tu^2}{i+2j} \Eld{\alpha+\beta+\psi}{tu}{i+j}} \Eld{\psi}{u}{j}  \Eld{\alpha+\beta}{t}{i}.
\end{multline*}
\end{relation}
and similarly from \Cnref{eq:comm:mid:inv}, \Cnref{rel:b3-large:comm:alpha+beta:psi}, and \Cnref{rel:b3-large:sub:inv:alpha+beta}:
\begin{relation}[\RelNameBLg\RelNameOrder{\Rt{\psi}}{\Rt{\alpha+\beta}}]\label{rel:b3-large:order:psi:alpha+beta}
\[ \Eld{\psi}{u}{j} \Eld{\alpha+\beta}{t}{i} = \Eld{\alpha+\beta}{t}{i} \asite{\Eld{\alpha+\beta+2\psi}{-tu^2}{i+2j} \Eld{\alpha+\beta+\psi}{-tu}{i+j}} \Eld{\psi}{u}{j}. \]
\end{relation}

\subsubsection{Establishing $\alpha+2\beta+2\psi$}

\paragraph{Lifting}

\begin{relation}[\RelNameBLg\RelNameHomLift\RelNameInterchange{\Rt{\alpha+2\beta+2\psi}}]\label{rel:b3-large:lift:alpha+2beta+2psi}
    \[
    \DQuant{i,j,k \in [1]}{t,u,v \in \LiftScalars} \Comm{\Eld{\alpha+\beta}{t}{i+j}}{\Eld{\beta+2\psi}{2uv}{j+2k}} = \Comm{\Eld{\alpha+\beta+\psi}{tu}{i+j+k}}{\Eld{\beta+\psi}{v}{j+k}}
    \]
    and
    \[
    \DQuant{i,j,k \in [1]}{t,u,v \in \LiftScalars} \Comm{\Eld{\alpha+\beta+\psi}{t}{i+j+k}}{\Eld{\beta+\psi}{uv}{j+k}} = \Comm{\Eld{\alpha+\beta+2\psi}{2tu}{i+j+2k}}{\Eld{\beta}{v}{j}}.
    \]
\end{relation}

\begin{proof}
    In both equations, if either $u$ or $v$ is zero, one term in the commutators on both sides is $\Id$ and we have nothing to prove. So, we proceed assuming they are both nonzero.
    
    For the first equation, we use the first equation of \Cnref{rel:b3-large:raw:lift:alpha+2beta+2psi} at $\liftC=(\frac{tu}v,\frac{v}u,u)$ and then apply \Cnref{prop:b3-large:est:alpha+beta+psi}; the coefficients on the $\Rt{\alpha+\beta}$, $\Rt{\beta+2\psi}$, $\Rt{\alpha+\beta+\psi}$, and $\Rt{\beta+\psi}$ elements are \[ (\frac{tu}v)(\frac{v}u) = t, 2(\frac{v}u)(u^2) = 2uv, (\frac{tu}v)(\frac{v}u)(u) = tu, (\frac{v}u)(u) = v, \] respectively.

    For the second equation, we use the second equation of \Cnref{rel:b3-large:raw:lift:alpha+2beta+2psi} at $\liftC=(\frac{t}{uv},v,u)$ and then apply \Cnref{rel:b3-large:expr:alpha+beta+2psi:alpha:beta+2psi}. Then the coefficients on the $\Rt{\alpha+\beta+\psi}$, $\Rt{\beta+\psi}$, $\Rt{\alpha+\beta+2\psi}$, and $\Rt{\beta}$ elements are \[ (\frac{t}{uv})(v)(u) = t, uv, 2(\frac{t}{uv})(v)(u^2) = 2tu, v, \] respectively.
\end{proof}

\input{figures/b3-large/alpha+2beta+2psi}

\begin{relation}[\RelNameBLg\RelNameHomLift\RelNameCommutator{\Rt{\psi}}{\Comm{\Rt{\alpha+\beta}}{\Rt{\beta+2\psi}}}]\label{rel:b3-large:comm:lift:psi:alpha+beta:beta+2psi}
\[
    \DQuant{i,j,k \in [1]}{t,u,v \in \LiftScalars} \Comm{\Eld{\psi}{t}{i}}{\Comm{\Eld{\alpha+\beta}{u}{i+j}}{\Eld{\beta+2\psi}{v}{j+2k}}} = \Id.
\]
\end{relation}

\begin{proof}
    If $t = 0$, then $\Eld{\psi}{t}{i} = \Id$ (\Cnref{rel:b3-large:sub:id:psi}), and we have nothing to prove. Otherwise, use \Cnref{rel:b3-large:comm:raw:lift:psi:alpha+beta:beta+2psi} with $\liftC=(\frac{ut^2}{v},\frac{v}{t^2},t)$.
\end{proof}

\begin{relation}[\RelNameBLg\RelNameHomLift\RelNameCommutator{\Rt{\alpha+\beta}}{\Comm{\Rt{\alpha+\beta}}{\Rt{\beta+2\psi}}}]\label{rel:b3-large:comm:lift:alpha+beta:alpha+beta:beta+2psi}
\[
\DQuant{i,j,k \in [1]}{t,u \in \LiftScalars, s \in \{\pm 1\}} \Comm{\Eld{\alpha+\beta}{t}{i+j}}{\Comm{\Eld{\alpha+\beta}{s t}{i+j}}{\Eld{\beta+2\psi}{u}{j+2k}}} = \Id.
\]
\end{relation}

\begin{proof}
    If $u = 0$, then $\Eld{\beta+2\psi}{u}{j+2k} = \Id$ (\Cnref{rel:b3-large:sub:id:beta+2psi}) and we have nothing to prove. Otherwise, use \Cnref{rel:b3-large:comm:raw:lift:alpha+beta:alpha+beta:beta+2psi} with $\liftC=(\frac{t}u,u,1)$.
\end{proof}

\begin{relation}[\RelNameBLg\RelNameHomLift\RelNameInvDoub{\Comm{\Rt{\alpha+\beta}}{\Rt{\beta+2\psi}}}]\label{rel:b3-large:inv-doub:lift:alpha+beta:beta+2psi}
    \begin{gather*}
        \DQuant{i,j,k \in [1]}{t,u \in \LiftScalars} \Comm{\Eld{\alpha+\beta}{t}{i+j}}{\Eld{\beta+2\psi}{u}{j+2k}}  = \Comm{\Eld{\alpha+\beta}{-t}{i+j}}{\Eld{\beta+2\psi}{-u}{j+2k}} , \\
        \DQuant{i,j,k \in [1]}{t,u \in \LiftScalars} \Comm{\Eld{\alpha+\beta}{t}{i+j}}{\Eld{\beta+2\psi}{u}{j+2k}} \Comm{\Eld{\alpha+\beta}{-t}{i+j}}{\Eld{\beta+2\psi}{u}{j+2k}} = \Id, \\
        \DQuant{i,j,k \in [1]}{t,u \in \LiftScalars} \Comm{\Eld{\alpha+\beta}{t}{i+j}}{\Eld{\beta+2\psi}{u}{j+2k}} \Comm{\Eld{\alpha+\beta}{t}{i+j}}{\Eld{\beta+2\psi}{u}{j+2k}} = \Comm{\Eld{\alpha+\beta}{2t}{i+j}}{\Eld{\beta+2\psi}{u}{j+2k}}.
    \end{gather*}
\end{relation}

\begin{proof}
    If $u = 0$, then $\Eld{\beta+2\psi}{u}{j+2k} = \Id$ (\Cnref{rel:b3-large:sub:id:beta+2psi}) and we have nothing to prove. Otherwise, use \Cnref{rel:b3-large:inv-doub:raw:lift:alpha+beta:beta+2psi} with $\liftC=(\frac{t}u,u,1)$.
\end{proof}

\begin{relation}[\RelNameBLg\RelNameHomLift\RelNameInvDoub{\Comm{\Rt{\beta}}{\Rt{\alpha+\beta+2\psi}}}]\label{rel:b3-large:inv-doub:lift:beta:alpha+beta+2psi}
    \begin{gather*}
        \DQuant{i,j,k \in [1]}{t,u \in \LiftScalars} \Comm{\Eld{\beta}{t}{i}}{\Eld{\alpha+\beta+2\psi}{u}{i+j+2k}} = \Comm{\Eld{\beta}{-t}{i}}{\Eld{\alpha+\beta+2\psi}{-u}{i+j+2k}}, \\
        \DQuant{i,j,k \in [1]}{t,u \in \LiftScalars} \Comm{\Eld{\beta}{t}{i}}{\Eld{\alpha+\beta+2\psi}{u}{i+j+2k}}\Comm{\Eld{\beta}{-t}{i}}{\Eld{\alpha+\beta+2\psi}{u}{i+j+2k}} = \Id, \\
        \DQuant{i,j,k \in [1]}{t,u \in \LiftScalars} \Comm{\Eld{\beta}{t}{i}}{\Eld{\alpha+\beta+2\psi}{u}{i+j+2k}} \Comm{\Eld{\beta}{t}{i}}{\Eld{\alpha+\beta+2\psi}{u}{i+j+2k}}= \Comm{\Eld{\beta}{2t}{i}}{\Eld{\alpha+\beta+2\psi}{u}{i+j+2k}}.
    \end{gather*}
\end{relation}

\begin{proof}
    If $t = 0$, then $\Eld{\beta}{t}{i} = \Id$ (\Cnref{rel:b3-large:sub:id:beta}) and we have nothing to prove. Otherwise, use \Cnref{rel:b3-large:inv-doub:raw:lift:beta:alpha+beta+2psi} at $\liftC=(\frac{u}t,t,1)$ and \Cnref{rel:b3-large:expr:alpha+beta+2psi:alpha:beta+2psi}.
\end{proof}

\begin{relation}[\RelNameBLg\RelNameHomLift\RelNameCommutator{\Rt{\beta+\psi}}{\Rt{\alpha+\beta+2\psi}}]\label{rel:b3-large:comm:lift:beta+psi:alpha+beta+2psi}
        \[
    \DQuant{i,j,k \in [1]}{t,u \in \LiftScalars} \Comm{\Eld{\beta+\psi}{t}{j+2k}}{\Eld{\alpha+\beta+2\psi}{u}{i+j+2k}} = \Id.
    \]
\end{relation}

\begin{proof}
    If $t = 0$, then $\Eld{\beta+\psi}{t}{j+2k} = \Id$ (\Cnref{rel:b3-large:sub:id:beta+psi}) and we have nothing to prove. Otherwise, use \Cnref{rel:b3-large:comm:raw:lift:beta+psi:alpha+beta+2psi} at $\liftC=(\frac{u}t,t,1)$ and \Cnref{rel:b3-large:expr:alpha+beta+2psi:alpha:beta+2psi}.
\end{proof}

\begin{relation}[\RelNameBLg\RelNameHomLift\RelNameCommutator{\Rt{\beta+2\psi}}{\Rt{\alpha+\beta+2\psi}}]\label{rel:b3-large:comm:lift:beta+2psi:alpha+beta+2psi}
        \[
    \DQuant{i,j,k \in [1]}{t,u \in \LiftScalars} \Comm{\Eld{\beta+2\psi}{t}{j+2k}}{\Eld{\alpha+\beta+2\psi}{u}{i+j+2k}} = \Id.
    \]
\end{relation}

\begin{proof}
    If $t=0$, then $\Eld{\beta+2\psi}{t}{j+2k} = \Id$ (\Cnref{rel:b3-large:sub:id:beta+2psi}) and we have nothing to prove. Otherwise, use \Cnref{rel:b3-large:comm:raw:lift:beta+2psi:alpha+beta+2psi} at $\liftC=(\frac{u}t,t,1)$ and \Cnref{rel:b3-large:expr:alpha+beta+2psi:alpha:beta+2psi}.
\end{proof}

\paragraph{Simple analysis}

\begin{relation}[\RelNameBLg\RelNameCommute{\Rt{\beta+\psi}}{\Comm{\Rt{\alpha+\beta}}{\Rt{\beta+2\psi}}}]\label{rel:b3-large:rough:comm:beta+psi:alpha+beta:beta+2psi}
        \[
    \DQuant{i \in [2], j \in [2], k \in [3]}{t,u,v \in \LiftScalars} \Comm{\Eld{\beta+\psi}{t}{i}}{\Comm{\Eld{\alpha+\beta}{u}{j}}{\Eld{\beta+2\psi}{v}{k}}} = \Id.
    \]
\end{relation}

\begin{proof}
    $\Rt{\beta+\psi}$ elements commute with $\Rt{\alpha+\beta}$ elements (\Cnref{rel:b3-large:comm:alpha+beta:beta+psi}) and $\Rt{\beta+2\psi}$ elements (\Cnref{rel:b3-large:sub:comm:beta+psi:beta+2psi}).
\end{proof}

\begin{relation}[\RelNameBLg\RelNameCommute{\Rt{\alpha+\beta}}{\Comm{\Rt{\beta+\psi}}{\Rt{\alpha+\beta+\psi}}}]\label{rel:b3-large:rough:comm:alpha+beta:beta+psi:alpha+beta+psi}
        \[
    \DQuant{i \in [2], j \in [2], k \in [3]}{t,u,v \in \LiftScalars} \Comm{\Eld{\alpha+\beta}{t}{i}}{\Comm{\Eld{\beta+\psi}{u}{j}}{\Eld{\alpha+\beta+\psi}{v}{k}}} = \Id.
    \]
\end{relation}

\begin{proof}
    $\Rt{\alpha+\beta}$ elements commute with $\Rt{\alpha+\beta+\psi}$ elements (\Cnref{rel:b3-large:comm:alpha+beta:alpha+beta+psi}) and $\Rt{\beta+\psi}$ elements (\Cnref{rel:b3-large:comm:alpha+beta:beta+psi}).
\end{proof}

\begin{relation}[\RelNameBLg\RelNameCommute{\Rt{\beta}}{\Comm{\Rt{\alpha+\beta}}{\Rt{\beta+2\psi}}}]\label{rel:b3-large:rough:comm:beta:alpha+beta:beta+2psi}
        \[
    \DQuant{i \in [1], j \in [2], k \in [3]}{t,u,v \in \LiftScalars} \Comm{\Eld{\beta}{t}{i}}{\Comm{\Eld{\alpha+\beta}{u}{j}}{\Eld{\beta+2\psi}{v}{k}}} = \Id.
    \]
\end{relation}

\begin{proof}
    $\Rt{\beta}$ elements commute with $\Rt{\alpha+\beta}$ elements (\Cnref{rel:b3-large:sub:comm:beta:alpha+beta}) and $\Rt{\beta+2\psi}$ elements (\Cnref{rel:b3-large:sub:comm:beta:beta+2psi}).
\end{proof}

\begin{relation}[\RelNameBLg\RelNameCommute{\Rt{\beta}}{\Comm{\Rt{\alpha+\beta+\psi}}{\Rt{\beta+\psi}}}]\label{rel:b3-large:rough:comm:beta:alpha+beta+psi:beta+psi}
        \[
    \DQuant{i \in [1], j \in [3], k \in [2]}{t,u,v \in \LiftScalars} \Comm{\Eld{\beta}{t}{i}}{\Comm{\Eld{\alpha+\beta+\psi}{u}{j}}{\Eld{\beta+\psi}{v}{k}}} = \Id.
    \]
\end{relation}

\begin{proof}
    $\Rt{\beta}$ elements commute with $\Rt{\alpha+\beta+\psi}$ elements (\Cnref{rel:b3-large:comm:beta:alpha+beta+psi}) and $\Rt{\beta+\psi}$ elements (\Cnref{rel:b3-large:sub:comm:beta:beta+psi}).
\end{proof}

\paragraph{``\newedgeA'' edges}

\begin{proposition}[Sufficient conditions for commutator of $\Rt{\alpha+\beta+\psi}$ and $\Rt{\beta+\psi}$]\label{rel:b3-large:comm:expr:abs:alpha+2beta+2psi:A}
    Let $i \in [2], j \in [1], k \in [2]$. Suppose that:
    \begin{subequations}
    \begin{gather}
    \Quant{t,u,v \in \LiftScalars} \Comm{\Eld{\psi}{t}{j}}{\Comm{\Eld{\alpha+\beta}{u}{i}}{\Eld{\beta+2\psi}{v}{j+k}}} = \Id, \label{ass:b3-large:comm:expr:abs:alpha+2beta+2psi:A:1} \\
    \Quant{t,u,v \in \LiftScalars} \Comm{\Eld{\alpha+\beta}{t}{i}}{\Comm{\Eld{\alpha+\beta}{u}{i}}{\Eld{\beta+2\psi}{v}{j+k}}} = \Id, \label{ass:b3-large:comm:expr:abs:alpha+2beta+2psi:A:2}\\
    \Quant{t,u \in \LiftScalars} \Comm{\Eld{\alpha+\beta}{t}{i}}{\Eld{\beta+2\psi}{u}{j+k}} = \Comm{\Eld{\alpha+\beta}{-t}{i}}{\Eld{\beta+2\psi}{-u}{j+k}}, \label{ass:b3-large:comm:expr:abs:alpha+2beta+2psi:A:3}\\
    \Quant{t,u \in \LiftScalars} \Comm{\Eld{\alpha+\beta}{t}{i}}{\Eld{\beta+2\psi}{u}{j+k}} \Comm{\Eld{\alpha+\beta}{t}{i}}{\Eld{\beta+2\psi}{u}{j+k}} = \Comm{\Eld{\alpha+\beta}{2t}{i}}{\Eld{\beta+2\psi}{u}{j+k}}\label{ass:b3-large:comm:expr:abs:alpha+2beta+2psi:A:4}.
    \end{gather}
    \end{subequations}
    Then:
    \[
    \Quant{t,u \in \LiftScalars} \Comm{\Eld{\alpha+\beta+\psi}{tu}{i+j}}{\Eld{\beta+\psi}{v}{k}} = \Comm{\Eld{\alpha+\beta}{t}{i}}{\Eld{\beta+2\psi}{2uv}{j+k}}.
    \]
\end{proposition}

\begin{proof}
From \Cnref{eq:comm:mid:inv}, \Cnref{rel:b3-large:sub:inv:alpha+beta}, \Cnref{rel:b3-large:sub:inv:beta+2psi}, and by assumption (\Cref{ass:b3-large:comm:expr:abs:alpha+2beta+2psi:A:3}), we have:
\begin{multline}\label{subeq:b3-large:comm:expr:abs:alpha+2beta+2psi:A:1}
    \Eld{\beta+2\psi}{-uv}{j+k} \Eld{\alpha+\beta}{t}{i} = \Eld{\alpha+\beta}{t}{i} \Comm{\Eld{\alpha+\beta}{-t}{i}}{\Eld{\beta+2\psi}{-uv}{j+k}} \Eld{\beta+2\psi}{-uv}{j+k} \\
    = \Eld{\alpha+\beta}{t}{i} \Comm{\Eld{\alpha+\beta}{t}{i}}{\Eld{\beta+2\psi}{uv}{j+k}} \Eld{\beta+2\psi}{-uv}{j+k}
\end{multline}
and from \Cnref{eq:comm:right:str}, \Cnref{rel:b3-large:sub:inv:alpha+beta}, and \Cnref{rel:b3-large:sub:inv:beta+2psi}, we have:
\begin{equation}\label{subeq:b3-large:comm:expr:abs:alpha+2beta+2psi:A:2}
\Eld{\alpha+\beta}{-t}{i} \Eld{\beta+2\psi}{-uv}{j+k} = \Eld{\beta+2\psi}{-uv}{j+k} \Eld{\alpha+\beta}{-t}{i}  \Comm{\Eld{\alpha+\beta}{t}{i}}{\Eld{\beta+2\psi}{uv}{j+k}}.
\end{equation}

Now, we write a product of $\Rt{\alpha+\beta+\psi}$ and $\Rt{\beta+\psi}$ elements:
\begin{align*}
    & \asite{\Eld{\alpha+\beta+\psi}{tu}{i+j}} \Eld{\beta+\psi}{v}{k} \\
    \intertext{Expand the $\Rt{\alpha+\beta+\psi}$ element into a product of $\Rt{\alpha+\beta}$ and $\Rt{\psi}$ elements (\Cnref{prop:b3-large:est:alpha+beta+psi}):}
    &= \Eld{\psi}{-u/2}{j} \Eld{\alpha+\beta}{t}{i} \Eld{\psi}{u}{j} \Eld{\alpha+\beta}{-t}{i} \Eld{\psi}{-u/2}{j} \asite{\Eld{\beta+\psi}{v}{k}} \\
    \intertext{Bring the $\Rt{\beta+\psi}$ element from the right fully to the left, creating no commutators with $\Rt{\alpha+\beta}$ elements (\Cnref{rel:b3-large:comm:alpha+beta:beta+psi}) and $\Rt{\beta+2\psi}$ commutators with $\Rt{\psi}$ elements (\Cnref{rel:b3-large:order:psi:beta+psi}):}
    &= \Eld{\beta+\psi}{v}{k} \Eld{\psi}{-u/2}{j} \asite{\Eld{\beta+2\psi}{-uv}{j+k} \Eld{\alpha+\beta}{t}{i}} \Eld{\beta+2\psi}{2uv}{j+k} \Eld{\psi}{u}{j} \asite{\Eld{\alpha+\beta}{-t}{i} \Eld{\beta+2\psi}{-uv}{j+k}} \Eld{\psi}{-u/2}{j} \\
    \intertext{Move the $\Rt{\beta+2\psi}$ elements towards each other across $\Rt{\alpha+\beta}$ elements, creating generic commutators (\Cref{subeq:b3-large:comm:expr:abs:alpha+2beta+2psi:A:1} and \Cref{subeq:b3-large:comm:expr:abs:alpha+2beta+2psi:A:2}:}
    &= \Eld{\beta+\psi}{v}{k} \Eld{\psi}{-u/2}{j} \Comm{\Eld{\alpha+\beta}{t}{i}}{\Eld{\beta+2\psi}{uv}{j+k}} \Eld{\alpha+\beta}{t}{i} \asite{\Eld{\beta+2\psi}{-uv}{j+k}} \\
    &\hspace{1in} \cdot \asite{\Eld{\beta+2\psi}{2uv}{j+k} \Eld{\psi}{u}{j}  \Eld{\beta+2\psi}{-uv}{j+k}} \Eld{\alpha+\beta}{-t}{i} \Comm{\Eld{\alpha+\beta}{t}{i}}{\Eld{\beta+2\psi}{uv}{j+k}} \Eld{\psi}{-u/2}{j} \\
    \intertext{Cancel the $\Rt{\beta+2\psi}$ elements (\Cnref{rel:b3-large:sub:comm:psi:beta+2psi} and \Cnref{rel:b3-large:sub:inv:beta+2psi}):}
    &= \Eld{\beta+\psi}{v}{k} \Eld{\psi}{-u/2}{j} \asite{\Comm{\Eld{\alpha+\beta}{t}{i}}{\Eld{\beta+2\psi}{uv}{j+k}}} \Eld{\alpha+\beta}{t}{i} \Eld{\psi}{u}{j}  \Eld{\alpha+\beta}{-t}{i} \asite{\Comm{\Eld{\alpha+\beta}{t}{i}}{\Eld{\beta+2\psi}{uv}{j+k}}} \Eld{\psi}{-u/2}{j}. \\
    \intertext{Move the $\Comm{\Rt{\alpha+\beta}}{\Rt{\beta+2\psi}}$ elements to the left. They commute with every element, in particular, with $\Rt{\psi}$ and $\Rt{\alpha+\beta}$ elements by assumption (\Cref{ass:b3-large:comm:expr:abs:alpha+2beta+2psi:A:1} and \Cref{ass:b3-large:comm:expr:abs:alpha+2beta+2psi:A:2}) and with $\Rt{\beta+\psi}$ elements by \Cnref{rel:b3-large:rough:comm:beta+psi:alpha+beta:beta+2psi}:}
    &= \asite{\Comm{\Eld{\alpha+\beta}{t}{i}}{\Eld{\beta+2\psi}{uv}{j+k}} \Comm{\Eld{\alpha+\beta}{t}{i}}{\Eld{\beta+2\psi}{uv}{j+k}}} \Eld{\beta+\psi}{v}{k} \asite{\Eld{\psi}{-u/2}{j} \Eld{\alpha+\beta}{t}{i} \Eld{\psi}{u}{j}  \Eld{\alpha+\beta}{-t}{i}  \Eld{\psi}{-u/2}{j}} \\
    \intertext{Reducing using \Cref{ass:b3-large:comm:expr:abs:alpha+2beta+2psi:A:4}, \Cnref{rel:b3-large:inv-doub:lift:alpha+beta:beta+2psi}, and \Cnref{prop:b3-large:est:alpha+beta+psi}:}
    &= \Comm{\Eld{\alpha+\beta}{t}{i}}{\Eld{\beta+2\psi}{2uv}{j+k}} \Eld{\beta+\psi}{v}{k} \Eld{\alpha+\beta+\psi}{tu}{i+j}.
\end{align*}
\end{proof}

\begin{relation}[\RelNameBLg\RelNamePart\newedgeA:\RelNameInterchange{\Rt{\alpha+2\beta+2\psi}}]\label{cor:b3-large:expr:alpha+2beta+2psi:A}
\begin{align*}
    \Quant{t,u,v \in \LiftScalars} \Comm{\Eld{\alpha+\beta+\psi}{tu}{0}}{\Eld{\beta+\psi}{v}{2}} &= \Comm{\Eld{\alpha+\beta}{t}{0}}{\Eld{\beta+2\psi}{2uv}{2}}, \\
    \Quant{t,u,v \in \LiftScalars} \Comm{\Eld{\alpha+\beta+\psi}{tu}{2}}{\Eld{\beta+\psi}{v}{0}} &= \Comm{\Eld{\alpha+\beta}{t}{1}}{\Eld{\beta+2\psi}{2uv}{1}}.
\end{align*}
\end{relation}

\begin{proof}
We apply the previous proposition (\Cnref{rel:b3-large:comm:expr:abs:alpha+2beta+2psi:A}) in both cases.
\begin{itemize}
\item For the first equation, we use the previous proposition (\Cnref{rel:b3-large:comm:expr:abs:alpha+2beta+2psi:A}) with $i=0, j=0, k=2$. This gives the desideratum if we can show that:
\begin{gather*}
    \Comm{\Eld{\psi}{t}{0}}{\Comm{\Eld{\alpha+\beta}{u}{0}}{\Eld{\beta+2\psi}{v}{2}}} = \Id, \\
    \Comm{\Eld{\alpha+\beta}{t}{0}}{\Comm{\Eld{\alpha+\beta}{u}{0}}{\Eld{\beta+2\psi}{v}{2}}} = \Id, \\
    \Comm{\Eld{\alpha+\beta}{t}{0}}{\Eld{\beta+2\psi}{u}{2}} = \Comm{\Eld{\alpha+\beta}{-t}{0}}{\Eld{\beta+2\psi}{-u}{2}}, \\
    \Comm{\Eld{\alpha+\beta}{t}{0}}{\Eld{\beta+2\psi}{u}{2}} \Comm{\Eld{\alpha+\beta}{t}{0}}{\Eld{\beta+2\psi}{u}{2}} = \Comm{\Eld{\alpha+\beta}{t}{0}}{\Eld{\beta+2\psi}{2u}{2}}.
\end{gather*}
The second condition follows immediately from homogeneous lifting (\Cnref{rel:b3-large:comm:lift:alpha+beta:alpha+beta:beta+2psi} with $i=0,j=0,k=1$). For the remaining conditions, we use that
\[
\Comm{\Eld{\alpha+\beta}{t}{0}}{\Eld{\beta+2\psi}{u}{2}} = \Comm{\Eld{\alpha+\beta}{t}{1}}{\Eld{\beta+2\psi}{u}{1}}
\]
via \Cnref{rel:b3-large:lift:alpha+2beta+2psi} (cf. \Cref{fig:b3-large:def:alpha+2beta+2psi}). Then, we use homogeneous lifting (\Cnref{rel:b3-large:comm:lift:psi:alpha+beta:beta+2psi} and \Cnref{rel:b3-large:inv-doub:lift:beta:alpha+beta+2psi} with $i=0,j=1,k=0$).
\item For the second equation, we use the previous proposition (\Cnref{rel:b3-large:comm:expr:abs:alpha+2beta+2psi:A}) with $i=1, j=1, k=0$. This gives the desideratum if we can show that:
\begin{gather*}
    \Comm{\Eld{\psi}{t}{1}}{\Comm{\Eld{\alpha+\beta}{u}{1}}{\Eld{\beta+2\psi}{v}{1}}} = \Id, \\
    \Comm{\Eld{\alpha+\beta}{t}{1}}{\Comm{\Eld{\alpha+\beta}{u}{1}}{\Eld{\beta+2\psi}{v}{1}}} = \Id, \\
    \Comm{\Eld{\alpha+\beta}{t}{1}}{\Eld{\beta+2\psi}{u}{1}} = \Comm{\Eld{\alpha+\beta}{-t}{1}}{\Eld{\beta+2\psi}{-u}{1}}, \\
    \Comm{\Eld{\alpha+\beta}{t}{1}}{\Eld{\beta+2\psi}{u}{1}} \Comm{\Eld{\alpha+\beta}{t}{1}}{\Eld{\beta+2\psi}{u}{1}} = \Comm{\Eld{\alpha+\beta}{t}{1}}{\Eld{\beta+2\psi}{2u}{1}}.
\end{gather*}
The second condition follows immediately from homogeneous lifting (\Cnref{rel:b3-large:comm:lift:alpha+beta:alpha+beta:beta+2psi} with $i=0,j=1,k=0$). For the remaining conditions, we use that
\[
\Comm{\Eld{\alpha+\beta}{t}{1}}{\Eld{\beta+2\psi}{u}{1}} = \Comm{\Eld{\alpha+\beta}{t}{0}}{\Eld{\beta+2\psi}{u}{2}}
\]
via \Cnref{rel:b3-large:lift:alpha+2beta+2psi} (cf. \Cref{fig:b3-large:def:alpha+2beta+2psi}). Then, we use homogeneous lifting (\Cnref{rel:b3-large:comm:lift:psi:alpha+beta:beta+2psi} and \Cnref{rel:b3-large:inv-doub:lift:alpha+beta:beta+2psi} with $i=0,j=0,k=1$).
\end{itemize}
\end{proof}

\paragraph{``\newedgeB'' edges}

\begin{proposition}[More sufficient conditions for commutator of $\Rt{\alpha+\beta+\psi}$ and $\Rt{\beta+\psi}$]\label{rel:b3-large:comm:expr:abs:alpha+2beta+2psi:B}
    Let $i \in [3], j \in [1], k \in [2]$. Suppose that:
    \begin{subequations}
    \begin{gather}
        \Quant{t,u,v \in \LiftScalars} \Comm{\Eld{\beta}{t}{k}}{\Comm{\Eld{\alpha+\beta+2\psi}{u}{i+j}}{\Eld{\beta}{v}{k}}} = \Id, \label{ass:b3-large:comm:expr:abs:alpha+2beta+2psi:B:1}, \\
        \Quant{t,u,v \in \LiftScalars} \Comm{\Eld{\psi}{t}{j}}{\Comm{\Eld{\alpha+\beta+2\psi}{u}{i+j}}{\Eld{\beta}{v}{k}}} = \Id \label{ass:b3-large:comm:expr:abs:alpha+2beta+2psi:B:2}, \\
        \Quant{t,u \in \LiftScalars} \Comm{\Eld{\beta}{u}{k}}{\Eld{\alpha+\beta+2\psi}{-t}{i+j}} = \Comm{\Eld{\alpha+\beta+2\psi}{t}{i+j}}{\Eld{\beta}{u}{k}} \label{ass:b3-large:comm:expr:abs:alpha+2beta+2psi:B:3}, \\
        \Quant{t,u \in \LiftScalars} \Comm{\Eld{\alpha+\beta+2\psi}{t}{i+j}}{\Eld{\beta}{u}{k}} \Comm{\Eld{\alpha+\beta+2\psi}{t}{i+j}}{\Eld{\beta}{u}{k}} = \Comm{\Eld{\alpha+\beta+2\psi}{2t}{i+j}}{\Eld{\beta}{u}{k}} \label{ass:b3-large:comm:expr:abs:alpha+2beta+2psi:B:4}.
    \end{gather}
    \end{subequations}
    Then:
    \[
    \Quant{t,u \in \LiftScalars} \Comm{\Eld{\alpha+\beta+\psi}{t}{i}}{\Eld{\beta+\psi}{uv}{j+k}} = \Comm{\Eld{\alpha+\beta+2\psi}{2tu}{i+j}}{\Eld{\beta}{v}{k}}.
    \]
\end{proposition}

\begin{proof}
    From \Cnref{eq:comm:left:str}, we have:
    \begin{equation}\label{subeq:b3-large:comm:expr:abs:alpha+2beta+2psi:B:2}
        \Eld{\alpha+\beta+2\psi}{tu}{i+j} \Eld{\beta}{v}{k} = \asite{\Comm{\Eld{\alpha+\beta+2\psi}{tu}{i+j}}{\Eld{\beta}{v}{k}}} \Eld{\beta}{v}{k} \Eld{\alpha+\beta+2\psi}{tu}{i+j},
    \end{equation}
    and from \Cnref{eq:comm:right:inv} and by assumption (\Cref{ass:b3-large:comm:expr:abs:alpha+2beta+2psi:B:3}), we have:
    \begin{equation}\label{subeq:b3-large:comm:expr:abs:alpha+2beta+2psi:B:3}
        \Eld{\beta}{-v}{k} \Eld{\alpha+\beta+2\psi}{tu}{i+j}
        = \Eld{\alpha+\beta+2\psi}{tu}{i+j} \Eld{\beta}{-v}{k} \asite{\Comm{\Eld{\alpha+\beta+2\psi}{tu}{i+j}}{\Eld{\beta}{v}{k}}}
    \end{equation}

    We write a product of $\Rt{\alpha+\beta+\psi}$ and $\Rt{\beta+\psi}$ elements:
    \begin{align*}
        & \Eld{\alpha+\beta+\psi}{t}{i} \asite{\Eld{\beta+\psi}{uv}{j+k}} \\
        \intertext{Expand the $\Rt{\beta+\psi}$ element into a product of $\Rt{\psi}$ and $\Rt{\beta}$ elements (\Cnref{rel:b3-large:sub:expr:beta+psi}):}
        &= \asite{\Eld{\alpha+\beta+\psi}{t}{i}} \Eld{\psi}{-u/2}{j} \Eld{\beta}{v}{k} \Eld{\psi}{u}{j} \Eld{\beta+\psi}{-v}{k} \Eld{\psi}{-u/2}{j} \\
        \intertext{Move the $\Rt{\alpha+\beta+\psi}$ from the left fully to the right, creating no commutators with $\Rt{\beta}$ elements (\Cnref{rel:b3-large:comm:beta:alpha+beta+psi}) and $\Rt{\alpha+\beta+2\psi}$ commutators with $\Rt{\psi}$ elements (\Cnref{rel:b3-large:order:alpha+beta+psi:psi}):}
        &=  \Eld{\psi}{-u/2}{j} \asite{\Eld{\alpha+\beta+2\psi}{tu}{i+j} \Eld{\beta}{v}{k}} \Eld{\alpha+\beta+2\psi}{-2tu}{i+j} \Eld{\psi}{u}{j} \asite{\Eld{\beta}{-v}{k} \Eld{\alpha+\beta+2\psi}{tu}{i+j}} \Eld{\psi}{-u/2}{j} \Eld{\alpha+\beta+\psi}{t}{i} \\
        \intertext{Bringing $\Rt{\alpha+\beta+2\psi}$ elements closer together by commuting across $\Rt{\beta}$ elements, creating generic commutators (\Cref{subeq:b3-large:comm:expr:abs:alpha+2beta+2psi:B:2} and \Cref{subeq:b3-large:comm:expr:abs:alpha+2beta+2psi:B:3}):}
        &= \Eld{\psi}{-u/2}{j} \Comm{\Eld{\alpha+\beta+2\psi}{tu}{i+j}}{\Eld{\beta}{v}{k}} \Eld{\beta}{v}{k} \asite{\Eld{\alpha+\beta+2\psi}{tu}{i+j}   \Eld{\alpha+\beta+2\psi}{-2tu}{i+j} \Eld{\psi}{u}{j}} \\
        & \hspace{0.5in} \cdot \asite{\Eld{\alpha+\beta+2\psi}{tu}{i+j}} \Eld{\beta}{-v}{k} \Comm{\Eld{\alpha+\beta+2\psi}{tu}{i+j}}{\Eld{\beta}{v}{k}} \Eld{\psi}{-u/2}{j} \Eld{\alpha+\beta+\psi}{t}{i} \\
        \intertext{Using that $\Rt{\alpha+\beta+2\psi}$ elements commute with $\Rt{\psi}$ elements (\Cnref{rel:b3-large:comm:psi:alpha+beta+2psi}), we cancel the $\Rt{\alpha+\beta+2\psi}$ elements (\Cnref{rel:b3-large:inv-doub:alpha+beta+2psi}):}
        &= \Eld{\psi}{-u/2}{j} \asite{\Comm{\Eld{\alpha+\beta+2\psi}{tu}{i+j}}{\Eld{\beta}{v}{k}}} \Eld{\beta}{v}{k} \Eld{\psi}{u}{j} \Eld{\beta}{-v}{k} \asite{\Comm{\Eld{\alpha+\beta+2\psi}{tu}{i+j}}{\Eld{\beta}{v}{k}}} \Eld{\psi}{-u/2}{j} \Eld{\alpha+\beta+\psi}{t}{i} \\
        \intertext{Commuting the $\Comm{\Rt{\alpha+\beta+2\psi}}{\Rt{\beta}}$ elements fully to the left, using that they commute with the $\Rt{\beta}$ and $\Rt{\psi}$ elements by assumption (\Cref{ass:b3-large:comm:expr:abs:alpha+2beta+2psi:B:1} and \Cref{ass:b3-large:comm:expr:abs:alpha+2beta+2psi:B:2}):}
        &= \asite{\Comm{\Eld{\alpha+\beta+2\psi}{tu}{i+j}}{\Eld{\beta}{v}{k}} \Comm{\Eld{\alpha+\beta+2\psi}{tu}{i+j}}{\Eld{\beta}{v}{k}}} \asite{\Eld{\psi}{-u/2}{j} \Eld{\beta}{v}{k} \Eld{\psi}{u}{j} \Eld{\beta}{-v}{k} \Eld{\psi}{-u/2}{j}} \Eld{\alpha+\beta+\psi}{t}{i} \\
        \intertext{Reducing using the doubling assumption (\Cref{ass:b3-large:comm:expr:abs:alpha+2beta+2psi:B:4}) for $\Comm{\Rt{\alpha+\beta+2\psi}}{\Rt{\beta}}$ and \Cnref{rel:b3-large:sub:expr:beta+psi}:}
        &= \Comm{\Eld{\alpha+\beta+2\psi}{2tu}{i+j}}{\Eld{\beta}{v}{k}} \Eld{\beta+\psi}{uv}{j+k} \Eld{\alpha+\beta+\psi}{t}{i}.
    \end{align*}
\end{proof}

\begin{proposition}[Sufficient conditions for commutator of $\Rt{\alpha+\beta+2\psi}$ and $\Rt{\beta+2\psi}$]\label{rel:b3-large:comm:impl:beta+2psi:alpha+beta+2psi}
    Let $i \in [3], j \in [1], k \in [3]$. Suppose that:
    \[
    \Quant{t,u \in \LiftScalars} \Comm{\Eld{\beta+2\psi}{t}{k}}{\Eld{\alpha+\beta+2\psi}{u}{i+j}} = \Id.
    \]
    Then:
    \[
    \Quant{t,u \in \LiftScalars} \Comm{\Eld{\beta+2\psi}{t}{i}}{\Eld{\alpha+\beta+2\psi}{u}{j+k}} = \Id.
    \]
\end{proposition}

\begin{proof}
    We write the product of $\Rt{\alpha+\beta+2\psi}$ and $\Rt{\beta+2\psi}$ elements in the desideratum:
    \begin{align*}
        & \asite{\Eld{\alpha+\beta+2\psi}{u}{j+k}} \Eld{\beta+2\psi}{t}{i} \\
        \intertext{Expanding the $\Rt{\alpha+\beta+2\psi}$ element into a product of $\Rt{\alpha}$ and $\Rt{\beta+2\psi}$ elements (\Cnref{rel:b3-large:expr:alpha+beta+2psi:alpha:beta+2psi}):}
        &=  \Eld{\alpha}u{j} \Eld{\beta+2\psi}1{k}  \Eld{\alpha}{-u}{j} \Eld{\beta+2\psi}{-1}{k} \asite{\Eld{\beta+2\psi}{t}{i}} \\
        \intertext{Commuting the right $\Rt{\beta+2\psi}$ element fully to the left, creating no commutators with $\Rt{\beta+2\psi}$ elements (\Cnref{rel:b3-large:sub:comm:self:beta+2psi}) and $\Rt{\alpha+\beta+2\psi}$ commutators with $\Rt{\alpha}$ elements (\Cnref{rel:b3-large:order:alpha:beta+2psi}):}
        &= \Eld{\beta+2\psi}{t}{i} \Eld{\alpha}u{j} \asite{\Eld{\alpha+\beta+2\psi}{tu}{i+j} \Eld{\beta+2\psi}{1}{k} \Eld{\alpha+\beta+2\psi}{-tu}{i+j}} \Eld{\alpha}{-u}{j} \Eld{\beta+2\psi}{-1}{k} . \\
        \intertext{Now by assumption, the $\Rt{\alpha+\beta+2\psi}$ elements and the $\Rt{\beta+2\psi}$ element commute, so we can cancel the $\Rt{\alpha+\beta+2\psi}$ elements (\Cnref{rel:b3-large:inv-doub:alpha+beta+2psi}):}
        &= \Eld{\beta+2\psi}{t}{i} \asite{\Eld{\alpha}u{j} \Eld{\beta+2\psi}{1}{k} \Eld{\alpha}{-u}{j} \Eld{\beta+2\psi}{-1}{k}}  \\
        \intertext{Reducing back into an $\Rt{\alpha+\beta+2\psi}$ element (\Cnref{rel:b3-large:expr:alpha+beta+2psi:alpha:beta+2psi}):}
        &=  \Eld{\alpha+\beta+2\psi}{u}{j+k} \Eld{\alpha}u{j}.
    \end{align*}
\end{proof}

We prove two commutators of $\Rt{\beta+2\psi}$ and $\Rt{\alpha+\beta+2\psi}$ here; only one is used in this section, while one is used in a later section:

\begin{relation}[\RelNameBLg\RelNamePart\RelNameCommute{\Rt{\beta+2\psi}}{\Rt{\alpha+\beta+2\psi}}]\label{cor:b3-large:comm:cases:beta+2psi:alpha+beta+2psi}
    \begin{align*}
        \Quant{t,u \in \LiftScalars} \Comm{\Eld{\beta+2\psi}{t}{2}}{\Eld{\alpha+\beta+2\psi}{u}{1}} &= \Id, \\
        \Quant{t,u \in \LiftScalars} \Comm{\Eld{\beta+2\psi}{t}{0}}{\Eld{\alpha+\beta+2\psi}{u}{2}} &= \Id.
    \end{align*}
\end{relation}

\begin{proof}
In both cases, we use \Cnref{rel:b3-large:comm:impl:beta+2psi:alpha+beta+2psi}.
\begin{itemize}
    \item For the first equation, we apply the previous proposition (\Cnref{rel:b3-large:comm:impl:beta+2psi:alpha+beta+2psi}) with $i=2, j = 0, k= 1$. For this, we need \[
    \Comm{\Eld{\beta+2\psi}{t}{1}}{\Eld{\alpha+\beta+2\psi}{u}{2}} = \Id,
    \]
    which follows from homogeneous lifting (\Cnref{rel:b3-large:comm:lift:beta+2psi:alpha+beta+2psi} with $i=1,j=1,k=0$).
    \item For the second equation, we apply the previous proposition (\Cnref{rel:b3-large:comm:impl:beta+2psi:alpha+beta+2psi}) with $i=0, j=1, k=1$. For this, we need \[
    \Comm{\Eld{\beta+2\psi}{t}{1}}{\Eld{\alpha+\beta+2\psi}{u}{1}} = \Id,
    \]
    which follows from homogeneous lifting (\Cnref{rel:b3-large:comm:lift:beta+2psi:alpha+beta+2psi} with $i=0,j=1,k=0$).
\end{itemize}
\end{proof}

\begin{proposition}[Sufficient conditions for commutator of $\Rt{\psi}$ and $\Comm{\Rt{\alpha+\beta+2\psi}}{\Rt{\beta}}$]\label{rel:b3-large:comm:impl:psi:alpha+beta+2psi:beta}
    Let $i\in[1],j\in[4],k\in[1]$. Suppose that:
    \begin{subequations}
    \begin{gather}
        \Quant{t,u \in \LiftScalars} \Comm{\Eld{\beta+2\psi}{t}{2i+k}}{\Eld{\alpha+\beta+2\psi}{u}{j}} = \Id \label{ass:b3-large:comm:impl:psi:alpha+beta+2psi:beta:1}, \\
        \Quant{t,u \in \LiftScalars} \Comm{\Eld{\beta+\psi}{t}{i+k}}{\Eld{\alpha+\beta+2\psi}{u}{j}} = \Id \label{ass:b3-large:comm:impl:psi:alpha+beta+2psi:beta:2}.
    \end{gather}
    \end{subequations}
    Then:
    \[
    \Quant{t,u,v \in \LiftScalars} \Comm{\Eld{\psi}{t}{i}}{\Comm{\Eld{\alpha+\beta+2\psi}{u}{j}}{\Eld{\beta}{v}{k}}} = \Id.
    \]
\end{proposition}

\begin{proof}
    We write a product of $\Rt{\psi}$ and $\Comm{\Rt{\alpha+\beta+2\psi}}{\Rt{\beta}}$ elements:
    \begin{align*}
        & \Eld{\psi}{t}{i} \asite{\Comm{\Eld{\alpha+\beta+2\psi}{u}{j}}{\Eld{\beta}{v}k}} \\
        \intertext{Expanding the $\Comm{\Rt{\alpha+\beta+2\psi}}{\Rt{\beta}}$ element using \Cnref{rel:b3-large:inv-doub:lift:beta:alpha+beta+2psi} and \Cnref{rel:b3-large:sub:inv:beta}:}
        &= \asite{\Eld{\psi}{t}{i}} \Eld{\alpha+\beta+2\psi}{u}{j} \Eld{\beta}{v}{k} \Eld{\alpha+\beta+2\psi}{-u}{j}\Eld{\beta}{-v}{k} \\
        \intertext{Moving the $\Rt{\psi}$ element from the left fully to the right, creating no commutators with $\Rt{\alpha+\beta+2\psi}$ elements (\Cnref{rel:b3-large:comm:psi:alpha+beta+2psi}) and $\Rt{\beta+\psi}$ and $\Rt{\beta+2\psi}$ commutators with $\Rt{\beta}$ elements (\Cnref{rel:b3-large:order:psi:beta}):}
        &= \Eld{\alpha+\beta+2\psi}{u}{j} \Eld{\beta}{v}{k} \asite{\Eld{\beta+\psi}{-tv}{i+k} \Eld{\beta+2\psi}{-t^2v}{2i+k} \Eld{\alpha+\beta+2\psi}{-u}{j} \Eld{\beta+2\psi}{t^2v}{2i+k} \Eld{\beta+\psi}{tv}{i+k}} \Eld{\beta}{-v}{k} \Eld{\psi}{t}{i} \\
        \intertext{By our assumptions, the $\Rt{\beta+2\psi}$ elements commute with the $\Rt{\alpha+\beta+2\psi}$ element (\Cref{ass:b3-large:comm:impl:psi:alpha+beta+2psi:beta:1}), so we can cancel them (\Cnref{rel:b3-large:sub:inv:beta+2psi}), and similarly with the $\Rt{\beta+\psi}$ elements (\Cref{ass:b3-large:comm:impl:psi:alpha+beta+2psi:beta:2} and \Cnref{rel:b3-large:sub:inv:beta+psi}):}
        &= \asite{\Eld{\alpha+\beta+2\psi}{u}{j} \Eld{\beta}{v}{k} \Eld{\alpha+\beta+2\psi}{-u}{j} \Eld{\beta}{-v}{k}} \Eld{\psi}{t}{i} \\
        \intertext{Reducing back into a $\Comm{\Rt{\alpha+\beta+2\psi}}{\Rt{\beta}}$ element:}
        &= \Comm{\Eld{\alpha+\beta+2\psi}{u}{j}}{\Eld{\beta}{v}k} \Eld{\psi}{t}{i}.
    \end{align*}
\end{proof}

\begin{relation}[\RelNameBLg\RelNamePart\RelNameCommute{\Rt{\psi}}{\Comm{\Rt{\alpha+\beta+2\psi}}{\Rt{\beta}}}]\label{cor:b3-large:comm:case:psi:alpha+beta+2psi:beta:110}
\[
\Quant{t,u,v \in \LiftScalars} \Comm{\Eld{\psi}{v}{1}}{\Comm{\Eld{\alpha+\beta+2\psi}{t}{1}}{\Eld{\beta}{u}{0}}} = \Id.
\]    
\end{relation}

\begin{proof}
    We apply the previous proposition (\Cnref{rel:b3-large:comm:impl:psi:alpha+beta+2psi:beta}) with $i=1, j=1,k=0$. For the desideratum, we need that: 
    \begin{align*}
    \Quant{t,u \in \LiftScalars} \Comm{\Eld{\beta+2\psi}{t}{2}}{\Eld{\alpha+\beta+2\psi}{u}{1}} &= \Id, \\
    \Quant{t,u\in \LiftScalars} \Comm{\Eld{\beta+\psi}{t}{1}}{\Eld{\alpha+\beta+2\psi}{u}{1}} &= \Id.
    \end{align*}
    The second equation follows from homogeneous lifting (\Cnref{rel:b3-large:comm:lift:beta+psi:alpha+beta+2psi} with $i=0,j=1,k=0$). For the first, we use \Cnref{cor:b3-large:comm:cases:beta+2psi:alpha+beta+2psi} above.
\end{proof}

\begin{relation}[\RelNameBLg\RelNamePart\newedgeB:\RelNameInterchange{\Rt{\alpha+2\beta+2\psi}}]\label{cor:b3-large:expr:alpha+2beta+2psi:B}
\begin{align*}
    \Quant{t,u,v \in \LiftScalars}\Comm{\Eld{\alpha+\beta+\psi}{0}{t}}{\Eld{\beta+\psi}{1}{uv}} &= \Comm{\Eld{\alpha+\beta+2\psi}{1}{2tu}}{\Eld{\beta}{0}{v}}.
\end{align*}
\end{relation}

\begin{proof}
    We invoke \Cnref{rel:b3-large:comm:expr:abs:alpha+2beta+2psi:B}, setting $i = 0, j=1, k=0$. For the desideratum, we need:
    \begin{gather*}
        \Comm{\Eld{\beta}{t}{0}}{\Comm{\Eld{\alpha+\beta+2\psi}{u}{1}}{\Eld{\beta}{v}{0}}} = \Id, \\
        \Comm{\Eld{\psi}{t}{1}}{\Comm{\Eld{\alpha+\beta+2\psi}{u}{1}}{\Eld{\beta}{v}{0}}} = \Id, \\
        \Comm{\Eld{\beta}{u}{0}}{\Eld{\alpha+\beta+2\psi}{-t}{1}} = \Comm{\Eld{\alpha+\beta+2\psi}{t}{1}}{\Eld{\beta}{u}{0}}, \\
        \Comm{\Eld{\alpha+\beta+2\psi}{t}{1}}{\Eld{\beta}{u}{0}} \Comm{\Eld{\alpha+\beta+2\psi}{t}{1}}{\Eld{\beta}{u}{0}} = \Comm{\Eld{\alpha+\beta+2\psi}{2t}{1}}{\Eld{\beta}{u}{0}}.
    \end{gather*}
    The second equation is precisely the previous corollary (\Cnref{cor:b3-large:comm:case:psi:alpha+beta+2psi:beta:110}). For the final two equations, we apply \Cnref{rel:b3-large:inv-doub:lift:beta:alpha+beta+2psi}. For the first equation, we use that
    \[
    \Comm{\Eld{\alpha+\beta+2\psi}{u}{1}}{\Eld{\beta}{v}{0}} = \Comm{\Eld{\alpha+\beta}{u}{1}}{\Eld{\beta+2\psi}{v}{0}}
    \]
    via \Cnref{rel:b3-large:lift:alpha+2beta+2psi} (cf. \Cref{fig:b3-large:def:alpha+2beta+2psi}) and then apply \Cnref{rel:b3-large:rough:comm:beta:alpha+beta:beta+2psi}.\footnote{Alternatively, we could use homogeneous lifting.}
\end{proof}

\paragraph{``\newedgeC'' edges}

\begin{proposition}[Sufficient conditions for commutator of $\Rt{\alpha+\beta}$ and $\Rt{\beta+2\psi}$]\label{rel:b3-large:comm:expr:abs:alpha+2beta+2psi:C}
    Let $i \in [2], j \in [1], k \in [2]$. Suppose that:
    \begin{subequations}
    \begin{gather}
    \Quant{t,u \in \LiftScalars} \Comm{\Eld{\alpha+\beta+2\psi}{t}{i+2j}}{\Eld{\beta+\psi}{u}{k}} = \Id \label{ass:b3-large:comm:expr:abs:alpha+2beta+2psi:1}, \\
    \Quant{t,u,v \in \LiftScalars} \Comm{\Eld{\psi}{t}{j}}{\Comm{\Eld{\alpha+\beta+\psi}{u}{i+j}}{\Eld{\beta+\psi}{v}{k}}} = \Id \label{ass:b3-large:comm:expr:abs:alpha+2beta+2psi:2}.
    \end{gather}
    \end{subequations}
    Then
    \[
    \Quant{t,u,v \in \LiftScalars} \Comm{\Eld{\alpha+\beta}{t}{i}}{\Eld{\beta+2\psi}{2uv}{j+k}} = \Comm{\Eld{\alpha+\beta+\psi}{tu}{i+j}}{\Eld{\beta+\psi}{v}{k}}
    \]
\end{proposition}

\begin{proof}
By \Cnref{eq:comm:left:str}, we have:
\begin{equation}\label{subeq:b3-large:comm:expr:abs:alpha+2beta+2psi:C:3}
    \Eld{\alpha+\beta+\psi}{tu}{i+j} \Eld{\beta+\psi}{v}{k} = \Comm{ \Eld{\alpha+\beta+\psi}{tu}{i+j}}{\Eld{\beta+\psi}{v}{k}} \Eld{\beta+\psi}{v}{k} \Eld{\alpha+\beta+\psi}{tu}{i+j} 
\end{equation}

    Now, we write a product of $\Rt{\alpha+\beta}$ and $\Rt{\beta+2\psi}$ elements:
    \begin{align*}
        & \Eld{\alpha+\beta}{t}{i} \asite{\Eld{\beta+2\psi}{2uv}{j+k}} \\
        \intertext{Expanding the $\Rt{\beta+2\psi}$ element as a product of $\Rt{\beta+\psi}$ and $\Rt{\psi}$ elements (\Cnref{rel:b3-large:sub:expr:beta+2psi}):}
        &= \asite{\Eld{\alpha+\beta}{t}{i}}  \Eld{\psi}{u}{j} \Eld{\beta+\psi}{v}{k} \Eld{\psi}{-u}{j}\Eld{\beta+\psi}{-v}{k} \\
        \intertext{Moving the $\Rt{\alpha+\beta}$ element on the left fully to the right, creating no commutators with $\Rt{\beta+\psi}$ elements (\Cnref{rel:b3-large:comm:alpha+beta:beta+psi}) and $\Rt{\alpha+\beta+\psi}$ and $\Rt{\alpha+\beta+2\psi}$ commutators with $\Rt{\psi}$ elements (\Cnref{rel:b3-large:order:alpha+beta:psi}):}
        &=  \Eld{\psi}{u}{j} \Eld{\alpha+\beta+\psi}{tu}{i+j} \asite{\Eld{\alpha+\beta+2\psi}{-tu^2}{i+2j} \Eld{\beta+\psi}{v}{k} \Eld{\alpha+\beta+2\psi}{tu^2}{i+2j}} \Eld{\alpha+\beta+\psi}{-tu}{i+j} \Eld{\psi}{-u}{j} \Eld{\beta+\psi}{-v}{k} \Eld{\alpha+\beta}{t}{i}
        \intertext{By assumption (\Cref{ass:b3-large:comm:expr:abs:alpha+2beta+2psi:1}), we can commute the $\Rt{\alpha+\beta+2\psi}$ elements together across the $\Rt{\beta+\psi}$ element, and therefore cancel them with \Cnref{rel:b3-large:inv-doub:alpha+beta+2psi}:}
        &=  \Eld{\psi}{u}{j} \asite{\Eld{\alpha+\beta+\psi}{tu}{i+j} \Eld{\beta+\psi}{v}{k}} \Eld{\alpha+\beta+\psi}{-tu}{i+j} \Eld{\psi}{-u}{j} \Eld{\beta+\psi}{-v}{k} \Eld{\alpha+\beta}{t}{i} \\
        \intertext{and now we commute the $\Rt{\alpha+\beta+\psi}$ element across the $\Rt{\beta+\psi}$ element, introducing a commutator (\Cref{subeq:b3-large:comm:expr:abs:alpha+2beta+2psi:C:3}), and cancel the $\Rt{\alpha+\beta+\psi}$ elements with \Cnref{rel:b3-large:inv-doub:alpha+beta+psi}:}
        &=  \Eld{\psi}{u}{j} \Comm{\Eld{\alpha+\beta+\psi}{tu}{i+j}}{\Eld{\beta+\psi}{v}{k}} \Eld{\beta+\psi}{v}{k} \Eld{\psi}{-u}{j} \Eld{\beta+\psi}{-v}{k}  \Eld{\alpha+\beta}{t}{i} \\
        \intertext{Finally, the $\Comm{\Rt{\alpha+\beta+\psi}}{\Rt{\beta+\psi}}$ element commutes with the $\Rt{\psi}$ element to its left by assumption (\Cref{ass:b3-large:comm:expr:abs:alpha+2beta+2psi:2}), leaving:}
        &= \Comm{\Eld{\alpha+\beta+\psi}{tu}{i+j}}{ \Eld{\beta+\psi}{v}{k}} \asite{\Eld{\psi}{u}{j}   \Eld{\beta+\psi}{v}{k} \Eld{\psi}{-u}{j}\Eld{\beta+\psi}{-v}{k}} \Eld{\alpha+\beta}{t}{i}  \\
        \intertext{Reducing back into a $\Rt{\beta+2\psi}$ element (\Cnref{rel:b3-large:sub:expr:beta+2psi}):}
        &= \Comm{\Eld{\alpha+\beta+\psi}{tu}{i+j}}{ \Eld{\beta+\psi}{v}{k}} \Eld{\beta+2\psi}{2uv}{j+k} \Eld{\alpha+\beta}{t}{i},
    \end{align*}
    as desired.
\end{proof}

\begin{proposition}[Sufficient condition for commutator of $\Rt{\alpha+\beta+2\psi}$ and $\Rt{\beta+\psi}$]\label{rel:b3-large:comm:impl:alpha+beta+2psi:beta+psi}
Let $i\in [4],j\in[1],k\in[1]$. Suppose that:
\begin{subequations}
    \begin{gather}
        \Quant{t,u,v \in \LiftScalars} \Comm{\Comm{\Eld{\alpha+\beta+2\psi}{t}{i}}{\Eld{\beta}{u}{j}}}{\Eld{\psi}{v}{k}} = \Id \label{ass:b3-large:comm:impl:alpha+beta+2psi:beta+psi:1}, \\
        \Quant{t,u \in \LiftScalars} \Comm{\Eld{\beta}{-u}{j}} {\Eld{\alpha+\beta+2\psi}{t}{i}} = \Comm{\Eld{\alpha+\beta+2\psi}{t}{i}}{\Eld{\beta}{u}{j}} \label{ass:b3-large:comm:impl:alpha+beta+2psi:beta+psi:2}, \\
        \Quant{t,u \in \LiftScalars} \Comm{\Eld{\alpha+\beta+2\psi}{t}{i}}{\Eld{\beta}{u}{j}} \Comm{\Eld{\alpha+\beta+2\psi}{-t}{i}}{\Eld{\beta}{u}{j}} = \Id \label{ass:b3-large:comm:impl:alpha+beta+2psi:beta+psi:3}.
    \end{gather}
\end{subequations}
Then
\[
\Quant{t,u \in \LiftScalars} \Comm{\Eld{\alpha+\beta+2\psi}{t}{i}}{\Eld{\beta+\psi}{u}{j+k}} = \Id.
\]
\end{proposition}

\begin{proof}
    From \Cnref{eq:comm:mid:inv} and the assumption (\Cref{ass:b3-large:comm:impl:alpha+beta+2psi:beta+psi:2}):
    \begin{equation}\label{rel:b3-large:comm:impl:alpha+beta+2psi:beta+psi:1}
        \Eld{\alpha+\beta+2\psi}{t}{i} \Eld{\beta}{u}{j} = \Eld{\beta}{u}{j} \asite{\Comm{\Eld{\alpha+\beta+2\psi}{t}{i}}{\Eld{\beta}{u}{j}}} \Eld{\alpha+\beta+2\psi}{t}{i}
    \end{equation}
    and from \Cnref{eq:comm:right:str}, \Cnref{rel:b3-large:inv-doub:alpha+beta+2psi}, and \Cnref{rel:b3-large:sub:inv:beta}:
    \begin{equation}\label{rel:b3-large:comm:impl:alpha+beta+2psi:beta+psi:2}
        \Eld{\alpha+\beta+2\psi}{t}{i} \Eld{\beta}{-u}{j} = \Eld{\beta}{-u}{j} \Eld{\alpha+\beta+2\psi}{t}{i} \asite{\Comm{\Eld{\alpha+\beta+2\psi}{-t}{i}}{\Eld{\beta}{u}{j}}}
    \end{equation}
    
    We write a product of $\Rt{\alpha+\beta+2\psi}$ and $\Rt{\beta+\psi}$ elements:
    \begin{align*}
        & \Eld{\alpha+\beta+2\psi}{t}{i} \asite{\Eld{\beta+\psi}{u}{j+k}} \\
        \intertext{Expand the $\Rt{\beta+\psi}$ element into a product of $\Rt{\beta}$ and $\Rt{\psi}$ elements (\Cnref{rel:b3-large:sub:expr:beta+psi}):}
        &= \asite{\Eld{\alpha+\beta+2\psi}{t}{i}} \Eld{\psi}{-1/2}{k} \Eld{\beta}{u}{j} \Eld{\psi}{1}{k} \Eld{\beta}{-u}{j} \Eld{\psi}{-1/2}{k} \\
        \intertext{Move the $\Rt{\alpha+\beta+2\psi}$ element from the left fully to the right, creating no commutators with $\Rt{\psi}$ elements (\Cnref{rel:b3-large:comm:psi:alpha+beta+2psi}) and commutators with $\Rt{\beta}$ elements (\Cref{rel:b3-large:comm:impl:alpha+beta+2psi:beta+psi:1,rel:b3-large:comm:impl:alpha+beta+2psi:beta+psi:2}):}
        &= \Eld{\psi}{-1/2}{k} \Eld{\beta}{u}{j} \asite{\Comm{\Eld{\alpha+\beta+2\psi}{t}{i}}{\Eld{\beta}{u}{j}} \Eld{\psi}{1}{k} \Comm{\Eld{\alpha+\beta+2\psi}{-t}{i}}{\Eld{\beta}{u}{j}}} \Eld{\beta}{-u}{j} \Eld{\psi}{-1/2}{k} \Eld{\alpha+\beta+2\psi}{t}{i} \\
        \intertext{Now, the $\Comm{\Rt{\alpha+\beta+2\psi}}{\Rt{\beta}}$ elements commute across the $\Rt{\psi}$ element by assumption (\Cref{ass:b3-large:comm:impl:alpha+beta+2psi:beta+psi:1}), so we can cancel them by assumption (\Cref{ass:b3-large:comm:impl:alpha+beta+2psi:beta+psi:3}):}
        &= \asite{\Eld{\psi}{-1/2}{k} \Eld{\beta}{u}{j} \Eld{\psi}{1}{k} \Eld{\beta}{-u}{j} \Eld{\psi}{-1/2}{k}} \Eld{\alpha+\beta+2\psi}{t}{i}. \\
        \intertext{Reduce back into a $\Rt{\beta+\psi}$ element (\Cnref{rel:b3-large:sub:expr:beta+psi}:}
        &= \Eld{\beta+\psi}{u}{j+k} \Eld{\alpha+\beta+2\psi}{t}{i}. \\
    \end{align*}
\end{proof}

\begin{relation}[\RelNameBLg\RelNamePart\RelNameCommute{\Rt{\beta+\psi}}{\Rt{\alpha+\beta+2\psi}}]\label{cor:b3-large:comm:case:psi:alpha+beta+2psi:beta+psi:20}
    \[ \Quant{t,u \in \LiftScalars} \Comm{\Eld{\alpha+\beta+2\psi}{t}{2}}{\Eld{\beta+\psi}{u}{0}} = \Id. \]
\end{relation}

\begin{proof}
    We apply the previous proposition (\Cnref{rel:b3-large:comm:impl:alpha+beta+2psi:beta+psi}) with $i=2,j=0,k=0$. We need
    \begin{gather*}
        \Quant{t,u,v \in \LiftScalars} \Comm{\Comm{\Eld{\alpha+\beta+2\psi}{t}{2}}{\Eld{\beta}{u}{0}}}{\Eld{\psi}{v}{0}} = \Id, \\
        \Quant{t,u \in \LiftScalars} \Comm{\Eld{\beta}{-u}{0}} {\Eld{\alpha+\beta+2\psi}{t}{2}} = \Comm{\Eld{\alpha+\beta+2\psi}{t}{2}}{\Eld{\beta}{u}{0}}, \\
        \Quant{t,u \in \LiftScalars} \Comm{\Eld{\alpha+\beta+2\psi}{t}{2}}{\Eld{\beta}{u}{0}} \Comm{\Eld{\alpha+\beta+2\psi}{-t}{2}}{\Eld{\beta}{u}{0}} = \Id.
    \end{gather*}
    These follow since
    \[
    \Quant{t,u \in \LiftScalars} \Comm{\Eld{\alpha+\beta+2\psi}{t}{2}}{\Eld{\beta}{u}{0}} = \Comm{\Eld{\alpha+\beta}{t}{1}}{\Eld{\beta+2\psi}{u}{1}}
    \]
    by \Cnref{rel:b3-large:lift:alpha+2beta+2psi} (cf. \Cref{fig:b3-large:def:alpha+2beta+2psi}), then we can use homogeneous lifting (\Cnref{rel:b3-large:comm:lift:psi:alpha+beta:beta+2psi} and \Cnref{rel:b3-large:inv-doub:lift:alpha+beta:beta+2psi} with $i=0,j=1,k=0$).
\end{proof}

\begin{relation}[\RelNameBLg\RelNamePart\newedgeC:\RelNameInterchange{\Rt{\alpha+2\beta+2\psi}}]\label{cor:b3-large:expr:alpha+2beta+2psi:C}
    \begin{align*}
        \Quant{t,u,v \in \LiftScalars} \Comm{\Eld{\alpha+\beta}{t}{0}}{\Eld{\beta+2\psi}{uv}{1}} &= \Comm{\Eld{\alpha+\beta+\psi}{tu}{1}}{\Eld{\beta+\psi}{2v}{0}}, \\
        \Quant{t,u,v \in \LiftScalars} \Comm{\Eld{\alpha+\beta}{t}{2}}{\Eld{\beta+2\psi}{uv}{0}} &= \Comm{\Eld{\alpha+\beta+\psi}{tu}{2}}{\Eld{\beta+\psi}{2v}{0}}.
    \end{align*}
\end{relation}

\begin{proof}
We apply \Cnref{rel:b3-large:comm:expr:abs:alpha+2beta+2psi:C}.

\begin{itemize}
    \item Applying \Cnref{rel:b3-large:comm:expr:abs:alpha+2beta+2psi:C} with $i=0, j = 1, k = 0$, we get the desideratum if we can show that:
    \begin{gather*}
    \Quant{t,u \in \LiftScalars} \Comm{\Eld{\alpha+\beta+2\psi}{t}{2}}{\Eld{\beta+\psi}{u}{0}} = \Id, \\
    \Quant{t,u,v \in \LiftScalars} \Comm{\Eld{\psi}{t}{1}}{\Comm{\Eld{\alpha+\beta+\psi}{u}{1}}{\Eld{\beta+\psi}{v}{0}}} = \Id.
    \end{gather*}
    The first equation is the previous corollary (\Cnref{cor:b3-large:comm:case:psi:alpha+beta+2psi:beta+psi:20}). For the latter, we rewrite \[
    \Comm{\Eld{\alpha+\beta+\psi}{u}{1}}{\Eld{\beta+\psi}{v}{0}} = \Comm{\Eld{\alpha+\beta+2\psi}{2u}{1}}{\Eld{\beta}{v}{0}}
    \]
    using \Cnref{rel:b3-large:lift:alpha+2beta+2psi} (cf. \Cref{fig:b3-large:def:alpha+2beta+2psi}). Then, we use \Cnref{cor:b3-large:comm:case:psi:alpha+beta+2psi:beta:110}.
    
    \item Applying \Cnref{rel:b3-large:comm:expr:abs:alpha+2beta+2psi:C} with $i=2, j = 0, k = 0$, we get the desideratum if we can show that:
    \begin{gather*}
    \Quant{t,u \in \LiftScalars} \Comm{\Eld{\alpha+\beta+2\psi}{t}{2}}{\Eld{\beta+\psi}{u}{0}} = \Id, \\
    \Quant{t,u,v \in \LiftScalars} \Comm{\Eld{\psi}{t}{0}}{\Comm{\Eld{\alpha+\beta+\psi}{u}{2}}{\Eld{\beta+\psi}{v}{0}}} = \Id.
    \end{gather*}
    The first is the previous corollary (\Cnref{cor:b3-large:comm:case:psi:alpha+beta+2psi:beta+psi:20}). For the latter, we rewrite
    \[
    \Comm{\Eld{\alpha+\beta+\psi}{u}{2}}{\Eld{\beta+\psi}{v}{0}} = \Comm{\Eld{\alpha+\beta+\psi}{u}{1}}{\Eld{\beta+\psi}{v}{1}}
    \]
    using \Cnref{rel:b3-large:lift:alpha+2beta+2psi} and \Cnref{cor:b3-large:expr:alpha+2beta+2psi:A} (cf. \Cref{fig:b3-large:def:alpha+2beta+2psi}). Then, we use homogeneous lifting (\Cnref{rel:b3-large:comm:lift:psi:alpha+beta:beta+2psi} with $i=0,j=1,k=0$).
\end{itemize}
\end{proof}

\paragraph{``\newedgeD'' edges}

\begin{proposition}[Sufficient conditions for commutator of $\Rt{\alpha+\beta+2\psi}$ and $\Rt{\beta}$]\label{rel:b3-large:comm:expr:abs:alpha+2beta+2psi:D}
Let $i \in [3], j \in [1], k \in [1]$. Suppose:
\begin{subequations}
\begin{gather}
\Quant{t,u \in \LiftScalars} \Comm{\Eld{\alpha+\beta+\psi}{t}{i}}{\Eld{\beta+2\psi}{u}{2j+k}} = \Id \label{ass:b3-large:comm:expr:abs:alpha+2beta+2psi:D:1}, \\
\Quant{t,u,v \in \LiftScalars} \Comm{\Comm{\Eld{\alpha+\beta+\psi}{t}{i}}{\Eld{\beta+\psi}{u}{j+k}}}{\Eld{\psi}{v}{j}} = \Id \label{ass:b3-large:comm:expr:abs:alpha+2beta+2psi:D:2}, \\
\Quant{t,u \in \LiftScalars} \Comm{\Eld{\beta+\psi}{-u}{j+k}}{\Eld{\alpha+\beta+\psi}{t}{i}}= \Comm{\Eld{\alpha+\beta+\psi}{t}{i}}{\Eld{\beta+\psi}{u}{j+k}} \label{ass:b3-large:comm:expr:abs:alpha+2beta+2psi:D:3}.
\end{gather}
\end{subequations}
Then:
\[
\Quant{t,u,v \in \LiftScalars} \Comm{\Eld{\alpha+\beta+2\psi}{2tu}{i+j}}{\Eld{\beta}{v}{k}} = \Comm{\Eld{\alpha+\beta+\psi}{t}{i}}{\Eld{\beta+\psi}{uv}{j+k}}.
\]
\end{proposition}

\begin{proof}
    From \Cnref{eq:comm:left:inv} and by assumption (\Cref{ass:b3-large:comm:expr:abs:alpha+2beta+2psi:D:3}), we have:
    \begin{equation}\label{subeq:b3-large:comm:expr:abs:alpha+2beta+2psi:D:2}
        \Eld{\beta+\psi}{-uv}{j+k} \Eld{\alpha+\beta+\psi}{t}{i} = \asite{\Comm{\Eld{\alpha+\beta+\psi}{t}{i}}{\Eld{\beta+\psi}{uv}{j+k}}} \Eld{\alpha+\beta+\psi}{t}{i} \Eld{\beta+\psi}{-uv}{j+k} 
    \end{equation}
    
    Now, we write a product of $\Rt{\alpha+\beta+2\psi}$ and $\Rt{\beta}$ elements:
    \begin{align*}
        &\asite{\Eld{\alpha+\beta+2\psi}{2tu}{i+j}} \Eld{\beta}{v}{k} \\
        \intertext{Expanding the $\Rt{\alpha+\beta+2\psi}$ element as a product of $\Rt{\alpha+\beta+\psi}$ and $\Rt{\psi}$ elements (\Cnref{rel:b3-large:expr:alpha+beta+2psi:alpha+beta+psi:psi}):}
        &= \Eld{\psi}{u}{j} \Eld{\alpha+\beta+\psi}{t}{i} \Eld{\psi}{-u}{j} \Eld{\alpha+\beta+\psi}{-t}{i}  \asite{\Eld{\beta}{v}{k}} \\
        \intertext{Commuting the $\Rt{\beta}$ element on the right fully to the left, creating no commutators with the $\Rt{\alpha+\beta+\psi}$ elements (\Cnref{rel:b3-large:comm:beta:alpha+beta+psi}) and $\Rt{\beta+\psi}$ and $\Rt{\beta+2\psi}$ commutators with the $\Rt{\psi}$ elements (\Cnref{rel:b3-large:order:psi:beta}):}
        &= \Eld{\beta}{v}{k} \Eld{\psi}{u}{j} \Eld{\beta+\psi}{-uv}{j+k} \asite{\Eld{\beta+2\psi}{-u^2 v}{2j+k} \Eld{\alpha+\beta+\psi}{t}{i}  \Eld{\beta+2\psi}{u^2 v}{2j+k}} \Eld{\beta+\psi}{uv}{j+k} \Eld{\psi}{-u}{j} \Eld{\alpha+\beta+\psi}{-t}{i} . \\
        \intertext{Now by assumption (\Cref{ass:b3-large:comm:expr:abs:alpha+2beta+2psi:D:1}), we can commute the $\Rt{\beta+2\psi}$ elements together across the $\Rt{\alpha+\beta+\psi}$ element, and then cancel them (\Cnref{rel:b3-large:sub:inv:beta+2psi}):}
        &= \Eld{\beta}{v}{k} \Eld{\psi}{u}{j} \asite{\Eld{\beta+\psi}{-uv}{j+k} \Eld{\alpha+\beta+\psi}{t}{i}} \Eld{\beta+\psi}{uv}{j+k} \Eld{\psi}{-u}{j} \Eld{\alpha+\beta+\psi}{-t}{i} . \\
        \intertext{Next we commute the two $\Rt{\beta+\psi}$ elements together, creating a generic commutator (\Cref{subeq:b3-large:comm:expr:abs:alpha+2beta+2psi:D:2}) and cancelling the $\Rt{\beta+\psi}$ elements (\Cnref{rel:b3-large:sub:inv:beta+psi}):}
        &= \Eld{\beta}{v}{k} \Eld{\psi}{u}{j} \asite{\Comm{\Eld{\alpha+\beta+\psi}{t}{i}}{\Eld{\beta+\psi}{uv}{j+k}}} \Eld{\alpha+\beta+\psi}{t}{i} \Eld{\psi}{-u}{j} \Eld{\alpha+\beta+\psi}{-t}{i} . \\
        \intertext{We can commute the $\Comm{\Rt{\alpha+\beta+\psi}}{\Rt{\beta+\psi}}$ element with the $\Rt{\psi}$ element by assumption (\Cref{ass:b3-large:comm:expr:abs:alpha+2beta+2psi:D:2}) and the $\Rt{\beta}$ element by \Cnref{rel:b3-large:rough:comm:beta:alpha+beta+psi:beta+psi}:}
        &= \Comm{\Eld{\alpha+\beta+\psi}{t}{i}}{\Eld{\beta+\psi}{uv}{j+k}} \Eld{\beta}{v}{k} \asite{\Eld{\psi}{u}{j}  \Eld{\alpha+\beta+\psi}{t}{i} \Eld{\psi}{-u}{j} \Eld{\alpha+\beta+\psi}{-t}{i}}. \\
        \intertext{Finally, we reduce to an $\Rt{\alpha+\beta+2\psi}$ element (\Cnref{rel:b3-large:expr:alpha+beta+2psi:alpha+beta+psi:psi}):}
        &= \Comm{\Eld{\alpha+\beta+\psi}{t}{i}}{\Eld{\beta+\psi}{uv}{j+k}} \Eld{\beta}{v}{k} \Eld{\alpha+\beta+2\psi}{2tu}{i+j}.
    \end{align*}
\end{proof}

\begin{proposition}[Sufficient conditions for commutator of $\Rt{\alpha+\beta+\psi}$ and $\Rt{\beta+2\psi}$]\label{rel:b3-large:comm:impl:alpha+beta+psi:beta+2psi:b}
    Let $i \in [1], j \in [2], k \in [3]$. Suppose:
    \[
    \Quant{t,u \in \LiftScalars} \Comm{\Eld{\alpha+\beta+2\psi}{u}{i+k}}{\Eld{\beta+\psi}{t}{j}} = \Id.
    \]
    Then:
    \[
    \Quant{t,u \in \LiftScalars} \Comm{\Eld{\alpha+\beta+\psi}{t}{i+j}}{\Eld{\beta+2\psi}{u}{k}}=\Id.
    \]
\end{proposition}

\begin{proof}
    From \Cnref{eq:comm:left:str}, \Cnref{eq:comm:right:str}, and \Cnref{prop:b3-large:est:alpha+beta+2psi}, we have:
    \begin{equation}\label{subeq:b3-large:comm:impl:alpha+beta+psi:beta+2psi:b:1}
        \Eld{\alpha}{t}{i} \Eld{\beta+2\psi}{u}{k} = \Eld{\beta+2\psi}{u}{k} \Eld{\alpha}{t}{i} \asite{\Eld{\alpha+\beta+2\psi}{tu}{i+k}} = \asite{\Eld{\alpha+\beta+2\psi}{tu}{i+k}} \Eld{\beta+2\psi}{u}{k} \Eld{\alpha}{t}{i}.
    \end{equation}

    Now, we write a product of $\Rt{\alpha+\beta+\psi}$ and $\Rt{\beta+2\psi}$ elements:
    \begin{align*}
        & \asite{\Eld{\alpha+\beta+\psi}{t}{i+j}} \Eld{\beta+2\psi}{u}{k} \\
        \intertext{Expanding the $\Rt{\alpha+\beta+\psi}$ element as a product of $\Rt{\alpha}$ and $\Rt{\beta+\psi}$ elements (\Cnref{prop:b3-large:est:alpha+beta+psi}):}
        &= \Eld{\beta+\psi}{-t/2}{j} \Eld{\alpha}{1}{i} \Eld{\beta+\psi}{t}{j} \Eld{\alpha}{-1}{i}\Eld{\beta+\psi}{-t/2}{j} \asite{\Eld{\beta+2\psi}{u}{k}} \\
        \intertext{Moving the $\Rt{\beta+2\psi}$ element from the right fully to the left, which creates no commutators with $\Rt{\beta+\psi}$ elements (\Cnref{rel:b3-large:sub:comm:beta+psi:beta+2psi}) and $\Rt{\alpha+\beta+2\psi}$ commutators with $\Rt{\alpha}$ elements (\Cref{subeq:b3-large:comm:impl:alpha+beta+psi:beta+2psi:b:1}):}
        &= \Eld{\beta+2\psi}{u}{k} \Eld{\beta+\psi}{-t/2}{j} \Eld{\alpha}{1}{i} \asite{\Eld{\alpha+\beta+2\psi}{u}{i+k} \Eld{\beta+\psi}{t}{j} \Eld{\alpha+\beta+2\psi}{-u}{i+k}} \Eld{\alpha}{-1}{i} \Eld{\beta+\psi}{-t/2}{j} \\
        \intertext{By assumption, the $\Rt{\alpha+\beta+2\psi}$ elements commute across the $\Rt{\beta+\psi}$ element. Therefore, we can cancel them (\Cnref{rel:b3-large:inv-doub:alpha+beta+2psi}):}
        &= \Eld{\beta+2\psi}{u}{k} \asite{\Eld{\beta+\psi}{-t/2}{j} \Eld{\alpha}{1}{i} \Eld{\beta+\psi}{t}{j} \Eld{\alpha}{-1}{i} \Eld{\beta+\psi}{-t/2}{j}} \\
        \intertext{And reduce back into an $\Rt{\alpha+\beta+\psi}$ element (\Cnref{prop:b3-large:est:alpha+beta+psi}):}
        &= \Eld{\beta+2\psi}{u}{k}  \Eld{\alpha+\beta+\psi}{t}{i+j} .
    \end{align*}
\end{proof}

\begin{relation}[\RelNameBLg\RelNamePart\RelNameCommute{\Rt{\alpha+\beta+\psi}}{\Rt{\beta+2\psi}}]\label{rel:b3-large:comm:case:alpha+beta+psi:beta+2psi:01}
    \[
    \Quant{t,u \in \LiftScalars} \Comm{\Eld{\alpha+\beta+\psi}{t}{0}}{\Eld{\beta+2\psi}{u}{1}}=\Id.
    \]
\end{relation}

\begin{proof}
We apply the prior proposition (\Cnref{rel:b3-large:comm:impl:alpha+beta+psi:beta+2psi:b}) with $i=0,j=0,k=1$. For the desideratum, we need that \[
\Comm{\Eld{\alpha+\beta+2\psi}{u}{1}}{\Eld{\beta+\psi}{t}{0}} = \Id.
\] This follows from homogeneous lifting (\Cnref{rel:b3-large:comm:lift:beta+psi:alpha:beta+2psi} with $i=1,j=0,k=0$).
\end{proof}

\begin{relation}[\RelNameBLg\RelNamePart\newedgeD:\RelNameInterchange{\Rt{\alpha+2\beta+2\psi}}]\label{cor:b3-large:expr:alpha+2beta+2psi:D}
    \[
        \Quant{t,u,v \in \LiftScalars} \Comm{\Eld{\alpha+\beta+\psi}{t}{0}}{\Eld{\beta+\psi}{uv}{1}} = \Comm{\Eld{\alpha+\beta+2\psi}{tu}{0}}{\Eld{\beta}{2v}{1}}.
    \]
\end{relation}

\begin{proof}
    We apply \Cnref{rel:b3-large:comm:expr:abs:alpha+2beta+2psi:D} with $i=0,j=0, k=1$. This gives the desideratum if we can show that
    \begin{gather*}
    \Quant{t,u \in \LiftScalars} \Comm{\Eld{\alpha+\beta+\psi}{t}{0}}{\Eld{\beta+2\psi}{u}{1}} = \Id, \\
    \Quant{t,u,v \in \LiftScalars} \Comm{\Comm{\Eld{\alpha+\beta+\psi}{t}{0}}{\Eld{\beta+\psi}{u}{1}}}{\Eld{\psi}{v}{0}} = \Id, \\
    \Quant{t,u \in \LiftScalars} \Comm{\Eld{\beta+\psi}{-u}{1}}{\Eld{\alpha+\beta+\psi}{t}{0}}= \Comm{\Eld{\alpha+\beta+\psi}{t}{0}}{\Eld{\beta+\psi}{u}{1}}.
    \end{gather*}
    The first equation is \Cnref{rel:b3-large:comm:case:alpha+beta+psi:beta+2psi:01}. For the remaining two, we use the fact that
    \[
    \Comm{\Eld{\alpha+\beta+\psi}{t}{0}}{\Eld{\beta+\psi}{u}{1}} = \Comm{\Eld{\alpha+\beta}{t}{1}}{\Eld{\beta+2\psi}{2u}{0}}
    \]
    by \Cnref{cor:b3-large:expr:alpha+2beta+2psi:B} and \Cnref{rel:b3-large:lift:alpha+2beta+2psi} (cf. \Cref{fig:b3-large:def:alpha+2beta+2psi}). For the second equation, we then use homogeneous lifting (\Cnref{rel:b3-large:comm:lift:psi:alpha+beta:beta+2psi}), and similarly for the third equation we use \Cnref{rel:b3-large:inv-doub:lift:alpha+beta:beta+2psi}.
\end{proof}

\paragraph{Conclusions}

\begin{bigproposition}[Establishing $\Rt{\alpha+2\beta+2\psi}$]\label{prop:b3-large:est:alpha+2beta+2psi}
    There exist elements $\Eld{\alpha+2\beta+2\psi}{t}{i}$ for $t \in \LiftScalars$ and $i \in [5]$ such that the following hold:
    \begin{gather*}
    \DQuant{i \in [2], j \in [3]}{t,u \in \LiftScalars} \Comm{\Eld{\alpha+\beta}{t}{i}}{\Eld{\beta+2\psi}{u}{j}} = \Eld{\alpha+2\beta+2\psi}{-tu}{i+j}, \\
    \DQuant{i \in [3], j \in [2]}{t,u \in \LiftScalars} \Comm{\Eld{\alpha+\beta+\psi}{t}{i}}{\Eld{\beta+\psi}{u}{j}} = \Eld{\alpha+2\beta+2\psi}{-2tu}{i+j}, \\
    \DQuant{i \in [4],j\in[1]}{t,u \in \LiftScalars} \Comm{\Eld{\alpha+\beta+2\psi}{t}{i}}{\Eld{\beta}{u}{j}} = \Eld{\alpha+2\beta+2\psi}{-tu}{i+j}.
    \end{gather*}
\end{bigproposition}

\begin{proof}
    Uses \Cref{prop:prelim:equal-conn} and \Cnref{cor:b3-large:expr:alpha+2beta+2psi:A}, \Cnref{cor:b3-large:expr:alpha+2beta+2psi:B}, \Cnref{cor:b3-large:expr:alpha+2beta+2psi:C}, \Cnref{cor:b3-large:expr:alpha+2beta+2psi:D}, and \Cnref{rel:b3-large:lift:alpha+2beta+2psi} (i.e., the blocks in \Cref{fig:b3-large:def:alpha+2beta+2psi} are connected).
\end{proof}

\begin{relation}[\RelNameBLg\RelNameCommute{\Rt{\beta}}{\Rt{\alpha+2\beta+2\psi}}]\label{rel:b3-large:comm:beta:alpha+2beta+2psi}
        \[
    \DQuant{i \in [1], j \in [5]}{t,u \in \LiftScalars} \Comm{\Eld{\beta}{t}{i}}{\Eld{\alpha+2\beta+2\psi}{u}{j}} = \Id.
    \]
\end{relation}

\begin{proof}
    Follows from \Cnref{prop:b3-large:est:alpha+2beta+2psi} and \Cnref{rel:b3-large:rough:comm:beta:alpha+beta:beta+2psi}.
\end{proof}

\begin{relation}[\RelNameBLg\RelNameCommute{\Rt{\alpha+\beta}}{\Rt{\alpha+2\beta+2\psi}}]\label{rel:b3-large:comm:alpha+beta:alpha+2beta+2psi}
        \[
    \DQuant{i \in [2], j \in [5]}{t,u \in \LiftScalars} \Comm{\Eld{\alpha+\beta}{t}{i}}{\Eld{\alpha+2\beta+2\psi}{u}{j}} = \Id.
    \]
\end{relation}

\begin{proof}
    Follows from \Cnref{prop:b3-large:est:alpha+2beta+2psi} and \Cnref{rel:b3-large:rough:comm:alpha+beta:beta+psi:alpha+beta+psi}.
\end{proof}

\begin{relation}[\RelNameBLg\RelNameCommute{\Rt{\beta+\psi}}{\Rt{\alpha+2\beta+2\psi}}]\label{rel:b3-large:comm:beta+psi:alpha+2beta+2psi}
        \[
    \DQuant{i \in [2], j \in [5]}{t,u \in \LiftScalars} \Comm{\Eld{\beta+\psi}{t}{i}}{\Eld{\alpha+2\beta+2\psi}{u}{j}} = \Id.
    \]
\end{relation}

\begin{proof}
    Follows from \Cnref{prop:b3-large:est:alpha+2beta+2psi} and \Cnref{rel:b3-large:rough:comm:beta+psi:alpha+beta:beta+2psi}.
\end{proof}

\begin{relation}[\RelNameBLg\RelNameInvDoub{\Rt{\alpha+2\beta+2\psi}}]\label{rel:b3-large:inv-doub:alpha+2beta+2psi}
    \begin{align*}
    \DQuant{i \in [5]}{t  \in \LiftScalars} \Eld{\alpha+2\beta+2\psi}{t}{i} \Eld{\alpha+2\beta+2\psi}{-t}{i} &= \Id, \\
    \DQuant{i \in [5]}{t \in \LiftScalars} \Eld{\alpha+2\beta+2\psi}{t}{i} \Eld{\alpha+2\beta+2\psi}{t}{i} &= \Eld{\alpha+2\beta+2\psi}{2t}{i}.
    \end{align*}
\end{relation}

\begin{proof}
    Follows from \Cnref{prop:b3-large:est:alpha+2beta+2psi} and \Cnref{rel:b3-large:inv-doub:lift:alpha+beta:beta+2psi}.
\end{proof}

Now from \Cnref{prop:b3-large:est:alpha+2beta+2psi}, \Cnref{rel:b3-large:inv-doub:alpha+2beta+2psi}, \Cnref{rel:b3-large:sub:inv:alpha+beta}, and \Cnref{rel:b3-large:sub:inv:beta+2psi}:
\begin{relation}[\RelNameBLg\RelNameExprn{\Rt{\alpha+2\beta+2\psi}}{\Rt{\alpha+\beta}}{\Rt{\beta+2\psi}}]\label{rel:b3-large:expr:alpha+2beta+2psi:alpha+beta:beta+2psi}
    \[ \Eld{\alpha+2\beta+2\psi}{-tu}{i+j} = \Eld{\alpha+\beta}{t}{i} \Eld{\beta+2\psi}{u}{j} \Eld{\alpha+\beta}{-t}{i} \Eld{\beta+2\psi}{-u}{j} = \Eld{\beta+2\psi}{-u}{j} \Eld{\alpha+\beta}{t}{i} \Eld{\beta+2\psi}{u}{j} \Eld{\alpha+\beta}{-t}{i}. \]
\end{relation}

From \Cnref{rel:b3-large:expr:alpha+2beta+2psi:alpha+beta:beta+2psi}, \Cnref{eq:comm:right:str}, \Cnref{eq:comm:mid:str}, \Cnref{rel:b3-large:inv-doub:alpha+2beta+2psi}, \Cnref{rel:b3-large:sub:inv:beta+2psi}, and \Cnref{rel:b3-large:sub:inv:alpha+beta}:
\begin{relation}[\RelNameBLg\RelNameOrder{\Rt{\alpha+\beta}}{\Rt{\beta+2\psi}}]\label{rel:b3-large:order:alpha+beta:beta+2psi}
    \[ \Eld{\alpha+\beta}{t}{i} \Eld{\beta+2\psi}{u}{j} = \Eld{\beta+2\psi}{u}{j} \asite{\Eld{\alpha+2\beta+2\psi}{-tu}{i+j}} \Eld{\alpha+\beta}{t}{i} = \Eld{\beta+2\psi}{u}{j} \Eld{\alpha+\beta}{t}{i} \asite{\Eld{\alpha+2\beta+2\psi}{-tu}{i+j}}. \]
\end{relation}
Similarly, from \Cnref{prop:b3-large:est:alpha+2beta+2psi}, \Cnref{rel:b3-large:inv-doub:alpha+2beta+2psi}, \Cnref{eq:comm:mid:inv}, \Cnref{eq:comm:right:inv}, \Cnref{rel:b3-large:inv-doub:alpha+beta+2psi}, and \Cnref{rel:b3-large:sub:inv:beta}:
\begin{relation}[\RelNameBLg\RelNameOrder{\Rt{\beta}}{\Rt{\alpha+\beta+2\psi}}]\label{rel:b3-large:order:beta:alpha+beta+2psi}
    \[ \Eld{\beta}{t}{i} \Eld{\alpha+\beta+2\psi}{u}{j} = \Eld{\alpha+\beta+2\psi}{u}{j} \asite{\Eld{\alpha+2\beta+2\psi}{tu}{i+j}} \Eld{\beta}{t}{i} = \Eld{\alpha+\beta+2\psi}{u}{j} \Eld{\beta}{t}{i} \asite{\Eld{\alpha+2\beta+2\psi}{tu}{i+j}}. \]
\end{relation}
Similarly, from \Cnref{prop:b3-large:est:alpha+2beta+2psi}, \Cnref{rel:b3-large:inv-doub:alpha+2beta+2psi}, \Cnref{eq:comm:mid:inv}, \Cnref{eq:comm:right:inv}, \Cnref{rel:b3-large:sub:inv:beta+psi}, and \Cnref{rel:b3-large:inv-doub:alpha+beta+psi}:
\begin{relation}[\RelNameBLg\RelNameOrder{\Rt{\beta+\psi}}{\Rt{\alpha+\beta+\psi}}]\label{rel:b3-large:order:beta+psi:alpha+beta+psi}
\[ \Eld{\beta+\psi}{t}{i} \Eld{\alpha+\beta+\psi}{u}{j} = \Eld{\alpha+\beta+\psi}{u}{j} \asite{\Eld{\alpha+2\beta+2\psi}{2tu}{i+j}} \Eld{\beta+\psi}{t}{i} =  \Eld{\alpha+\beta+\psi}{u}{j} \Eld{\beta+\psi}{t}{i} \asite{\Eld{\alpha+2\beta+2\psi}{2tu}{i+j}}. \]
\end{relation}

\begin{relation}[\RelNameBLg\RelNameCommutator{\Rt{\alpha}}{\Rt{\beta+\psi}}]\label{rel:b3-large:comm:alpha:beta+psi}
        \[
    \DQuant{i \in [1], j \in [2]}{t,u \in \LiftScalars} \Comm{\Eld{\alpha}{t}{i}}{\Eld{\beta+\psi}{u}{j}} = \Eld{\alpha+\beta+\psi}{tu}{i+j} \Eld{\alpha+2\beta+2\psi}{tu^2}{i+2j} = \Eld{\alpha+2\beta+2\psi}{tu^2}{i+2j} \Eld{\alpha+\beta+\psi}{tu}{i+j}.
    \]
\end{relation}

\begin{proof}
    Follows from \Cnref{rel:b3-large:rough:comm:alpha:beta+psi} and \Cnref{prop:b3-large:est:alpha+2beta+2psi}.
\end{proof}

\subsubsection{Remaining commutation relations}

\begin{proposition}[Sufficient conditions for commutator of $\Rt{\psi}$ and $\Rt{\alpha+2\beta+2\psi}$]\label{rel:b3-large:comm:impl:psi:alpha+2beta+2psi}
    Let $i \in [1], j \in [2], k \in [3]$. Suppose that:
    \begin{subequations}
    \begin{gather}
    \Quant{t,u \in \LiftScalars} \Comm{\Eld{\alpha+\beta+\psi}{t}{i+j}}{\Eld{\beta+2\psi}{u}{k}} = \Id \label{ass:b3-large:comm:impl:psi:alpha+2beta+2psi:1}, \\
    \Quant{t,u \in \LiftScalars} \Comm{\Eld{\alpha+\beta+2\psi}{t}{2i+j}}{\Eld{\beta+2\psi}{u}{k}} = \Id \label{ass:b3-large:comm:impl:psi:alpha+2beta+2psi:2}.
    \end{gather}
    \end{subequations}
    Then:
    \[
    \Quant{t,u \in \LiftScalars} \Comm{\Eld{\psi}{t}{i}}{\Eld{\alpha+2\beta+2\psi}{u}{j+k}} = \Id.
    \]
\end{proposition}

\begin{proof}
    We write a product of $\Rt{\alpha+2\beta+2\psi}$ and  $\Rt{\psi}$ elements:
    \begin{align*}
        &\asite{\Eld{\alpha+2\beta+2\psi}{u}{j+k}} \Eld{\psi}{t}{i}  \\
        \intertext{Expand the $\Rt{\alpha+2\beta+2\psi}$ elements into a product of $\Rt{\alpha+\beta}$ and $\Rt{\beta+2\psi}$ elements (\Cnref{rel:b3-large:expr:alpha+2beta+2psi:alpha+beta:beta+2psi}):}
        &=  \Eld{\alpha+\beta}{u}{j} \Eld{\beta+2\psi}{-1}{k} \Eld{\alpha+\beta}{-u}{j} \Eld{\beta+2\psi}{1}{k} \asite{\Eld{\psi}{t}{i}} \\
        \intertext{Move the $\Rt{\psi}$ element from the right fully to the left, creating $\Rt{\alpha+\beta+\psi}$ and $\Rt{\alpha+\beta+2\psi}$ commutators with the $\Rt{\alpha+\beta}$ elements (\Cnref{rel:b3-large:order:alpha+beta:psi}) and no commutators with the $\Rt{\beta+2\psi}$ elements (\Cnref{rel:b3-large:sub:comm:psi:beta+2psi}):}
        &= \Eld{\psi}{t}{i} \Eld{\alpha+\beta}{u}{j} \asite{\Eld{\alpha+\beta+\psi}{tu}{i+j} \Eld{\alpha+\beta+2\psi}{-t^2u}{2i+j}} \\
        &\hspace{1in} \cdot \asite{\Eld{\beta+2\psi}{-1}{k} \Eld{\alpha+\beta+2\psi}{t^2u}{2i+j} \Eld{\alpha+\beta+\psi}{-tu}{i+j}} \Eld{\alpha+\beta}{-u}{j} \Eld{\beta+2\psi}{1}{k}  \\
        \intertext{These $\Rt{\alpha+\beta+\psi}$ and $\Rt{\alpha+\beta+2\psi}$ elements commute across the $\Rt{\beta+2\psi}$ element by assumption (\Cref{ass:b3-large:comm:impl:psi:alpha+2beta+2psi:1} and \Cref{ass:b3-large:comm:impl:psi:alpha+2beta+2psi:2}), so we may cancel them (\Cnref{rel:b3-large:inv-doub:alpha+beta+psi} and \Cnref{rel:b3-large:inv-doub:alpha+beta+2psi}):}
        &= \Eld{\psi}{t}{i} \asite{\Eld{\alpha+\beta}{u}{j} \Eld{\beta+2\psi}{-1}{k} \Eld{\alpha+\beta}{-u}{j} \Eld{\beta+2\psi}{1}{k}}  \\
        \intertext{Reducing into an $\Rt{\alpha+2\beta+2\psi}$ element (\Cnref{rel:b3-large:expr:alpha+2beta+2psi:alpha+beta:beta+2psi}):}
        &= \Eld{\psi}{t}{i} \Eld{\alpha+2\beta+2\psi}{u}{j+k}.
    \end{align*}
\end{proof}

\begin{relation}[\RelNameBLg\RelNamePart\RelNameCommute{\Rt{\psi}}{\Rt{\alpha+2\beta+2\psi}}]\label{cor:b3-large:comm:case:psi:alpha+2beta+2psi:10}
    \begin{align*}
        \Quant{t,u \in \LiftScalars} \Comm{\Eld{\psi}{t}{1}}{\Eld{\alpha+2\beta+2\psi}{u}{0}} &= \Id.
    \end{align*}
\end{relation}

\begin{proof}
    We apply the previous proposition (\Cnref{rel:b3-large:comm:impl:psi:alpha+2beta+2psi}) with $i=1,j=0,k=0$; this gives the desideratum if 
    \[
    \Quant{t,u \in \LiftScalars} \Comm{\Eld{\alpha+\beta+\psi}{t}{1}}{\Eld{\beta+2\psi}{u}{0}} = \Id.
    \]
    and
    \[
    \Quant{t,u \in \LiftScalars} \Comm{\Eld{\alpha+\beta+2\psi}{t}{2}}{\Eld{\beta+2\psi}{u}{0}} = \Id
    \]
    The first equation follows from homogeneous lifting (\Cnref{rel:b3-large:comm:lift:beta+2psi:alpha+beta+psi} with $i=1,j=0,k=0$). The second equation follows from \Cnref{cor:b3-large:comm:cases:beta+2psi:alpha+beta+2psi}.
\end{proof}

\begin{bigrelation}[\RelNameBLg\RelNameCommute{\Rt{\psi}}{\Rt{\alpha+2\beta+2\psi}}]\label{rel:b3-large:comm:psi:alpha+2beta+2psi}
        \[
    \DQuant{i \in [1], j \in [5]}{t,u \in \LiftScalars} \Comm{\Eld{\psi}{t}{i}}{\Eld{\alpha+2\beta+2\psi}{u}{j}} = \Id.
    \]
\end{bigrelation}

\begin{proof}
    The $12$ cases in $[1] \times [5]$ are covered between homogeneous lifting (via \Cnref{rel:b3-large:comm:lift:psi:alpha+beta:beta+2psi} and \Cnref{prop:b3-large:est:alpha+2beta+2psi}), which provides eight pairs (see also \Cref{tab:b3-large:homog}), and \Cnref{cor:b3-large:comm:case:psi:alpha+2beta+2psi:10} and \Cnref{cor:b3-large:comm:case:psi:alpha+beta+2psi:beta:110} with \Cnref{prop:b3-large:est:alpha+2beta+2psi}, which each provide two.
\end{proof}

Now we get lots of commutators for free:

\begin{relation}[\RelNameBLg\RelNameCommute{\Rt{\beta+2\psi}}{\Rt{\alpha+2\beta+2\psi}}]\label{rel:b3-large:comm:beta+2psi:alpha+2beta+2psi}
        \[
    \DQuant{i \in [3], j \in [5]}{t,u \in \LiftScalars} \Comm{\Eld{\beta+2\psi}{t}{i}}{\Eld{\alpha+2\beta+2\psi}{u}{j}} = \Id.
    \]
\end{relation}

\begin{proof}
    By \Cnref{rel:b3-large:sub:comm:psi:beta+psi}, $\Rt{\beta+2\psi}$ elements may be expressed as products of $\Rt{\beta+\psi}$ elements and $\Rt{\psi}$ elements. These types of elements commute with $\Rt{\alpha+2\beta+2\psi}$ elements by \Cnref{rel:b3-large:comm:beta+psi:alpha+2beta+2psi} and \Cnref{rel:b3-large:comm:psi:alpha+2beta+2psi}, respectively.
\end{proof}

\begin{relation}[\RelNameBLg\RelNameCommute{\Rt{\alpha+\beta+\psi}}{\Rt{\alpha+2\beta+2\psi}}]\label{rel:b3-large:comm:alpha+beta+psi:alpha+2beta+2psi}
        \[
    \DQuant{i \in [3], j \in [5]}{t,u \in \LiftScalars} \Comm{\Eld{\alpha+\beta+\psi}{t}{i}}{\Eld{\alpha+2\beta+2\psi}{u}{j}} = \Id.
    \]
\end{relation}

\begin{proof}
    By \Cnref{prop:b3-large:est:alpha+beta+psi}, $\Rt{\alpha+\beta+\psi}$ elements may be expressed as products of $\Rt{\psi}$ elements and $\Rt{\alpha+\beta}$ elements. These types of elements commute with $\Rt{\alpha+2\beta+2\psi}$ elements by \Cnref{rel:b3-large:comm:psi:alpha+2beta+2psi} and \Cnref{rel:b3-large:comm:alpha+beta:alpha+2beta+2psi}, respectively.
\end{proof}

\begin{relation}[\RelNameBLg\RelNameCommute{\Rt{\alpha+\beta+2\psi}}{\Rt{\alpha+2\beta+2\psi}}]\label{rel:b3-large:comm:alpha+beta+2psi:alpha+2beta+2psi}
        \[
    \DQuant{i \in [4], j \in [5]}{t,u \in \LiftScalars} \Comm{\Eld{\alpha+\beta+2\psi}{t}{i}}{\Eld{\alpha+2\beta+2\psi}{u}{j}} = \Id.
    \]
\end{relation}

\begin{proof}
    By \Cnref{rel:b3-large:expr:alpha+beta+2psi:alpha+beta+psi:psi}, $\Rt{\alpha+\beta+2\psi}$ elements may be expressed as products of $\Rt{\psi}$ elements and $\Rt{\alpha+\beta+\psi}$ elements. These types of elements commute with $\Rt{\alpha+2\beta+2\psi}$ elements by \Cnref{rel:b3-large:comm:psi:alpha+2beta+2psi} and \Cnref{rel:b3-large:comm:alpha+beta+psi:alpha+2beta+2psi}, respectively.
\end{proof}

\begin{relation}[\RelNameBLg\RelNameSelfCommute{\Rt{\alpha+2\beta+2\psi}}]\label{rel:b3-large:comm:self:alpha+2beta+2psi}
        \[
    \DQuant{i \in [5], j \in [5]}{t,u \in \LiftScalars} \Comm{\Eld{\alpha+2\beta+2\psi}{t}{i}}{\Eld{\alpha+2\beta+2\psi}{u}{j}} = \Id.
    \]
\end{relation}

\begin{proof}
    By \Cnref{prop:b3-large:est:alpha+beta+psi}, $\Rt{\alpha+2\beta+2\psi}$ elements may be expressed as products of $\Rt{\beta}$ elements and $\Rt{\alpha+\beta+2\psi}$ elements. These types of elements commute with $\Rt{\alpha+2\beta+2\psi}$ elements by \Cnref{rel:b3-large:comm:beta:alpha+2beta+2psi} and \Cnref{rel:b3-large:comm:alpha+beta+2psi:alpha+2beta+2psi}, respectively.
\end{proof}

\begin{relation}[\RelNameBLg\RelNameCommute{\Rt{\alpha+\beta+\psi}}{\Rt{\beta+2\psi}}]\label{rel:b3-large:comm:alpha+beta+psi:beta+2psi}
        \[
    \DQuant{i \in [3], j \in [3]}{t,u \in \LiftScalars} \Comm{\Eld{\alpha+\beta+\psi}{t}{i}}{\Eld{\beta+2\psi}{u}{j}} = \Id.
    \]
\end{relation}

\begin{proof}
    Writing a product of $\Rt{\alpha+\beta+\psi}$ and $\Rt{\beta+2\psi}$ elements:
\begin{align*}
    & \asite{\Eld{\alpha+\beta+\psi}{t}{i}} \Eld{\beta+2\psi}{u}{j} \\
    \intertext{Expanding the $\Rt{\alpha+\beta+\psi}$ element into a product of $\Rt{\alpha+\beta}$ and $\Rt{\psi}$ elements (\Cnref{prop:b3-large:est:alpha+beta+psi}):}
    &= \Eld{\psi}{-1/2}{i_2} \Eld{\alpha+\beta}{t}{i_1} \Eld{\psi}{1}{i_2} \Eld{\alpha+\beta}{-t}{i_1} \Eld{\psi}{-1/2}{i_2} \asite{\Eld{\beta+2\psi}{u}{j}} \\
    \intertext{Moving the $\Rt{\beta+2\psi}$ element from right to left, creating no commutators with $\Rt{\psi}$ elements (\Cnref{rel:b3-large:sub:comm:psi:beta+2psi}) and $\Rt{\alpha+2\beta+2\psi}$ commutators with $\Rt{\alpha+\beta}$ elements (\Cnref{rel:b3-large:order:alpha+beta:beta+2psi}):}
    &= \Eld{\beta+2\psi}{u}{j} \Eld{\psi}{-1/2}{i_2} \Eld{\alpha+\beta}{t}{i_1} \asite{\Eld{\alpha+2\beta+2\psi}{-tu}{i_1+j} \Eld{\psi}{1}{i_2} \Eld{\alpha+2\beta+2\psi}{tu}{i_1+j}} \Eld{\alpha+\beta}{-t}{i_1} \Eld{\psi}{-1/2}{i_2} \\
    \intertext{Commuting the $\Rt{\alpha+2\beta+2\psi}$ elements across the $\Rt{\psi}$ element (\Cnref{rel:b3-large:comm:psi:alpha+2beta+2psi}) and cancelling them (\Cnref{rel:b3-large:inv-doub:alpha+2beta+2psi}):}
    &= \Eld{\beta+2\psi}{u}{j} \asite{\Eld{\psi}{-1/2}{i_2}\Eld{\alpha+\beta}{t}{i_1} \Eld{\psi}{1}{i_2} \Eld{\alpha+\beta}{-t}{i_1} \Eld{\psi}{-1/2}{i_2}} \\
    \intertext{Reducing to an $\Rt{\alpha+\beta+\psi}$ element (\Cnref{prop:b3-large:est:alpha+beta+psi}):}
    &= \Eld{\beta+2\psi}{u}{j} \Eld{\alpha+\beta+\psi}{t}{i}.
\end{align*}
\end{proof}

\begin{relation}[\RelNameBLg\RelNameCommute{\Rt{\beta+\psi}}{\Rt{\alpha+\beta+2\psi}}]\label{rel:b3-large:comm:beta+psi:alpha+beta+2psi}
        \[
    \DQuant{i \in [2], j \in [4]}{t,u \in \LiftScalars} \Comm{\Eld{\beta+\psi}{t}{i}}{\Eld{\alpha+\beta+2\psi}{u}{j}} = \Id.
    \]
\end{relation}

\begin{proof}
    Decompose arbitrarily $i = i_1 + i_2$ for $i_1,i_2 \in [1]$. We write a product of $\Rt{\beta+\psi}$ and $\Rt{\alpha+\beta+2\psi}$ elements:
    \begin{align*}
        & \asite{\Eld{\beta+\psi}{t}{i}} \Eld{\alpha+\beta+2\psi}{u}{j} \\
        \intertext{Expanding the $\Rt{\beta+\psi}$ element into a product of $\Rt{\beta}$ and $\Rt{\psi}$ elements (\Cnref{rel:b3-large:sub:expr:beta+psi}):}
        &= \Eld{\psi}{-1/2}{i_2} \Eld{\beta}{t}{i_1} \Eld{\psi}{1}{i_2} \Eld{\beta}{-t}{i_1} \Eld{\psi}{-1/2}{i_2} \asite{\Eld{\alpha+\beta+2\psi}{u}{j}} \\
        \intertext{Moving the $\Rt{\alpha+\beta+2\psi}$ element fully to the left, creating no commutators with $\Rt{\psi}$ elements (\Cnref{rel:b3-large:comm:psi:alpha+beta+2psi}) and $\Rt{\alpha+2\beta+2\psi}$ commutators with $\Rt{\beta}$ elements (\Cnref{rel:b3-large:order:beta:alpha+beta+2psi}):}
        &= \Eld{\alpha+\beta+2\psi}{u}{j} \Eld{\psi}{-1/2}{i_2} \Eld{\beta}{t}{i_1} \asite{\Eld{\alpha+2\beta+2\psi}{tu}{i_1+j} \Eld{\psi}{1}{i_2} \Eld{\alpha+2\beta+2\psi}{-tu}{i_1+j}} \Eld{\beta}{-t}{i_1} \Eld{\psi}{-1/2}{i_2} \\
        \intertext{Moving the $\Rt{\alpha+2\beta+2\psi}$ elements together across the $\Rt{\psi}$ element (\Cnref{rel:b3-large:comm:psi:alpha+2beta+2psi}) and canceling them (\Cnref{rel:b3-large:inv-doub:alpha+2beta+2psi}):}
        &= \Eld{\alpha+\beta+2\psi}{u}{j} \asite{\Eld{\psi}{-1/2}{i_2} \Eld{\beta}{t}{i_1} \Eld{\psi}{1}{i_2} \Eld{\beta}{-t}{i_1} \Eld{\psi}{-1/2}{i_2}} \\
        \intertext{Reducing the $\Rt{\beta+\psi}$ element (\Cnref{rel:b3-large:sub:expr:beta+psi}):}
        &= \Eld{\alpha+\beta+2\psi}{u}{j} \Eld{\beta+\psi}{t}{i}.
    \end{align*}
\end{proof}

\begin{relation}[\RelNameBLg\RelNameCommute{\Rt{\beta+2\psi}}{\Rt{\alpha+\beta+2\psi}}]\label{rel:b3-large:comm:beta+2psi:alpha+beta+2psi}
        \[
    \DQuant{i \in [3], j \in [4]}{t,u \in \LiftScalars} \Comm{\Eld{\beta+2\psi}{t}{i}}{\Eld{\alpha+\beta+2\psi}{u}{j}} = \Id.
    \]
\end{relation}

\begin{proof}
    By \Cnref{rel:b3-large:sub:expr:beta+2psi}, $\Rt{\beta+2\psi}$ elements may be expressed as products of $\Rt{\psi}$ elements and $\Rt{\beta+\psi}$ elements. These types of elements commute with $\Rt{\alpha+\beta+2\psi}$ elements by \Cnref{rel:b3-large:comm:psi:alpha+beta+2psi} and \Cnref{rel:b3-large:comm:beta+psi:alpha+beta+2psi}, respectively.
\end{proof}

\begin{relation}[\RelNameBLg\RelNameCommute{\Rt{\alpha+\beta+\psi}}{\Rt{\alpha+\beta+2\psi}}]\label{rel:b3-large:comm:alpha+beta+psi:alpha+beta+2psi}
        \[
    \DQuant{i \in [3], j \in [4]}{t,u \in \LiftScalars} \Comm{\Eld{\alpha+\beta+\psi}{t}{i}}{\Eld{\alpha+\beta+2\psi}{u}{j}} = \Id.
    \]
\end{relation}

\begin{proof}
    By \Cnref{prop:b3-large:est:alpha+beta+psi}, $\Rt{\alpha+\beta+\psi}$ elements may be expressed as products of $\Rt{\alpha}$ elements and $\Rt{\beta+\psi}$ elements. These types of elements commute with $\Rt{\alpha+\beta+2\psi}$ elements by \Cnref{rel:b3-large:comm:alpha:alpha+beta+2psi} and \Cnref{rel:b3-large:comm:beta+psi:alpha+beta+2psi}, respectively.
\end{proof}

\begin{relation}[\RelNameBLg\RelNameSelfCommute{\Rt{\alpha+\beta+2\psi}}]\label{rel:b3-large:comm:self:alpha+beta+2psi}
        \[
    \DQuant{i \in [4], j \in [4]}{t,u \in \LiftScalars} \Comm{\Eld{\alpha+\beta+2\psi}{t}{i}}{\Eld{\alpha+\beta+2\psi}{u}{j}} = \Id.
    \]
\end{relation}

\begin{proof}
    By \Cnref{rel:b3-large:expr:alpha+beta+2psi:alpha+beta+psi:psi}, $\Rt{\alpha+\beta+2\psi}$ elements may be expressed as products of $\Rt{\psi}$ elements and $\Rt{\alpha+\beta+\psi}$ elements. These types of elements commute with $\Rt{\alpha+\beta+2\psi}$ elements by \Cnref{rel:b3-large:comm:psi:alpha+beta+2psi} and \Cnref{rel:b3-large:comm:alpha+beta+psi:alpha+beta+2psi}, respectively.
\end{proof}

\subsubsection{Linearity relations for $\alpha+\beta+2\psi$ and $\alpha+2\beta+2\psi$}

The next two propositions use the ``standard'' linearity proof for commutators:

\begin{relation}[\RelNameBLg\RelNameLinearity{\Rt{\alpha+\beta+2\psi}}]\label{rel:b3-large:lin:alpha+beta+2psi}
        \[
    \DQuant{i \in [4]}{t,u \in \LiftScalars} \Eld{\alpha+\beta+2\psi}{t}{i}
    \Eld{\alpha+\beta+2\psi}{u}{i} =
    \Eld{\alpha+\beta+2\psi}{t+u}{i}.
    \]
\end{relation}

\begin{proof}
    Decompose arbitrarily $i = i_1+i_2$ for $i_1 \in [1], i_2 \in [3]$. We write a product of $\Rt{\alpha+\beta+2\psi}$ elements:
    \begin{align*}
    & \asite{\Eld{\alpha+\beta+2\psi}{t}{i}} \Eld{\alpha+\beta+2\psi}{u}{i} \\
    \intertext{Expanding one $\Rt{\alpha+\beta+2\psi}$ into a product of $\Rt{\alpha}$ and $\Rt{\beta+2\psi}$ elements with \Cnref{rel:b3-large:expr:alpha+beta+2psi:alpha:beta+2psi}:}
    &= \Eld{\alpha}{t}{i_1} \Eld{\beta+2\psi}{1}{i_2} \Eld{\alpha}{-t}{i_1} \Eld{\beta+2\psi}{-1}{i_2} \asite{\Eld{\alpha+\beta+2\psi}{u}{i}} \\
    \intertext{Moving the other $\Rt{\alpha+\beta+2\psi}$ element to the left since it commutes with $\Rt{\alpha}$ elements (\Cnref{rel:b3-large:comm:alpha:alpha+beta+2psi}) and $\Rt{\beta+2\psi}$ elements (\Cnref{rel:b3-large:comm:beta+2psi:alpha+beta+2psi}):}
    &= \Eld{\alpha}{t}{i_1} \asite{\Eld{\alpha+\beta+2\psi}{u}{i}} \Eld{\beta+2\psi}{1}{i_2} \Eld{\alpha}{-t}{i_1} \Eld{\beta+2\psi}{-1}{i_2} \\
    \intertext{Expanding the other $\Rt{\alpha+\beta+2\psi}$ element (\Cnref{rel:b3-large:expr:alpha+beta+2psi:alpha:beta+2psi}):}
    &= \asite{\Eld{\alpha}{t}{i_1} \Eld{\alpha}{u}{i_1}}  \Eld{\beta+2\psi}{1}{i_2} \asite{\Eld{\alpha}{-u}{i_1} \Eld{\beta+2\psi}{-1}{i_2} \Eld{\beta+2\psi}{1}{i_2} \Eld{\alpha}{-t}{i_1}} \Eld{\beta+2\psi}{-1}{i_2} \\
    \intertext{Using inverses for $\Rt{\beta+2\psi}$ elements (\Cnref{rel:b3-large:sub:inv:beta+2psi}) and linearity for $\Rt{\alpha}$ elements (\Cnref{rel:b3-large:sub:lin:alpha}):}
    &= \Eld{\alpha}{t+u}{i_1} \Eld{\beta+2\psi}{1}{i_2} \Eld{\alpha}{-(t+u)}{i_1} \Eld{\beta+2\psi}{-1}{i_2} \\
    \intertext{Reducing into an $\Rt{\alpha+\beta+2\psi}$ element (\Cnref{rel:b3-large:expr:alpha+beta+2psi:alpha:beta+2psi}):}
    &= \Eld{\alpha+\beta+2\psi}{t+u}{i}.
    \end{align*}
\end{proof}

\begin{relation}[\RelNameBLg\RelNameLinearity{\Rt{\alpha+2\beta+2\psi}}]\label{rel:b3-large:lin:alpha+2beta+2psi}
        \[
    \DQuant{i \in [5]}{t,u \in \LiftScalars} \Eld{\alpha+2\beta+2\psi}{t}{i}
    \Eld{\alpha+2\beta+2\psi}{u}{i} =
    \Eld{\alpha+2\beta+2\psi}{t+u}{i}.
    \]
\end{relation}

\begin{proof}
    Decompose arbitrarily $i = i_1+i_2$ for $i_1 \in [2], i_2 \in [3]$. We write a product of $\Rt{\alpha+2\beta+2\psi}$ elements:
    \begin{align*}
    & \asite{\Eld{\alpha+2\beta+2\psi}{t}{i}} \Eld{\alpha+2\beta+2\psi}{u}{i} \\
    \intertext{Expanding one $\Rt{\alpha+2\beta+2\psi}$ element into a product of $\Rt{\alpha+\beta}$ and $\Rt{\beta+2\psi}$ elements with \Cnref{rel:b3-large:expr:alpha+2beta+2psi:alpha+beta:beta+2psi}:}
    &= \Eld{\alpha+\beta}{t}{i_1} \Eld{\beta+2\psi}{-1}{i_2} \Eld{\alpha+\beta}{-t}{i_1} \Eld{\beta+2\psi}{1}{i_2} \asite{\Eld{\alpha+2\beta+2\psi}{u}{i}} \\
    \intertext{Moving the other $\Rt{\alpha+2\beta+2\psi}$ element to the left since it commutes with $\Rt{\alpha+\beta}$ elements (\Cnref{rel:b3-large:comm:alpha+beta:alpha+2beta+2psi}) and $\Rt{\beta+2\psi}$ elements (\Cnref{rel:b3-large:comm:beta+2psi:alpha+2beta+2psi}):}
    &= \Eld{\alpha+\beta}{t}{i_1} \asite{\Eld{\alpha+2\beta+2\psi}{u}{i}} \Eld{\beta+2\psi}{-1}{i_2} \Eld{\alpha+\beta}{-t}{i_1} \Eld{\beta+2\psi}{1}{i_2} \\
    \intertext{Expanding the other $\Rt{\alpha+2\beta+2\psi}$ element (\Cnref{rel:b3-large:expr:alpha+2beta+2psi:alpha+beta:beta+2psi}):}
    &= \asite{\Eld{\alpha+\beta}{t}{i_1} \Eld{\alpha+\beta}{u}{i_1}}  \Eld{\beta+2\psi}{-1}{i_2} \asite{\Eld{\alpha+\beta}{-u}{i_1} \Eld{\beta+2\psi}{-1}{i_2} \Eld{\beta+2\psi}{1}{i_2} \Eld{\alpha+\beta}{-t}{i_1}} \Eld{\beta+2\psi}{1}{i_2} \\
    \intertext{Using inverses for $\Rt{\beta+2\psi}$ elements (\Cnref{rel:b3-large:sub:inv:beta+2psi}) and linearity for $\Rt{\alpha+\beta}$ elements (\Cnref{rel:b3-large:sub:lin:alpha+beta}):}
    &= \Eld{\alpha}{t+u}{i_1} \Eld{\beta+2\psi}{-1}{i_2} \Eld{\alpha}{-(t+u)}{i_1} \Eld{\beta+2\psi}{1}{i_2} \\
    \intertext{Reducing into an $\Rt{\alpha+2\beta+2\psi}$ element (\Cnref{rel:b3-large:expr:alpha+2beta+2psi:alpha+beta:beta+2psi}):}
    &= \Eld{\alpha+2\beta+2\psi}{t+u}{i}.
    \end{align*}
\end{proof}

\subsubsection{The more difficult relations}\label{sec:b3-large:remaining-relations}

The remaining Steinberg relations are:

\begin{itemize}
    \item $\Rt{\alpha+\beta+\psi}$ self-commutes.
    \item $\Rt{\alpha+\beta+2\psi}$ and $\Rt{\alpha+\beta}$ commute.
    \item $\Rt{\alpha+2\beta+2\psi}$ and $\Rt{\alpha}$ commute.
\end{itemize}

If we inspect the corresponding columns in \Cref{tab:b3-large:homog}, we see that homogeneous lifting could give us at most $4/16$, $6/18$, and $6/12$ degree-pairs, respectively, for these relations.\footnote{ ``At most'' is important here: One also has to consider which pairs of \emph{scalar} arguments can be generated via lifting. For instance, lifting a self-commutation relation always gives the same scalars on both arguments.} However, there are some issues stemming from the fact that the \emph{differences} between these pairs of roots --- namely, $0$, $2\psi$, and $2\beta+2\psi$, respectively --- are all ``even''. For instance, consider the commutator of $\Rt{\alpha+2\beta+2\psi}$ and $\Rt{\alpha}$. The image of the base relation \[ \Comm{\El{\alpha}{1}}{\El{\alpha+2\beta+2\psi}{1}} = \Id \] under a homogeneous lift is \[ \Comm{\Eld{\alpha}{a}{i}}{\Eld{\alpha+2\beta+2\psi}{ab^2c^2}{i+2j+2k}} = \Id. \] There are two things to notice about this lifted relation:
\begin{enumerate}
    \item The ratio between the scalars is $\frac{ab^2c^2}{a} = (bc)^2$, a square.
    \item The difference between the degrees is $i+2j+2k-i = 2(j+k)$, which is even.
\end{enumerate}

The former causes an issue because it means that direct homogeneous lifting does not immediately give that for \emph{every} $t,u \in \LiftScalars$, \[ \Comm{\Eld{\alpha}{t}{i}}{\Eld{\alpha+2\beta+2\psi}{u}{i+2j+2k}}, \] i.e., it only gives it in cases where the ratio of $t$ and $u$ is a square. We can deal with this, though, using the fact that every element in a finite field is the sum of two squares (\Cref{prop:prelim:two-qrs} below) and the linearity relation we have already proven for $\Rt{\alpha+2\beta+2\psi}$ (\Cnref{rel:b3-large:lin:alpha+2beta+2psi}). The latter also causes an issue because the natural ways to prove an implication between these commutators will preserve the parity of the difference between the degrees, as we will see below. So, we have to resort to \emph{nonhomogeneous} lifting at this point.


\begin{relation}[\RelNameBLg\RelNameHomLift\RelNameCommute{\Rt{\alpha}}{\Rt{\alpha+2\beta+2\psi}}\RelSquare]\label{rel:b3-large:comm:lift:alpha:alpha+2beta+2psi:square}
    \[
    \Quant{i,j,k \in [1]}{t,r \in \LiftScalars} \Comm{\Eld{\alpha}{t}{i}}{\Eld{\alpha+2\beta+2\psi}{tr^2}{i+2j+2k}}.
    \]
\end{relation}

\begin{proof}
    Use \Cnref{rel:b3-large:comm:raw:lift:alpha:alpha+2beta+2psi:square} with $(t_1,t_0,u_1,u_0,v_1,v_0)$ such that $t_i = t,t_{1-i}=0, u_j = 1, u_{1-j} = 0, v_k = r, v_{1-k} = 0$ and \Cnref{rel:b3-large:expr:alpha+2beta+2psi:alpha+beta:beta+2psi}.
\end{proof}

The following elementary fact is discussed in, e.g.,~\cite{BBI21}:
\begin{proposition}\label{prop:prelim:two-qrs}
    In any finite field $\BF_{q^k}$, every element is the sum of two squares.
\end{proposition}

\begin{relation}[\RelNameBLg\RelNameHomLift\RelNameCommute{\Rt{\alpha}}{\Rt{\alpha+2\beta+2\psi}}]\label{rel:b3-large:comm:lift:alpha:alpha+2beta+2psi:homog}
    \[
    \DQuant{i,j,k \in [1]}{t,u \in \LiftScalars} \Comm{\Eld{\alpha}{t}{i}}{\Eld{\alpha+2\beta+2\psi}{u}{i+2j+2k}}.
    \]
\end{relation}

\begin{proof}
    If $t = 0$, then $\Eld{\alpha}{t}{i} = \Id$ (\Cnref{rel:b3-large:sub:id:alpha}) and we have nothing to prove. Otherwise, let $u/t = r^2 + s^2$ (using \Cref{prop:prelim:two-qrs}), so that $tr^2 + ts^2 = u$. By the previous proposition (\Cnref{rel:b3-large:comm:lift:alpha:alpha+2beta+2psi:square}), \[
    \Comm{\Eld{\alpha}{t}{i}}{\Eld{\alpha+2\beta+2\psi}{tr^2}{i+2j+2k}} = \Comm{\Eld{\alpha}{t}{i}}{\Eld{\alpha+2\beta+2\psi}{ts^2}{i+2j+2k}} = \Id.
    \] And by \Cnref{rel:b3-large:lin:alpha+2beta+2psi},
    \[
    \Comm{\Eld{\alpha}{t}{i}}{\Eld{\alpha+2\beta+2\psi}{u}{i+2j+2k}} = \Comm{\Eld{\alpha}{t}{i}}{\Eld{\alpha+2\beta+2\psi}{tr^2}{i+2j+2k} \Eld{\alpha+2\beta+2\psi}{ts^2}{i+2j+2k}}.
    \]
\end{proof}

\begin{relation}[\RelNameBLg\RelNameNonHomLift\RelNameCommute{\Rt{\alpha}}{\Rt{\alpha+2\beta+2\psi}}]\label{rel:b3-large:comm:lift:alpha:alpha+2beta+2psi:nonhomog}
\[
\DQuant{i,j \in [1]}{t,u \in \LiftScalars} \Comm{\Eld{\alpha}{t}{i}}{\Eld{\alpha+2\beta+2\psi}{u}{i+2j+1}} = \Id.
\]
\end{relation}

\begin{proof}
If $t = 0$, then $\Eld{\alpha}{t}{i} = \Id$ (\Cnref{rel:b3-large:sub:id:alpha}) and we have nothing to prove. Otherwise, use \Cnref{rel:b3-large:comm:raw:lift:alpha:alpha+2beta+2psi:square} with $t_i = t, t_{1-i} = 0, u_j = 1, u_{1-j} = 0, v_1 = 1, v_0 = \frac{u}{2t}$. Note that under this lift, the $\Rt{\alpha}$, $\Rt{\alpha+\beta}$, and $\Rt{\beta+2\psi}$ terms are
    \[ \Eld{\alpha}{t}{i}, \Eld{\alpha+\beta}{t}{i+j}, \Eld{\beta+2\psi}{1}{j+2} \Eld{\beta+2\psi}{u/t}{j+1} \Eld{\beta+2\psi}{u^2/(4t^2)}{j},
    \]
    respectively, and we can reorder the $\Rt{\beta+2\psi}$ terms arbitrarily by \Cnref{rel:b3-large:sub:comm:self:beta+2psi}. Thus, the commutator in \Cnref{rel:b3-large:comm:raw:lift:alpha:alpha+2beta+2psi:square} becomes via \Cnref{rel:b3-large:sub:inv:alpha}, \Cnref{rel:b3-large:sub:inv:alpha+beta}, and \Cnref{rel:b3-large:sub:inv:beta+2psi}:
    
    \begin{align*}
    \Id &= \Eld{\alpha}{t}{i} \Eld{\alpha+\beta}{t}{i+j}  \Eld{\beta+2\psi}{u/t}{j+1} \Eld{\beta+2\psi}{1}{j+2} \asite{\Eld{\beta+2\psi}{u^2/(4t^2)}{j} \Eld{\alpha+\beta}{-t}{i+j} \Eld{\beta+2\psi}{-u^2/(4t^2)}{j}} \Eld{\beta+2\psi}{-1}{j+2} \Eld{\beta+2\psi}{-u/t}{j+1} \\
    & \hspace{.2in} \cdot \Eld{\alpha}{-t}{i} \Eld{\alpha+\beta}{-t}{i+j} \Eld{\beta+2\psi}{u/t}{j+1} \Eld{\beta+2\psi}{1}{j+2} \asite{\Eld{\beta+2\psi}{u^2/(4t^2)}{j} \Eld{\alpha+\beta}{t}{i+j} \Eld{\beta+2\psi}{-u^2/(4t^2)}{j}} \Eld{\beta+2\psi}{-1}{j+2} \Eld{\beta+2\psi}{-u/t}{j+1} \\
    \intertext{Moving pairs of $\Eld{\beta+2\psi}{\pm u^2/(4t^2)}{j}$ elements together across $\Eld{\alpha+\beta}{\pm t}{i+j}$, creating $\Eld{\alpha+2\beta+2\psi}{\pm u^2/(4t)}{i+2j}$ commutators (\Cnref{rel:b3-large:order:alpha+beta:beta+2psi}), then cancelling (\Cnref{rel:b3-large:sub:inv:beta+2psi}):}
    &= \Eld{\alpha}{t}{i} \Eld{\alpha+\beta}{t}{i+j} \Eld{\beta+2\psi}{u/t}{j+1} \Eld{\beta+2\psi}{1}{j+2} \asite{\Eld{\alpha+2\beta+2\psi}{-u^2/(4t)}{i+2j}} \Eld{\alpha+\beta}{-t}{i+j} \Eld{\beta+2\psi}{-1}{j+2} \Eld{\beta+2\psi}{-u/t}{j+1} \\
    & \hspace{.2in} \cdot \Eld{\alpha}{-t}{i} \Eld{\alpha+\beta}{-t}{i+j} \Eld{\beta+2\psi}{u/t}{j+1} \Eld{\beta+2\psi}{1}{j+2} \asite{\Eld{\alpha+2\beta+2\psi}{u^2/(4t)}{i+2j}} \Eld{\alpha+\beta}{t}{i+j} \Eld{\beta+2\psi}{-1}{j+2} \Eld{\beta+2\psi}{-u/t}{j+1} \\
    \intertext{The newly created $\Rt{\alpha+2\beta+2\psi}$ elements commute with the $\Rt{\alpha+\beta}$, $\Rt{\beta+2\psi}$, and $\Rt{\alpha}$ elements using \Cnref{rel:b3-large:comm:alpha+beta:alpha+2beta+2psi}, \Cnref{rel:b3-large:comm:beta+2psi:alpha+2beta+2psi}, and \Cnref{rel:b3-large:comm:lift:alpha:alpha+2beta+2psi:homog} respectively (note that in the third case, we need that $i+2j$ and $i$ have the same parity!), so we can cancel them (\Cnref{rel:b3-large:inv-doub:alpha+2beta+2psi}):}
    &= \Eld{\alpha}{t}{i} \Eld{\alpha+\beta}{t}{i+j} \Eld{\beta+2\psi}{u/t}{j+1} \asite{\Eld{\beta+2\psi}{1}{j+2} \Eld{\alpha+\beta}{-t}{i+j} \Eld{\beta+2\psi}{-1}{j+2}} \Eld{\beta+2\psi}{-u/t}{j+1} \\
    & \hspace{.2in} \cdot \Eld{\alpha}{-t}{i} \Eld{\alpha+\beta}{-t}{i+j} \Eld{\beta+2\psi}{u/t}{j+1} \asite{\Eld{\beta+2\psi}{1}{j+2} \Eld{\alpha+\beta}{t}{i+j} \Eld{\beta+2\psi}{-1}{j+2}} \Eld{\beta+2\psi}{-u/t}{j+1} \\
    \intertext{Moving pairs of $\Eld{\beta+2\psi}{\pm u/t}{j+2}$ elements together across $\Eld{\alpha+\beta}{\pm t}{i+j}$, creating $\Eld{\alpha+2\beta+2\psi}{\pm t}{i+2j}$ commutators (\Cnref{rel:b3-large:order:alpha+beta:beta+2psi}), then cancelling (\Cnref{rel:b3-large:sub:inv:beta+2psi}):}
    &= \Eld{\alpha}{t}{i} \Eld{\alpha+\beta}{t}{i+j} \Eld{\beta+2\psi}{u/t}{j+1} \asite{\Eld{\alpha+2\beta+2\psi}{-t}{i+2j+2}} \Eld{\alpha+\beta}{-t}{i+j} \Eld{\beta+2\psi}{-u/t}{j+1} \\
    & \hspace{.2in} \cdot \Eld{\alpha}{t}{i} \Eld{\alpha+\beta}{-t}{i+j} \Eld{\beta+2\psi}{u/t}{j+1} \asite{\Eld{\alpha+2\beta+2\psi}{-t}{i+2j+2}}\Eld{\alpha+\beta}{t}{i+j} \Eld{\beta+2\psi}{-u/t}{j+1} \\
    \intertext{Again, the newly created $\Rt{\alpha+2\beta+2\psi}$ elements commute with the $\Rt{\alpha+\beta}$, $\Rt{\beta+2\psi}$, and $\Rt{\alpha}$ elements using \Cnref{rel:b3-large:comm:alpha+beta:alpha+2beta+2psi}, \Cnref{rel:b3-large:comm:beta+2psi:alpha+2beta+2psi}, and \Cnref{rel:b3-large:comm:lift:alpha:alpha+2beta+2psi:homog} respectively ($i+2j+2$ and $i$ have the same parity), so we can cancel them (\Cnref{rel:b3-large:inv-doub:alpha+2beta+2psi}):}
    &= \Eld{\alpha}{t}{i} \Eld{\alpha+\beta}{t}{i+j} \Eld{\beta+2\psi}{u/t}{j+1} \Eld{\alpha+\beta}{-t}{i+j} \Eld{\beta+2\psi}{-u/t}{j+1} \Eld{\alpha}{-t}{i} \Eld{\alpha+\beta}{-t}{i+j} \Eld{\beta+2\psi}{u/t}{j+1} \Eld{\alpha+\beta}{t}{i+j} \Eld{\beta+2\psi}{-u/t}{j+1} \\
    \intertext{Reducing into $\Rt{\alpha+2\beta+2\psi}$ elements (\Cnref{rel:b3-large:expr:alpha+2beta+2psi:alpha+beta:beta+2psi}) and applying \Cnref{rel:b3-large:sub:inv:alpha} and \Cnref{rel:b3-large:inv-doub:alpha+2beta+2psi}:}
    &= \Comm{\Eld{\alpha}{t}{i}}{\Eld{\beta+2\psi}{u}{i+2j+1}}.
    \end{align*}
\end{proof}

\begin{table}
    \centering
    \begin{tabular}{|c|c||c|}
    \hline
        $i$ & $j$ & $i+2j+1$ \\ \hline \hline
        $0$ & $0$ & $1$ \\ \hline
        $0$ & $1$ & $3$ \\ \hline
        $1$ & $0$ & $2$ \\ \hline
        $1$ & $1$ & $4$ \\ \hline
    \end{tabular}
    \caption{The trace of \Cnref{rel:b3-large:comm:lift:alpha:alpha+2beta+2psi:nonhomog}. Observe that between \Cnref{rel:b3-large:comm:lift:alpha:alpha+2beta+2psi:homog} and \Cnref{rel:b3-large:comm:lift:alpha:alpha+2beta+2psi:nonhomog} we are still missing the pairs $(1,0)$ and $(0,5)$.}\label{tab:b3-large:comm:lift:alpha:alpha+2beta+2psi:nonhomog}
\end{table}

\begin{proposition}[Sufficient conditions for commutator of $\Rt{\alpha+\beta}$ and $\Rt{\alpha+\beta+2\psi}$]\label{rel:b3-large:comm:impl:alpha+beta:alpha+beta+2psi}
    Let $i,j \in [1], k \in [4]$. Suppose that:
    \[
    \Quant{t,u \in \LiftScalars} \Comm{\Eld{\alpha}{t}{i}}{\Eld{\alpha+2\beta+2\psi}{u}{j+k}} = \Id.
    \]
    Then:
    \[
    \Quant{t,u \in \LiftScalars} \Comm{\Eld{\alpha+\beta}{t}{i+j}}{\Eld{\alpha+\beta+2\psi}{u}{k}} = \Id.
    \]
\end{proposition}

\begin{proof}
    We write a product of $\Rt{\alpha+\beta}$ and $\Rt{\alpha+\beta+2\psi}$ elements:
    \begin{align*}
        & \asite{\Eld{\alpha+\beta}{t}{i+j}} \Eld{\alpha+\beta+2\psi}{u}{k} \\
        \intertext{Expanding the $\Rt{\alpha+\beta}$ element into a product of $\Rt{\alpha}$ and $\Rt{\beta}$ elements (\Cnref{rel:b3-large:sub:expr:alpha+beta}):}
        &= \Eld{\alpha}{t}{i} \Eld{\beta}{1}{j} \Eld{\alpha}{-t}{i} \Eld{\beta}{-1}{j} \asite{\Eld{\alpha+\beta+2\psi}{u}{k}} \\
        \intertext{Moving the $\Rt{\alpha+\beta+2\psi}$ element to the left, creating no commutators with $\Rt{\alpha}$ elements (\Cnref{rel:b3-large:comm:alpha:alpha+beta+2psi}) and $\Rt{\alpha+2\beta+2\psi}$ commutators with $\Rt{\beta}$ elements (\Cnref{rel:b3-large:order:beta:alpha+beta+2psi}):}
        &= \Eld{\alpha+\beta+2\psi}{u}{k} \Eld{\alpha}{t}{i} \Eld{\beta}{1}{j} \asite{\Eld{\alpha+2\beta+2\psi}{tu}{i+k} \Eld{\alpha}{-t}{i} \Eld{\alpha+2\beta+2\psi}{-tu}{i+k}} \Eld{\beta}{-1}{j} \\
        \intertext{Canceling the $\Rt{\alpha+2\beta+2\psi}$ elements (\Cnref{rel:b3-large:inv-doub:alpha+2beta+2psi}) since they commute with $\Rt{\alpha}$ elements (\Cnref{rel:b3-large:comm:alpha:alpha+2beta+2psi}):}
        &= \Eld{\alpha+\beta+2\psi}{u}{k} \asite{\Eld{\alpha}{t}{i} \Eld{\beta}{1}{j} \Eld{\alpha}{-t}{i} \Eld{\beta}{-1}{j}} \\
        \intertext{Reducing into an $\Rt{\alpha+\beta}$ element (\Cnref{rel:b3-large:sub:expr:alpha+beta}):}
        &= \Eld{\alpha+\beta+2\psi}{u}{k} \Eld{\alpha+\beta}{t}{i+j}.
    \end{align*}
\end{proof}

\begin{relation}[\RelNameBLg\RelNamePart\RelNameCommute{\Rt{\alpha+\beta}}{\Rt{\alpha+2\beta+2\psi}}]\label{cor:b3-large:comm:spec:alpha+beta:alpha+beta+2psi:01}
    \[ \Quant{t,u \in \LiftScalars} \Comm{\Eld{\alpha+\beta}{t}{0}}{\Eld{\alpha+\beta+2\psi}{u}{1}} = \Id. \]
\end{relation}

\begin{proof}
    We apply the previous proposition (\Cnref{rel:b3-large:comm:impl:alpha+beta:alpha+beta+2psi}) with $i=0,j=0,k=1$. This gives the desideratum if 
    \[
    \Quant{t,u \in \LiftScalars} \Comm{\Eld{\alpha}{t}{0}}{\Eld{\alpha+2\beta+2\psi}{u}{1}} = \Id,
    \]
    which follows from nonhomogeneous lifting (\Cnref{rel:b3-large:comm:lift:alpha:alpha+2beta+2psi:nonhomog} with $i=j=0$, see also \Cref{tab:b3-large:comm:lift:alpha:alpha+2beta+2psi:nonhomog}).
\end{proof}

\begin{proposition}[Sufficient conditions for commutator of $\Rt{\alpha}$ and $\Rt{\alpha+2\beta+2\psi}$]\label{rel:b3-large:comm:impl:alpha:alpha+2beta+2psi}
    Let $i \in [1], j \in [2], k \in [3]$. Suppose that:
    \[
    \Quant{t,u \in \LiftScalars} \Comm{\Eld{\alpha+\beta}{t}{j}}{\Eld{\alpha+\beta+2\psi}{u}{i+k}} = \Id.
    \]
    Then:
    \[
    \Quant{t,u \in \LiftScalars} \Comm{\Eld{\alpha}{t}{i}}{\Eld{\alpha+2\beta+2\psi}{u}{j+k}} = \Id.
    \]
\end{proposition}

\begin{proof}
    We write a product of $\Rt{\alpha}$ and $\Rt{\alpha+2\beta+2\psi}$ elements:
    \begin{align*}
        & \Eld{\alpha}{t}{i} \asite{\Eld{\alpha+2\beta+2\psi}{u}{j+k}} \\
        \intertext{Expanding the $\Rt{\alpha+2\beta+2\psi}$ element into a product of $\Rt{\alpha+\beta}$ and $\Rt{\beta+2\psi}$ elements (\Cnref{rel:b3-large:expr:alpha+2beta+2psi:alpha+beta:beta+2psi}):}
        &= \asite{\Eld{\alpha}{t}{i}} \Eld{\alpha+\beta}{u}{j} \Eld{\beta+2\psi}{-1}{k} \Eld{\alpha+\beta}{-u}{j} \Eld{\beta+2\psi}{1}{k} \\
        \intertext{Moving the $\Rt{\alpha}$ element fully to the right, creating no commutators with $\Rt{\alpha+\beta}$ elements (\Cnref{rel:b3-large:sub:comm:alpha:alpha+beta}) and $\Rt{\alpha+\beta+2\psi}$ commutators with $\Rt{\beta+2\psi}$ elements (\Cnref{rel:b3-large:order:alpha:beta+2psi}):}
        &= \Eld{\alpha+\beta}{u}{j} \Eld{\beta+2\psi}{-1}{k} \asite{\Eld{\alpha+\beta+2\psi}{-t}{i+k} \Eld{\alpha+\beta}{-u}{j} \Eld{\alpha+\beta+2\psi}{t}{i+k}} \Eld{\beta+2\psi}{1}{k} \Eld{\alpha}{t}{i} \\
        \intertext{Moving the $\Rt{\alpha+\beta+2\psi}$ elements together across the $\Rt{\alpha+\beta}$ element by assumption and cancelling them (\Cnref{rel:b3-large:inv-doub:alpha+beta+2psi}):}
        &= \asite{\Eld{\alpha+\beta}{u}{j} \Eld{\beta+2\psi}{-1}{k} \Eld{\alpha+\beta}{-u}{j} \Eld{\beta+2\psi}{1}{k}} \Eld{\alpha}{t}{i} \\
        \intertext{Reducing into an $\Rt{\alpha+2\beta+2\psi}$ element (\Cnref{rel:b3-large:expr:alpha+2beta+2psi:alpha+beta:beta+2psi}):}
        &= \Eld{\alpha+2\beta+2\psi}{u}{j+k} \Eld{\alpha}{t}{i}.
    \end{align*}
\end{proof}

\begin{relation}[\RelNameBLg\RelNamePart\RelNameCommute{\Rt{\alpha}}{\Rt{\alpha+2\beta+2\psi}}]\label{cor:b3-large:comm:spec:alpha:alpha+2beta+2psi:10}
    \[
    \Quant{t,u \in \LiftScalars} \Comm{\Eld{\alpha}{t}{1}}{\Eld{\alpha+2\beta+2\psi}{u}{0}} = \Id.
    \]
\end{relation}

\begin{proof}
    Using the previous proposition (\Cnref{rel:b3-large:comm:impl:alpha:alpha+2beta+2psi}) with $i=1,j=0,k=0$, we get the desideratum if:
    \[
    \Quant{t,u \in \LiftScalars} \Comm{\Eld{\alpha+\beta}{t}{0}}{\Eld{\alpha+\beta+2\psi}{u}{1}} = \Id.
    \]
    But this is precisely \Cnref{cor:b3-large:comm:spec:alpha+beta:alpha+beta+2psi:01}.
\end{proof}

\begin{bigrelation}[\RelNameBLg\RelNameCommute{\Rt{\alpha}}{\Rt{\alpha+2\beta+2\psi}}]\label{rel:b3-large:comm:alpha:alpha+2beta+2psi}
\[
\DQuant{i\in[1],j\in[5]}{t,u \in \LiftScalars} \Comm{\Eld{\alpha}{t}{i}}{\Eld{\alpha+2\beta+2\psi}{u}{j}} = \Id.
\]
\end{bigrelation}

\begin{proof}
    The degree-pairs $[1] \times [5]$ are covered by \Cnref{cor:b3-large:comm:spec:alpha:alpha+2beta+2psi:10}, \Cnref{rel:b3-large:comm:lift:alpha:alpha+2beta+2psi:nonhomog}, and \Cnref{rel:b3-large:comm:lift:alpha:alpha+2beta+2psi:homog} (see also \Cref{tab:b3-large:homog} and \Cref{tab:b3-large:comm:lift:alpha:alpha+2beta+2psi:nonhomog}).
\end{proof}

\begin{relation}[\RelNameBLg\RelNameCommute{\Rt{\alpha+\beta}}{\Rt{\alpha+\beta+2\psi}}]\label{rel:b3-large:comm:alpha+beta:alpha+beta+2psi}
\[
\DQuant{i\in[2],j\in[4]}{t,u \in \LiftScalars} \Comm{\Eld{\alpha+\beta}{t}{i}}{\Eld{\alpha+\beta+2\psi}{u}{j}} = \Id.
\]
\end{relation}

\begin{proof}
    Follows from the implication from ``$\Rt{\alpha}$ and $\Rt{\alpha+2\beta+2\psi}$ commute'' to ``$\Rt{\alpha+\beta}$ and $\Rt{\alpha+\beta+2\psi}$ commute'' (\Cnref{rel:b3-large:comm:impl:alpha+beta:alpha+beta+2psi}) and the fact that $\Rt{\alpha}$ and $\Rt{\alpha+2\beta+2\psi}$ commute (\Cnref{rel:b3-large:comm:alpha:alpha+2beta+2psi}).
\end{proof}

\begin{relation}[\RelNameBLg\RelNameSelfCommute{\Rt{\alpha+\beta+\psi}}]\label{rel:b3-large:comm:self:alpha+beta+psi}
        \[
    \DQuant{i,j \in [3]}{t,u \in \LiftScalars} \Comm{\Eld{\alpha+\beta+\psi}{t}{i}}{\Eld{\alpha+\beta+\psi}{u}{j}} = \Id.
    \]
\end{relation}

\begin{proof}
    We write a product of $\Rt{\alpha+\beta+\psi}$ elements:
    \begin{align*}
        & \asite{\Eld{\alpha+\beta+\psi}{t}{i}} \Eld{\alpha+\beta+\psi}{u}{j} \\
        \intertext{Expand the left $\Rt{\alpha+\beta+\psi}$ elements into a product of $\Rt{\alpha}$ and $\Rt{\beta+\psi}$ elements (\Cnref{prop:b3-large:est:alpha+beta+psi}) using an arbitrary decomposition $i = i_1 + i_2$ for $i_1 \in [1], i_2 \in [2]$:}
        &= \Eld{\alpha}{-1/2}{i_1} \Eld{\beta+\psi}{t}{i_2} \Eld{\alpha}{1}{i_1} \Eld{\beta+\psi}{-t}{i_2} \Eld{\alpha}{-1/2}{i_1}  \asite{\Eld{\alpha+\beta+\psi}{u}{j}} \\
        \intertext{Move the right $\Rt{\alpha+\beta+\psi}$ element to the left, creating no commutators with $\Rt{\alpha}$ elements (\Cnref{rel:b3-large:comm:alpha:alpha+beta+psi}) and $\Rt{\alpha+2\beta+2\psi}$ commutators with $\Rt{\beta+\psi}$ elements (\Cnref{rel:b3-large:order:beta+psi:alpha+beta+psi}):}
        &= \Eld{\alpha+\beta+\psi}{u}{j} \Eld{\alpha}{-1/2}{i_1} \Eld{\beta+\psi}{t}{i_2} \asite{\Eld{\alpha+2\beta+2\psi}{2tu}{i_2+j} \Eld{\alpha}{1}{i_1} \Eld{\alpha+2\beta+2\psi}{-2tu}{i_2+j}} \Eld{\beta+\psi}{-t}{i_2} \Eld{\alpha}{-1/2}{i_1} \\
        \intertext{Now we can commute the $\Rt{\alpha+2\beta+2\psi}$ elements together across the $\Rt{\alpha}$ element (\Cnref{rel:b3-large:comm:alpha:alpha+2beta+2psi}) and then cancel them (\Cnref{rel:b3-large:inv-doub:alpha+2beta+2psi}):}
        &= \Eld{\alpha+\beta+\psi}{u}{j} \asite{\Eld{\alpha}{-1/2}{i_1} \Eld{\beta+\psi}{t}{i_2} \Eld{\alpha}{1}{i_1} \Eld{\beta+\psi}{-t}{i_2} \Eld{\alpha}{-1/2}{i_1}} \\
        \intertext{Reducing into an $\Rt{\alpha+\beta+\psi}$ element (\Cnref{prop:b3-large:est:alpha+beta+psi}):}
        &= \Eld{\alpha+\beta+\psi}{u}{j} \Eld{\alpha+\beta+\psi}{t}{i}.
    \end{align*}
\end{proof}

\subsubsection{Linearity for $\alpha+\beta+\psi$}
Do something similar to what we did before but more complicated because the definition has five elements. Anyhow, commute so that $t,u$ and $-t,-u$ are close to each other. 

\begin{relation}[\RelNameBLg\RelNameLinearity{\alpha+\beta+\psi}]\label{rel:b3-large:lin:alpha+beta+psi}
        \[
    \DQuant{i \in [3]}{t,u \in \LiftScalars} \Eld{\alpha+\beta+\psi}{t}{i}
    \Eld{\alpha+\beta+\psi}{u}{i} =
    \Eld{\alpha+\beta+\psi}{t+u}{i}.
    \]
\end{relation}

\begin{proof}
Decompose $i = i_1+i_2$ for $i_1 \in [2],i_2 \in [1]$. We write the product of two $\Rt{\alpha+\beta+\psi}$ elements and compute:
\begin{align*}
    & \asite{\Eld{\alpha+\beta+\psi}{t}{i_1+i_2}} \Eld{\alpha+\beta+\psi}{u}{i_1+i_2} \\
    \intertext{Expanding one $\Rt{\alpha+\beta+\psi}$ element into a product of $\Rt{\alpha+\beta}$ and $\Rt{\psi}$ elements (\Cnref{prop:b3-large:est:alpha+beta+psi}):}
    &= \Eld{\psi}{-1/2}{i_2} \Eld{\alpha+\beta}{t}{i_1} \Eld{\psi}{1}{i_2} \Eld{\alpha+\beta}{-t}{i_1} \Eld{\psi}{-1/2}{i_2} \asite{\Eld{\alpha+\beta+\psi}{u}{i_1+i_2}} \\
    \intertext{Moving the other $\Rt{\alpha+\beta+\psi}$ element from right to \emph{middle}, creating an $\Rt{\alpha+\beta+2\psi}$ commutator with the $\Rt{\psi}$ element (\Cnref{rel:b3-large:order:psi:alpha+beta+psi}) and no commutator with the $\Rt{\alpha+\beta}$ element (\Cnref{rel:b3-large:comm:alpha+beta:alpha+beta+psi}):}
    &= \Eld{\psi}{-1/2}{i_2} \Eld{\alpha+\beta}{t}{i_1} \Eld{\psi}{1}{i_2} \asite{\Eld{\alpha+\beta+\psi}{u}{i_1+i_2}} \Eld{\alpha+\beta}{-t}{i_1} \Eld{\psi}{-1/2}{i_2} \Eld{\alpha+\beta+2\psi}{-u}{i_1+2i_2} \\
    \intertext{Expanding the other $\Rt{\alpha+\beta+\psi}$ element, again with \Cnref{prop:b3-large:est:alpha+beta+psi}:}
    &= \Eld{\psi}{-1/2}{i_2} \Eld{\alpha+\beta}{t}{i_1} \asite{\Eld{\psi}{1}{i_2} \Eld{\psi}{-1/2}{i_2}} \Eld{\alpha+\beta}{u}{i_1} \Eld{\psi}{1}{i_2} \Eld{\alpha+\beta}{-u}{i_1} \Eld{\psi}{-1/2}{i_2} \Eld{\alpha+\beta}{-t}{i_1} \Eld{\psi}{-1/2}{i_2} \Eld{\alpha+\beta+2\psi}{-u}{i_1+2i_2} \\
    \intertext{Simplifying the product of $\Rt{\psi}$ elements (\Cnref{rel:b3-large:sub:lin:psi}):}
    &= \Eld{\psi}{-1/2}{i_2} \Eld{\alpha+\beta}{t}{i_1} \asite{\Eld{\psi}{1/2}{i_2} \Eld{\alpha+\beta}{u}{i_1}} \Eld{\psi}{1}{i_2} \asite{\Eld{\alpha+\beta}{-u}{i_1} \Eld{\psi}{-1/2}{i_2}} \Eld{\alpha+\beta}{-t}{i_1} \Eld{\psi}{-1/2}{i_2} \Eld{\alpha+\beta+2\psi}{-u}{i_1+2i_2} \\
    \intertext{Commuting $\Rt{\psi}$ and $\Rt{\alpha+\beta}$ elements (\Cnref{rel:b3-large:order:alpha+beta:psi} and \Cnref{rel:b3-large:order:psi:alpha+beta}):}
    &= \Eld{\psi}{-1/2}{i_2} \asite{\Eld{\alpha+\beta}{t}{i_1} \Eld{\alpha+\beta}{u}{i_1}} \Eld{\alpha+\beta+2\psi}{-u/4}{i_1+2i_2} \Eld{\alpha+\beta+\psi}{-u/2}{i_1+i_2} \asite{\Eld{\psi}{1/2}{i_2} \Eld{\psi}{1}{i_2} \Eld{\psi}{-1/2}{i_2}} \\
    &\hspace{1in} \cdot \Eld{\alpha+\beta+\psi}{u/2}{i_1+i_2} \Eld{\alpha+\beta+2\psi}{u/4}{i_1+2i_2} \asite{\Eld{\alpha+\beta}{-u}{i_1} \Eld{\alpha+\beta}{-t}{i_1}} \Eld{\psi}{-1/2}{i_2} \Eld{\alpha+\beta+2\psi}{-u}{i_1+2i_2}. \\
    \intertext{Now we can simplify the products of $\Rt{\psi}$ elements (\Cnref{rel:b3-large:sub:lin:psi}) and of $\Rt{\alpha+\beta}$ elements (\Cnref{rel:b3-large:sub:lin:alpha+beta}):}
    &= \Eld{\psi}{-1/2}{i_2} \Eld{\alpha+\beta}{t+u}{i_1}  \Eld{\alpha+\beta+2\psi}{-u/4}{i_1+2i_2} \asite{\Eld{\alpha+\beta+\psi}{-u/2}{i_1+i_2} \Eld{\psi}{1}{i_2}} \\
    &\hspace{1in} \cdot \asite{\Eld{\alpha+\beta+\psi}{u/2}{i_1+i_2}} \Eld{\alpha+\beta+2\psi}{u/4}{i_1+2i_2} \Eld{\alpha+\beta}{-(t+u)}{i_1} \Eld{\psi}{-1/2}{i_2} \Eld{\alpha+\beta+2\psi}{-u}{i_1+2i_2} \\
    \intertext{Commute the $\Rt{\alpha+\beta+\psi}$ elements across the $\Rt{\psi}$ element, creating a $\Rt{\alpha+\beta+2\psi}$ commutator (\Cnref{rel:b3-large:order:psi:alpha+beta+psi}) and cancelling them (\Cnref{rel:b3-large:inv-doub:alpha+beta+psi}):}
    &= \Eld{\psi}{-1/2}{i_2} \Eld{\alpha+\beta}{t+u}{i_1}  \asite{\Eld{\alpha+\beta+2\psi}{-u/4}{i_1+2i_2} \Eld{\alpha+\beta+2\psi}{u}{i_1+2i_2}} \Eld{\psi}{1}{i_2} \\
    &\hspace{1in} \cdot \asite{\Eld{\alpha+\beta+2\psi}{u/4}{i_1+2i_2}} \Eld{\alpha+\beta}{-(t+u)}{i_1} \Eld{\psi}{-1/2}{i_2} \asite{\Eld{\alpha+\beta+2\psi}{-u}{i_1+2i_2}} \\
    \intertext{Now, we can cancel the $\Rt{\alpha+\beta+2\psi}$ elements (\Cnref{rel:b3-large:inv-doub:alpha+beta+2psi}) since they commute across $\Rt{\psi}$ elements (\Cnref{rel:b3-large:comm:psi:alpha+beta+2psi}) and $\Rt{\alpha+\beta}$ elements (\Cnref{rel:b3-large:comm:alpha+beta:alpha+beta+2psi}):}
    &= \asite{\Eld{\psi}{-1/2}{i_2} \Eld{\alpha+\beta}{t+u}{i_1} \Eld{\psi}{1}{i_2} \Eld{\alpha+\beta}{-(t+u)}{i_1} \Eld{\psi}{-1/2}{i_2}} \\
    \intertext{Reducing into an $\Rt{\alpha+\beta+\psi}$ element (\Cnref{prop:b3-large:est:alpha+beta+psi}):}
    &= \Eld{\alpha+\beta+\psi}{t+u}{i_1+i_2}.
\end{align*}
as desired.
\end{proof}

%% file: figures/b3-large/alpha+beta+psi.tex
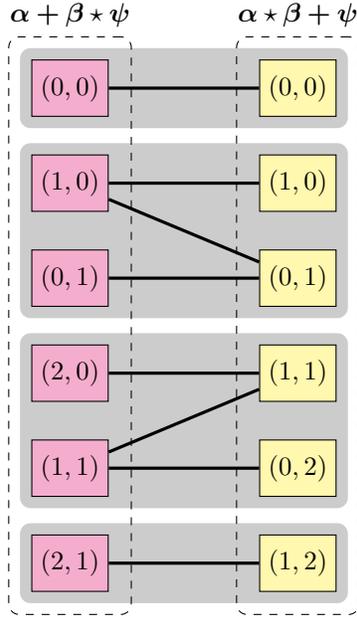
\begin{figure}
    \centering
    \begin{tikzpicture}
\node[algnode0] (v000) {$(0,0)$};
\node[algnode0] (v010) [below=of v000] {$(1,0)$};
\node[algnode0] (v001) [below=of v010] {$(0,1)$};
\node[algnode0] (v020) [below=of v001] {$(2,0)$};
\node[algnode0] (v011) [below=of v020] {$(1,1)$};
\node[algnode0] (v021) [below=of v011] {$(2,1)$};
\node[algnode1] (v100) [right=of v000] {$(0,0)$};
\node[algnode1] (v110) [below=of v100] {$(1,0)$};
\node[algnode1] (v101) [below=of v110] {$(0,1)$};
\node[algnode1] (v111) [below=of v101] {$(1,1)$};
\node[algnode1] (v102) [below=of v111] {$(0,2)$};
\node[algnode1] (v112) [below=of v102] {$(1,2)$};
\draw[algedge] (v011) -- (v111);
\draw[algedge] (v021) -- (v112);
\draw[algedge] (v000) -- (v100);
\draw[algedge] (v010) -- (v110);
\draw[algedge] (v011) -- (v102);
\draw[algedge] (v001) -- (v101);
\draw[algedge] (v010) -- (v101);
\draw[algedge] (v020) -- (v111);
\begin{scope}[on background layer]
    \node[algbackh, fit={(v000) (v100)}] {};
    \node[algbackh, fit={(v010) (v101)}] {};
    \node[algbackh, fit={(v020) (v102)}] {};
    \node[algbackh, fit={(v021) (v112)}] {};
    \node[algback, fit={(v000) (v021)}] (b0) {};
    \node[above=0 of b0] {$\Rt{\alpha+\beta} \star \Rt{\psi}$};
    \node[algback, fit={(v100) (v112)}] (b1) {};
    \node[above=0 of b1] {$\Rt{\alpha} \star \Rt{\beta+\psi}$};
\end{scope}
\end{tikzpicture}
\caption{Establishing $\Rt{\alpha+\beta+\psi}$: A bipartite graph with left vertex-set $[2] \times [1]$ and right vertex-set $[1] \times [2]$, with an edge $(i,j) \sim (k,\ell)$ only if for all $t,u,v \in R$, $\Eld{\psi}{-v/2}{j} \Eld{\alpha+\beta}{tu}{i} \Eld{\psi}{v}{j} \Eld{\alpha+\beta}{-tu}{i} \Eld{\psi}{-v/2}{j} = \Eld{\beta+\psi}{-uv/2}{\ell} \Eld{\alpha}{t}{k} \Eld{\beta+\psi}{uv}{\ell} \Eld{\alpha}{-t}{k} \Eld{\beta+\psi}{-uv/2}{\ell}$ from \Cnref{rel:b3-large:raw:lift:alpha+beta+psi}. Additionally, grey blocks partition the vertices based on the sum of coordinates in $[3]$. In this case, the blocks also correspond to connected components in the graph.}\label{fig:b3-large:def:alpha+beta+psi}
\end{figure}

%% file: figures/b3-large/alpha+beta+2psi.tex
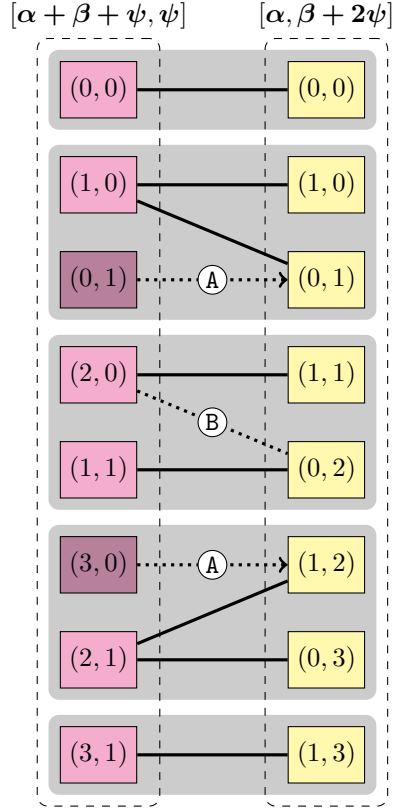
\begin{figure}
    \centering
    \begin{tikzpicture}
\node[algnode0] (v000) {$(0,0)$};
\node[algnode0] (v010) [below=of v000] {$(1,0)$};
\node[algnode0M] (v001) [below=of v010] {$(0,1)$};
\node[algnode0] (v020) [below=of v001] {$(2,0)$};
\node[algnode0] (v011) [below=of v020] {$(1,1)$};
\node[algnode0M] (v030) [below=of v011] {$(3,0)$};
\node[algnode0] (v021) [below=of v030] {$(2,1)$};
\node[algnode0] (v031) [below=of v021] {$(3,1)$};
\node[algnode1] (v100) [right=of v000] {$(0,0)$};
\node[algnode1] (v110) [below=of v100] {$(1,0)$};
\node[algnode1] (v101) [below=of v110] {$(0,1)$};
\node[algnode1] (v111) [below=of v101] {$(1,1)$};
\node[algnode1] (v102) [below=of v111] {$(0,2)$};
\node[algnode1] (v112) [below=of v102] {$(1,2)$};
\node[algnode1] (v103) [below=of v112] {$(0,3)$};
\node[algnode1] (v113) [below=of v103] {$(1,3)$};

\draw[algedge] (v010) -- (v101);
\draw[algedge] (v021) -- (v103);
\draw[algedge] (v031) -- (v113);
\draw[algedge] (v021) -- (v112);
\draw[algedge] (v020) -- (v111);
\draw[algedge] (v010) -- (v110);
\draw[algedge] (v000) -- (v100);
\draw[algedge] (v011) -- (v102);

\draw[algedge1,->] (v001) -- node[midway, algedgelabel] {\newedgeA} (v101);
\draw[algedge1,->] (v030) -- node[midway, algedgelabel] {\newedgeA} (v112);

\draw[algedge1] (v020) -- node[midway, algedgelabel] {\newedgeB} (v102);

\begin{scope}[on background layer]
    \node[algbackh, fit={(v000) (v100)}] {};
    \node[algbackh, fit={(v010) (v101)}] {};
    \node[algbackh, fit={(v020) (v102)}] {};
    \node[algbackh, fit={(v030) (v103)}] {};
    \node[algbackh, fit={(v031) (v113)}] {};
    \node[algback, fit={(v000) (v031)}] (b0) {};
    \node[above=0 of b0] {$\Comm{\Rt{\alpha+\beta+\psi}}{\Rt{\psi}}$};
    \node[algback, fit={(v100) (v113)}] (b1) {};
    \node[above=0 of b1] {$\Comm{\Rt{\alpha}}{\Rt{\beta+2\psi}}$};
\end{scope}

    \end{tikzpicture}
    \caption{Establishing $\Rt{\alpha+\beta+2\psi}$: A bipartite graph with left vertex-set $[3] \times [1]$ and right vertex-set $[1] \times [3]$, with an edge $(i,j) \sim (k,\ell)$ only if $\forall t,u,v \in R, \Comm{\Eld{\alpha+\beta+\psi}{tu}{i}}{\Eld{\psi}{v}{j}} = \Comm{\Eld{\alpha}{t}{k}}{\Eld{\beta+2\psi}{-2 uv}{\ell}}$. The solid edges are implied via lifting (\Cref{rel:b3-large:raw:lift:alpha+beta+2psi}). Vertices are darker-colored if they are not touched by solid edges (and therefore not ``reached'' by lifting). For the dotted edges, ``\newedgeA'' and ``\newedgeB'' edges are proven in \Cref{cor:b3-large:comm:alpha+beta+psi:psi:A} and \Cref{cor:b3-large:comm:alpha+beta+psi:psi:B} respectively. Additionally, grey blocks partition the vertices based on the sum of coordinates in $[4]$. Darker-colored}\label{fig:b3-large:def:alpha+beta+2psi}
\end{figure}

%% file: figures/b3-large/alpha+2beta+2psi.tex
\begin{figure}
    \centering
    \begin{tikzpicture}
\node[algnode0] (v000) {$(0,0)$};
\node[algnode0] (v010) [below=of v000] {$(1,0)$};
\node[algnode0M] (v001) [below=of v010] {$(0,1)$};
\node[algnode0M] (v020) [below=of v001] {$(2,0)$};
\node[algnode0] (v011) [below=of v020] {$(1,1)$};
\node[algnode0] (v002) [below=of v011] {$(0,2)$};
\node[algnode0] (v021) [below=of v002] {$(2,1)$};
\node[algnode0] (v012) [below=of v021] {$(1,2)$};
\node[algnode0M] (v003) [below=of v012] {$(0,3)$};
\node[algnode0M] (v022) [below=of v003] {$(2,2)$};
\node[algnode0] (v013) [below=of v022] {$(1,3)$};
\node[algnode0] (v023) [below=of v013] {$(2,3)$};

\node[algnode1] (v100) [right=of v000] {$(0,0)$};
\node[algnode1] (v110) [below=of v100] {$(1,0)$};
\node[algnode1M] (v101) [below=of v110] {$(0,1)$};
\node[algnode1M] (v120) [below=of v101] {$(2,0)$};
\node[algnode1] (v111) [below=of v120] {$(1,1)$};
\node[algnode1M] (v102) [below=of v111] {$(0,2)$};
\node[algnode1M] (v130) [below=of v102] {$(3,0)$};
\node[algnode1] (v121) [below=of v130] {$(2,1)$};
\node[algnode1M] (v112) [below=of v121] {$(1,2)$};
\node[algnode1M] (v131) [below=of v112] {$(3,1)$};
\node[algnode1] (v122) [below=of v131] {$(2,2)$};
\node[algnode1] (v132) [below=of v122] {$(3,2)$};

\node[algnode2] (v200) [right=of v100] {$(0,0)$};

\node[algnode2] (v210) [below=of v200] {$(1,0)$};
\node[algnode2M] (v201) [below=of v210] {$(0,1)$};
\node[algdummy] (vdummy1) at ($(v120)!0.5!(v111)$) {};
\node[algnode2] (v220) [below=of v201, right=of vdummy1] {$(2,0)$};
\node[algnode2] (v211) [below=of v220] {$(1,1)$};
\node[algdummy] (vdummy2) at ($(v130)!0.5!(v121)$) {};
\node[algnode2] (v230) [below=of v211, right=of vdummy2] {$(3,0)$};
\node[algnode2] (v221) [below=of v230] {$(2,1)$};
\node[algnode2M] (v240) [below=of v221, right=of v131] {$(4,0)$};
\node[algnode2] (v231) [below=of v240] {$(3,1)$};
\node[algnode2] (v241) [below=of v231, right=of v132] {$(4,1)$};

\draw[algedge] (v121) -- (v230);
\draw[algedge] (v122) -- (v231);
\draw[algedge] (v002) -- (v111);
\draw[algedge] (v110) -- (v210);
\draw[algedge] (v100) -- (v200);
\draw[algedge] (v000) -- (v100);
\draw[algedge] (v010) -- (v110);
\draw[algedge] (v012) -- (v121);
\draw[algedge] (v021) -- (v121);
\draw[algedge] (v023) -- (v132);
\draw[algedge] (v132) -- (v241);
\draw[algedge] (v013) -- (v122);
\draw[algedge] (v111) -- (v220);
\draw[algedge] (v011) -- (v111);
\draw[algedge] (v111) -- (v211);
\draw[algedge] (v121) -- (v221);

\draw[algedge1,<-] (v011) -- node[midway, algedgelabel] {\newedgeA} (v120);
\draw[algedge1,<-] (v002) -- node[midway, algedgelabel] {\newedgeA} (v102);
\draw[algedge1,<-] (v021) -- node[midway, algedgelabel] {\newedgeA} (v130);
\draw[algedge1,<-] (v012) -- node[midway, algedgelabel] {\newedgeA} (v112);

\draw[algedge1,->] (v101) -- node[midway, algedgelabel] {\newedgeB} (v210);
\draw[algedge1,->] (v131) -- node[midway, algedgelabel] {\newedgeB} (v231);

\draw[algedge1,->] (v001) -- node[midway, algedgelabel] {\newedgeC} (v110);
\draw[algedge1,->] (v020) -- node[midway, algedgelabel] {\newedgeC} (v120);
\draw[algedge1,->] (v003) -- node[midway, algedgelabel] {\newedgeC} (v112);
\draw[algedge1,->] (v022) -- node[midway, algedgelabel] {\newedgeC} (v122);

\draw[algedge1,<-] (v101) -- node[midway, algedgelabel] {\newedgeD} (v201);
\draw[algedge1,<-] (v131) -- node[midway, algedgelabel] {\newedgeD} (v240);

\begin{scope}[on background layer]
    \node[algbackh, fit={(v000) (v200)}] {};
    \node[algbackh, fit={(v010) (v201)}] {};
    \node[algbackh, fit={(v020) (v102) (v211)}] {};
    \node[algbackh, fit={(v021) (v112) (v221)}] {};
    \node[algbackh, fit={(v022) (v231)}] {};
    \node[algbackh, fit={(v023) (v241)}] {};
    \node[algback, fit={(v000) (v023)}] (b0) {};
    \node[above=0 of b0] {$\Comm{\Rt{\alpha+\beta}}{\Rt{\beta+2\psi}}$};
    \node[algback, fit={(v100) (v132)}] (b1) {};
    \node[above=0 of b1] {$\Comm{\Rt{\alpha+\beta+\psi}}{\Rt{\beta+\psi}}$};
    \node[algback, fit={(v200) (v241)}] (b2) {};
    \node[above=0 of b2] {$\Comm{\Rt{\alpha+\beta+2\psi}}{\Rt{\beta}}$};
\end{scope}

    \end{tikzpicture}
    \caption{Establishing $\Rt{\alpha+2\beta+2\psi}$: A tripartite graph with left vertex-set $\{0,1,2\}\times\{0,1,2,3\}$, middle vertex-set $\{0,1,2,3\}\times\{0,1,2\}$, and right vertex-set $\{0,1,2,3,4\} \times \{0,1\}$, with a left-middle edge $(i,j) \sim (k,\ell)$ only if $\forall t,u,v \in R, \Comm{\Eld{\alpha+\beta}{t}{i}}{\Eld{\beta+2\psi}{2uv}{j}} =\Comm{\Eld{\alpha+\beta+\psi}{tu}{k}}{\Eld{\beta+\psi}{v}{\ell}}$, and a middle-right edge $(i,j) \sim (k,\ell)$ only if $\forall t,u,v \in R, \Comm{\Eld{\alpha+\beta+\psi}{t}{i}}{\Eld{\beta+\psi}{uv}{j}} = \Comm{\Eld{\alpha+\beta+2\psi}{2tu}{k}}{\Eld{\beta}{v}{\ell}}$. Solid edges are implied by lifting (\Cref{rel:b3-large:raw:lift:alpha+2beta+2psi}). Vertices are darker-colored if they are not touched by solid edges (and therefore not ``reached'' by lifting). For the dotted edges, the ``\newedgeA'', ``\newedgeB'', ``\newedgeC'', and ``\newedgeD'' edges are proven in \Cref{cor:b3-large:expr:alpha+2beta+2psi:A}, \Cref{cor:b3-large:expr:alpha+2beta+2psi:B}, \Cref{cor:b3-large:expr:alpha+2beta+2psi:C}, and \Cref{cor:b3-large:expr:alpha+2beta+2psi:D}, respectively.}\label{fig:b3-large:def:alpha+2beta+2psi}
\end{figure}
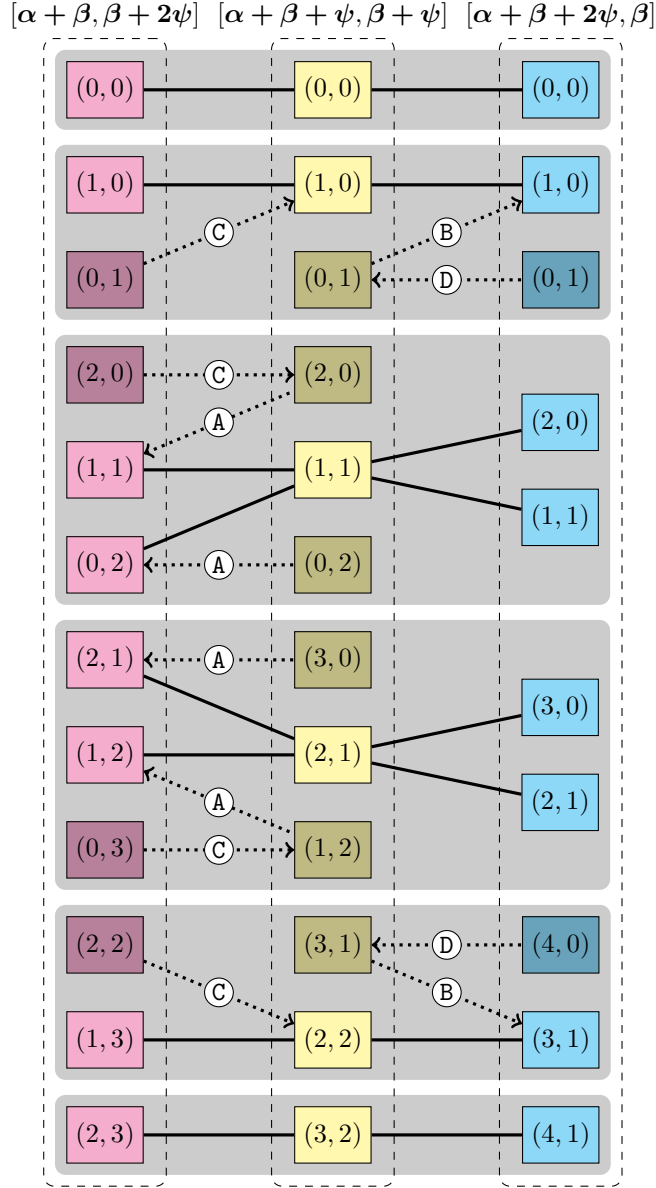

%% file: sections/AA-explicit-realizations.tex
\section{Explicit realizations of the root systems}

\newcommand{\AIndMat}[2]{\VecMatStyle{E}_{#1,#2}}
\newcommand{\BIndMat}[2]{\VecMatStyle{E}_{#1,#2}}
\newcommand{\IdMat}[1]{\VecMatStyle{I}^{(#1)}}
\newcommand{\AIdMat}{\VecMatStyle{I}}
\newcommand{\BIdMat}{\VecMatStyle{I}}
\newcommand{\AMat}[4]{\VecMatStyle{M}^{(#1)}_{#2,#3} \pbra{#4}}
\newcommand{\BShortMat}[4]{\VecMatStyle{S}^{(#1)}_{#2#3} \pbra{#4}}
\newcommand{\BLongUMat}[6]{\VecMatStyle{\tilde{L}}^{(#1)}_{#2#4,#3#5} \pbra{#6}}
\newcommand{\BLongMat}[6]{\VecMatStyle{L}^{(#1)}_{#2#4,#3#5} \pbra{#6}}

\newcommand{\XMat}{\VecMatStyle{X}}
\newcommand{\YMat}{\VecMatStyle{Y}}

In this appendix, we present explicit groups of matrices which we prove satisfy the Steinberg relations for the root systems $A_n$ and $B_n$, following \cite[\S11]{Car89}, and therefore realize their respective Chevalley groups.\footnote{Carter~\cite{Car89} describes which matrix to construct corresponding to each root in the root systems $A_n,B_n,C_n,D_n$, but does not give calculations showing why these matrices satisfy the Steinberg relations. The calculations are mechanical but are difficult to find in the literature; we hope that this appendix may independently be a useful reference.} Over a finite field $\BF_q$ these groups are, respectively, $\mathsf{SL}_{n+1}(\BF_q)$, the group of $(n+1) \times (n+1)$ determinant-$1$ matrices over $\BF_q$, and $\Omega_{2n+1}(\BF_q)$, the commutator subgroup of $\mathsf{SO}_{2n+1}(\BF_q)$, the group of $(2n+1) \times (2n+1)$ determinant-$1$ orthogonal matrices over $\BF_q$.

Recall some notations: $[n] = \{0,\ldots,n\}$ denotes the integers between $0$ and $n$ inclusive. $\IndVec{n}{i} \in \BR^n$ denotes the elementary indicator vector with a $1$ in the $i$-th coordinates and $0$'s elsewhere. From \Cref{sec:prelim:root-systems}, we have the root systems \[ A_n = \{\IndVec{n+1}{i} - \IndVec{n+1}{j} : i \neq j \in [n]\} \] and \[ B_n = \{a \IndVec{n}{i} + b \IndVec{n}{j} : i \neq j \in [n-1], a,b\in\{\pm1\}\} \cup \{a \IndVec{n}{i} : i \in [n-1], a \in \{\pm1\}\}. \]

Fix a global field $\BF$. $\IdMat{n} \in \BF^{n \times n}$ denotes the $n \times n$ identity matrix. We use the extremely simple but useful ``foiling'' identity:

\begin{equation}\label{eq:explicit:foil}
    (\IdMat{n} + \VecMatStyle{X})(\IdMat{n} + \VecMatStyle{Y}) = \IdMat{n} + \VecMatStyle{X} + \VecMatStyle{Y} + \VecMatStyle{X}\VecMatStyle{Y}.
\end{equation}

Now for $i,j \in [n-1]$, $\IndMat{n}{i}{j} \in \BF^{n \times n}$ denotes the $n \times n$ ``indicator'' matrix with a $1$ in the $(i,j)$-th entry and $0$'s elsewhere.

\begin{relation}[Products of indicator matrices]\label{rel:explicit:E-prod}
    Fix any field $\BF$. For every $n$ and $i,j,k,\ell \in [n-1]$, \[ \IndMat{n}{i}{j} \IndMat{n}{k}{\ell} = \begin{cases} \IndMat{n}{i}{\ell} & j=k, \\ 0 & j \neq k. \end{cases} \] In particular, if $i \neq j \in [n-1]$, then $(\IndMat{n}{i}{j})^2 = 0$.
\end{relation}

\begin{proof}
    We compute an entry of the product:
    \[
    (\IndMat{n}{i}{j} \IndMat{n}{k}{\ell})_{i',j'} = \sum_{k'=1}^n (\IndMat{n}{i}{j})_{i',k'} (\IndMat{n}{k}{\ell})_{k',j'}.
    \]
    The factor $(\IndMat{n}{i}{j})_{i',k'}$ is $1$ if $i=i'$ and $j=k'$, otherwise zero, and the factor $(\IndMat{n}{k}{\ell})_{k',j'}$ is $1$ if $k=k'$ and $\ell=j'$, otherwise zero.
\end{proof}

This implies the following standard fact about multiplying diagonal matrices:

\begin{relation}[Products of diagonal matrices]\label{rel:explicit:diag}
    Fix any field $\BF$. For every $n$ and $t_1,\ldots,t_n,u_1,\ldots,u_n \in \BF$, \[
    \pbra{\IdMat{n} + \sum_{i=1}^n (t_i-1) \IndMat{n}{i}{i}} \pbra{\IdMat{n} + \sum_{i=1}^n (u_i-1) \IndMat{n}{i}{i}} = \IdMat{n} + \sum_{i=1}^n (t_iu_i-1) \IndMat{n}{i}{i}.
    \]
\end{relation}

As a warmup, we prove it using the previous relation:

\begin{proof}
    By \Cref{eq:explicit:foil}, we have:
    \begin{multline*}
    \pbra{\IdMat{n} + \sum_{i=1}^n (t_i-1) \IndMat{n}{i}{i}} \pbra{\IdMat{n} + \sum_{i=1}^n (u_i-1) \IndMat{n}{i}{i}} \\
    = \IdMat{n} + \sum_{i=1}^n (t_i+u_i-2) \IndMat{n}{i}{i} + \pbra{\sum_{i=1}^n (t_i-1) \IndMat{n}{i}{i}} \pbra{ \sum_{i=1}^n (u_i-1) \IndMat{n}{i}{i}}.
    \end{multline*}
    By \Cref{rel:explicit:E-prod}, we have:
    \[
    \pbra{\sum_{i=1}^n (t_i-1) \IndMat{n}{i}{i}} \pbra{ \sum_{i=1}^n (u_i-1) \IndMat{n}{i}{i}} = \sum_{i,j=1}^n (t_i-1) (u_j-1) \IndMat{n}{i}{i} \IndMat{n}{j}{j} = \sum_{i=1}^n (t_i-1)(u_i-1) \IndMat{n}{i}{i}
    \]
    and $t+u-2+(t-1)(u-1) = tu-1$.
\end{proof}

\subsection{Realizing $A_n$} \label{app:Ad}

Fix $n$. We realize roots in $A_n$ as $(n+1)\times(n+1)$ matrices over $\BF$. Every root in $A_n$ has the form $\IndVec{n+1}{i} - \IndVec{n+1}{j}$ for $i \neq j \in [n-1]$. For each pair $i \neq j \in [n-1]$ and field element $t \in \BF$, we define the matrix
\begin{equation}\label{eq:explicit:A:M}
    \AMat{n}{i}{j}{t} := \IdMat{n+1} + t \IndMat{n+1}{i}{j} \in \BF^{d +1 \times n+1}.
\end{equation}

\subsubsection{Linearity for type-$A$ matrices}

We have a simple identity relation:

\begin{relation}\label{rel:explicit:A:id}
    For every $i \neq j \in [n-1]$, \[ \AMat{n}{i}{j}{0} = \IdMat{n+1}. \]
\end{relation}

\begin{proof}
    Obvious from \Cref{eq:explicit:A:M}.
\end{proof}

\begin{relation}[Linearity for type-$A$ matrices]\label{rel:explicit:A:lin}
    For every $i \neq j \in [n-1]$, \[ \Quant{t,u \in \BF} \AMat{n}{i}{j}{t} \cdot \AMat{n}{i}{j}{u} = \AMat{n}{i}{j}{t+u}. \]
\end{relation}

\begin{proof}
    Compute
    \begin{align*}
        \AMat{n}{i}{j}{t} \cdot \AMat{n}{i}{j}{u} &= (\IdMat{n+1} + t\IndMat{n+1}{i}{j})(\IdMat{n+1} + u \IndMat{n+1}{i}{j}) \tag{\Cref{eq:explicit:A:M}} \\
        &= \IdMat{n+1} + (t+u) \IndMat{n+1}{i}{j} + tu (\IndMat{n+1}{i}{j})^2 \tag{\Cref{eq:explicit:foil}} \\
        &= \IdMat{n+1} + (t+u) \IndMat{n+1}{i}{j} \tag{$i \neq j$ and \Cref{rel:explicit:E-prod}} \\
        &= \AMat{n}{i}{j}{t+u} \tag{\Cref{eq:explicit:A:M}}.
    \end{align*}
\end{proof}

\begin{relation}[Inverses of type-$A$ matrices]\label{rel:explicit:A:inv}
    For every $i \neq j \in [n-1]$, we have \[ \Quant{t \in \BF} (\AMat{n}{i}{j}{t})^{-1} = \AMat{n}{i}{j}{-t}. \]
\end{relation}

\begin{proof}
    Use the previous two propositions.
\end{proof}

\subsubsection{Commutators of type-$A$ matrices}

\begin{relation}[Commutators of type-$A$ matrices]
    \begin{enumerate}[label=\Roman*.]
        \item Let $i\neq j \in [n-1]$ and $k \neq \ell \in [n-1]$. Suppose $j \neq k$ and $i \neq \ell$. Then \[
        \Quant{t,u \in \BF} \Comm{\AMat{n}{i}{j}{t}}{\AMat{n}{k}{\ell}{u}} = \Id.
        \]
        \item Let $i,j,k \in [n-1]$ be all distinct. Then \[
        \Quant{t,u \in \BF} \Comm{\AMat{n}{i}{j}{t}}{\AMat{n}{j}{k}{u}} = \AMat{n}{i}{k}{tu}.
        \]
    \end{enumerate}
\end{relation}

Note that the second item also implies
\[ \Quant{t,u \in \BF} \Comm{\AMat{n}{i}{j}{t}}{\AMat{n}{k}{i}{u}} = \AMat{n}{j}{k}{-tu} \]
via taking inverses of both sides.

\begin{proof}
    For conciseness, write $\IdMat{n+1} = \AIdMat$ and $\IndMat{n+1}{i}{j} = \AIndMat{i}{j}$.

    Now
    \begin{align*}
        \AMat{n}{i}{j}{t} \cdot \AMat{n}{k}{\ell}{u} &= (\AIdMat+t\AIndMat{i}{j})(\AIdMat+u\AIndMat{k}{\ell}) \\
        &= \AIdMat+ \XMat + \YMat,
    \end{align*}
    where we define
    \begin{align*}
        \XMat &:= t\AIndMat{i}{j} + tu\AIndMat{k}{\ell}, \\
        \YMat &:= tu \AIndMat{i}{j} \AIndMat{k}{\ell}.
    \end{align*}
    
    The utility of this notation is that $\XMat$ and $\YMat$ are the ``odd'' and ``even'' parts of the expansion, respectively. In particular, 
    \[
    \AMat{n}{i}{j}{-t} \cdot \AMat{n}{k}{\ell}{-u} =  \AIdMat - \XMat + \YMat.
    \]
    Therefore,
    \[ \Comm{\AMat{n}{i}{j}{t}}{\AMat{n}{k}{\ell}{u}} = \AIdMat + 2\YMat + (\XMat + \YMat)(-\XMat + \YMat). \]

    \paragraph*{Case I.} Since $j \neq k$ by assumption, \Cref{rel:explicit:E-prod} implies $\YMat = 0$. Hence
    \[
    \Comm{\AMat{n}{i}{j}{t}}{\AMat{n}{k}{\ell}{u}} = \AIdMat + \XMat^2.
    \]
    But $\XMat^2 = 0$ since all pairwise products of $\{\AIndMat{i}{j},\AIndMat{k}{\ell}\}$ vanish by \Cref{rel:explicit:E-prod} and by assumption.

    \paragraph*{Case II.} Reusing the previous expressions with the new indices, we have
    \begin{align*}
        \XMat &:= t\AIndMat{i}{j} + u\AIndMat{j}{k}, \\
        \YMat &:= tu \AIndMat{i}{j} \AIndMat{j}{k} = tu \AIndMat{i}{k}
    \end{align*}
    (using \Cref{rel:explicit:E-prod} for the second line). Expanding,
    \begin{align*}
    (\XMat + \YMat)(-\XMat + \YMat) &= (t\AIndMat{i}{j} + u\AIndMat{j}{k} + tu \AIndMat{i}{k})(-t\AIndMat{i}{j} - u\AIndMat{j}{k} + tu \AIndMat{i}{k}).
    \intertext{Now look at the two factors being multiplied. In the first factor, the right-indices of $\AIndMat{\cdot}{\cdot}$ terms are $\{j,k\}$; in the second factor, the left-indices are $\{i,j\}$. By \Cref{rel:explicit:E-prod}, the product of two terms vanishes unless the right-index of the first and the left-index of the second are the same. So since the two sets intersect only at $j$:}
    &= (t \AIndMat{i}{j})(-u \AIndMat{j}{k}) \\
    &= -tu \AIndMat{i}{k}.
    \end{align*}
    Hence we conclude
    \[
    \Comm{\AMat{n}{i}{j}{t}}{\AMat{n}{k}{\ell}{u}} = \AIdMat + 2\YMat + (\XMat + \YMat)(-\XMat + \YMat) = \AIdMat + tu \AIndMat{i}{k} = \AMat{n}{i}{k}{tu}.
    \]
\end{proof}

\subsubsection{The diagonal relation for type-$A$ matrices}

For $i \neq j \in [n-1]$ and $t \neq 0 \in \BF$, define
\[
g_{i,j}(t) := \AMat{n}{i}{j}{t} \AMat{n}{j}{i}{-t^{-1}} \AMat{n}{i}{j}{t}
\]
and
\[
h_{i,j}(t) := g_{i,j}(t) g_{i,j}(-1).
\]

\begin{relation}
    For every $i \neq j \in [n-1]$ and $t \neq 0 \in \BF$,
    \[
    g_{i,j}(t) = \AIdMat + t\AIndMat{i}{j} - t^{-1} \AIndMat{j}{i} - \AIndMat{i}{i} - \AIndMat{j}{j}.
    \]
\end{relation}

\begin{proof}
    Expand the definitions:
    \begin{align*}
        g_{i,j}(t) &= \AMat{n}{i}{j}{t} \AMat{n}{j}{i}{-t^{-1}} \AMat{n}{i}{j}{t} \\
        &= \pbra{\AIdMat + t\AIndMat{i}{j}} \pbra{\AIdMat - t^{-1} \AIndMat{j}{i}} (\AIdMat + t\AIndMat{i}{j}) \\
        \intertext{and then expand fully using \Cref{rel:explicit:E-prod}:}
        &= \AIdMat + 2t\AIndMat{i}{j} - t^{-1} \AIndMat{j}{i} - \AIndMat{i}{i} - \AIndMat{j}{j} - t\AIndMat{i}{j},
    \end{align*}
    which rearranges to the desired right-hand side.
\end{proof}

\begin{relation}
    For every $i \neq j \in [n-1]$ and $t \neq 0 \in \BF$,
    \[
    h_{i,j}(t) = \AIdMat + (t-1) \AIndMat{i}{i} + \pbra{t^{-1}-1} \AIndMat{j}{j}.
    \]
\end{relation}

\begin{proof}
    Expanding using the definitions and the previous proposition:
    \begin{align*}
        & h_{i,j}(t) \\
        &= g_{i,j}(t) g_{i,j}(-1) \\
        &= \pbra{\AIdMat + t\AIndMat{i}{j} - t^{-1} \AIndMat{j}{i} - \AIndMat{i}{i} - \AIndMat{j}{j}} \pbra{\AIdMat - \AIndMat{i}{j} + \AIndMat{j}{i} - \AIndMat{i}{i} - \AIndMat{j}{j}} \\
        \intertext{Foiling (\Cref{eq:explicit:foil}):}
        &= \AIdMat + (t-1) \AIndMat{i}{j} + \pbra{1-t^{-1}} \AIndMat{j}{i} - 2\AIndMat{i}{i} - 2\AIndMat{j}{j} + \pbra{t\AIndMat{i}{j} - t^{-1} \AIndMat{j}{i} - \AIndMat{i}{i} - \AIndMat{j}{j}} \pbra{- \AIndMat{i}{j} + \AIndMat{j}{i} - \AIndMat{i}{i} - \AIndMat{j}{j}} \\
        \intertext{Expanding the product (\Cref{rel:explicit:E-prod}):}
        &= \AIdMat + (t-1) \AIndMat{i}{j} + \pbra{1-t^{-1}} \AIndMat{j}{i} - 2\AIndMat{i}{i} - 2\AIndMat{j}{j} + t\AIndMat{i}{i} - t\AIndMat{i}{j} + t^{-1} \AIndMat{j}{j} + t^{-1} \AIndMat{j}{i} + \AIndMat{i}{j} + \AIndMat{i}{i} - \AIndMat{j}{i} + \AIndMat{j}{j} \\
        &= \AIdMat + (t-1) \AIndMat{i}{i} + \pbra{t^{-1} - 1} \AIndMat{j}{j},
    \end{align*}
    as desired.
\end{proof}

Note that $h_{i,j}$ is a diagonal matrix, and by \Cref{rel:explicit:diag}, we have $h_{i,j}(t) h_{i,j}(u) = h_{i,j}(tu)$.

\subsection{Realizing $B_n$} \label{app:Bd}

Fix $d$. We realize the roots in $B_n$ as $(2n+1)\times(2n+1)$ matrices over $\BF$. We enumerate the row- and column-indices of these matrices as $[\pm n] = \{-n,\ldots,+n\}$ for convenience.

For $t \in \BF$, $1 \leq i \leq \ell$ and $a \in \{\pm1\}$, we define the matrix
\begin{equation}\label{eq:explicit:B:S}
    \BShortMat{n}{a}{i}{t} := \IdMat{2n+1} + 2at \IndMat{2n+1}{ai}{0} - at \IndMat{2n+1}{0}{-ai} - t^2 \IndMat{2n+1}{ai}{-ai}.
\end{equation}
(We identify this matrix with the root $a\IndVec{n}{i}$.)

For $t \in \BF$ and $i \neq j \in [n-1]$ and $a, b \in \{\pm 1\}$, we define the matrix
\begin{equation}\label{eq:explicit:B:L}
\BLongUMat{n}{a}{b}{i}{j}{t} := \IdMat{2n+1} + at \IndMat{2n+1}{ai}{-bj} - at \IndMat{2n+1}{bj}{-ai}.
\end{equation}
(Note that there is a symmetry here: $\BLongUMat{n}{a}{b}{i}{j}{t} = \BLongUMat{n}{b}{a}{j}{i}{-t}$. Both these matrices correspond to the root $a\IndVec{n}{i}+b\IndVec{n}{j}$, but it is notationally convenient for now to not break the symmetry and force a choice of sign.) 

\subsubsection{Simple identities}

\begin{relation}[Identity for short matrices]\label{rel:explicit:B:S:id}
For all $i \in [n-1]$ and $a \in \{\pm 1\}$, \[ \BShortMat{n}{a}{i}{0} = \IdMat{2n+1}. \]
\end{relation}
\begin{proof}
    Immediate from \Cref{eq:explicit:B:S}.
\end{proof}
\begin{relation}[Identity for long matrices]\label{rel:explicit:B:L:id}
For all $i < j \in [n-1]$ and $a,b \in \{\pm 1\}$, \[ \BLongUMat{n}{a}{b}{i}{j}{0} = \IdMat{2n+1}. \]
\end{relation}
\begin{proof}
    Immediate from \Cref{eq:explicit:B:L}.
\end{proof}

\begin{relation}[Linearity for short matrices]\label{rel:explicit:B:S:lin}
For all $i \in [n-1]$ and $a \in \{\pm 1\}$, \[ \Quant{t,u \in \BF} \BShortMat{n}{a}{i}{t} \cdot \BShortMat{n}{a}{i}{u} = \BShortMat{n}{a}{i}{t+u}. \]
\end{relation}
\begin{proof}
    We write a product of elements:
    \begin{align*}
        & \BShortMat{n}{a}{i}{t} \cdot \BShortMat{n}{a}{i}{u} \\
        \intertext{Expand with \Cref{eq:explicit:B:S}:}
        &= (\IdMat{2n+1} + 2at \IndMat{2n+1}{ai}{0} - at \IndMat{2n+1}{0}{-ai} - t^2 \IndMat{2n+1}{ai}{-ai}) (\IdMat{2n+1} + 2au \IndMat{2n+1}{ai}{0} - au \IndMat{2n+1}{0}{-ai} - u^2 \IndMat{2n+1}{ai}{-ai}) \\
        \intertext{Foil (\Cref{eq:explicit:foil}):}
        &= \IdMat{2n+1} + 2at \IndMat{2n+1}{ai}{0} - at \IndMat{2n+1}{0}{-ai} - t^2 \IndMat{2n+1}{ai}{-ai} + 2au \IndMat{2n+1}{ai}{0} - au \IndMat{2n+1}{0}{-ai} - u^2 \IndMat{2n+1}{ai}{-ai} \\ &+ (2at \IndMat{2n+1}{ai}{0} - at \IndMat{2n+1}{0}{-ai} - t^2 \IndMat{2n+1}{ai}{-ai})(2au \IndMat{2n+1}{ai}{0} - au \IndMat{2n+1}{0}{-ai} - u^2 \IndMat{2n+1}{ai}{-ai}) \\
        \intertext{Condense the first line with linearity. For the second line, For the second line, we consider expanding the product with \Cref{rel:explicit:E-prod}. Observe that in the first factor, the only right-indices appearing are $0$ and $-ai$, and in the second factor, the only left-indices appearing are $ai$ and $0$. Thus, all pairs vanish unless the right-index of the left-term and the left-index of the right-term are both $0$:}
        &= \IdMat{2n+1} + 2a(t+u) \IndMat{2n+1}{ai}{0} - a(t+u) \IndMat{2n+1}{0}{-ai} - (t^2+u^2) \IndMat{2n+1}{ai}{-ai} - 2a^2tu^2 \IndMat{2n+1}{ai}{-ai} \\
        \intertext{Since $a^2=1$, we can reduce the square:}
        &= \IdMat{2n+1} + 2a(t+u) \IndMat{2n+1}{ai}{0} - a(t+u) \IndMat{2n+1}{0}{-ai} - (t+u)^2 \IndMat{2n+1}{ai}{-ai} \\
        \intertext{Reduce with \Cref{eq:explicit:B:S}:}
        &= \BShortMat{n}{a}{i}{t+u}.
    \end{align*}
\end{proof}

\begin{relation}[Linearity for long matrices]\label{rel:explicit:B:L:lin}
For all $i \neq j \in [n-1]$ and $a,b \in \{\pm 1\}$, \[ \Quant{t,u \in \BF} \BLongUMat{n}{a}{b}{i}{j}{t} \cdot \BLongUMat{n}{a}{b}{i}{j}{u} = \BLongUMat{n}{a}{b}{i}{j}{t+u}. \]
\end{relation}
\begin{proof}
    We write a product of elements:
    \begin{align*}
        & \BLongUMat{n}{a}{b}{i}{j}{t} \cdot \BLongUMat{n}{a}{b}{i}{j}{u} \\
        \intertext{Expand with \Cref{eq:explicit:B:L}:}
        &= (\IdMat{2n+1} + at \IndMat{2n+1}{ai}{-bj} - at \IndMat{2n+1}{bj}{-ai}) (\IdMat{2n+1} + au \IndMat{2n+1}{ai}{-bj} - au \IndMat{2n+1}{bj}{-ai}) \\
        \intertext{Foil (\Cref{eq:explicit:foil}):}
        &= \IdMat{2n+1} + at \IndMat{2n+1}{ai}{-bj} - at \IndMat{2n+1}{bj}{-ai} +  au \IndMat{2n+1}{ai}{-bj} - au \IndMat{2n+1}{bj}{-ai}) \\
        &+ (at \IndMat{2n+1}{ai}{-bj} - at \IndMat{2n+1}{bj}{-ai}) (au \IndMat{2n+1}{ai}{-bj} - au \IndMat{2n+1}{bj}{-ai}) \\
        \intertext{Condense the first line with linearity. For the second line, we consider expanding the product with \Cref{rel:explicit:E-prod}. Observe that in the first factor, the only right-indices appearing are $-ai$ and $-bj$, and in the second factor, the only left-indices appearing are $ai$ and $bj$. Since $i \neq j$ and $a,b \neq 0$, these two sets are disjoint, so all pairs vanish:}
        &= \IdMat{2n+1} + a(t+u) \IndMat{2n+1}{ai}{-bj} - a(t+u) \IndMat{2n+1}{bj}{-ai} \\
        \intertext{Reduce with \Cref{eq:explicit:B:L}:}
        &= \BLongUMat{n}{a}{b}{i}{j}{t+u}.
    \end{align*}
\end{proof}

\begin{relation}[Inverses for short matrices]\label{rel:explicit:B:S:inv}
For all $i \in [n-1]$ and $a \in \{\pm 1\}$, \[ \Quant{t \in \BF} (\BShortMat{n}{a}{i}{t})^{-1} = \BShortMat{n}{a}{i}{-t}. \]
\end{relation}
\begin{proof}
    Immediate from \Cref{rel:explicit:B:S:id,rel:explicit:B:S:lin}.
\end{proof}

\begin{relation}[Inverses for long matrices]\label{rel:explicit:B:L:inv}
For all $i \neq j \in [n-1]$ and $a,b \in \{\pm 1\}$, \[ \Quant{t \in \BF} (\BLongUMat{n}{a}{b}{i}{j}{t})^{-1} = \BLongUMat{n}{a}{b}{i}{j}{-t}. \]
\end{relation}
\begin{proof}
    Immediate from \Cref{rel:explicit:B:L:id,rel:explicit:B:L:lin}.
\end{proof}

\subsubsection{Commutators}

\begin{relation}[Commutator of two short matrices]
For every $i\neq j \in [n-1]$ and $a,b \in \{\pm1\}$, \[ \Quant{t,u \in \BF} \Comm{\BShortMat{n}{a}{i}{t}}{\BShortMat{n}{b}{j}{t}} = \BLongUMat{n}{a}{b}{i}{j}{2btu}. \]
\end{relation}

Geometric interpretation: The matrices $\BShortMat{n}{a}{i}{t},\BShortMat{n}{b}{j}{u}$ correspond to, respectively, the short roots $a\IndVec{n}{i}, b\IndVec{n}{j} \in B_n$. Their sum is the long root $a\IndVec{n}{i} + b\IndVec{n}{j} \in B_n$ which corresponds to matrices $\BLongUMat{n}{a}{b}{i}{j}{\cdot}$.

\begin{proof}
    We begin by expanding the product of two short matrices:
    \begin{align*}
        & \BShortMat{n}{a}{i}{t} \cdot \BShortMat{n}{b}{j}{u} \\
        \intertext{Expanding with \Cref{eq:explicit:B:S}:}
        &= (\BIdMat + 2at \BIndMat{ai}{0} - at \BIndMat{0}{-ai} - t^2 \BIndMat{ai}{-ai}) (\BIdMat + 2bu \BIndMat{bj}{0} - bu \BIndMat{0}{-bj} - u^2 \BIndMat{bj}{-bj}) \\
        \intertext{Foiling (\Cref{eq:explicit:foil}):}
        &= \BIdMat + 2at \BIndMat{ai}{0} - at \BIndMat{0}{-ai} - t^2 \BIndMat{ai}{-ai} + 2bu \BIndMat{bj}{0} - bu \BIndMat{0}{-bj} - u^2 \BIndMat{bj}{-bj} \\
        &\hspace{0.5in} + (2at \BIndMat{ai}{0} - at \BIndMat{0}{-ai} - t^2 \BIndMat{ai}{-ai}) (2bu \BIndMat{bj}{0} - bu \BIndMat{0}{-bj} - u^2 \BIndMat{bj}{-bj}) \\
        \intertext{Similarly to the proof of \Cref{rel:explicit:B:S:inv}, we expand the second line using \Cref{rel:explicit:E-prod}. Since $i \neq j$ by assumption, the set of right-indices in the first term ($\{0,-ai\}$) and the set of left-indices in the left-term ($\{bj,0\}$) intersect only at zero, so we deduce:}
        &= \BIdMat + 2at \BIndMat{ai}{0} - at \BIndMat{0}{-ai} - t^2 \BIndMat{ai}{-ai} + 2bu \BIndMat{bj}{0} - bu \BIndMat{0}{-bj} - u^2 \BIndMat{bj}{-bj} - 2abtu \BIndMat{ai}{-bj} \\
        &= \BIdMat + \XMat + \YMat
    \end{align*}
    where we separate the odd and even terms, defining:
    \begin{align*}
        \XMat &:= 2at \BIndMat{ai}{0} - at \BIndMat{0}{-ai}  + 2bu \BIndMat{bj}{0} - bu \BIndMat{0}{-bj}, \\
        \YMat &:= - t^2 \BIndMat{ai}{-ai} - u^2 \BIndMat{bj}{-bj} - 2abtu \BIndMat{ai}{-bj}.
    \end{align*}
    In particular,
    \[
    \BShortMat{n}{a}{i}{-t} \cdot \BShortMat{n}{b}{j}{-u} = \BIdMat - \XMat + \YMat
    \]
    and so
    \begin{align*}
        \Comm{\BShortMat{n}{a}{i}{t}}{\BShortMat{n}{b}{j}{u}} &= \BShortMat{n}{a}{i}{t} \cdot \BShortMat{n}{b}{j}{u} \cdot \BShortMat{n}{a}{i}{-t} \cdot \BShortMat{n}{b}{j}{-u} \\
        &= (\BIdMat + \XMat + \YMat) (\BIdMat - \XMat + \YMat)  \\
        &= \BIdMat + 2\YMat + (\XMat + \YMat) (-\XMat + \YMat).
    \end{align*}
    Now the set of right-indices in $\XMat + \YMat$ is $\{0,-ai,-bj\}$ and the set of left-indices in $-\XMat + \YMat$ is $\{ai,0,bj\}$. These two sets only intersect at $0$ (which is not present in $\YMat$!), so
    \begin{align*}
    (\XMat + \YMat) (-\XMat + \YMat) &= (2at \BIndMat{ai}{0} + 2bu \BIndMat{bj}{0}) (at \BIndMat{0}{-ai} + bu \BIndMat{0}{-bj}) \\
    &= 2t^2 \BIndMat{ai}{-ai} + 2atbu \BIndMat{ai}{-bj} + 2atbu \BIndMat{bj}{-ai} + 2u^2 \BIndMat{bj}{-bj}.
    \end{align*}
    Hence we deduce the commutator is \[ \BIdMat - 2abtu \BIndMat{ai}{-bj} + 2abtu \BIndMat{bj}{-ai} = \BLongUMat{n}{a}{b}{i}{j}{-2btu}. \] as desired.
\end{proof}

\begin{relation}[Commutator of two long matrices]
\begin{enumerate}[label=\Roman*.]
    \item Suppose $i\neq j\in [n-1]$, $k \neq \ell \in [n-1]$, $a,b,c,d \in \{\pm1\}$ satisfy $\{ai,bj\} \cap \{-ck,-d\ell\} = \emptyset$. Then \[ \Quant{t,u \in \BF} \Comm{\BLongUMat{n}{a}{b}{i}{j}{t}}{\BLongUMat{n}{c}{d}{k}{\ell}{u}} = \BIdMat. \]
    \item Suppose $i \neq j \in [n-1]$, $k \neq j \in [n-1]$, and $a,b,c \in \{\pm 1\}$ satisfy $ai \neq -ck$. Then \[ \Quant{t,u \in \BF} \Comm{\BLongUMat{n}{a}{b}{i}{j}{t}}{\BLongUMat{n}{c}{-b}{k}{j}{u}} = \BLongUMat{n}{a}{-c}{i}{k}{-ctu}. \]
\end{enumerate}
\end{relation}

Note that the left-hand side of the second is the same as the first with $bj = -d\ell$. We will work out the remaining possible intersections --- keeping track of signs --- below.


\begin{proof}
We begin by expanding a product of two long matrices:
\begin{align*}
    & \BLongUMat{n}{a}{b}{i}{j}{t} \cdot \BLongUMat{n}{c}{d}{k}{\ell}{u} \\
    \intertext{Expanding with \Cref{eq:explicit:B:L}:}
    &= (\BIdMat + at \BIndMat{ai}{-bj} - at \BIndMat{bj}{-ai})(\BIdMat + cu \BIndMat{ck}{-d\ell} - cu \BIndMat{d\ell}{-ck}) \\
    \intertext{Foiling (\Cref{eq:explicit:foil}):}
    &= \BIdMat + at \BIndMat{ai}{-bj} - at \BIndMat{bj}{-ai} + cu \BIndMat{ck}{-d\ell} - cu \BIndMat{d\ell}{-ck} + (at \BIndMat{ai}{-bj} - at \BIndMat{bj}{-ai})(cu \BIndMat{ck}{-d\ell} - cu \BIndMat{d\ell}{-ck}) \\
    \intertext{which we abbreviate as:}
    &= \BIdMat + \XMat + \YMat
\end{align*}
where we separate long and short terms:
\begin{align*}
    \XMat &:= at \BIndMat{ai}{-bj} - at \BIndMat{bj}{-ai} + cu \BIndMat{ck}{-d\ell} - cu \BIndMat{d\ell}{-ck}, \\
    \YMat &:= (at \BIndMat{ai}{-bj} - at \BIndMat{bj}{-ai})(cu \BIndMat{ck}{-d\ell} - cu \BIndMat{d\ell}{-ck}).
\end{align*}

Hence
\[ \BLongUMat{n}{a}{b}{i}{j}{-t} \cdot \BLongUMat{n}{c}{d}{k}{\ell}{-u} = \BIdMat - \XMat + \YMat, \]
and so the commutator of two long matrices is
\begin{align*}
\Comm{\BLongUMat{n}{a}{b}{i}{j}{t}}{\BLongUMat{n}{c}{d}{k}{\ell}{t}} &= \BLongUMat{n}{a}{b}{i}{j}{t} \cdot \BLongUMat{n}{c}{d}{k}{\ell}{u} \cdot \BLongUMat{n}{a}{b}{i}{j}{-t} \cdot \BLongUMat{n}{c}{d}{k}{\ell}{-u} \\
&= (\BIdMat + \XMat + \YMat) (\BIdMat - \XMat + \YMat) \\
&= \BIdMat + 2\YMat + (\XMat + \YMat)(-\XMat + \YMat).
\end{align*}

Now we apply \Cref{rel:explicit:E-prod} to calculate this expansion. Observe that in the definition of $\YMat$, the right-indices in the first factor are $\{-ai,-bj\}$, while the left-indices in the second factor are $\{ck,d\ell\}$. Further, in $\XMat$, the right-indices are $\{-bj,-ai,-d\ell,-ck\}$ and the left-indices are $\{ai,bj,ck,d\ell\}$.

\paragraph*{Case 1.} $\YMat$'s first factor has right-indices $\{-bj,-ai\}$ and $\YMat$'s second factor has left-indices $\{ck,d\ell\}$. By assumption, these sets are disjoint, so $\YMat = 0$. Further, $\XMat$ has left-indices $\{ai,bj,ck,d\ell\}$ and right-indices $\{-bj,-ai,-d\ell,-ck\}$. We claim these two sets are disjoint: Indeed, e.g., $-ai$ cannot be $ai$ (since $a \in \{\pm1\}$ and $i \neq 0$), $bj$ (since $a,b \in \{\pm 1\}$ and $i, j > 0$), nor $ck$ or $d\ell$ by assumption. Thus, $\XMat^2 = 0$ and therefore the commutator vanishes.

\paragraph*{Case 2.} Using the new indices, we have
\begin{align*}
    \XMat &= at \BIndMat{ai}{-bj} - at \BIndMat{bj}{-ai} + cu \BIndMat{ck}{bj} - cu \BIndMat{-bj}{-ck}, \\
    \YMat &= (at \BIndMat{ai}{-bj} - at \BIndMat{bj}{-ai})(cu \BIndMat{ck}{bj} - cu \BIndMat{-bj}{-ck}).
\end{align*}

$\YMat$'s first factor has right-indices $\{-bj,-ai\}$ and its second factor has left-indices $\{ck,-bj\}$. These two sets intersect at $-bj$ (but not elsewhere, since $i \neq j$, $k \neq j$, and $ai \neq -ck$ by assumption). Thus, $\YMat = -actu \BIndMat{ai}{-ck}$. 

Now in the expansion of $(\XMat + \YMat)(-\XMat + \YMat)$, the first term has right-indices $\{-bj,-ai,bj,-ck\}$ and the second term has left-indices $\{ai,bj,ck,-bj\}$. These intersect at $\{\pm bj\}$ and we need only consider these terms (in particular $\YMat$ becomes irrelevant!):
\[
(\XMat + \YMat)(-\XMat + \YMat) = (at \BIndMat{ai}{-bj})(cu \BIndMat{-bj}{-ck}) + (cu \BIndMat{ck}{bj})(at \BIndMat{bj}{-ai}) = actu \BIndMat{ai}{-ck} + actu \BIndMat{ck}{-ai}.
\]

Hence the commutator is altogether 
\begin{align*}
    \BIdMat + 2(-actu \BIndMat{ai}{-ck}) + (actu \BIndMat{ai}{-ck} + actu \BIndMat{ck}{-ai}) \\
    &= \BIdMat - actu \BIndMat{ai}{-ck} + actu \BIndMat{ck}{-ai} \\
    &= \BLongUMat{n}{a}{-c}{i}{k}{-ctu}.
\end{align*}
\end{proof}

\begin{relation}[Commutator of long and short matrix]
\begin{enumerate}[label=\Roman*.]
\item Suppose $i \neq j \in [n-1], k \in [n-1], a,b, c \in \{\pm 1\}$ are such that $-ck \not\in \{ai,bj\}$. Then  \[ \Quant{t,u \in \BF} \Comm{\BLongUMat{n}{a}{b}{i}{j}{t}}{\BShortMat{n}{c}{k}{u}} = \Id. \]
\item Suppose $i \neq j \in [n-1], a,b \in \{\pm 1\}$. Then  \[ \Quant{t,u \in \BF} \Comm{\BLongUMat{n}{a}{b}{i}{j}{t}}{\BShortMat{n}{-a}{i}{u}} = \BLongUMat{n}{-a}{-b}{i}{j}{-tu^2} \cdot \BShortMat{n}{b}{j}{btu}. \]
\end{enumerate}
\end{relation}

\begin{proof}
We can expand a product of long and short matrices:
\begin{align*}
&\BLongUMat{n}{a}{b}{i}{j}{t} \cdot \BShortMat{n}{c}{k}{u} \\
&= (\BIdMat + at \BIndMat{ai}{-bj} - at \BIndMat{bj}{-ai}) (\BIdMat + 2cu \BIndMat{ck}{0} - cu \BIndMat{0}{-ck} - u^2 \BIndMat{ck}{-ck}) \\
&= \BIdMat + at \BIndMat{ai}{-bj} - at \BIndMat{bj}{-ai} + 2cu \BIndMat{ck}{0} - cu \BIndMat{0}{-ck} - u^2 \BIndMat{ck}{-ck} \\
& \hspace{0.5in} + (at \BIndMat{ai}{-bj} - at \BIndMat{bj}{-ai}) (2cu \BIndMat{ck}{0} - cu \BIndMat{0}{-ck} - u^2 \BIndMat{ck}{-ck}).
\intertext{In the second line, we can eliminate the term in the second factor with left-index $0$, since no factor in the first term has this right-index. As usual, we separate into odd and even parts:}
&= \BIdMat + \XMat + \YMat
\end{align*}
where
\begin{align*}
    \XMat &:= at \BIndMat{ai}{-bj} - at \BIndMat{bj}{-ai} + 2cu \BIndMat{ck}{0} - cu \BIndMat{0}{-ck} + (at \BIndMat{ai}{-bj} - at \BIndMat{bj}{-ai}) (- u^2 \BIndMat{ck}{-ck}), \\
    \YMat &:= - u^2 \BIndMat{ck}{-ck} + (at \BIndMat{ai}{-bj} - at \BIndMat{bj}{-ai}) (2cu \BIndMat{ck}{0}).
\end{align*}
As usual,
\[
\BLongUMat{n}{a}{b}{i}{j}{t} \cdot \BShortMat{n}{c}{k}{u} = \BIdMat - \XMat + \YMat
\]
and so
\begin{align*}
    \Comm{\BLongUMat{n}{a}{b}{i}{j}{t}}{\BShortMat{n}{c}{k}{u}} &= \BLongUMat{n}{a}{b}{i}{j}{t} \cdot \BShortMat{n}{c}{k}{u} \cdot \BLongUMat{n}{a}{b}{i}{j}{-t} \cdot \BShortMat{n}{c}{k}{-u} \\
    &= (\BIdMat + \XMat + \YMat) (\BIdMat - \XMat + \YMat) \\
    &= \BIdMat + 2\YMat + (\XMat + \YMat) (-\XMat + \YMat).
\end{align*}

\paragraph*{Case I.} In this case, the unexpanded terms in the definitions of $\XMat$ and $\YMat$ vanish, leaving
\begin{align*}
    \XMat &= at \BIndMat{ai}{-bj} - at \BIndMat{bj}{-ai} + 2cu \BIndMat{ck}{0} - cu \BIndMat{0}{-ck}, \\
    \YMat &= -u^2 \BIndMat{ck}{-ck}.
\end{align*}
As an aside, this lets us deduce that as long as $-ck \not\in \{ai,bj\}$,
\begin{equation}\label{eq:explicit:long-times-short}
    \BLongUMat{n}{a}{b}{i}{j}{t} \cdot \BShortMat{n}{c}{k}{u} = \BIdMat + at \BIndMat{ai}{-bj} - at \BIndMat{bj}{-ai} + 2cu \BIndMat{ck}{0} - cu \BIndMat{0}{-ck} -u^2 \BIndMat{ck}{-ck},
\end{equation}
which is a formula we'll use below.

Now, when expanding $(\XMat + \YMat) (-\XMat + \YMat)$, the right-indices in the first term are $\{-bj,-ai,0,-ck\}$ and left-indices in the second term $\{ai,bj,ck,0\}$. These two sets intersect only at $0$ (we can ignore $\YMat$!). So,
\[
(\XMat + \YMat) (-\XMat + \YMat) = (2cu \BIndMat{ck}{0}) (cu \BIndMat{0}{-ck}) = 2u^2 \BIndMat{ck}{-ck},
\]
which is exactly $-2\YMat$, and therefore the commutator vanishes.

\paragraph*{Case II.} In this case, we rewrite
\begin{align*}
    \XMat &= at \BIndMat{ai}{-bj} - at \BIndMat{bj}{-ai} - 2au \BIndMat{-ai}{0} + au \BIndMat{0}{ai} + (at \BIndMat{ai}{-bj} - at \BIndMat{bj}{-ai}) (- u^2 \BIndMat{-ai}{ai}) \\
    &= at \BIndMat{ai}{-bj} - at \BIndMat{bj}{-ai} - 2au \BIndMat{-ai}{0} + au \BIndMat{0}{ai} + atu^2 \BIndMat{bj}{ai}, \\
    \YMat &= - u^2 \BIndMat{-ai}{ai} + (at \BIndMat{ai}{-bj} - at \BIndMat{bj}{-ai}) (-2au \BIndMat{-ai}{0}) \\
    &= - u^2 \BIndMat{-ai}{ai} + 2tu \BIndMat{bj}{0}.
\end{align*}
Expanding $(\XMat + \YMat)(-\XMat + \YMat)$, the first term has right-indices $\{-bj,-ai,0,ai\}$ and the second term has left-indices $\{ai,bj,-ai,0\}$. These intersect at $\{ai,-ai,0\}$. So:
\begin{align*}
    (\XMat+\YMat)(-\XMat+\YMat) &= (at \BIndMat{ai}{-bj} - at \BIndMat{bj}{-ai} - 2au \BIndMat{-ai}{0} + au \BIndMat{0}{ai} + atu^2 \BIndMat{bj}{ai} - u^2 \BIndMat{-ai}{ai} + 2tu \BIndMat{bj}{0}) \\
    &\hspace{0.5in} \cdot (-at \BIndMat{ai}{-bj} + at \BIndMat{bj}{-ai} + 2au \BIndMat{-ai}{0} - au \BIndMat{0}{ai} - atu^2 \BIndMat{bj}{ai} - u^2 \BIndMat{-ai}{ai} + 2tu \BIndMat{bj}{0}) \\
    &= (- at \BIndMat{bj}{-ai})(2au \BIndMat{-ai}{0} - u^2 \BIndMat{-ai}{ai}) + (- 2au \BIndMat{-ai}{0} + 2tu \BIndMat{bj}{0}) ( - au \BIndMat{0}{ai}) \\
    &\hspace{0.5in} + (au \BIndMat{0}{ai} + atu^2 \BIndMat{bj}{ai} - u^2 \BIndMat{-ai}{ai}) (-at \BIndMat{ai}{-bj}) \\
    &= -2tu \BIndMat{bj}{0} + atu^2 \BIndMat{bj}{ai} + 2u^2 \BIndMat{-ai}{ai} - 2atu^2 \BIndMat{bj}{ai} - tu \BIndMat{0}{-bj} - t^2u^2 \BIndMat{bj}{-bj} + atu^2 \BIndMat{-ai}{-bj} \\
    &= -2tu \BIndMat{bj}{0} - atu^2  \BIndMat{bj}{ai} + 2u^2 \BIndMat{-ai}{ai} - tu \BIndMat{0}{-bj} - t^2u^2 \BIndMat{bj}{-bj} + atu^2 \BIndMat{-ai}{-bj}.
\end{align*}
Hence the commutator is altogether
\begin{align*}
    &\BIdMat + 2(- u^2 \BIndMat{-ai}{ai} + 2tu \BIndMat{bj}{0}) -2tu \BIndMat{bj}{0} - atu^2  \BIndMat{bj}{ai} + 2u^2 \BIndMat{-ai}{ai} - tu \BIndMat{0}{-bj} - t^2u^2 \BIndMat{bj}{-bj} + atu^2 \BIndMat{-ai}{-bj} \\
    &= \BIdMat + 2tu \BIndMat{bj}{0} - atu^2  \BIndMat{bj}{ai} - tu \BIndMat{0}{-bj} - t^2u^2 \BIndMat{bj}{-bj} + atu^2 \BIndMat{-ai}{-bj} \\
    &= \BIdMat + atu^2 \BIndMat{-ai}{-bj} - atu^2  \BIndMat{bj}{ai} + 2tu \BIndMat{bj}{0}- tu \BIndMat{0}{-bj} - t^2u^2 \BIndMat{bj}{-bj} \\
    \intertext{Using \Cref{eq:explicit:long-times-short}:}
    &= \BLongUMat{n}{-a}{-b}{i}{j}{-tu^2} \cdot \BShortMat{n}{b}{j}{btu}.
\end{align*}
\end{proof}

\subsubsection{Diagonal relations for type-$B$ matrices}

First, we handle the diagonal relation for the long roots. For $i < j \in [n-1]$ and $t \neq 0 \in \BF$, define
\[
g_{ai,bj}(t) := \BLongMat{n}{a}{b}{i}{j}{t} \BLongMat{n}{-a}{-b}{i}{j}{-t^{-1}} \BLongMat{n}{a}{b}{i}{j}{t}
\]
and
\[
h_{ai,bj}(t) := g_{ai,bj}(t) g_{ai,bj}(-1).
\]

We have:

\begin{relation}
    For $i < j \in [n-1]$ and $t \neq 0 \in \BF$:
    \[
    g_{ai,bj}(t) = \BIdMat + at\BIndMat{ai}{-bj} - at\BIndMat{bj}{-ai} + at^{-1} \BIndMat{-ai}{bj} - at^{-1} \BIndMat{-bj}{ai} - \BIndMat{ai}{ai} - \BIndMat{bj}{bj} - \BIndMat{-bj}{-bj} - \BIndMat{-ai}{-ai}.
    \]
\end{relation}

\begin{proof}
    We write and expand with the definitions:
    \begin{align*}
        & g_{ai,bj}(t) \\
        &= \BLongMat{n}{a}{b}{i}{j}{t} \BLongMat{n}{-a}{-b}{i}{j}{-t^{-1}} \BLongMat{n}{a}{b}{i}{j}{t} \\
        &= \pbra{\BIdMat + at\BIndMat{ai}{-bj} - at\BIndMat{bj}{-ai}} \pbra{\BIdMat + at^{-1} \BIndMat{-ai}{bj} - at^{-1} \BIndMat{-bj}{ai}} \pbra{\BIdMat + at\BIndMat{ai}{-bj} - at\BIndMat{bj}{-ai}} \\
        \intertext{Expanding fully and using \Cref{rel:explicit:E-prod}:}
        &= \BIdMat + 2at\BIndMat{ai}{-bj} - 2at\BIndMat{bj}{-ai} + at^{-1} \BIndMat{-ai}{bj} - at^{-1} \BIndMat{-bj}{ai} - a^2 \BIndMat{ai}{ai} - a^2 \BIndMat{bj}{bj} - a^2 \BIndMat{-bj}{-bj} - a^2 \BIndMat{-ai}{-ai} \\
        &\hspace{1in} - a^3 t \BIndMat{ai}{-bj} + a^3 t \BIndMat{bj}{-ai} \\
        \intertext{Gathering like terms:}
        &= \BIdMat + at\BIndMat{ai}{-bj} - at\BIndMat{bj}{-ai} + at^{-1} \BIndMat{-ai}{bj} - at^{-1} \BIndMat{-bj}{ai} - \BIndMat{ai}{ai} - \BIndMat{bj}{bj} - \BIndMat{-bj}{-bj} - \BIndMat{-ai}{-ai}.
    \end{align*}
\end{proof}

\begin{relation}
    For $i < j \in [n-1]$ and $t \neq 0 \in \BF$:
    \[
    h_{ai,bj}(t) = \BIdMat + (t-1) \BIndMat{ai}{ai} + (t-1) \BIndMat{bj}{bj} + (t^{-1}-1) \BIndMat{-bj}{-bj} + (t^{-1}-1) \BIndMat{-ai}{-ai}
    \]
\end{relation}

\begin{proof}
    \begin{align*}
        & h_{ai,bj}(t) \\
        &= g_{ai,bj}(t) g_{ai,bj}(-1) \\
        \intertext{Using the previous relation:}
        &= \pbra{\BIdMat + at\BIndMat{ai}{-bj} - at\BIndMat{bj}{-ai} + at^{-1} \BIndMat{-ai}{bj} - at^{-1} \BIndMat{-bj}{ai} - \BIndMat{ai}{ai} - \BIndMat{bj}{bj} - \BIndMat{-bj}{-bj} - \BIndMat{-ai}{-ai}} \\
        &\hspace{1in} \cdot \pbra{\BIdMat - a \BIndMat{ai}{-bj} + a \BIndMat{bj}{-ai} - a \BIndMat{-ai}{bj} + a \BIndMat{-bj}{ai} - \BIndMat{ai}{ai} - \BIndMat{bj}{bj} - \BIndMat{-bj}{-bj} - \BIndMat{-ai}{-ai}} \\
        \intertext{Foiling (\Cref{eq:explicit:foil}):}
        &= \BIdMat + a(t-1)\BIndMat{ai}{-bj} - a(t-1)\BIndMat{bj}{-ai} + a(t^{-1}-1) \BIndMat{-ai}{bj} - a(t^{-1}-1) \BIndMat{-bj}{ai} \\
        &\hspace{.5in} - 2\BIndMat{ai}{ai} - 2\BIndMat{bj}{bj} - 2\BIndMat{-bj}{-bj} - 2\BIndMat{-ai}{-ai} \\
        &\hspace{.5in} + \pbra{at\BIndMat{ai}{-bj} - at\BIndMat{bj}{-ai} + at^{-1} \BIndMat{-ai}{bj} - at^{-1} \BIndMat{-bj}{ai} - \BIndMat{ai}{ai} - \BIndMat{bj}{bj} - \BIndMat{-bj}{-bj} - \BIndMat{-ai}{-ai}} \\
        &\hspace{.5in} \cdot \pbra{- a \BIndMat{ai}{-bj} + a \BIndMat{bj}{-ai} - a \BIndMat{-ai}{bj} + a \BIndMat{-bj}{ai} - \BIndMat{ai}{ai} - \BIndMat{bj}{bj} - \BIndMat{-bj}{-bj} - \BIndMat{-ai}{-ai}}.
    \end{align*}
    Gathering compatible terms as in \Cref{rel:explicit:E-prod}, the product can be expanded as:
    \begin{align*}
        & (at\BIndMat{ai}{-bj} - \BIndMat{-bj}{-bj})(a \BIndMat{-bj}{ai}- \BIndMat{-bj}{-bj}) + (- at\BIndMat{bj}{-ai} - \BIndMat{-ai}{-ai})(- a \BIndMat{-ai}{bj} - \BIndMat{-ai}{-ai}) \\
        &\hspace{.5in} (at^{-1} \BIndMat{-ai}{bj} - \BIndMat{bj}{bj}) (a \BIndMat{bj}{-ai} - \BIndMat{bj}{bj}) + (- at^{-1} \BIndMat{-bj}{ai} - \BIndMat{ai}{ai})(- a \BIndMat{ai}{-bj}- \BIndMat{ai}{ai}) \\
        &= a^2 t \BIndMat{ai}{ai} - at \BIndMat{ai}{-bj} - a \BIndMat{-bj}{ai} + \BIndMat{-bj}{-bj} + a^2 t \BIndMat{bj}{bj} + at \BIndMat{bj}{-ai} + a \BIndMat{-ai}{bj} + \BIndMat{-ai}{-ai} \\
        &\hspace{.5in} + a^2 t^{-1} \BIndMat{-ai}{-ai} - at^{-1} \BIndMat{-ai}{bj} - a \BIndMat{bj}{-ai} + \BIndMat{bj}{bj} + a^2 t^{-1} \BIndMat{-bj}{-bj} + at^{-1} \BIndMat{-bj}{ai} + a \BIndMat{ai}{-bj} + \BIndMat{ai}{ai} \\
        \intertext{Collecting like terms and using $a^2=1$:}
        &= (t+1) \BIndMat{ai}{ai} - a(t-1)\BIndMat{ai}{-bj} + a(t^{-1}-1) \BIndMat{-bj}{ai} + (1 + t^{-1}) \BIndMat{-bj}{-bj} \\
        &\hspace{.5in} + (t + 1) \BIndMat{bj}{bj} + a(t-1) \BIndMat{bj}{-ai} + a (1-t^{-1}) \BIndMat{-ai}{bj} + (1 + t^{-1}) \BIndMat{-ai}{-ai}.
    \end{align*}
    Hence we deduce
    \[
    h_{ai,bj}(t) = \BIdMat + (t-1) \BIndMat{ai}{ai} + (t-1) 2\BIndMat{bj}{bj} + (t^{-1}-1) \BIndMat{-bj}{-bj} + (t^{-1}-1) \BIndMat{-ai}{-ai}.
    \]
\end{proof}

Therefore $h_{ai,bj}(t) h_{ai,bj}(u) = h_{ai,bj}(tu)$ by \Cref{rel:explicit:diag}.

A similar proof is possible for the short roots, but it is more verbose.

%% file: sections/BB-algebra.tex
\section{Commutators and group theory}\label{app:group}

Recall our definition of the commutator (\Cref{eq:def:comm}): We have
\begin{equation}\label{eq:comm:inv}
    \Comm{x}{y}^{-1} = \Comm{y}{x}.
\end{equation}
and
\begin{equation}
    \Comm{x}{\Id} = \Comm{\Id}{x} = \Id.
\end{equation}
And we have a number of ``rewriting'' forms, which we will name individually.

\begin{relation}[\RelNameAlg\RelNameOrder{\RelNameLeft}{\RelNameStraight}]\label{eq:comm:left:str}
    \[ xy = \Comm{x}{y} y x. \]
\end{relation}
\begin{relation}[\RelNameAlg\RelNameOrder{\RelNameLeft}{\RelNameRev}]\label{eq:comm:left:inv}
    \[ xy = \Comm{y}{x}^{-1} y x. \]
\end{relation}
\begin{relation}[\RelNameAlg\RelNameOrder{\RelNameMid}{\RelNameStraight}]\label{eq:comm:mid:str}
    \[ xy = y \Comm{x}{y^{-1}}^{-1} x \]
\end{relation}
\begin{relation}[\RelNameAlg\RelNameOrder{\RelNameMid}{\RelNameRev}]\label{eq:comm:mid:inv}
    \[ xy = y \Comm{y^{-1}}{x} x. \]
\end{relation}
\begin{relation}[\RelNameAlg\RelNameOrder{\RelNameRight}{\RelNameStraight}]\label{eq:comm:right:str}
    \[ xy = y x \Comm{x^{-1}}{y^{-1}}. \]
\end{relation}
\begin{relation}[\RelNameAlg\RelNameOrder{\RelNameRight}{\RelNameRev}]\label{eq:comm:right:inv}
    \[ xy = y x \Comm{y^{-1}}{x^{-1}}^{-1}. \]
\end{relation}

Also, we have the following equation which will be useful below: For every $x,y \in G$, let $x \star y := y^{-1} x y^2 x^{-1} y^{-1}$. Then 
\begin{equation}\label{eq:comm:conj}
    (x \star y) \Comm{(x \star y)^{-1}}{y} = \Comm{x \star y}{y}^{-1} (x \star y)= y (x \star y) y^{-1} = \Comm{x}{y^2}.
\end{equation}

We also record here some other simple but useful group theory facts. We use the following lemma:

\begin{proposition}\label{prop:prelim:equal-conn}
    Let $R$ be a ring and $S$ a finite set. Let $G=(V,E)$ be a directed (simple, unweighted) graph with the following properties:
    \begin{enumerate}
        \item To every vertex $x \in V$, there is an associated function $f_x : R \times R \to S$.
        \item For every directed edge $(x,y) \in E$ and for every $t,u,v \in R$,
        \[
        f_x(tu,v) = f_y(t,uv).
        \]
    \end{enumerate}
    Suppose that $G$ contains at least two vertices and is weakly connected.\footnote{That is, connected as an undirected graph when replacing directed edges $(u,v)$ with undirected edges $\{u,v\}$.} Then there exists a function $g : R \to S$ such that for every $x \in V$ and $t, u \in R$, $f_x(t,u) = g(tu)$.
\end{proposition}

\begin{proof}
    For any edge $(x',y') \in E$, we observe that using the assumption, $f_{x'}(t,u) = f_{x'}(t \cdot 1, u) = f_{y'}(t, 1 \cdot u) = f_{y'}(t,u)$.

    Pick any edge $(x,y) \in E$. (There is at least one edge since $G$ has two vertices and is weakly connected.) For every $t \in R$, we define $g(t) := f_x(t,1)$. Now for every $t, u \in R$, by the assumption,
    \[
    f_y(t,u) = f_y(t,u \cdot 1) = f_x(tu, 1) = g(tu).
    \]

    Now, it remains to show that $f_z(t,u) = g(t,u)$ for every $z \neq y \in V$. Consider an arbitrary such vertex. By connectivity, $z$ and $y$ are connected by some undirected path, i.e., $z = w_0, \ldots, w_k = z$ and either $(w_i,w_{i+1}) \in E$ or $(w_{i+1},w_i) \in E$. In both cases, we use the observation from the previous paragraph to deduce that $f_{w_i}(t,u) = f_{w_{i+1}}(t,u)$, and so inductively that $f_z(t,u) = f_y(t,u) = g(tu)$, as desired.
\end{proof}

And the following very standard implication:

\begin{proposition}\label{prop:prelim:group-homo}
    Let $G,H$ be groups written using additive and multiplicative notations, respectively, and $\phi : G \to H$ a map such that for every $x,y \in G$, $\phi(x+y) = \phi(x) \phi(y)$ (i.e., a group homomorphism). Then $\phi(0) = \Id$ and for every $x \in G$, $(\phi(x))^{-1} = \phi(-x)$.
\end{proposition}

\begin{proof}
    By assumption, $\phi(0) = \phi(0 + 0) = \phi(0) \phi(0)$ and so $\phi(0) = \Id$. Then $\phi(x) \phi(-x) = \phi(x-x) = \phi(0) = \Id$.
\end{proof}

%% file: sections/CC-biss-dasgupta.tex
\section{Why $A_3$'s link is simply connected and $B_3$'s might not be}\label{sec:bd}

In this appendix, we discuss why in the $A_3$ setting, \cite{KO21} had a base theorem (\Cref{thm:story:a3:base}), while in the $B_3$ setting, we conjecture that such a theorem is not true.

The \cite{KO21} base theorem --- stating that all Steinberg relations in $\UnipA{\BZ/p\BZ}$ can be derived from in-subgroup relations in $\UnipA{\BZ/p\BZ}$ --- is due essentially to Biss and Dasgupta~\cite{BD01}.\footnote{The original \cite{BD01} proof did not emphasize the quantitative bound on the length of the derivation, and in particular, the fact that the length of the derivation does \emph{not} depend on the size~$p$ of the ring.} The key step of the \cite{BD01} proof (essentially the only difficulty one) is to derive the relation ``$\Rt{\alpha+\beta}$ and $\Rt{\beta+\gamma}$ elements commute''. (This makes sense as a first step because it is the \emph{only} relation in the $A_3$ case which does not name the ``missing'' root $\Rt{\alpha+\beta+\gamma}$.) We give our own shorter (and perhaps conceptually clearer) proof of this relation, and then discuss why we do \emph{not} know how to prove analogous relations in $B_3$.

\subsection{An alternative proof of Biss--Dasgupta}
In this section we give an alternative proof of the key Biss--Dasgupta result~\cite[Sec.~4]{BD01}, which states that over the ring $\BZ/p\BZ$ with $p$ odd, the fact that $(\Rt{\alpha+\beta})$-type and $(\Rt{\beta+\gamma})$-type elements commute can be derived only using ``in-subgroup relations''.
In fact, our proof works within the unipotent group $U_{A_3}(R)$ for any any commutative ring~$R$ in which~$1/2$ exists (i.e., $1+1$ is a unit).

Our proof, as in the original proof in \cite[\S4]{BD01} proof, starts by deriving the following useful ``rewriting'' relation in $U_{A_3}(R)$:

\begin{relation}\label{rel:bd:swap}
For any ring $R$, whenever the elements $uv/(t+v)$ and $tu/(t+v)$ exist, one can derive \[ 
\El{\alpha}{t} \El{\beta}{u} \El{\alpha}{v} = \El{\beta}{uv/(t+v)} \El{\alpha}{t+v} \El{\beta}{tu/(t+v)} \] 
using only in-subgroup relations in $U_{A_3}(R)$, and the number of steps in the proof does not depend on~$R$. The same holds replacing $\Rt{\alpha}$'s with $\Rt{\beta}$'s and $\Rt{\beta}$'s with $\Rt{\gamma}$'s.
\end{relation}

\begin{proof}
    First, note that $u = uv/(t+v) + tu/(t+v)$. Thus, we can write, by linearity of $\Rt{\beta}$ elements:
    \begin{align*}
        & \El{\alpha}{t} \El{\beta}{u} \El{\alpha}{v} \\
        &= \El{\alpha}{t} \El{\beta}{uv/(t+v)} \El{\beta}{tu/(t+v)} \El{\alpha}{v} \\
        \intertext{And reordering elements (\Cnref{eq:comm:right:str} and \Cnref{eq:comm:left:inv}):}
        &=  \El{\beta}{uv/(t+v)} \El{\alpha}{t} \Comm{\El{\alpha}{t}^{-1}}{\El{\beta}{uv/(t+v)}^{-1}} \Comm{\El{\beta}{tu/(t+v)}}{\El{\alpha}{v}}^{-1} \El{\alpha}{v} \El{\beta}{tu/(t+v)} .
    \end{align*}
    Using the fact that inverting $\Rt{\alpha}$, $\Rt{\beta}$, and $\Rt{\alpha+\beta}$ elements is the same as negating the respective entries, and that the commutator of $\Rt{\alpha}$ and $\Rt{\beta}$ elements is an $\Rt{\alpha+\beta}$ element whose product is the element of the entries, these two commutators cancel (they are $\Rt{\alpha+\beta}$ elements with entries $tuv/(t+v)$ and $-tuv/(t+v)$, respectively. Linearity of $\Rt{\alpha}$ elements gives the desired relation.
\end{proof}

Now, we can prove:

\begin{relation}\label{rel:ungr:a3:comm:alpha+beta:beta+gamma}
In any ring $R$ containing $1/2$, one can derive that \[ \Quant{t,u \in R} \Comm{\El{\alpha+\beta}{t}}{\El{\beta+\gamma}{u}} = \Id. \] using only in-subgroup relations in $U_{A_3}(R)$, and the number of steps in the proof does not depend on~$R$. 

\end{relation}

\begin{proof}
    For this proof, we need some subgroup identities. We use four identities which expand elements as commutators (using, again, that inverting a $\Rt{\zeta}$ elements negates its entry for $\Rt{\zeta} \in \{\Rt{\alpha},\Rt{\beta},\Rt{\gamma},\Rt{\alpha+\beta},\Rt{\beta+\gamma}\}$). These identities require that $1/2$ exist in $R$:
    \begin{align}
        \El{\alpha+\beta}{t} \label{eq:a3:comm:alpha+beta:beta+gamma:1} &= \El{\beta}{-1} \El{\alpha}{t} \El{\beta}{1} \El{\alpha}{-t}, \\
        \El{\beta+\gamma}{u} &= \El{\gamma}{-2u} \El{\beta}{1/2} \El{\gamma}{2u} \El{\beta}{-1/2}, \label{eq:a3:comm:alpha+beta:beta+gamma:2} \\
        \El{\alpha+\beta}{-t} &= \El{\beta}{1/2} \El{\alpha}{2t} \El{\beta}{-1/2} \El{\alpha}{-2t}, \label{eq:a3:comm:alpha+beta:beta+gamma:3} \\
        \El{\beta+\gamma}{-u} &= \El{\gamma}{-u} \El{\beta}{-1} \El{\gamma}{u} \El{\beta}{1}. \label{eq:a3:comm:alpha+beta:beta+gamma:4} 
    \end{align}
    We also require four identities for swaps following from \Cref{rel:bd:swap}; again, these require $1/2$ to exist in~$R$:
    \begin{align}
        \El{\alpha}{-t} \El{\beta}{1/2} \El{\alpha}{2t} &= \El{\beta}{1} \El{\alpha}{t} \El{\beta}{-1/2}, \label{eq:a3:comm:alpha+beta:beta+gamma:5} \\
        \El{\gamma}{2u} \El{\beta}{-1/2} \El{\gamma}{-u} &= \El{\beta}{1/2} \El{\gamma}{u} \El{\beta}{-1}, \label{eq:a3:comm:alpha+beta:beta+gamma:6} \\
        \El{\beta}{1} \El{\gamma}{-2u} \El{\beta}{1} &= \El{\gamma}{-u} \El{\beta}{2} \El{\gamma}{-u}, \label{eq:a3:comm:alpha+beta:beta+gamma:7} \\
        \El{\beta}{-1} \El{\alpha}{-2t} \El{\beta}{-1} &= \El{\alpha}{-t} \El{\beta}{-2} \El{\alpha}{-t}. \label{eq:a3:comm:alpha+beta:beta+gamma:8}
    \end{align}

    \noindent The basic plan is now to write down $\Comm{\El{\alpha+\beta}{t}}{\El{\beta+\gamma}{u}}$ and expand using \Cref{eq:a3:comm:alpha+beta:beta+gamma:1,eq:a3:comm:alpha+beta:beta+gamma:2,eq:a3:comm:alpha+beta:beta+gamma:3,eq:a3:comm:alpha+beta:beta+gamma:4}.
    We proceed as follows:
    \begin{align*}
        & \asite{\Comm{\El{\alpha+\beta}{t}}{\El{\beta+\gamma}{u}}} \\
        \intertext{Expand the commutators:}
        &= \El{\alpha+\beta}{t} \asite{\El{\beta+\gamma}{u}} \asite{\El{\alpha+\beta}{-t}} \El{\beta+\gamma}{-u} \\
        \intertext{Expand with \Cref{eq:a3:comm:alpha+beta:beta+gamma:2,eq:a3:comm:alpha+beta:beta+gamma:3}:}
        &= \El{\alpha+\beta}{t} \El{\gamma}{-2u} \El{\beta}{1/2} \El{\gamma}{2u} \asite{\El{\beta}{-1/2}} \asite{\El{\beta}{1/2}} \El{\alpha}{2t} \El{\beta}{-1/2} \El{\alpha}{-2t}  \El{\beta+\gamma}{-u} \\
        \intertext{Cancel the $\Rt{\beta}$ elements:}
        &= \asite{\El{\alpha+\beta}{t}} \El{\gamma}{-2u} \El{\beta}{1/2} \El{\gamma}{2u} \El{\alpha}{2t} \El{\beta}{-1/2} \El{\alpha}{-2t}  \asite{\El{\beta+\gamma}{-u}} \\
        \intertext{Expand with \Cref{eq:a3:comm:alpha+beta:beta+gamma:1,eq:a3:comm:alpha+beta:beta+gamma:4}:}
        &= \El{\beta}{-1} \El{\alpha}{t} \El{\beta}{1} \asite{\El{\alpha}{-t} \El{\gamma}{-2u}} \El{\beta}{1/2} \asite{\El{\gamma}{2u} \El{\alpha}{2t}} \El{\beta}{-1/2} \asite{\El{\alpha}{-2t}  \El{\gamma}{-u}} \El{\beta}{-1} \El{\gamma}{u} \El{\beta}{1} \\
        \intertext{Commute $\Rt{\alpha}$ and $\Rt{\gamma}$ elements where possible:}
        &= \El{\beta}{-1} \El{\alpha}{t} \El{\beta}{1} \El{\gamma}{-2u} \asite{\El{\alpha}{-t} \El{\beta}{1/2} \El{\alpha}{2t}} \asite{\El{\gamma}{2u} \El{\beta}{-1/2} \El{\gamma}{-u}} \El{\alpha}{-2t} \El{\beta}{-1} \El{\gamma}{u}  \El{\beta}{1} \\
        \intertext{Swap $\Rt{\alpha} \cdot \Rt{\beta} \cdot \Rt{\alpha}$ products for $\Rt{\beta} \cdot \Rt{\alpha} \cdot \Rt{\beta}$ products and similarly for $\Rt{\gamma}$ and $\Rt{\beta}$ (\Cref{eq:a3:comm:alpha+beta:beta+gamma:5,eq:a3:comm:alpha+beta:beta+gamma:6}):}
        &= \El{\beta}{-1} \El{\alpha}{t} \El{\beta}{1} \El{\gamma}{-2u} \El{\beta}{1} \El{\alpha}{t} \asite{\El{\beta}{-1/2} \El{\beta}{1/2}} \El{\gamma}{u} \El{\beta}{-1} \El{\alpha}{-2t} \El{\beta}{-1} \El{\gamma}{u} \El{\beta}{1} \\
        \intertext{Cancel adjacent $\Rt{\beta}$ elements:}
        &= \El{\beta}{-1} \El{\alpha}{t} \asite{\El{\beta}{1} \El{\gamma}{-2u} \El{\beta}{1}} \El{\alpha}{t} \El{\gamma}{u} \asite{\El{\beta}{-1} \El{\alpha}{-2t} \El{\beta}{-1}} \El{\gamma}{u}  \El{\beta} \\
        \intertext{Do another swap (\Cref{eq:a3:comm:alpha+beta:beta+gamma:7,eq:a3:comm:alpha+beta:beta+gamma:8}):}
        &= \El{\beta}{-1} \El{\alpha}{t} \El{\gamma}{-u} \El{\beta}{2} \asite{\El{\gamma}{-u} \El{\alpha}{t} \El{\gamma}{u} \El{\alpha}{-t}} \El{\beta}{-2} \El{\alpha}{-t} \El{\gamma}{u}  \El{\beta}{1}  \\
        \intertext{Commute $\Rt{\alpha}$ and $\Rt{\gamma}$ elements and therefore cancel them:}
        &= \El{\beta}{-1} \El{\alpha}{t} \El{\gamma}{-u} \asite{\El{\beta}{2}  \El{\beta}{-2}} \El{\alpha}{-t} \El{\gamma}{u}  \El{\beta}{1}\\
        \intertext{Cancel $\Rt{\beta}$ elements:}
        &= \El{\beta}{-1} \asite{\El{\alpha}{t} \El{\gamma}{-u}  \El{\alpha}{-t} \El{\gamma}{u}}  \El{\beta}{1} \\
        \intertext{Again commute and cancel $\Rt{\alpha}$ and $\Rt{\gamma}$ elements:}
        &= \El{\beta}{-1} \El{\beta}{1}
        \intertext{Finally cancel $\Rt{\beta}$ elements:}
        &= \Id,
    \end{align*}
    as desired.
\end{proof}

Incidentally, the fact that \Cref{rel:ungr:a3:comm:alpha+beta:beta+gamma} holds implies that in any link complex $\LinkCplxA{q}$, there must be an $\BF_2$-filling of the loop corresponding to the critical relation $\Comm{\El{\alpha+\beta}{t}}{\El{\beta+\gamma}{u}} = \Id$.  
We found a simple (human-verifiable) proof of this fact, illustrated by the diagrams in \Cref{fig:squares}. The $\Comm{\El{\alpha+\beta}{t}}{\El{\beta+\gamma}{u}} = \Id$ relation corresponds to a loop of length~$4$, alternating between $\BoxR$ and $\BoxB$ subgroups; this is the orange outer loop in the figures. 
In the coset complex, some (but not all) $\BoxR$-$\BoxB$-$\BoxR$-$\BoxB$ ``squares'' have the property that they are $\BF_2$-filled by $4$ triangles, all incident on the same $\BoxG$ vertex.  Although the orange outer loop does \emph{not} have this property, we find five other squares that \emph{are} fillable in this way, and such that the $\BF_2$-boundary of these squares is the orangle loop to be filled.  The arrangement of these squares is depicted in the left figure, and the full filling by triangles is depicted in the right figure.  Overall, this gives an $\BF_2$-filling of the critical relation $\Comm{\El{\alpha+\beta}{t}}{\El{\beta+\gamma}{u}} = \Id$ using $20$ triangles.

\input{figures/other/figure}

\subsection{$A_3$ vs.\ $B_3$}
Let us now reflect on some key features of $\UnipA{q}$ that let us write this proof. Firstly, $\Rt{\alpha+\beta}$ and $\Rt{\beta+\gamma}$ were expressible as $\Rt{\alpha}$-$\Rt{\beta}$ and $\Rt{\beta}$-$\Rt{\gamma}$ commutators, respectively. By using factors of $2$ and $1/2$ in two of the expansions (\Cref{eq:a3:comm:alpha+beta:beta+gamma:2,eq:a3:comm:alpha+beta:beta+gamma:3}) but not the others (\Cref{eq:a3:comm:alpha+beta:beta+gamma:1,eq:a3:comm:alpha+beta:beta+gamma:4}), we were able to introduce some asymmetries allowing us to eventually swap $\Rt{\eta} \Rt{\zeta} \Rt{\eta}$ type products for $\Rt{\zeta} \Rt{\eta} \Rt{\zeta}$ type products. We also got mileage out of the complete symmetry between the pairs $\Rt{\alpha}, \Rt{\beta}$ and $\Rt{\beta} \Rt{\gamma}$.

We tried extensively to find a similar derivation for the group $\UnipBSm{q}$: A proof, using only in-subgroup relations, that $\Rt{\beta+\psi}$ and $\Rt{\psi+\omega}$ commute. However, the situation there is not as nice: $\Rt{\beta+\psi}$ is not the only root in the positive  span of $\Rt{\beta}$ and $\Rt{\psi}$, $\Rt{\beta+2\psi}$ is as well! Thus, $\Rt{\beta+\psi}$ elements are not commutators of $\Rt{\beta}$ and $\Rt{\psi}$ elements, though we can express them as $\Rt{\psi} \Rt{\beta} \Rt{\psi} \Rt{\beta} \Rt{\psi}$ products (see \Cref{rel:b3-small:sub:expr:beta+psi}). Importantly, the entries on the first and last $\Rt{\psi}$ elements in this expansion are the same, so we have less asymmetry to play with.

We have some more concrete evidence that it is not the case that in-subgroup relations in $\UnipBSm{q}$ are enough to prove all relations in $\UnipBSm{q}$, and in particular the commutation relation for $\Rt{\beta+\psi}$ and $\Rt{\psi+\omega}$ roots, from our computer calculations of the homology of the corresponding coset complex $\LinkCplxBSm{q}$:
\begin{enumerate}
    \item Firstly, we observed that while $\LinkCplxBSm{5}$ and $\LinkCplxBLg{5}$ both have vanishing $1$-homology over $\BF_2$ (\Cref{thm:story:b3:base}), $\LinkCplxBSm{5}$ \emph{does not} have vanishing $1$-homology over $\BF_5$. This means that in-subgroup relations in $\UnipBSm{5}$ do not generate all relations, since the complex $\LinkCplxBSm{5}$ is not simply connected. The $1$- homology of $\LinkCplxBSm{5}$ does vanish over $\BF_3$, $\BF_7$, and $\BF_{11}$; we conjecture that the $1$-homology of $\LinkCplxBSm{q}$ does not vanish over $\BF_q$ for general $q$.
    \item We attempted to replicate the $A_3$ success depicted in \Cref{fig:squares}; that is, we used a computer to look for reasonably small, ``nicely structured'' fillings of the cycle corresponding to the commutation relation for $\Rt{\beta+\psi}$ and $\Rt{\psi+\omega}$ roots in $\LinkCplxBSm{5}$.
    Unfortunately, we concluded that there is no comparably small and symmetric filling, in comparison to the proof of \Cref{rel:a3:comm:alpha+beta:beta+gamma} in the $A_3$ case.
\end{enumerate}

%% file: figures/other/figure.tex
\begin{figure}[H]
\centering
\begin{subfigure}[b]{0.45\textwidth}
\def\dscale{4}
\def\inscale{0.3}
\usetikzlibrary{calc}
\centering
\begin{tikzpicture}

    \coordinate (TL) at (0,0);
    \coordinate (TR) at (0,\dscale);
    \coordinate (BL) at (\dscale,0);
    \coordinate (BR) at (\dscale,\dscale);
    
    \coordinate (iTL) at (\inscale*\dscale,\inscale*\dscale);
    \coordinate (iTR) at (\inscale*\dscale,{(1-\inscale)*\dscale});
    \coordinate (iBL) at ({(1-\inscale)*\dscale},\inscale*\dscale);
    \coordinate (iBR) at ({(1-\inscale)*\dscale},{(1-\inscale)*\dscale});
    
    \draw[line width=1.5pt, yellow] (TL) rectangle (BR);
    \draw[black] (iTL) rectangle (iBR);
    \draw[black] (TL) -- (iTL);
    \draw[black] (TR) -- (iTR);
    \draw[black] (BL) -- (iBL);
    \draw[black] (BR) -- (iBR);

    \node[fill=red, text=white, circle, inner sep=3pt] at (TL) {1};
    \node[fill=blue, text=white, circle, inner sep=3pt] at (TR) {2};
    \node[fill=blue, text=white, circle, inner sep=3pt] at (BL) {3};
    \node[fill=red, text=white, circle, inner sep=3pt] at (BR) {4};

    \node[fill=blue, text=white, circle, inner sep=3pt] at (iTL) {5};
    \node[fill=red, text=white, circle, inner sep=3pt] at (iTR) {6};
    \node[fill=red, text=white, circle, inner sep=3pt] at (iBL) {7};
    \node[fill=blue, text=white, circle, inner sep=3pt] at (iBR) {8};
\end{tikzpicture}
\end{subfigure}%
\hfill
\begin{subfigure}[b]{0.45\textwidth}
\centering
\def\dscale{4}
\def\inscale{0.3}
\usetikzlibrary{calc}
\begin{tikzpicture}

    \coordinate (TL) at (0,0);
    \coordinate (TR) at (0,\dscale);
    \coordinate (BL) at (\dscale,0);
    \coordinate (BR) at (\dscale,\dscale);
    
    \coordinate (iTL) at (\inscale*\dscale,\inscale*\dscale);
    \coordinate (iTR) at (\inscale*\dscale,{(1-\inscale)*\dscale});
    \coordinate (iBL) at ({(1-\inscale)*\dscale},\inscale*\dscale);
    \coordinate (iBR) at ({(1-\inscale)*\dscale},{(1-\inscale)*\dscale});

    \coordinate (cM) at (\dscale/2,\dscale/2);
    \coordinate (cT) at (\inscale*\dscale/2,\dscale/2);
    \coordinate (cL) at (\dscale/2,\inscale*\dscale/2);
    \coordinate (cB) at ({(1-\inscale/2)*\dscale},\dscale/2);
    \coordinate (cR) at (\dscale/2,{(1-\inscale/2)*\dscale});
    
    \draw[line width=1.5pt, yellow] (TL) rectangle (BR);
    \draw[black] (iTL) rectangle (iBR);
    \draw[black] (TL) -- (iTL);
    \draw[black] (TR) -- (iTR);
    \draw[black] (BL) -- (iBL);
    \draw[black] (BR) -- (iBR);

    \draw[black, dashed] (iTL) -- (cM);
    \draw[black, dashed] (iTR) -- (cM);
    \draw[black, dashed] (iBL) -- (cM);
    \draw[black, dashed] (iBR) -- (cM);

    \draw[black, dashed] (TL) -- (cL);
    \draw[black, dashed] (BL) -- (cL);
    \draw[black, dashed] (iTL) -- (cL);
    \draw[black, dashed] (iBL) -- (cL);

    \draw[black, dashed] (TR) -- (cR);
    \draw[black, dashed] (BR) -- (cR);
    \draw[black, dashed] (iTR) -- (cR);
    \draw[black, dashed] (iBR) -- (cR);
    
    \draw[black, dashed] (TL) -- (cT);
    \draw[black, dashed] (TR) -- (cT);
    \draw[black, dashed] (iTL) -- (cT);
    \draw[black, dashed] (iTR) -- (cT);
    
    \draw[black, dashed] (BL) -- (cB);
    \draw[black, dashed] (BR) -- (cB);
    \draw[black, dashed] (iBL) -- (cB);
    \draw[black, dashed] (iBR) -- (cB);

    \node[fill=red, circle, inner sep=3pt] at (TL) {};
    \node[fill=blue, circle, inner sep=3pt] at (TR) {};
    \node[fill=blue, circle, inner sep=3pt] at (BL) {};
    \node[fill=red, circle, inner sep=3pt] at (BR) {};

    \node[fill=blue, circle, inner sep=3pt] at (iTL) {};
    \node[fill=red, circle, inner sep=3pt] at (iTR) {};
    \node[fill=red, circle, inner sep=3pt] at (iBL) {};
    \node[fill=blue, circle, inner sep=3pt] at (iBR) {};

    \node[fill=green, circle, inner sep=3pt] at (cM) {};
    \node[fill=green, circle, inner sep=3pt] at (cL) {};
    \node[fill=green, circle, inner sep=3pt] at (cT) {};
    \node[fill=green, circle, inner sep=3pt] at (cR) {};
    \node[fill=green, circle, inner sep=3pt] at (cB) {};
\end{tikzpicture}
\end{subfigure}
\caption{Depiction of the $\BF_2$-triangulation of relation $\Comm{\El{\alpha+\beta}{1}}{\El{\beta+\gamma}{1}} = \Id$ in $\LinkCplxA{q}$. In the left-hand figure, the vertices are labeled $1,\ldots,8$. These correspond to the following cosets: $H_{\setminus \Rt{\alpha}}$, $H_{\setminus \Rt{\gamma}}$, $\El{\alpha+\beta}1 H_{\setminus \Rt{\alpha}}$, $\El{\beta+\gamma}1 H_{\setminus \Rt{\gamma}}$, $\El{\beta}1 \El{\gamma}2 H_{\setminus \Rt{\gamma}}$, $\El{\beta}1 \El{\alpha}1 H_{\setminus \Rt{\alpha}}$, $\El{\alpha+\beta}1 \El{\beta}2 \El{\gamma}1 H_{\setminus \Rt{\gamma}}$, and $\El{\beta+\gamma}1 \El{\beta}{1/2} \El{\gamma}2 H_{\setminus \Rt{\gamma}}$, respectively.}
\label{fig:squares}
\end{figure}
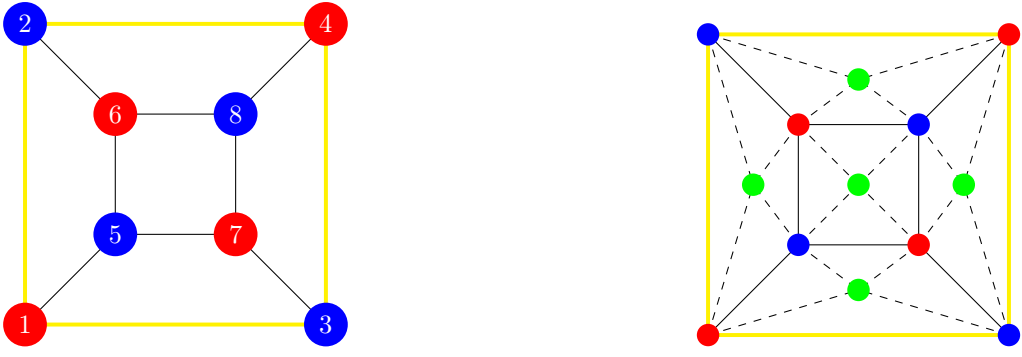